\documentclass[11pt]{amsart}
\usepackage{amscd,amsmath,amssymb,amsfonts}
\usepackage{url,enumerate,color}
\usepackage{dsfont}
\usepackage[cmtip, all]{xy}

\usepackage[top=30mm,right=30mm,bottom=30mm,left=20mm]{geometry}

\theoremstyle{plain}
\newtheorem{thm}{Theorem}
\newtheorem{lem}[thm]{Lemma}
\newtheorem{cor}[thm]{Corollary}
\newtheorem{prop}[thm]{Proposition}
\theoremstyle{definition}
\newtheorem{defn}[thm]{Definition}

\newtheorem{rmk}[thm]{Remark}

\numberwithin{thm}{section}
\numberwithin{equation}{thm}

\newcommand{\al}{\alpha}
\newcommand{\om}{\omega}

\newcommand{\surj}{\twoheadrightarrow}

\newcommand{\rank}{{\rm rank}}

\newcommand{\Spec}{{\rm Spec \,}}
\newcommand{\End}{{\mathrm {End}}}
\newcommand{\Char}{{\rm char}}
\newcommand{\Tr}{{\rm Tr}}

\newcommand{\Trace}{{\rm Trace}}
\newcommand{\Stab}{{\rm Stab}}
\newcommand{\Ker}{{\rm Ker}}

\newcommand{\Aut}{\mathrm{Aut}}
\newcommand{\Out}{\mathrm{Out}}
\newcommand{\Irr}{\mathrm{Irr}}

\newcommand{\eps}{\epsilon}

\newcommand{\Ind}{\mathrm{Ind}}
\newcommand{\diag}{\mathrm{diag}}
\newcommand{\Wild}{{\sf Wild}}
\newcommand{\Tame}{{\sf Tame}}

\newcommand{\SL}{\mathrm{SL}}
\newcommand{\PSL}{\mathrm{PSL}}
\newcommand{\GL}{\mathrm{GL}}
\newcommand{\PGL}{\mathrm{PGL}}
\newcommand{\GU}{\mathrm{GU}}
\newcommand{\SU}{\mathrm{SU}}
\newcommand{\PSU}{\mathrm{PSU}}
\newcommand{\PGU}{\mathrm{PGU}}
\newcommand{\Sp}{\mathrm{Sp}}
\newcommand{\PSp}{\mathrm{PSp}}
\newcommand{\CSp}{\mathrm{CSp}}
\newcommand{\PCSp}{\mathrm{PCSp}}

\newcommand{\GO}{\mathrm{O}}

\newcommand{\sA}{{\mathcal A}}
\newcommand{\sB}{{\mathcal B}}

\newcommand{\sD}{{\mathcal D}}

\newcommand{\sF}{{\mathcal F}}
\newcommand{\sG}{{\mathcal G}}
\newcommand{\sH}{{\mathcal H}}

\newcommand{\sK}{{\mathcal K}}
\newcommand{\sL}{{\mathcal L}}

\newcommand{\sO}{{\mathcal O}}

\newcommand{\sW}{{\mathcal W}}
\newcommand{\sX}{{\mathcal X}}
\newcommand{\sY}{{\mathcal Y}}

\newcommand{\A}{{\mathbb A}}

\newcommand{\C}{{\mathbb C}}
\newcommand{\D}{{\mathbb D}}{

\newcommand{\F}{{\mathbb F}}
\newcommand{\G}{{\mathbb G}}

\newcommand{\K}{{\mathbb K}}

\renewcommand{\P}{{\mathbb P}}
\newcommand{\Q}{{\mathbb Q}}

\newcommand{\W}{{\mathbb W}}

\newcommand{\Z}{{\mathbb Z}}

\newcommand{\ppd}{{\rm ppd}}
\newcommand{\ZB}{{\mathbf Z}}
\newcommand{\CB}{{\mathbf C}}
\newcommand{\NB}{{\mathbf N}}
\newcommand{\OB}{{\mathbf O}}

\newcommand{\ABS}{\mathsf{A}}

\newcommand{\Gauss}{\mathsf{Gauss}}

\newcommand{\Swan}{\mathsf {Swan}}
\newcommand{\Sym}{\mathsf {S}}
\newcommand{\dl}{{\mathfrak {d}}}
\newcommand{\Cstar}{(\star)}
\newcommand{\meo}{{\mathrm{meo}}}
\newcommand{\obar}{\bar{\mathsf{o}}}
\newcommand{\ord}{\mathsf{o}}
\newcommand{\cyc}{{\mathrm {cyc}}}
\newcommand{\ssp}{\mathsf{ss}}
\newcommand{\geo}{\mathrm{geom}}
\newcommand{\ari}{\mathrm{arith}}
\newcommand{\gbar}{\bar{g}}
\newcommand{\td}{\tilde\delta}
\newcommand{\tx}{\tilde\xi}

\newcommand{\lcm}{\mathrm {lcm}}
\newcommand{\bj}{\boldsymbol{j}}
\newcommand{\triv}{{\mathds{1}}}
\newcommand{\edit}[1]{{\color{black} #1}}
\newcommand{\Ch}{\mathsf{Char}}
\newcommand{\CSP}{\mathrm {({\bf S+})}}
\newcommand{\tw}[1]{{}^#1\!}

\begin{document}

\title{Monodromy groups of certain Kloosterman and hypergeometric sheaves}
\author{Nicholas M. Katz and Pham Huu Tiep}
\address{Department of Mathematics, Princeton University, Princeton, NJ 08544}
\email{nmk@math.princeton.edu}
\address{Department of Mathematics, Rutgers University, Piscataway, NJ 08854}
\email{tiep@math.rutgers.edu}
\thanks{The second author gratefully acknowledges the support of the NSF (grants 
DMS-1839351 and DMS-1840702), and the Joshua Barlaz Chair in Mathematics.}
\thanks{Part of this work was done while the second author visited the Isaac Newton Institute for Mathematical
Sciences, Cambridge, UK. It is a pleasure to thank the Institute for hospitality and support.}
\thanks{The authors are grateful to Richard Lyons and Will Sawin for helpful discussions.}

\maketitle

\begin{abstract}
A certain ``condition ({\bf S})" on reductive algebraic groups was introduced in \cite{G-T}, in which a slightly stronger condition ({\bf S+}) 
was shown to have very strong consequences. We show that a wide class of Kloosterman and hypergeometric sheaves satisfy ({\bf S+}). For this class of sheaves, we determine possible structure of their monodromy groups.
\end{abstract}

\tableofcontents

\section*{Introduction}
Given a prime $p$, it was conjectured by Abhyankar \cite{Ab} and proven by Raynaud  \cite{Ray} (see also \cite{Pop}) that any finite group $G$ which is generated by its Sylow $p$-subgroups occurs as a quotient of the fundamental group of the affine line
$\A^1/\overline{\F_p}$. The analogous result for the multiplicative group $\G_m := \A^1 \setminus \{0\}$, also conjectured by Abhyankar and proven by Harbater \cite{Har} is that any finite group $G$ which, modulo the subgroup $\OB^{p'}(G)$ generated by its Sylow 
$p$-subgroups, is cyclic, occurs as a quotient of the fundamental group of $\G_m/\overline{\F_p}$. In the ideal world,
given such a finite group $G$, and a complex representation $V$ of $G$, we would be able, for any prime $\ell \neq p$, to choose an embedding of $\C$ into $\overline{\Q_\ell}$, and  to write down an explicit $\overline{\Q_\ell}$-local system on either 
$\A^1/\overline{\F_p}$ or on $\G_m/\overline{\F_p}$ whose geometric monodromy group is $G$, in the given representation.

In some earlier papers, we have been able to do this for some particular pairs $(G,V)$. When we were able to do this on $\A^1$, it was through one-parameter families of ``simple to remember" exponential sums, often but not always rigid local systems on $\A^1$.
When we have been able to do this on $\G_m$, it was through explicit irreducible hypergeometric sheaves.

Here we reverse this point of view, and investigate what possible $(G,V)$ can hypergeometric sheaves give rise to? The first
part of the paper is devoted to showing that for a wide class of hypergeometric sheaves $\sH$, their geometric monodromy groups $G_{\geo}$ (which need not be finite) in their given representations satisfy a certain condition $\CSP$ (which is a slightly strengthening of 
condition $({\bf S})$ introduced in \cite{G-T}, and roughly speaking, corresponds to 
Aschbacher's class ${\mathcal S}$ of maximal subgroups of classical groups \cite{Asch}), 
see Theorems \ref{Kl-S}, \ref{Hyp-notp-S}, \ref{Hyp-notpbis-S}, and \ref{Hyp-p-S}. When this condition holds,
it imposes strong restrictions on the pair $(G_{\geo},\sH)$. If $G$ is infinite, then the identity component $G_{\geo}^\circ$ of $G_\geo$ 
is a simple algebraic group, still acting irreducibly. If $G$ is finite, then either $G$ is almost quasisimple (that is, $S \lhd G/\ZB(G) \leq \Aut(S)$ for some non-abelian simple group $S$), or 
$G$ is an ``extraspecial normalizer'', in particular, the dimension of the representation is a prime power $r^n$ and there is an extraspecial $r$-group $E$ in $G$ of order $r^{1+2n}$ acting irreducibly.

In this paper, we consider only geometrically irreducible hypergeometric sheaves, i.e., those on which $G_{\geo}$ acts irreducibly. One also knows that if $G_{\geo}$ is finite, then a generator of local monodromy at $0$ is an element of $G$ which has all distinct eigenvalues in the given representation (a ``simple spectrum" element). And by Abhyankar, if $G_{\geo}$ is finite, then $G/\OB^{p'}(G)$ is cyclic.

Let us say that a triple $(G,V,g)$ satisfies {\it the Abhyankar condition at $p$} if 
$G$ is a finite group such that $G/\OB^{p'}(G)$ is cyclic, $V$ a faithful, irreducible, finite-dimensional complex representation of $G$, and 
$g \in G$ an element of order coprime to $p$ that has simple spectrum on $V$.   So a natural question is which triples 
$(G,V,g)$, 
with $G$ a finite group, almost quasisimple or an extraspecial normalizer, that satisfy the Abhyankar condition at $p$, 
occur ``hypergeometrically'', that is, as 
$(G_{\geo},\sH,g)$ 
for a hypergeometric sheaf $\sH$ and a generator $g \in G_{\geo}$ of local monodromy around $0$ on $\G_m/\overline{\F_p}$ 
(and $V$ realizes the action of $G=G_{\geo}$ on $\sH$).

Grosso modo, our main results essentially classify all such triples $(G,V,g)$ that can arise from hypergeometric sheaves, and 
also determine the structure of geometric monodromy groups of hypergeometric sheaves that satisfy the condition $\CSP$.

More precisely, 
in Theorems \ref{alt}, \ref{spor}, and \ref{simple} we classify all pairs $(G,V)$, where $G$ is a finite almost quasisimple group and $V$ a faithful, irreducible, finite-dimensional complex representation of $G$ such that 
some element $g$ has simple spectrum on $V$. Next, in Theorem \ref{char-sheaf1} we show that if such a group $G$ occurs as $G_{\geo}$
for a hypergeometric sheaf in characteristic $p$ and in addition $G$ is a finite group of Lie type in characteristic $r$, then $p=r$ unless 
$\dim(V)$ is small. Theorem \ref{char-sheaf2} gives an analogous result in the case $G$ is an extraspecial normalizer. With these results in hand, we complete the classification of triples $(G,V,g)$ that satisfy the Abhyankar condition at $p$, with $G$ being almost quasisimple or an extraspecial normalizer, in 
\S\ref{sec:simple-Lie}. Further constraints for a finite group $G$ to occur as $G_{\geo}$ of a hypergeometric sheaf are established in 
\S\S\ref{sec:bounds}, \ref{sec:hypergeom1}, \ref{sec:hypergeom2}. With an explicit, finite, list of exceptions, all the almost quasisimple 
triples $(G,V,g)$, that satisfy the Abhyankar condition at $p$ and in addition these extra constraints, are then shown (modulo a central subgroup) to occur hypergeometrically;
the respective hypergeometric sheaves $\sH$ are 
explicitly constructed in a series of companion papers 
\cite{Ka-RL-2J2}, \cite{Ka-RL-T-Co3}--\cite{Ka-RL-T-Spor}, \cite{KT1}--\cite{KT3}, \cite{KT5}--\cite{KT8}. 

The hypergeometric sheaves satisfying $\CSP$, but with infinite geometric monodromy groups, will be studied in a sequel to this paper.

\section{The basic  ({\bf S+}) setting}
\subsection*{1A. Conditions ({\bf S}) and ({\bf S+})}
We work over an algebraically closed field $\C$ of characteristic zero, which we will take to be $\overline{\Q_\ell}$ for some prime $\ell$ in the rest of this paper. Given a nonzero finite-dimensional  $\C$-vector space $V$ and a Zariski closed subgroup 
$G \leq \GL(V)$, recall from \cite[2.1]{G-T} that $G$ (or more precisely the pair $(G,V)$) is said {\it to satisfy condition {\rm ({\bf S})}} if each of the following four conditions is satisfied.
\begin{itemize}
\item[(i)]  The $G$-module $V$ is irreducible.
\item[(ii)]  The $G$-module $V$ is primitive.
\item[(iii)]  The $G$-module $V$ is tensor indecomposable.
\item[(iv)] The $G$-module $V$ is not tensor induced.
\end{itemize}

\begin{lem}\label{basic}
Suppose $1 \neq G \leq \GL(V)$ is a Zariski closed, irreducible subgroup. Then the following statements holds.
\begin{enumerate}[\rm(i)]
\item If $G$ satisfies $({\bf S})$, $\dim(V) > 1$, and $\ZB(G)$ is finite, then we have three possibilities:
\begin{itemize}
\item[(a)] The identity component $G^\circ$ is a simple algebraic group, and $V|_{G^\circ}$ is irreducible.
\item[(b)] $G$ is finite, and {\it almost quasisimple}, i.e. there is a finite non-abelian simple group $S$ such that $S \lhd G/\ZB(G) <\Aut(S)$.
\item[(c)]  $G$ is finite and it is an ``{\it extraspecial normalizer}'' (in characteristic $r$), that is, $\dim(V)=r^n$ for a prime $r$, and $G$ contains a normal $r$-subgroup $R=\ZB(R)E$, where $E$ is an extraspecial $r$-group $E$ of order $r^{1+2n}$ acting irreducibly on $V$, 
and either $R=E$ or $\ZB(R) \cong C_4$.
\end{itemize}
\item $\ZB(G)$ is finite if and only if $\det(G)$ is finite.
\end{enumerate}
\end{lem}

\begin{proof}
(i) The proof of  \cite[Prop. 2.8]{G-T} (taking $H=G$) shows that one of (a)--(c) holds. 

(ii) By Schur's lemma, $\ZB(G)$ consists of scalar matrices, hence the finiteness of $\det(G)$ implies $|\ZB(G)| < \infty$. 
Suppose now that $|\ZB(G)| < \infty$.  
Note that the unipotent radical of $G^\circ$ has nonzero fixed points on $V$ \cite[17.5]{Hum},
hence the irreducibility of $V \neq 0$ implies that $G^\circ$ is reductive, and so $G^\circ = T[G^\circ,G^\circ]$ with 
$T:=\ZB(G^\circ)^\circ$ and $[G^\circ,G^\circ] \leq \SL(V)$. As $G/G^\circ$ is finite, it suffices to show that $T \leq \SL(V)$.

We may assume the torus $T$ has dimension $d \geq 1$, and let 
$\lambda_1, \ldots,\lambda_n$ denote the distinct weights of $T$ acting on $V$. 
The irreducibility of $V$ over $G \rhd T$ implies that
$G/G^\circ$ acts transitively on $\{\lambda_1, \ldots,\lambda_m\}$; in particular, all these weights occur on $V$ with the same 
multiplicity $e \geq 1$. Let $A < \GL_d(\Z)$ denote the finite subgroup induced by the action of $G/G^\circ$ on the character 
group $X(T) \cong \Z^d$, and let $W:= X(T) \otimes_\Z \Q$. As $A$ is finite, we can find an $A$-invariant Euclidean scalar product
$(\cdot,\cdot)$ on $W$. Note that
\begin{equation}\label{eq:s10} 
   W = [W,A] \oplus W^A,
\end{equation}   
where $[W,A] := \langle a(v)-v \mid a \in A,v \in W \rangle_\Q$ and $W^A := \{v \in W \mid a(v)=v, ~\forall a \in A\}$. [Indeed,
for any $a \in A$, $v \in W$, and $w \in W^A$ we have 
$$(a(v)-v,w) = (a(v),w)-(v,w)=(a(v),a(w))-(v,w)=0,$$ 
showing $[W,A] \perp W^A$. Also we have 
$$|A| \cdot v = \sum_{a \in A}(v-a(v)) + \sum_{a \in A}a(v),$$ 
ensuring $W = [W,A]+W^A$.] 

Choose a basis $\alpha_1, \ldots,\alpha_l \in X(T)$ of $[W,A]$ (over $\Q$). 
Consider any $g \in G$ and the element $a \in A$ induced by the conjugation action of $g$ on $T$. Since $X(T)$ has finite
rank $d$, we can find an integer $N_a > 0$ such that $N_a(a(\beta)-\beta) \in \langle \alpha_1, \ldots, \alpha_l\rangle_\Z$ for all 
$\beta \in X(T)$. As $A$ is finite, taking $N:=\lcm(N_a \mid a \in A)$, we have that 
\begin{equation}\label{eq:s11}
   N(a(\beta)-\beta) \in \langle \alpha_1, \ldots, \alpha_l\rangle_\Z,\mbox{ for all } a \in A \mbox{ and }\beta \in X(T).
\end{equation}   

Now, if $l\leq d-1$, then  $T_1:=\bigl( \bigcap^l_{j=1}\Ker(\alpha_j)\bigr)^\circ$ has dimension $\geq 1$. On the other hand, by \eqref{eq:s11}, 
for any $t \in T_1$ and any $\beta \in X(T)$, $g \in G$, we have 
$$\beta(gt^Ng^{-1}t^{-N}) = \beta((gtg^{-1})^N)/\beta(t^N)=\bigl(a(\beta)(t)\bigr)^N/\beta(t)^N=\bigl(N(a(\beta)-\beta))\bigr)(t)=1$$
if $g$ induces $a \in A$. Thus $gt^Ng^{-1}=t^N$ for all $g \in G$, and so $t^N \in \ZB(G)$ for all $t \in T_1$, a contradiction since
$|\ZB(G)| < \infty$ and $\dim T_1 \geq 1$. It follows that $l=d$, and so $W^A = 0$ by \eqref{eq:s10}. 

Recall that $\lambda_1, \ldots,\lambda_m$ is an $A$-orbit in $W$. Hence $\sum^m_{i=1}\lambda_i \in W^A$, and so 
$\sum^m_{i=1}\lambda_i=0$. Finally, for any $t \in T$, note that 
$$\det(t|_V) = \bigl(\prod^m_{i=1}\lambda_i(t)\bigr)^e = \bigl(e\sum^m_{i=1}\lambda_i\bigr)(t)=1,$$
i.e. $T \leq \SL(V)$, as stated.
\end{proof}

\begin{defn}
A pair $(G,V)$ is said {\it to satisfy the condition $\CSP$}, if it satisfies $({\bf S})$ and, in addition, $|\ZB(G)|$ is finite (equivalently,
$\det(G)$ is finite).
\end{defn}

The following lemma is immediate from the definitions.
\begin{lem}Given a Zariski closed subgroup $G \subset \GL(V)$ and a Zariski closed subgroup $H \leq G$, suppose that $(H,V|_H)$ satisfies {\rm ({\bf S})}. Then $(G,V)$ satisfies {\rm ({\bf S})}. If in addition $\ZB(G)$ is finite,
then $(G,V)$ satisfies {\rm ({\bf S+})}.
\end{lem}
 Let us also recall the following lemma from \cite[Lemma 2.5]{G-T}.
 \begin{lem}Given a Zariski closed subgroup $G \subset \GL(V)$ and a Zariski closed normal subgroup $H \lhd G$, suppose that $(G,V)$ satisfies the first three conditions defining {\rm ({\bf S})}, i.e., suppose that $G$ is irreducible, primitive, and tensor indecomposable. Then either $H \leq \ZB(G)$ or $V|_H$ is irreducible.
 \end{lem}

\begin{defn}\label{changeofgroup}
More generally, if $\Gamma$ is any group given with a finite-dimensional representation $\Phi: \Gamma \to \GL(V)$, then
we say $(\Gamma,V)$ {\it satisfies $\CSP$}, if $(\Phi(\Gamma),V)$ satisfies the three conditions of $({\bf S})$ and, in addition,
$\det(\Phi(\Gamma))$ is finite.
\end{defn}

\begin{lem}\label{reducetopi_1}Let $\Gamma$ be a group, $\C$ an algebraically closed field of characteristic zero, 
$n \in \Z_{\ge 1}$, $\Phi:\Gamma \rightarrow \GL_n(\C)=\GL(V)$ a representation of $\Gamma$, and 
$G \leq \GL(V)$ the Zariski closure of $\Phi(\Gamma)$. Then $(\Gamma,V)$ satisfies {\rm ({\bf S+})} if and only if 
$(G,V)$ satisfies {\rm ({\bf S+})}. This equivalence holds separately for each of the four conditions defining  {\rm ({\bf S+})}.
\end{lem}

\begin{proof}
If $V$ is $G$-reducible, it is a fortiori $\Gamma$-reducible. Conversely, if $\Phi(\Gamma)$ stabilizes a proper subspace 
$U \neq 0$ of $V$, then, since the stabilizer of $U$ in $\GL(V)$ is closed, $G$ also stabilizes $U$ and so is reducible on
$V$.  If $V$ is $G$-imprimitive, any system of imprimitivity for $G$ remains one for $\Gamma$. 
Conversely, if $\Phi(\Gamma)$ stabilizes an imprimitive decomposition $V = \oplus^m_{i=1}V_i$ of $V$,
then, since the stabilizer of this decomposition in $\GL(V)$ is closed, $G$ also stabilizes the decomposition and so is 
imprimitive on $V$. If $V$ is tensor decomposable as a $G$-module, then a fortiori it is tensor decomposable for $\Gamma$. 
Conversely, if $\Phi(\Gamma)$ stabilizes a tensor decomposition $V = A \otimes B$ with $\dim A,\dim B > 1$, we use the fact that
the image of the``Kronecker product" map $\GL(A) \times \GL(B) \rightarrow \GL(A\otimes B)$, namely the stabilizer $\GL(A) \otimes \GL(B)$, is closed, cf. \cite[7.4, Prop. B]{Hum}. Therefore $G$ also stabilizes the decomposition and so is tensor
decomposable on $V$. The same argument shows that  $V$ is tensor induced for $G$ if and only if it is tensor induced for 
$\Gamma$. Indeed if $V$ is $V_1^{\otimes n}$ with $\dim(V_1) > 1$ and $n > 1$, use the fact that the image in 
$\GL(V_1^{\otimes n})$ of the wreath product $\GL(V_1)\wr {\sf S}_n$ is closed to see that  $\Phi(\Gamma)$ lands in this image if 
and only if $G$ does.
If $\det(G)$ fails to be finite, then $\det(\Phi(\Gamma))$ is infinite, by the Zariski density of $\Phi(\Gamma)$ in $G$. 
If $\det(G)$ is finite, then a fortiori $\det(\Phi(\Gamma))$ is finite.
\end{proof}

\subsection*{1B. Statements of theorems of type  ({\bf S+}) for Kloosterman and hypergeometric sheaves}

We work in characteristic $p$, and use $\overline{\Q_\ell}$-coefficients for a chosen prime $\ell \neq p$. We fix a nontrivial additive character $\psi$ of $\F_p$, with values in $\mu_p(\overline{\Q_\ell})$. We will consider Kloosterman and hypergeometric sheaves on
 $\G_m/\overline{\F_p}$ as representations of $\pi_1:=\pi_1(\G_m/\overline{\F_p})$, and prove that, under various hypotheses, they satisfy  {\rm ({\bf S+})} as representations of $\pi_1$. As noted in Lemma \ref{reducetopi_1}, 
this is equivalent to their satisfying  {\rm ({\bf S+})} as representations of their geometric monodromy groups.
 
On $\G_m/\overline{\F_p}$, we consider a Kloosterman sheaf $$\sK l:=\sK l_\psi(\chi_1,\ldots ,\chi_D)$$ of rank $D \ge 2$, defined by an unordered list of $D$ not necessarily distinct multiplicative characters of some finite subfield $\F_q$ of $\overline{\F_p}$.

One knows that $\sK l$ is absolutely irreducible, cf. \cite[4.1.2]{Ka-GKM}. One also knows, by a result of Pink \cite[Lemmas 11 and 12]{Ka-MG} that $\sK l$ is primitive so long as it is not 
Kummer induced. Recall that $\sK l$ is Kummer induced if and only if there exists a nontrivial multiplicative character $\rho$ such that the unordered list of the $\chi_i$ is equal to the unordered list of the $\rho \chi_i$. Thus primitivity (or imprimitivity) of $\sK l$ is immediately visible.

\begin{thm}\label{Kl-S} Let $\sK l$ be a Kloosterman sheaf of rank $D \ge 2$ in characteristic $p$ which is primitive. Suppose that $D$ is not $4$. If $p=2$, suppose also that $D \neq 8$. Then $\sK l$ satisfies {\rm ({\bf S+})}.
\end{thm}

\begin{rmk}We exclude $D=4$ because in any odd characteristic $p$, there are Kloosterman sheaves of rank $D=4$ which are $2$-tensor induced, cf. \cite[Theorem 6.3]{Ka-CC}. 
\end{rmk}

We next consider a hypergeometric sheaf $\sH$ of type $(D,m)$ with $D > m \ge 0$, thus
$$\sH=\sH yp_\psi(\chi_1,\ldots ,\chi_D;\rho_1,\ldots ,\rho_m).$$
Here the $\chi_i$ and $\rho_j$ are (possibly
trivial) multiplicative characters of some finite subfield $\F_q^\times$, with the proviso that no $\chi_i$ is any $\rho_j$. [The case $m=0$ is precisely the $\sK l$ case.] One knows \cite[8.4.2, (1)]{Ka-ESDE} that such an $\sH$ is lisse on $\G_m$, geometrically irreducible. Its local monodromy at $0$ is tame, a successive extension of the $\chi_i$. It is of finite order if and only if the $\chi_i$ are pairwise distinct, in which case that local monodromy is their direct sum $\oplus_i \chi_i$, cf. \cite[8.4.2, (5)]{Ka-ESDE}.
 Its local monodromy at $\infty$ is the direct sum of a tame part which is a successive extension of the $\rho_j$, with a totally wild representation $\Wild_{D-m}$ of rank $D-m$ and Swan conductor one, i.e. it has all $\infty$-breaks $1/(D-m)$. It
 is of finite order if and only the $\rho_j$ are pairwise distinct, in which case that local monodromy is the direct sum of
$\oplus_j \rho_j$ with $\Wild_{D-m}$. We denote by $W:=D-m$ the dimension of the wild part $\Wild$.

In the case of a hypergeometric sheaf $\sH$ with $m > 0$, primitivity is less easy to determine at first glance, because there is
also the possibility of Belyi induction, cf. \cite[Proposition 1.2]{Ka-RL-T-2Co1}. It is known that an $\sH$ of type $(D,1)$ is primitive unless $D$ is a power of $p$, cf. \cite[Cor 1.3]{Ka-RL-T-2Co1}. It is also known \cite[Proposition 1.4]{Ka-RL-T-2Co1} that  an $\sH$ of type $(D,m)$, with $D > m \ge 2$ and $D$ a power of $p$, is primitive.

\begin{thm}\label{Hyp-notp-S}Let $\sH$ be a hypergeometric sheaf of type $(D,m)$ with $D > m >0$, with $D \ge 4$. Suppose that $\sH$ is primitive, $p \nmid D$, and $W > D/2$. If $p$ is odd and $D=8$, suppose $W > 6$. If $ p \neq 3$, suppose that either $D \neq 9$, or that both $D=9$ and $W > 6$. Then $\sH$ satisfies   {\rm ({\bf S+})}.
\end{thm}

\begin{rmk}In the case $D=4$, the condition $W> D/2$ is sharp. In any odd characteristic $p$, there are hypergeometric sheaves of type $(4,2)$ which are  $2$-tensor induced, cf. \cite[Theorem 6.5]{Ka-CC}. There are also hypergeometric sheaves of type $(4,2)$ which are tensor decomposable, cf. \cite[Theorem 5.3]{Ka-CC}. 
\end{rmk}

Here is a slight variant, which visibly implies the above Theorem \ref{Hyp-notp-S}.

\begin{thm}\label{Hyp-notpbis-S}Let $\sH$ be a hypergeometric of type $(D,m)$ with $D > m >0$, with $D \ge 4$. Suppose that $\sH$ is primitive. Suppose that $D >4$ is prime to $p$. Denote by $p_0$ the least prime divisor of $D$. Suppose that either
\begin{itemize}
\item[(i)]$D=p_0$, or
\item[(ii)]$D=p_0^2$ and $W > 2p_0$, or
\item[(iii)]$D$ is neither $p_0$ nor $p_0^2$, and $W > D/p_0$, or
\item[(iv)]$D=4$ and $W=3$. 
\item[(v)]$D=8$ and $W > 6$. 
\end{itemize}
Then $\sH$ satisfies  {\rm ({\bf S+})}.
\end{thm}

In the case when $p$ divides $D$, we need stronger hypotheses to show that  ({\bf S+}) holds.

\begin{thm}\label{Hyp-p-S}Let $\sH$ be a hypergeometric of type $(D,m)$ with $D > m >0$, with $D > 4$. Suppose that $\sH$ is primitive. Suppose that $p|D$, and $W > (2/3)(D-1)$.  If $p=2$, suppose $D \neq 8$. If $p=3$, suppose $(D,m)$ is not $(9,1)$. Then $\sH$ satisfies  {\rm ({\bf S+})}.
\end{thm}

\section{Tensor indecomposability}
In this section, we will prove the tensor indecomposability for the Kloosterman and hypergeometric sheaves of Theorems \ref{Kl-S}, \ref{Hyp-notp-S}, \ref{Hyp-notpbis-S}, \ref{Hyp-p-S}. We being with a general statement on ``linearization''.

Let $k$ be an algebraically closed field of characteristic $p > 0$, and $U/k$ an affine curve which smooth and connected, $X/k$ the complete nonsingular model of $U/k$, and $\infty $ a $k$-point of $X \setminus U$. Denote by $\pi_1(U) $ the fundamental group of $U$ (with respect to some geometric point as base point), and denote by $I(\infty) \subset \pi_1(U)$ a choice of inertia group at $\infty$.
Fix a choice of a prime $\ell$. 

\begin{prop}\label{linearize}Suppose we are given a finite dimensional  $\overline{\Q_\ell}$-vector space $V$ with on which $\pi_1$ acts continuously, by a representation $\rho$. Suppose further that we are given an expression of the vector space $V$ as a tensor product
$V=A_1 \otimes A_2 \otimes \cdots \otimes A_n$
of $n \ge 2$ vector spaces $A_i$, each of dimension $\ge 2$, such that the image of $\rho(\pi_1(U))$ lands in the subgroup
$$\GL(A_1) \otimes \GL(A_2) \otimes \cdots \otimes\GL(A_n) < \GL(V).$$
{\rm [This is the subgroup of those automorphisms of $V$ which have (non-unique !) expressions as $n$-fold tensor products of
automorphisms of the $A_i$.]} Then we have the following results.
\begin{itemize}
\item[(i)] There exists a lifting of $\rho$ to a homomorphism
$$\tilde{\rho}:\pi_1(U) \rightarrow \GL(A_1) \times \GL(A_2) \times \cdots \times\GL(A_n).$$
\item[(ii)]Suppose that for $i=1$ to $n-1$, $\dim(A_i)$ is prime to $p$. Suppose that in the representation $\rho$, all the $\infty$-slopes are $\le r$ for some real number $r \ge 0$, i.e., for each real $x >r$, the upper numbering subgroup $I(\infty)^{(x)}$  acts trivially on $V$.
Then $\tilde{\rho}$ can be chosen so that each $A_i$ (viewed as a representation of $\pi_1(U)$ by applying $\tilde{\rho}$ and then projecting onto the $A_i$ factor) has all its $\infty$-slopes are $\le r$.
\end{itemize}
\end{prop}
\begin{proof} To prove the first assertion, we argue as follows. In an expression of an element of $\otimes_{i=1}^n \GL(A_i)$ as  $\otimes_{i=1}^n \alpha_i$, we are free to multiply each $\alpha_i$ by an invertible scalar $\lambda_i$, so long as $\prod_i \lambda_1=1$. Doing this, we can move the first $n-1$ of the $\alpha_i$ into $\SL(A_i)$. In other words, we have an equality of groups
$$(\otimes_{i=1}^{n-1}\SL(A_i))\otimes \GL(A_n) = \otimes_{i=1}^n \GL(A_i)$$
inside $\GL(V)$.
So we have a short exact sequence 
$$1 \rightarrow \prod_{i=1}^{n-1}\mu_{\dim(A_i)} \rightarrow  (\prod_{i=1}^{n-1}\SL(A_i))\times \GL(A_n) \rightarrow \otimes_{i=1}^n \GL(A_i) \rightarrow 1,$$
the first map sending $(\zeta_1,\cdots, \zeta_{n-1})$ to $(\zeta_1,\cdots, \zeta_{n-1}, 1/\prod_{i=1}^{n-1}\zeta_i)$.
Now use the fact that $\pi_1(U)$ has cohomological dimension $\le 1$, to lift $\rho$.

If the first $n-1$ factors $A_i$ have dimensions prime to $p$, then the group $ \prod_{i=1}^{n-1}\mu_{\dim(A_i)} $ has order prime to $p$.
If a given $I(\infty)^{(x)}$ with $x > r$ dies under $\rho$, then its image 
under $\tilde{\rho}$ lands in $ \prod_{i=1}^{n-1}\mu_{\dim(A_i)} $.
But $I(\infty)^{(x)}$ with $x >r$  is a pro-$p$ group, so must die in the prime to $p$ group $ \prod_{i=1}^{n-1}\mu_{\dim(A_i)} $. Thus 
$I(\infty)^{(x)}$ with $x >r$ dies under $\tilde{\rho}$. In other words, each $A_i$ has all its $I(\infty)$-slopes $\le r$.
\end{proof}

\begin{lem}\label{Kl-indec}Let $\sK l$ be a Kloosterman sheaf of rank $D \ge2$  in characteristic $p$. Then $\sK l$ is tensor indecomposable. 
\end{lem}
\begin{proof}If $D$ is a prime number, there is nothing to prove. If $D$ is not prime, suppose that $D=AB$ with $A,B$ both $\ge 2$. Suppose that the image of $\pi_1:=\pi_1(\G_m/\overline{\F_p})$ lies in $\GL(A)\otimes \GL(B)$. In view of Proposition \ref{linearize}, there exist local systems $\sA$ and $\sB$ on $\G_m/\overline{\F_p}$, of ranks $A$ and $B$ respectively, such that we have an isomorphism
$\sK l \cong \sA \otimes \sB$
as representations of $\pi_1$. We argue by contradiction.

Consider first the ``easy" case, in which $p^2$ does not divide $D$. Then $p$ does not divide at least one of $A$ or $B$.  The largest $\infty$-slope of $\sK l$ is $1/D$. In view of part (ii) of Proposition \ref{linearize}, we may choose the local systems $\sA$ and $\sB$ so that each of them has largest $\infty$-slope $\le 1/D$. Then their Swan conductors at $\infty$ satisfy
$$\Swan_{\infty}(\sA)\le A/D <1, \Swan_{\infty}(\sB)\le B/D <1.$$
But Swan conductors are nonnegative integers, so we have $\Swan_{\infty}(\sA) =\Swan_{\infty}(\sB)=0$, i.e., both $\sA$ and $\sB$ are tame at $\infty$. But then $\sK l \cong \sA \otimes \sB$ is tame at $\infty$, contradiction.

Suppose now that $\sK l \cong \sA \otimes \sB$, but both $A$ and $B$ are divisible by $p$. In this case, we use the argument of \v{S}uch, cf. \cite[Prop. 12.1, second paragraph]{Such}. We have
$$\begin{aligned}\End(\sK l) & \cong \End(\sA \otimes \sB) = \End(\sA)\otimes \End(\sB)= (\triv \oplus \End^0(\sA))\otimes  (\triv \oplus \End^0(\sB))\\
 & = \triv \oplus \End^0(\sA)) \oplus \End^0(\sB)) \oplus \End^0(\sA))\otimes \End^0(\sB)).\end{aligned}$$
In particular, each of $\End^0(\sA)), \End^0(\sB))$ is a direct factor of $\End(\sK l)$. To fix ideas, assume $A \le B$. Then $A^2 \le D$, and hence  $\End^0(\sA))$ has rank $\le D-1$. The largest $\infty$-slope of $\sK l$ is $1/D$, as is the largest slope of its dual (itself another Kloosterman sheaf of the same rank $D$). There $\End(\sK l)$ has all $\infty$-slopes $\le 1/D$. Therefore $\End^0(\sA))$ has $\Swan_{\infty} \le (D-1)/D < 1$. Just as above, this forces $\End^0(\sA))$ to be tame at $\infty$. Hence also $\End(\sA))$ (being the sum of $\End^0(\sA))$ and $\triv$) is tame at $\infty$. Thus the wild inertia group $P(\infty)$ acts trivially on $\End(\sA))$, and hence acts by a scalar character on $\sA$. Observe that $\sA$ is $I(\infty)$-irreducible, simply because $ \sA \otimes \sB$ is $I(\infty)$-irreducible. Recalling that $p|A$, write $A$ as $n_0 q$ with $n_0$ prime to $p$ and with $q$ a positive power of $p$. From \cite[1.14]{Ka-GKM}, we know that the restriction of $\sA$ to $P(\infty)$ is the sum of $n_0$ pairwise distinct irreducible representations of $P(\infty)$, each of dimension $q$. This contradicts having $P(\infty)$ act on $\sA$ by a scalar character.
\end{proof}

\begin{lem}\label{Hyp-indec}Let $\sH$ be a hypergeometric sheaf of type $(D,m)$ with $D> m > 0$ in characteristic $p$. Then $\sH$ is tensor indecomposable under each of the following hypotheses. 
\begin{enumerate}[\rm(i)]
\item $D \neq 4$.
\item $D=4$, $p$ odd, and $(D,m) \neq (4,2)$.
\item $D=4$, $p=2$, and $(D,m) \neq (4,1)$.
\end{enumerate}

\end{lem}
\begin{proof}This is proven in \cite[Cor. 10.3]{Ka-RL-T-2Co1}, in the stronger form that under the stated hypotheses, the $I(\infty)$ representation of $\sH$ is tensor indecomposable.
\end{proof}

\section{Tensor induced sheaves}
\subsection*{3A. Dealing with tensor induction: First steps}
Given $(G,V)$ as in the first section, and an integer $n \ge 2$, we say that $(G,V)$ is {\it $n$-tensor induced} if $D:=\dim(V)$ is an $n^{\mathrm {th}}$ power
$D=D_0^n$ with $D_0 \ge 2$ and there exists a tensor factorization of $V$ as
$V =A_1 \otimes A_2 \otimes \cdots \otimes A_n$
with each $\dim(A_i)=D_0$, such that
$G \leq (\otimes_{i=1}^n \GL(A_i)) \rtimes \Sym_n$,
with the symmetric group $\Sym_n$ acting by permuting the factors.

One says that $(G,V)$ is {\it not tensor induced} if it is not  $n$-tensor induced for any $n \ge 2$.

We have the following obvious but useful lemma.
\begin{lem}\label{notpower}Given $(G,V)$ whose dimension $D:=\dim(V) \ge 2$ not a power (i.e., not an $n^{\mathrm {th}}$ power for any $n \ge 2$),
then $(G,V)$ is not tensor induced.
\end{lem}

To deal with the case when $D$ is a power, we begin with the following lemma.
\begin{lem}\label{tametoSn}Let $\sF$ be either a Kloosterman sheaf $\sK l$ of rank $D \ge 4$ or a hypergeometric sheaf $\sH$ of type $(D,m)$ with $D > m > 0$ and  $D \ge 4$. Suppose $\sF$ is $n$-tensor induced for a given $n \ge 2$. Consider the composite homomorphism
$$\pi_1(\G_m/\overline{\F_p}) \rightarrow (\otimes_{i=1}^n \GL(A_i)) \rtimes \Sym_n \rightarrow \Sym_n,$$
obtained by projecting onto the last factor. 
Suppose we are in either of the following four situations.
\begin{itemize}
\item[(i)]$\sF$ is a Kloosterman sheaf of rank $D \ge 4$.
\item[(ii)]$\sF$ is a hypergeometric sheaf  $\sH$ of type $(D,m)$ with $D \neq 4$. Denote by $p_0$ the least prime dividing $D$, and suppose we have the inequality $W > D/p_0^2$.
\item[(iii)]$\sF$ is a hypergeometric sheaf  $\sH$ of type $(4,1)$ and $p$ is odd.
\item[(iv)]$\sF$ is a hypergeometric sheaf  $\sH$ of type $(4,2)$ and $p=2$.
\end{itemize}
Then this composite homomorphism factors through the tame quotient $\pi_1(\G_m/\overline{\F_p})^{{\rm tame \ at \ }0,\infty}$, and its image is an $n$-cycle in $S_n$. Moreover, $n$ is prime to $p$.
\end{lem}
\begin{proof}Via the deleted permutation representation, we have $\Sym_n \subset O(n-1)$. View the composite homomorphism as an
$n-1$ dimensional representation of $\pi_1$. It is tame at $0$, and its largest $\infty$ slope is $\le 1/W$. We first show that this
homomorphism is tame at $\infty$. For this, via the inequality $\Swan_{\infty} \le (n-1)/W$, it suffices to show that 
$W >n-1$.

In the Kloosterman case, $W=D$ and $D=D_0^n$ with $D_0 \ge 2$ and $n \ge 2$. So we must show in this case that
$D_0^n > n-1$ for $D_0 \ge 2$ and $n \ge 2$. Hence we are done, since $2^n > n-1$ for all $n \in \Z_{\geq 0}$.

In the two hypergeometric cases with $D=4$, the only possible $n$ is $n=2$. In both of these cases, we have $W  > 1$.

In the hypergeometric case with $D\neq 4$, we are given $W > D/p_0^2 =D_0^n/p_0^2$, so it suffices to show that $D_0^n/p_0^2 \ge n-1$ for
$n \ge 2$. Because $p_0 |D$ and $D=D_0^n$, $p_0$ must divide $D_0$. Thus $D_0^n/p_0^2 \ge p_0^{n-2}$, and it suffices to show that
$p_0^{n-2} \ge n-1$. Again for given $n \ge 2$, it suffices to show that $2^{n-2} \ge n-1$, which holds for all $n \ge 2$.

The tame quotient $\pi_1(\G_m/\overline{\F_p})^{{\rm tame \ at \ }0,\infty}$ is the pro-cyclic group $\prod_{\ell \neq p}\Z_\ell(1)$, of pro-order prime to p. So its image in $S_n$ is a cyclic group of order prime to $p$. But this image must be transitive, otherwise our
$\sK l$ would be tensor decomposed (never) or our $\sH$ would be tensor decomposed (not under the $D \neq 4$ and $(D,m)$ not (an even power of $p,1$) hypothesis). Thus the image is (the cyclic group generated by) an $n$-cycle. Because the tame quotient is pro-cyclic or pro-order prime to $p$, and cyclic quotient has order prime to $p$. Thus $n$ is prime to $p$.
\end{proof}

\begin{cor}\label{npullback}Let $\sF$ be either a Kloosterman sheaf or a hypergeometric sheaf which satisfies one of the hypotheses of Lemma \ref{tametoSn} above, if $\sF$ is $n$-tensor induced for a given $n \ge 2$ ( $n$ necessarily prime to $p$), then we have a tensor decomposition of the Kummer pullback $[n]^\star \sF$,
$$[n]^\star\sF =\sA_1 \otimes \sA_2 \otimes \cdots \otimes\sA_n$$
 with local systems $\sA_i$ each of rank $D_0 \ge 2$. 
Moreover, if $D$ is prime to $p$, then we can choose this tensor decomposition so that each $\sA_i$ has all $\infty$ slopes $\le n/W$.
\end{cor}
\begin{proof}In view of Lemma \ref{tametoSn}, after this Kummer pullback, $\pi_1$ lands in $\otimes_{i=1}^n\GL(A_i)$. Then apply the linearization Proposition \ref{linearize}. The largest  $\infty$ slope of $[n]^\star\sF$ is $n/W$, so in the case when $D$ is prime to $p$, we apply part (ii) of Proposition \ref{linearize}.
\end{proof}

\subsection*{3B. Tensor induction: the case when $p \nmid D$}
\begin{prop}\label{notp-tensind}Let $\sF$ be either a Kloosterman sheaf $\sK l$ of rank $D> 4$ or a hypergeometric sheaf $\sH$ of type $(D,m)$ with $D > m > 0$ and  $D\ge 4$. Suppose further we are in one of the following three situations.
\begin{itemize}
\item[(i)]$\sF$ is a Kloosteman sheaf of rank $D\ge 4$ and $D$ is prime to $p$..
\item[(ii)]$\sF$ is a hypergeometric sheaf  $\sH$ of type $(D,m)$ with $D \neq 4$ and $D$ prime to $p$. Denote by $p_0$ the least prime dividing $D$, and suppose we have the inequality $W > D/p_0$. If $D=p_0^2$ (possible only if $p_0 >2$, given that $D > 4$), suppose in addition that $W > 2p_0$. If $D=8$, suppose in addition that $W > 6$.
\item[(iii)]$\sF$ is a hypergeometric sheaf  $\sH$ of type $(4,1)$ and $p \neq 2$.
\end{itemize}
Then $\sF$ is not tensor induced.
\end{prop}
\begin{proof}
We treat first the case of a hypergeometric sheaf  $\sH$ of type $(4,1)$ in characteristic $p \neq 2$. We must show that $\sH$ is not $2$ tensor induced. If it were, then the $I(\infty)$ of $[2]^\star \sH$ would be tensor decomposed. But its slopes are $2/3$ repeated $3$ times, and $0$. Thus the  $I(\infty)$ of $[2]^\star \sH$ is the sum of a one-dimensional tame part and a single wild irreducible of dimension $3$,
hence is not tensor decomposable, cf. \cite[Cor. 10.4 (ii)]{Ka-RL-T-2Co1}.

The idea is to show that in the other cases, each $\sA_i$ is tame at $\infty$. [For this, it suffices to show that its Swan conductor is $< 1$.] This tameness forces $[n]^\star \sF$ to be tame at $\infty$, which is nonsense.

We begin with the Kloosterman case. If we are $n$-tensor induced, then $D=D_0^n$, each $A_i$ has rank $D_0$ and all  $\infty$ slopes
$\le n/W =n/D =n/D_0^n$. It suffices to show that each $\sA_i$ has $Swan_{\infty} <1$. This Swan conductor is $\le D_0(n/D_0^n)$, so
it suffices to show that
$$n < D_0^{n-1}$$
when $n \ge 2$ and $D_0 \ge 2$, except in the case $(n=2, D_0=2)$, which is ruled out by the $D > 4$ hypothesis. For $n=2$, the worst remaining case is $D_0=3$, and indeed $3>2$. For $n \ge 3$, the worst case is $2^{n-1}>n$, which indeed holds.

In the hypergeometric case, we again have $D_0(n/W)$ as an upper bound for the Swan conductor of any $\sA_i$. We have 
$W >D/p_0$, so we wish to show $nD_0 <W$, which is implied by
$$nD_0 \le D_0^n/p_0, {\rm i.e.,\ }D_0^{n-1} \ge np_0.$$
Because $p_0$ divides $D=D_0^n$, $p_0$ divides $D_0$, so we write $D_0=n_0p_0$ for some integer $n_0 \ge 1$. It suffices to show 
$$n_0^{n-1}p_0^{n-2} \ge n.$$

This last equality is visibly false for $n=2$ if $n_0=1$, i.e, if we are dealing with the case $D=p_0^2$. But in that case we assumed that
$W > 2p_p$, and with this estimate we do have $nD_0 <W$ in the $n=2$ case with $D=p_0^2$.

Suppose now that $n=3$. Then we need $3D_0 \le D_0^3/p_0$, i.e., we need $D_0^2 \ge 3p_0$, i.e., $n_0^2p_0 \ge 3$. This is fine so long as $p_0 \ge 3$ or $n_0 >1$. In the case $D_0=p_0=2$, the desired inequality for $n=3$ is $3.2 <W$, which is precisely what we assumed in the $D=8$ case.

Finally, for $n \ge 4$, where we need $nD_0 \le D_0^n/p_0$, this is implied by $p_0^{n-2} \ge n$, which for $n\ge 4$ already holds for the worst case $p_0=2$.
\end{proof}

\begin{rmk}In \cite[10.6.9 and 10.9.1]{Ka-ESDE}, there are examples of hypergeometric sheaves of type $(9,3)$ which are $2$-tensor induced. In \cite[10.8.1]{Ka-ESDE}, there are examples of hypergeometric sheaves of type $(8,2)$ which are $3$-tensor induced. 
\end{rmk}

\subsection*{3C. Tensor induction: the case when $p|D$}
\begin{prop}\label{p-tensind}Let $\sF$ be either a Kloosterman sheaf $\sK l$ of rank $D > 4$ or a hypergeometric sheaf $\sH$ of type $(D,m)$ with $D > m > 0$ and  $D \ge 4$. Suppose further we are in one of the following three situations.
\begin{itemize}
\item[(i)]$\sF$ is a hypergeometric sheaf  $\sH$ of type $(4,2)$ in characteristic $p=2$.
\item[(ii)]$\sF$ is a Kloosteman sheaf of rank $D > 4$  and $p|D$. If $p=2$, suppose also that $D \neq 8$.
\item[(iii)]$\sF$ is a hypergeometric sheaf  $\sH$ of type $(D,m)$ with $D > 4$ and $p|D$. Suppose that $W > (2/3)(D-1)$. If $p=2$, suppose  $D \neq 8$. If $p=3$, suppose $(D,m)$ is not $(9,1)$.
\end{itemize}
Then $\sF$ is not tensor induced.
\end{prop}

\begin{proof}We first treat case (i), a hypergeometric sheaf  $\sH$ of type $(4,2)$ in characteristic $p=2$. It could only possibly be $n$-tensor induced for $n=2$, but this is impossible as $p \nmid n$, cf. Lemma \ref{tametoSn}.

We next treat the Kloosterman case. If $\sK l$ is $n$-tensor induced for a given $n \ge 2$, then $n$ is prime to $p$,  $D=D_0^n$ and we have a tensor decomposition
$$[n]^\star \sK l =\sA_1 \otimes \sA_2 \cdots \otimes\sA_n,$$
with each $\sA_i$ of rank $D_0 \ge 2$.
We use the argument of \v{S}uch, cf. \cite[Prop. 12.1, second paragraph]{Such}, which we already used in the proof of Lemma \ref{Kl-indec}. Exactly as there, each $\End^0(\sA_i)$ is a direct factor of $\End([n]^\star \sK l )$, hence has all $\infty$ slopes $\le n/D =n/D_0^n$.

If $n=2$, then each $\sA_i$ has
$$\Swan_{\infty}( \End^0(\sA_i)) \le (D_0^2-1)(2/D_0^2) < 2.$$
Thus $\Swan_{\infty}( \End^0(\sA_i))$, which is equal to $\Swan_{\infty}( \End(\sA_i))$,  is either $0$ or $1$. If it is $1$ for at least one of $\sA_1$ or $\sA_2$, that $\End$ is the direct sum of a nonzero tame part (from scalar endomorphisms) and an $I(\infty)$-irreducible part with Swan conductor $1$, this latter part being totally wild. Its expression as an $\End$ violates its tensor indecomposability, cf. \cite[Cor. 10.3]{Ka-RL-T-2Co1}, because the rank of $\End$, here $D_0^2=D$, is not $4$. 

If both the $ \End^0(\sA_i)$ are tame at $\infty$, then $P(\infty)$ acts by a scalar character on each of $\sA_1$ and $\sA_2$, and hence 
$P(\infty)$ acts by a scalar character on $[2]^\star(\sK l)$. The $I(\infty)$ representation of $\sK l$ is irreducible of Swan conductor $1$.
Its rank $D$ is divisible by $p$, so we write $D=n_0q$ with $n_0 \ge 1$ prime to $p$ and with $q$ a positive power of $p$. Then
the $P(\infty)$ representation of $\sK l$ is the direct sum of $n_0$ inequivalent irreducible $P(\infty)$-representations, each of dimension $q$. The $P(\infty)$ representation does not change under Kummer pullback, so  $P(\infty)$ representation of  $[2]^\star(\sK l)$ is the direct sum of $n_0$ inequivalent irreducible $P(\infty)$-representations, each of dimension $q$. Therefore $P(\infty)$ does not act as scalars on any of these $q$ dimensional $P(\infty)$-irreducibles.

Suppose now that $\sK l$ is $n$-tensor induced for a given $n \ge 3$. Then each $\sA_i$ has
$$\Swan_{\infty}( \End^0(\sA_i)) \le (D_0^2 -1)(n/D_0^n) \le (n/D_0^{n-2})\frac{D_0^2 -1}{D_0^2} <n/D_0^{n-2}.$$
For $n=3$, we cannot have $D_0=2$ unless $p=2$, but we have ruled out $D=8$ when $p=2$. So for $n=3$, we have
$D_0 \ge 3$, and so each $\Swan_{\infty}( \End^0(\sA_i)) < 1$ in the $n=3$ case. For $n \ge 4$, we have $n/D_0^{n-2} \le 1$, as one sees already from the worst case $D_0=2$, where it amounts to the inequality $n \le 2^{n-2}$ for $n \ge 4$.

We now turn to the hypergeometric case. Because $D \neq 4$, $\sH$ is tensor indecomposable. So if $\sH$ is $n$-tensor induced for some $n \ge 2$ (necessarily prime to $p$), we have $D=D_0^n$ and a tensor decomposition
$$[n]^\star \sH=\sA_1 \otimes \sA_2 \otimes\cdots \otimes\sA_n,$$
with each $\sA_i$ of rank $D_0 \ge 2$.

We first consider the case $n \ge 3$. By the \v{S}uch argument, each $\End^0(\sA_i)$ is a direct factor of $\End([n]^\star \sH)$, hence has all $\infty$ slopes $\le n/W$. We claim that each $\End^0(\sA_i)$ is tame at $\infty$. If so, we reach a contradiction as follows. Each $\End(\sA_i)$ is then tame at $\infty$, so
$P(\infty)$ acts on each $\sA_i$ by a scalar character, and hence $P(\infty)$ acts on $[n]^\star \sH$ by a scalar character. Because $[n]^\star \sH$ has a tame part of rank $m > 0$, this scalar character must be trivial. This in turn implies that $[n]^\star \sH$ is tame at $\infty$, contradiction.

To show that each $\End^0(\sA_i)$ is tame at $\infty$, it suffices to show that its Swan conductor is $< 1$. Using the estimate $\Swan_{\infty}( \End^0(\sA_i))$ is $\le (D_0^2-1)(n/W)$, it suffices to show that
$$n(D_0^2-1) < W.$$
By hypothesis, $W > (2/3)(D-1)$.
So for $D \neq 27$ it suffices to show that
$$n(D_0^2-1) \le (2/3)(D-1)=2/3)(D_0^n-1),$$
so long $D$ is neither $8$ when $p=2$ nor $27$ when $p=3$. 
For $n=3$, we need
$$3(D_0^2-1) \le (2/3)(D_0^3-1)\ {\rm for\ }p \ge 3.$$
This inequality holds for $D_0 \ge 5$, which for $p$ odd rules out $D_0=3$. But this $p=3,n=3, D_0=3$ case does not arise, because in characteristic $p$, here $3$, we can only be $n$-tensor induced when $n$ is prime to $p$.. 
For $p=2$, it rules out $D_0=2$, the excluded $D=8$ case. 

But we must still deal with the case $n=3, p=2, D_0=4$. Here the estimate for
$\Swan_\infty(\End^0(\sA_i)$ 
$$\Swan_\infty(\End^0(\sA_i) < \frac{3(4^2-1)}{(2/3)(4^3-1)}=15/14 <2.$$
So each $\sA_i$ has $\Swan_\infty(\End^0(\sA_i)=\Swan_\infty(\End(\sA_i)$ either $0$ or $1$. The $\Swan_\infty(\End(\sA_i)=1$ is impossible, because it violates tensor indecomposability, cf. \cite[Cor. 10.3]{Ka-RL-T-2Co1}, because the rank of $\End$, here $D_0^2=16$, is not $4$. Thus each $\End^0(\sA_i)$ is tame at $\infty$ in this case as well.


For $n \ge 4$, we need
$$n(D_0^2-1) \le (2/3)(D_0^n-1)\ {\rm for\ }p \ge 3.$$
This holds for $D_0 \ge 3$ for all $n \ge 4$, and for $D_0 \ge 2$ for all $n \ge 5$. The case $n=4,D_0=2$ is excluded because when $p=2$, $n$-tensor induction is only possible when $n$ is odd.


It remains to treat the case $n=2$. In this case, $p$ must be odd. Thus $D=D_0^2$, 
$$[2]^\star \sH=\sA_1 \otimes \sA_2 ,$$
with each $\sA_i$ of rank $D_0$. We first claim that $\Swan_{\infty}(\End^0(\sA_i))$ is $0,1,$ or $2$, i.e. that
$$\Swan_{\infty}(\End^0(\sA_i)) < 3.$$
This Swan conductor is $\le (D_0^2 -1)(2/W)=2(D-1)$, so we must show $2(D-1)<3W$, which is precisely our hypothesis.

If both $\sA_i$ have $\End^0$, and hence $\End$, tame at $\infty$, then just as in the $n \ge 3$ case above, we reach a contradiction.

If one of the $\sA_i$ has $\Swan_{\infty}(\End^0(\sA_i)) =1$, then the rank $D$ of this $\End(\sA_i)$ must be $4$. But $D >4$ in the hypergeometric case $(ii)$ we are considering. [Alternatively, $D=4$ and $p|D$ forces $p=2$, in which case $n$-tensor induction for $n=2$ is impossible.]

If one of the $\sA_i$ has $\Swan_{\infty}(\End^0(\sA_i)) =2$, we argue as follows. Either $\End(\sA_i) $ is the sum of a nonzero tame part and a single $I(\infty)$-irreducible whose Swan conductor is $2$, or $\End(\sA_i) $ is the sum of a nonzero tame part and of two $I(\infty)$-irreducibles, each of Swan conductor $1$. In the first case, we again (by  \cite[Cor. 10.3]{Ka-RL-T-2Co1}) then the rank $D$ of this $\End(\sA_i)$ must be $4$, an excluded case.

It remains now to analyze the case when each  of the $\End(\sA_i) $, $i=1,2$, is the sum of a nonzero tame part and of two $I(\infty)$-irreducibles, each of Swan conductor $1$. We first show that in this case, the rank $D$ of $\End(\sA_i) $ must be $q^2$, for $q$ some positive power of $p$.
We show this in the next lemma.

\begin{lem}\label{swan2irred}Let $\sA$ be an $I(\infty)$-representation of dimension $D_0 \ge 2$ with $p|D_0$, $p$ odd, such that $\End(\sA):=\sA \otimes \sA^{\vee}$ is the sum of a nonzero tame part and of two wild $I(\infty)$-irreducibles. If such an $\sA$ exists, then it is an $I(\infty)$-irreducible of dimension $q$, for $q$ some positive power of $p$. 
\end{lem}
\begin{proof}We first show that $\sA$ is totally wild. It cannot be totally tame, otherwise its $\End$ would be tame. It cannot contain both a nonzero tame part $T$ and two wild two $I(\infty)$-irreducibles $W_1$ and $W_2$, for then its $\End$ contains the four totally wild components $T\otimes W_1^\vee, T\otimes W_2^\vee, T^\vee \otimes W_1, T^\vee \otimes W_2$, which each themselves contain at least one wild $I(\infty)$-irreducible. 

If it is of the form $T+W$ with $T$ a nonzero tame part and $W$ a wild $I(\infty)$-irreducible,then its $\End$ contains $T\otimes W^\vee, T^\vee \otimes W, W\otimes W^\vee$. If this  $\End$ contains only two wild $I(\infty)$-irreducibles, then
$T$ is one-dimensional and $W\otimes W^\vee$ is totally tame. But if $W\otimes W^\vee$ is totally tame, then $W$ is one-dimensional,
cf. \cite[Lemma 10.2]{Ka-RL-T-2Co1}). Thus our $\sA$, if not totally wild, has dimension $D_0= 2$. But as $p|D_0$, and $p$ is odd, this cannot happen.

Thus $\sA$ is totally wild. We next show that it is $I(\infty)$-irreducible. 
 If $\sA$  contains two wild irreducibles $W_1$ and $W_2$, at least one of which has dimension $> 1$, we reach a contradiction as follows. Then its $\End$ contains  the four terms $W_1\otimes W_2^\vee, W_2\otimes W_1^\vee, 
W_1\otimes W_1^\vee,W_2\otimes W_2^\vee$. Neither of ths two cross terms, nor whichever of $W_i\otimes W_i^\vee$ has dimension $> 1$, can be totally tame, again by  \cite[Lemma 10.2]{Ka-RL-T-2Co1}). 

To finish the proof that $\sA$ is $I(\infty)$-irreducible, we must rule out the case when $\sA$  contains only wild irreducibles of dimension one. In this case, $\sA$ contains at least $3$ such (because $D_0 >2$). Partition them according to the equivalence relation $W_1 \equiv W_2$ if and only if $W_1\otimes W_2^\vee$ is tame. Then $\sA$ is the sum of terms 
$T_i\otimes W_i$, with $T_i$ tame of dimension $d_i \ge 1$, $W_i$ wild of dimension one, and $W_i\otimes W_j^\vee$ is wild whenever $i \neq j$. Then its $\End$ contains precisely $\sum_{i\neq j}d_id_j$ wild summands (namely the $T_i\otimes T_j^\vee \otimes W_i\otimes W_j^\vee$), and $\sum_i d_i =D_0 \ge 3$. There must be more than one such summand, otherwise the $\End$ is tame.
If there are at least three such summands, then $\sA$ contains $W_1 + W_2 +W_3$, with $W_i\otimes W_j^\vee$ wild for $i \neq j$. In this case, the $ \End$ contains the six wild summands $W_i\otimes W_j^\vee$ with $i,j \in [1,3]$ and $i\neq j$. So $\sA$ must be
$T_1\otimes W_1 +T_2\otimes W_2$ with $d_1+d_2 =D_0 \ge 3$. Interchanging the two indices if necessary,  we may assume $d_1 \ge 2$ (and $d_1 \ge 1$). Then the $\End$ contains at least $2d_1 \ge 4$ wild summands, contradiction. 

Thus $\sA$ is $I(\infty)$-irreducible. We write its dimension $D_0$ as $n_0q$, with $n_0 \ge 1$ and $q$ a strictly positive power of $p$.
Then $\sA$ is the Kummer direct image 
$$\sA =[n_0]_\star \sB$$
for $\sB$ a $q$ dimensional $I(\infty)$-irreducible. We know further that $\sB$ is $P(\infty)$-irreducible, and that $\sB$ is, as $P(\infty)$-representation, the direct sum of $n_0$ pairwise inequivalent irreducibles. Indeed, under the multiplicative  translation action of
$\mu_{n_0}$, the $n_0$ multiplicative translates $\{{\rm MT}_\zeta \sB\}_{\zeta \in \mu_{n_0}}$ are pairwise inequivalent  $P(\infty)$-irreducibles.

We next claim that we have a direct sum decomposition
$$\End([n_0]_\star \sB)= \bigoplus_{\zeta \in \mu_{n_0}} [n_0]_\star (\sB \otimes {\rm MT}_\zeta \sB^\vee).$$
To see this, we argue as follows. Denote by $I(n_0) \lhd I(\infty)$ the open subgroup of index $n_0$. For any $I(n_0)$-representation $V$, the character of its direct image $[n_0]_\star V$ (i.e.the group theoretic induction of $V$ from $I(n_0)$ to $I(\infty)$) is supported in $I(n_0)$ (simply because  $I(n_0) \lhd I(\infty)$ is a normal subgroup). Then $\End(([n_0]_\star V)=([n_0]_\star V)\otimes ([n_0]_\star V^\vee)$ has its character supported in $I(n_0)$. Therefore the character of $\End([n_0]_\star V)$ is determined by its pullback to  $I(n_0)$.

We now apply this with $V$ taken to be $\sB$. 
Because $\sB$ is  $I(n_0)$-irreducible, 
its induction $[n_0]_\star \sB$ and its $\End([n_0]_\star \sB)$ are both $I(\infty)$-semisimple, so determined by their characters, and hence by the characters of their pullbacks $[n_0]^\star$. We have
$$[n_0]^\star [n_0]_\star \sB =  \bigoplus_{\zeta \in \mu_{n_0}} {\rm MT}_\zeta \sB,~~
   [n_0]^\star [n_0]_\star \sB^\vee =  \bigoplus_{\zeta \in \mu_{n_0}} {\rm MT}_\zeta \sB^\vee.$$
Thus
$$[n_0]^\star \End([n_0]_\star \sB) =\bigoplus_{(\zeta_1,\zeta_2) \in \mu_{n_0}\times \mu_{n_0}}( {\rm MT}_{\zeta_1} \sB)\otimes ({\rm MT}_{\zeta _2}\sB^\vee)
=\bigoplus_{{\zeta_2} \in \mu_{n_0}} \bigoplus_{{\zeta_1} \in \mu_{n_0}} {\rm MT}_{\zeta_1}(\sB \otimes ({\rm MT}_{\zeta _2}\sB^\vee)),$$
which is the pullback to $I(n_0)$ of the character of 
$$\bigoplus_{{\zeta_2} \in \mu_{n_0}} [n_0]_\star (\sA \otimes {\rm MT}_{\zeta_2} \sB^\vee).$$

With this formula at hand, we continue as follows. Because the various $ {\rm MT}_{\zeta} \sB$ are pairwise inequivalent irreducible $P(\infty)$-representations, we have
$$\sB\otimes \sB^\vee = \triv + {\rm totally\ wild},$$
and for each $\zeta \neq 1$,
$$\sB \otimes {\rm MT}_{\zeta} \sB^\vee= {\rm totally\ wild}.$$
Now $V \mapsto [n_0]_\star V$ preserves being totally wild, so we find

$$\begin{aligned}\End(\sA) & = \End([n_0]_\star \sB)=[n_0]_\star \triv + [n_0]_\star \End^0(\sB) \oplus_{\zeta \neq 1 \in \mu_{n_0}}[n_0]_\star(\sB \otimes {\rm MT}_\zeta \sB^\vee)\\
& =[n_0]_\star \triv + {\rm the\ sum\ of\ }n_0\ {\rm totally\  wild\ summands}.\end{aligned}$$

In order for there to be precisely two irreducible wild summands in $\End(\sA)$, we must have $n_0 \le 2$. 

If $n_0=2$, then both
$[2]_\star(\End^0(\sB))$ and $[2]_\star(\sB \otimes {\rm MT}_{-1} \sB^\vee)$ must be irreducible. In particular, $\End^0(\sB)$ must be irreducible, i.e. we must have $\End(\sB) =\triv +{\rm irreducible}$. This is possible only when $\sB$ has rank $2$, cf. \cite
[Cor. 10.4]{Ka-RL-T-2Co1}. Then $\sA=[2]_\star \sB$ has rank $D_0=4$. As $p|D_0$, $p$ must be $2$, an excluded case.

If $n_0=1$, then $\sA$ is an $I(\infty)$-irreducible of dimension $q$.
\end{proof}

Returning to our situation
$$[2]^\star \sH =\sA_1 \otimes \sA_2,$$
we now know that $D=q^2,D_0=q$, and that both $\sA_1$ and $\sA_2$ are $I(\infty)$-irreducibles. Having dimension $q$, they are each $P(\infty)$-irreducible. Because $\sH$ has type $(D,m)$ with $m > 0$, $[2]^\star \sH$ has an $I(\infty)$-tame part of dimension $m > 0$. 
At the expense of tensoring $\sH$ with a tame character, we may assume further that among the ``bottom" characters in $\sH$ is $\triv$.
Then $\sA_1 \otimes \sA_2$ as $I(\infty)$-representation contains $\triv$. The projection of $\sA_1 \otimes \sA_2$ onto $\triv$ is then 
a nonzero  $I(\infty)$-linear map $\sA_2 \rightarrow \sA_1^\vee$, which must be an $I(\infty)$-isomorphism because source and target are $I(\infty)$-irreducible. Thus $\sA_1 \otimes \sA_2$ is $I(\infty)$-isomorphic to $\End(\sA_1)$. Because  $\sA_1$ is $P(\infty)$-irreducible,
the space of $P(\infty)$-invariants in  $\End(\sA_1)$ is one dimensional. But this space of $P(\infty)$-invariants is precisely the tame part of 
$\sA_1 \otimes \sA_2=[2]^\star \sH$. Therefore $m=1$, and $\sH$ has type $(D,m)=(q^2,1)$. 

In the next section, we will deal with this $(q^2,1)$ case.

\subsection*{3D. Completion of the proof of Proposition \ref{p-tensind}}
$\;$\\
In this subsection, $q$ is a positive power $p^a$ of the odd prime $p$, and $\sH$ is a hypergeometric of type $(q^2,1)$ whose ``bottom" character is $\triv$.
The $I(\infty)$-representation of $\sH$ is the direct sum $\sW +\triv$, with $\sW$ totally wild of rank $q^2-1$ and Swan conductor one.
Because $q^2-1$ is prime to $p$, we know \cite[1.14]{Ka-GKM} that $\sW$ is the Kummer direct image $[q^2-1]_\star(\sL)$ for some rank one $\sL$ of Swan conductor one. Furthermore, the restriction of $\sW$ to $P(\infty)$ is the direct sum of $q^2 -1$ pairwise distinct characters of $P(\infty)$ which are cyclically permuted \cite[1.14(3)]{Ka-GKM} by $I(\infty)/P(\infty)$, acting through its  $\mu_{q^2-1}$ quotient.   

We denote by
$$J_1:={\rm \  the\  image\  of\  }I(\infty) {\rm\ acting\  on\  }
\sW+\triv.$$
Because our $\sH$ began life on $\G_m$ over a finite extension of $\F_p$, we know \cite[1.11 (3)]{Ka-GKM} that $J_1$ is finite,
with a normal Sylow $p$-subgroup $P_1$ such that $J_1/P_1$ is cyclic of $p'$-order $m(q^2-1)$ for some $m \in \Z_{\geq 1}$.  
Moreover, any element of $J_1$ of order $m(q^2-1)$ induces, by conjugation, an automorphism of $P_1$ of order $q^2-1$. 
[Indeed this action cyclically permutes $q^2-1$ distinct characters of $P_1$ on $\sW$.]

Our concern is with the Kummer pullback $[2]^\star \sH$, whose $I(\infty)$ representation is $[2]^\star \sW +\triv$.
We readily decompose
$$[2]^\star \sW =[2]^\star [q^2-1]_\star  \sL=[2]^\star [2]_\star [(q^2-1)/2]_\star  \sL=$$
$$=[2]^\star [2]^\star \sX=\sX +[-1]_\star \sX, {\rm \ for\ }\sX:= [(q^2-1)/2]_\star  \sL.$$
$\sX$ is itself irreducible of Swan conductor one. One knows that for any irreducible $I(\infty)$-representation $\sX$ of Swan conductor one, $\sX$ and its multiplicative translate $[-1]_\star \sX$ are inequivalent. [If they were isomorphic, $\sX$ would descend through $[2]$, i.e. would be of the form $[2]^\star \sY$, which would force its Swan conductor to be even, cf. \cite[proof of 3.7.6]{Ka-ESDE} for the $\sD$-module analogue.]

Thus the  $I(\infty)$ representation of $[2]^\star \sH$ is the sum of three distinct irreducibles:
$$[2]^\star \sH \cong \sX + [-1]_\star \sX +\triv.$$

We denote by 
$$J_2 \lhd J_1$$
the subgroup of index $2$ which is the image of $I(\infty)$ acting on $[2]^\star \sH$.
As $p > 2$, $J_2$ has the {\bf same} $P_1$ as its normal Sylow $p$-subgroup $P_2$, and the quotient $J_2/P_2$ is cyclic of $p'$-order $m(q^2-1)/2$. Moreover, any element of $J_2$ of order $m(q^2-1)/2$ induces, by conjugation, an automorphism of $P_2$ of order  $(q^2-1)/2$. 

Suppose now that $\sH$ is $2$-tensor induced. The $I(\infty)$ representation (indeed the $\pi_1$-representation, but we do not know how to use this much stronger information) of $[2]^\star \sH$ lands in $\GL(A_1)\otimes \GL(A_2)$, with each $A_i$ of dimension $q$. We wrote $\GL(A_1)\otimes \GL(A_2)$ 
as $\SL(A_1)\otimes \GL(A_2)$. This allowed us to lift the action of $I(\infty)$ (indeed of $\pi_1$) on $[2]^\star \sH$ to a
map 
$$I(\infty)\rightarrow \SL(A_1)\times \GL(A_2).$$
The image of this map we denote $J_3$. This group $J_3$ is finite, and maps
$J_3 \twoheadrightarrow J_2$
with kernel the intersection of $J_3$ with the subgroup $\mu_q$, embedded as $(\lambda,1/\lambda) \in  \SL(A_1)\times \GL(A_2)$.
Thus the kernel is a cyclic group of order dividing $q$, so lies in the Sylow $p$-subgroup $P_3$ of $J_3$.
Thus $J_3$ has a normal Sylow $p$-subgroup $P_3$, and $J_3/P_3$ is cyclic of $p'$-order $m(q^2-1)/2$, in fact the same order as $J_2/P_2$.
Thus we get $J_3$-representations
$\sA_1 \times \sA_2$ such that 
$$[2]^\star \sH \cong \sA_1 \otimes \sA_2,$$
and we showed that $\sA_1$ and $\sA_2$ are $I(\infty)$ (and hence $J_3$) duals of each other.  

Next, we define
$$J_4 < \SL(A_1)$$
to be the image of $J_3 < \SL(A_1)\times \GL(A_2)$ by the first projection. We denote by $P_4 \lhd J_4$ its normal Sylow $p$-subgroup.
The image of $J_4$ in $\End(\sA_1) \cong  [2]^\star \sH$ is $J_2$. The kernel $K$ of
the surjection
$J_4 \twoheadrightarrow J_2$
is the intersection of $J_4$ with the scalars $\mu_q$ of $\SL(A_1)$, so lies in $P_4$.
Therefore $P_4$ maps onto $P_2$ with kernel $K$, and $J_4/P_4$ maps isomorphically to $J_2/P_2$. Any element $x \in J_4$ of order  $m(q^2-1)/2$ induces, by conjugation, an automorphism $\varphi_x$ of $P_4$ of order  $(q^2-1)/2$. [Indeed, this is already the case in the quotient situation $(J_2, P_2)$, 
hence $\varphi_x$ has order divisible by $(q^2-1)/2$.  It also follows that $\varphi_x^{(q^2-1)/2}$ acts trivially on $P_2=P_4/K$ and on the central
$p$-subgroup $K$, and so the order of $\varphi_x^{(q^2-1)/2}$ is a $p$-power.
As $x$ is a $p'$-element, we conclude that $\varphi_x^{(q^2-1)/2} = 1$.]

\smallskip
We will apply the next lemmas with
$$J:=J_4 < \SL(\sA_1), V:= \sA_1.$$
We know that $J$ is a finite group with a normal $p$-Sylow subgroup $P$ (thus $\OB_p(J)=P$), that $J/P$ is cyclic of prime to $p$ order divisible by $(q^2-1)/2$ and that any element of $J$ of order  $m(q^2-1)/2$ induces, by conjugation, an automorphism of $P$ of order  $(q^2-1)/2$. We have a faithful irreducible $q$-dimensional representation $V$ of $J$, and we know that $\End(V)$ is the sum of three distinct irreducible submodules. Note that $J$ is solvable, and 
furthermore has cyclic Sylow $2$-subgroups if $2 \nmid q$. Hence the subsequent Lemmas \ref{m43}--\ref{m43-main} apply to $J$.


\begin{lem}\label{m43}
Let $V$ be a faithful irreducible $\C J$-module of dimension $d \geq 3$ and $d \neq 4$, where $J$ is a finite solvable group, which has 
abelian Sylow $2$-subgroups if $2 \nmid d$.
Suppose that $\End(V)$ is a sum of three irreducible submodules, 
but the $J$-module $V$ does not satisfy condition {\rm ({\bf S})}. 
Then the $J$-module $V$ is tensor indecomposable, not tensor induced, and every imprimitivity
decomposition for $V$ has the form $V = \oplus^d_{i=1}V_i$, with $\dim V_i = 1$ and $J$ permuting $\{V_1, \ldots,V_d\}$ primitively and $2$-homogeneously.
\end{lem}

\begin{proof}
Let $\chi$ denote the character afforded by the $\C J$-module $V$. By assumption, $\chi\overline\chi = 1_J + \al_1 + \al_2$, where 
$\al_i \in \Irr(J)$. First we note that $\al_1 \neq \al_2$. Indeed, by Burnside's theorem (3.15) of \cite{Is}, $\chi(g) = 0$ for some $g \in J$. Hence,
in the case $\al_1 = \al_2$ we would have that $\al_1(g) = -1/2$, which is not an algebraic integer, a contradiction. Next, the irreducibility of $\chi$
implies that $\al_i \neq 1_J$. It follows that 
$$3 = M_4(J,V) = [\chi\overline\chi,\chi\overline\chi]_J = [\chi^2,\chi^2]_J.$$
On the other hand, $V^{\otimes 2} = \Sym^2(V) \oplus \wedge^2(V)$. So we conclude that either $\Sym^2(V)$ is irreducible, or 
$\wedge^2(V)$ is irreducible. 

If $\Sym^2(V)$ is irreducible, then the $J$-module $V$ satisfies ({\bf S}) by \cite[Lemma 2.1]{G-T}. Hence $\wedge^2(V)$ is irreducible.
Assume in addition that $V$ is imprimitive: $J$ stabilizes a decomposition $V = \oplus^t_{i=1}V_i$ with $t > 1$ and $\dim V_i = d/t$. The proof
of \cite[Lemma 2.4]{G-T} then shows that $t=d$ and $J$ acts on the set $\{V_1, \ldots ,V_d\}$ $2$-homogeneously. 


In the same proof, it was shown that $V$ is tensor indecomposable, and furthermore, if it is tensor induced, then any tensor induced decomposition 
has the form $V = V_1 \otimes V_2$, with $J_0 := \Stab_J(V_1,V_2)$ of index $2$ in $J$, and $\Sym^2(V_i)$ and $\wedge^2(V_i)$ being irreducible
over $J_0$ for $i=1,2$. In the latter case, $\dim V_i = \sqrt{d} \geq 3$ by hypothesis. Furthermore, $V_i$ is irreducible, and 
$\Sym^2(V_i) \not\cong \wedge^2(V_i)$ by dimension consideration. It follows that $M_4(J_0,V) = 2$. Certainly, $J_0$ is solvable as so is $J$.
If furthermore $2 \nmid \dim V_i$, then $2 \nmid d$ and so Sylow $2$-subgroups of $J$ are abelian, whence so are Sylow $2$-subgroups of $J_0$.
Thus \cite[Theorem 2.3]{Ka-RL-T-2Co1} applies to the $J_0$-module $V_i$ and yields $M_4(J_0,V_i) \geq 3$, a contradiction.
\end{proof}

Next we will analyze the situations arising in Lemma \ref{m43}, under the assumption that $d = \dim V = q = p^a$, where $p$ is a prime and $a \geq 1$.
We will fix a {\it primitive prime divisor} $\ell = \ppd(p,2a)$ of $p^{2a}-1$, that is, a prime divisor of $p^{2a}-1$ that does not divide $\prod^{2a-1}_{i=1}(p^i-1)$,
when it exists. Such a prime always exists, unless either $(p,a) = (2,3)$, or $a=1$ and $p$ is a Mersenne prime, see \cite{Zs}. 


\begin{lem}\label{m43a}
In the situation of Lemma \ref{m43}, assume that $d = p^a \geq 3$ for a prime $p$ and that 
the conjugation by some element $h \in J$  induces an automorphism 
of $\OB_p(J)$ of order 
$(p^{2a}-1)/\gcd(2,p-1)$. Then $J$ acts primitively on $V$. 
\end{lem}

\begin{proof}
Assume the contrary. By Lemma \ref{m43}, the action of $J$ on $\{V_1, \ldots,V_d\}$ induces a solvable, primitive subgroup $H$ of $\Sym_d$. Since $H$ is 
solvable, it possesses an abelian minimal normal subgroup $N$. By the O'Nan-Scott theorem, see e.g. \cite{LPS}, $H$ is a subgroup of the affine group 
$\mathrm{AGL}(U) = \mathrm{AGL}_a(p)$ in its action on the points of $U = \F_p^a$ (with $N$ acting via translations). Let $B \lhd J$ consist of all
elements that fact trivially on $\{V_1, \ldots,V_d\}$, so that $H = J/B$. Then $B$ is contained in a maximal torus of $\GL(V)$ and so is abelian.

\smallskip
First we consider the case $\ell=\ppd(p,2a)$ exists. Then $\ell$ does not divide $|\GL_a(p)|$. It follows that any $\ell$-element $g \in J$ has trivial image
in $H$, that is, $g \in B$ and so $g \in \OB_\ell(B) \lhd J$ (since $B$ is abelian). For any $x \in \OB_p(J)$ we then have $[g,x] \in \OB_p(J) \cap \OB_\ell(B) = 1$. We have shown
that $[g,\OB_p(J)] =1$, and so $\OB_p(J)$ is centralized by $\OB^{\ell'}(J)$. Thus the action of $J$ on $\OB_p(J)$ induces a subgroup of 
$\Aut(\OB_p(J))$ of order coprime to $\ell$, a contradiction. 

\smallskip
Now we may assume that $\ell$ does not exist. Assume furthermore that $a=1$, but $p \geq 3$ is a Mersenne prime. 
Now we can find a $2$-element 
$g \in \langle h \rangle$ such that the conjugation by $g$ induces an automorphism $\varphi_g$ of order $4$ of $\OB_p(J)$.
Since the $2$-part of $|H|$ divides $|\GL_1(p)| = p-1$, $g^2$ has trivial image
in $H$, and so $g^2 \in \OB_2(B) \lhd J$. For any $x \in \OB_p(J)$ we then have $[g^2,x] \in \OB_p(J) \cap \OB_2(B) = 1$. We have shown
that $[g^2,\OB_p(J)] =1$, contrary to $|\varphi_g| = 4$.

\smallskip
It remains to consider the case $p^a = 8$. 
In this case we can find a $3$-element $g \in \langle h \rangle$ such that the conjugation by $g$ induces an automorphism $\varphi_g$ of order $9$ of $\OB_p(J)$.
Since the $3$-part of $|H|$ divides $|\GL_3(2)| = 168$, $g^3$ has trivial image
in $H$, and so $g^3 \in \OB_3(B) \lhd J$. For any $x \in \OB_p(J)$ we then have $[g^3,x] \in \OB_p(J) \cap \OB_3(B) = 1$. Thus
$[g^3,\OB_p(J)] =1$, again contradicting the equality $|\varphi_g| = 9$.
\end{proof}

Now we complete the analysis of the situations arising in Lemma \ref{m43}:

\begin{lem}\label{m43-main}
Let $J$ be a finite solvable group, and let $V$ be a faithful irreducible $\C J$-module of dimension $d = p^a\geq 5$ for some 
prime $p$. Assume in addition that $J$ has abelian Sylow $2$-subgroups if $2 \nmid d$, and that the conjugation by some element $h \in J$ induces an
automorphism of $\OB_p(J)$ of order $(p^{2a}-1)/\gcd(2,p-1)$.
Then $\End(V)$ cannot be a sum of three irreducible $J$-submodules.
\end{lem}

\begin{proof}
(i) Assume the contrary: $\End(V)$ a sum of three irreducible $J$-submodules. By Lemmas \ref{m43} and \ref{m43a}, the $J$-module $V$ satisfies
condition ({\bf S}). As explained in \S1, the proof of \cite[Proposition 3.8]{G-T} shows that $J$ contains a normal $p$-subgroup $Q$, where $Q = \ZB(Q)E$ for some extraspecial $p$-group $E$ of order $p^{1+2a}$ acting irreducibly on $V$; furthermore, either $Q = E$ or $|\ZB(Q)|=4$. Let $A$ denote the subgroup
of $\Aut(Q)$ induced by the conjugation action of $J$.

\smallskip
(ii) Observe that $\CB_J(\OB_p(J)) = \CB_J(Q) = \ZB(J)$. (Indeed, $\ZB(J) \leq \CB_J(\OB_p(J)) \leq \CB_J(Q)$ as $Q \lhd \OB_p(J)$. As $E \leq Q$ is 
irreducible on $V$, $\CB_J(Q) = \ZB(J)$ by Schur's lemma.) From the equality $\CB_J(\OB_p(J)) = \CB_J(Q)$, we see that the image of $J$ in $\Aut(\OB_p(J))$, namely $J/\CB_J(\OB_p(J)), $ maps isomorphically to $A \cong J/\CB_J(Q)$.

Hence, by hypothesis, $A \leq \Aut(Q)$ contains a cyclic $p'$-subgroup $C$ of order  
$(p^{2a}-1)/\gcd(2,p-1)$. In fact, $A$ acts trivially on $\ZB(Q)$, so $A$ is contained in $\Aut_0(Q)$, the subgroup of all automorphisms of
$Q$ that act trivially on $\ZB(Q)$.
Next, if $Q = E$, then $E/\ZB(E) \lhd \Aut_0(Q) \leq (E/\ZB(E)) \cdot \Sp_{2a}(p)$. The same also holds in the case 
$Q > E$, see \cite[\S1]{Gri}. As $p \nmid |C|$,
$C$ injects into $\Sp(U) \cong \Sp_{2a}(p)$, and we can view $C$ as a subgroup of $\Sp(U)$, with $U := E/\ZB(E) \cong \F_p^{2a}$.


\smallskip
(iii) Here we consider the case $\ell = \ppd(p,2a)$ exists. As $\ell$ divides $|C|$, $C < \Sp(U)$ acts irreducibly on $U$. 
As explained in part (a) of the proof of \cite[Theorem 5]{BNRT}, $|C|$ divides $p^a+1$. 
However, $|C| = (p^{2a}-1)/\gcd(2,p-1)$, so we obtain $p^a-1 \leq \gcd(2,p-1)$, and so $d = p^a \leq 3$, which is excluded.

Next we consider the case $a=1$ and $p \geq 5$ a Mersenne prime. Then $C$ is a cyclic subgroup of $\SL_2(p)$. Any such subgroup has 
order $\leq p+1 < (p^2-1)/2$, contrary to the assumptions.

Finally, the case $p^a = 8$ is excluded since $\Sp_6(2)$ does not contain any cyclic subgroup of order $2^6-1$, see \cite{ATLAS}.
\end{proof}

With this Lemma \ref{m43-main}, we have completed the proof of Proposition \ref{p-tensind}.
\end{proof}

\begin{rmk}\label{m43-2}
Here we construct two examples related to the situations in Lemma \ref{m43}. First we give an example of an imprimitive $\C J$-module of dimension $d=p^a$ that satisfies the conditions of Lemma \ref{m43}, but does not 
satisfy ({\bf S}), for any odd prime power $p^a \geq 3$. Consider the $d$-dimensional vector space $V = \C^d$ with basis 
$\{e_v \mid v \in \F_q\}$. Let $H := \mathrm{AGL}_1(q)$ act on this basis as follows: the normal $p$-subgroup $Q_1$ of order $q$ acts 
via translations $e_v \mapsto e_{u+v}$, $u \in \F_q$, and the complement $C:=\{c_\lambda \mid \lambda \in \F_q^\times\}$ 
acts via $e_v \mapsto e_{\lambda v}$. 
Also consider the unique elementary abelian subgroup $Q$ of order $p^q$ of $\GL(V)$ that acts diagonally in the given basis. Then 
$J := Q \rtimes H = (Q \rtimes Q_1) \rtimes C$ acts 
imprimitively on $V$ and has $M_4(J,V) = 3$. (Indeed, $\wedge^2(V)$ is irreducible and $\Sym^2(V)$ is the sum of two 
irreducible submodules. The equality $M_4(J,V) = 3$ then follows from the fact that $c_{-1}$ acts as $1$ on $e_v \otimes e_{-v} + e_{-v} \otimes e_v$ but as $-1$ on $e_v \otimes e_{-v} -e_{-v} \otimes e_v$ for any $0 \neq v \in V$.)

Next, let $d=p=3$ and consider the faithful irreducible representation of the extraspecial $3$-group $P = 3^{1+2}_+$ on $V=\C^3$. It is well known that this 
representation extends to $P \rtimes \SL_2(3)$. Now we can take $J = P \rtimes C_4$ inside $\GL(V)$ and observe that $M_4(J,V) = 3$. 
\end{rmk}

\begin{rmk}If we drop the hypothesis that $p|D_0$, there is a three dimensional $\sA$ in any characteristic $p\neq 3$ whose $\End$ consists of a nonzero tame part and two wild irreducibles.  Start with $\sL_{\psi(x)}$ and form its Kummer direct image $[3]_\star\sL_{\psi(x)}$. This is $I(\infty)$ irreducible. We first show that
$$\Swan_{\infty}(\End([3]_\star\sL_{\psi(x)})=2.$$
In fact, for any $n \ge 1$ prime to $p$, we will have
$$\Swan_{\infty}(\End([n]_\star\sL_{\psi(x)}))=n-1.$$
To see this, use the fact that for $n$ prime to $p$, and any $I(\infty)$ representation $V$, we have 
$$\Swan_{\infty}([n]^\star V)=n\Swan_{\infty}(V).$$
Applied to $\End([n]_\star\sL_{\psi(x)})$, this gives
$$\Swan_{\infty}(\End([n]_\star\sL_{\psi(x)})) = (1/n)\Swan_{\infty}([n]^\star \End([n]_\star \sL_{\psi(x)}))=$$
$$=(1/n)\Swan_{\infty}(\End([n]^\star [n]_\star \sL_{\psi(x)})).$$
But $[n]^\star [n]_\star \sL_{\psi(x)}$ is the direct sum $\oplus_{\zeta \in \mu_n}\sL_{\psi(\zeta x)}$, whose $\End$ is
$$\oplus_{(\zeta_1,\zeta_2) \in \mu_n \times \mu_n}\sL_{\psi((\zeta_1-\zeta_2) x)},$$
whose Swan conductor is visibly $n(n-1)$.

In fact, $\End([n]_\star\sL_{\psi(x)})$ is the direct sum of the tame piece $[n]_\star(\triv)$ with the direct sum of the $n-1$ irreducible wild summands
$$\oplus_{\zeta \neq 1, \zeta \in \mu_n}[n]_\star(\sL_{\psi((1-\zeta) x}),$$
each of which has Swan conductor one. To see this, we repeat the argument in Lemma \ref{swan2irred}. Denote by $I(n)$ the unique subgroup of $ I(\infty)$ of index $n$. Because $I(n) \lhd I(\infty)$ is a normal subgroup, the induction $[n]_\star\sL_{\psi(x)}$ has its character supported in $I(n)$. Hence also $\End([n]^\star [n]_\star \sL_{\psi(x)})$ has its character supported in $I(n)$. Similarly, each term [$n]_\star\sL_{\psi((1-\zeta) x)}$ has its character supported in $I(n)$. So it suffices to check that the  two sides of the asserted identity have, after $[n]^\star$, the same character, which is visibly the case.
\end{rmk}

\subsection*{3E. Interesting special cases}
For an integer $N \ge 2$ prime to $p$, we have the local system $\sF_N$ on $\A^1$ of rank $D=N-1$ in characteristic $p$ attached to the family of 
exponential sums
$$t \mapsto -\sum_x \psi(x^N + tx).$$
One knows that $\sF_N$ is the Kummer pullback
$$\sF_N \cong [N]^\star \sK l(\psi; {\rm all\ nontrivial\ }\chi \ {\rm with\ }\chi^N=\triv).$$
This $\sK l$ is visibly primitive (i.e., not Kummer induced), and thus satisfies ({\bf S+}) so long as $D:=N-1$ is not $4$ (or $8$, if $p=2$). We expect that $\sF_5$ satisfies ({\bf S+}), but we do not know how to prove it. We also note that $\sF_N$ itself is primitive when $D$ is not a power of $p$, but is imprimitive (indeed its
$G_{\geo}$ is a finite $p$-group) when $D$ is any power of $p$.

For an integer $D \ge 2$ prime to $p$, and a nontrivial multiplicative character $\chi$, we have the local system $\sG_D$ on $\A^1$ of rank $D$ in characteristic $p$ attached to the family of 
exponential sums
$$t \mapsto -\sum_x \psi(x^D + tx)\chi(x).$$
One knows that $\sF_N$ is the Kummer pullback
$$\sG_D \cong [D]^\star \sH(\psi; {\rm all\ }\chi \ {\rm with\ }\chi^D=\triv;\rho),$$
for any $\rho$ with $\rho^D=\chi$.
One knows \cite[Lemma 1.1]{Ka-RL-T-Co3} that $\sG_D$ is primitive for any $D$ prime to $p$ and any nontrivial $\chi$, so a fortiori
$\sH$ is primitive as well. Then by Theorem 2.3, applied in the $W=D-1$ case, $\sH$  satisfies ({\bf S+}) whenever $D $ is prime to $p$. 

\subsection*{3F.  Another ({\bf S+}) result}
In this section, we work in a fixed characteristic $p$, and we denote by $q$ a positive power of $p$.
\begin{thm}\label{qwild}Let $\sH$ be a hypergeometric sheaf in characteristic $p$ of type $(D, D-q)$, with $D > q=p^a$. 
Equivalently,  $\sH$  is of type
$(D,m)$ with $D >m >0$ and wild part of dimension $W=q=p^a$. If $D$ is a power, let $n$ be the largest integer such that $D$ is an $n^{\mathrm {th}}$ power,
and suppose we have the inequality $W \ge n$.
Then $\sH$ has  {\rm ({\bf S+})}.
\end{thm}
\begin{proof}We first show that such an $\sH$ is primitive. To see that it cannot be Kummer induced, let $d \ge 2$ be the (necessarily prime to $p$)
degree of the Kummer induction. Then $d|D$ and $d|m$, and hence $d|q$, a contradiction since $d$ is prime to $p$. 

We next show that $\sH$ is not Belyi induced. Looking at the three types of Belyi induction in \cite[Prop. 1.2]{ Ka-RL-T-2Co1}, we see that whenever there are fewer ``downstairs" characters than ``upstairs" characters, the difference, i.e. $W$, is always
of the form 
$$d_0p^r - d_0= d_0(p^r -1),$$
for some $d_0$ prime to $p$ and some positive power $p^r$ of $p$. This difference is prime to $p$, so cannot be $q$.

We next observe that $\sH$ is tensor indecomposable, since already its $I(\infty)$ representation is tensor indecomposable, cf. \cite[Prop. 10.1 or Thm. 2.1]{Ka-RL-T-2Co1}. This applies when $D \neq 4$. 

It also applies in the two cases $D=4$ and either $p=2$ and $m \ge 2$ or $p$ odd
and $m \neq 2$. Let us see that this is good enough for us. We are working in characteristic $p$, with $D=4 >q$. So either we are in
characteristic $p=2$ and $W=2$, which has $m=2$, an allowed case, or we are in characteristic $p=3$, $W=3$, and $m=1$, another allowed case.

If $D$ is not a power, then $\sH$ cannot be tensor induced, and we are done.
If $D$ is a power, recall that $n$ is the largest integer such that $D$ is an $n^{\mathrm {th}}$ power, in which case we assume 
$W \ge n$.
Suppose that $D$ is an $r^{\mathrm {th}}$ power, and that $\sH$ is $r$-tensor induced. Then $W \ge r$, and, as explained in the first paragraph of the proof of Lemma \ref{tametoSn}, the composite map from $\pi_1$ to $S_r$ is tame at both $0$ and $\infty$, so its image is (the cyclic group generated by) an $r$-cycle, and $r$ is prime to $p$. So we would find that the Kummer pullback $[r]^\star \sH$ is tensor decomposable (of a very specific shape, but this will not matter). The key point is that the wild part $\Wild$ of the $I(\infty)$ representation of $\sH$, having dimension $q$ is irreducible on $P(\infty)$. Therefore all of its Kummer pullbacks, e.g. $[r]^\star \Wild$, are irreducible
on $P(\infty)$, and hence a fortiori on $I(\infty)$. So we have only to apply \cite[Prop. 10.1 or Thm. 2.1]{Ka-RL-T-2Co1} to know that the
$I(\infty)$ representation of $[r]^\star \sH$ is tensor indecomposable.
\end{proof}

\section{General results on $G_\geo$}\label{sec:bounds}
In this section, we consider a ($\overline{\Q_\ell}$-adic) hypergeometric sheaf, 
$$\sH:=\sH yp_\psi(\chi_1,\ldots ,\chi_D;\rho_1,\ldots , \rho_m)$$ 
of type $(D,m)$ with $D > m \geq 0$, defined over some finite subfield of $\overline{\F_p}$, $p \neq \ell$, and  write
$W :=D-m$.
The $I(\infty)$ representation on $\sH$ is then the direct sum of a tame part of rank $m$ and a totally wild part of rank $W$,
all of whose $\infty$-breaks are $1/W$. Let us denote by
\begin{equation}\label{eq:j1}
  J:={\rm the\ image\ of\ }I(\infty) {\rm\ on\ }\sH.
\end{equation}  
One knows that $J$ is a finite group if and only if the $\rho_j$ are all distinct, and that $\sH$ is geometrically irreducible if and only if 
none of $\chi_i$ is among the $\rho_j$.


\begin{thm}\label{generation}
Let $\sH$ be an irreducible $\overline{\Q_\ell}$-hypergeometric sheaf on $\G_m/\overline{\F_p}$, with $p\neq \ell$, and of type $(D,m)$ with 
$W:= D-m \ge 2$. Denote by $G_0$ the Zariski closure inside the geometric monodromy group $G_{\geo}$ of the normal subgroup generated by all 
$G_{\geo}$-conjugates of the image of $I(0)$. Then $G_0=G_{\geo}$.
In particular, if $G_\geo$ is finite then it is generated by all $G_\geo$-conjugates of the image of $I(0)$, and $G_\geo= \OB^p(G_\geo)$.
\end{thm}
\begin{proof}Let $K:=G_{\geo}/G_0$. Because $\sH$ is geometrically irreducible, $G_{\geo}$ has a 
faithful irreducible representation, and hence is reductive. Therefore its quotient $K$ is reductive. 

Suppose $K$ is nontrivial. Then it has at least one nontrivial irreducible representation, say $\rho$. View $\rho$ as a representation of $\pi_1(\G_m/\overline{\F_p})$. So viewed, 
$\rho$ is trivial on $I(0)$, so may be viewed as a lisse $\overline{\Q_\ell}$-sheaf $\sF_\rho$ on the affine line $\A^1/\overline{\F_p}$ which is irreducible and nontrivial. Therefore $H^i_c(\A^1/\overline{\F_p},\sF_\rho)=0$ for $i \neq 1$. By the Euler-Poincare formula \cite[2.3.1]{Ka-GKM},
$$\chi_c(\A^1/\overline{\F_p},\sF_\rho)=\rank(\sF_\rho) -\Swan_\infty(\sF_\rho),$$
and hence
$$h^1_c(\A^1/\overline{\F_p},\sF_\rho) =\Swan_\infty(\sF_\rho)-\rank(\sF_\rho).$$
As $h^1_c \ge 0$, we find that
$$\Swan_\infty(\sF_\rho) \ge \rank(\sF_\rho).$$
On the other hand, the upper numbering subgroup $I(\infty)^{1/W + \epsilon}$ dies in $G_{\geo}$ for all $ \epsilon > 0$. So a fortiori, it dies in $K$, and hence all $\infty$-slopes of $\sF_\rho$ are $\le 1/W$. Hence
$$\Swan_\infty(\sF_\rho) \le (1/W)\rank(\sF_\rho) \le (1/2)\rank(\sF_\rho) ,$$
a contradiction.

In the case $G_\geo$ is finite, $G_0$ is the normal closure of the image of $I(0)$, and is contained in 
the subgroup $\OB^p(G_\geo)$ generated by
all $p'$-elements of $G_\geo$, whence the statements follow.
\end{proof}

If we replace $0$ by $\infty$ in Theorem \ref{generation}, we get the following result.

\begin{lem}\label{biginftygeneration} Let $\sH$ be an irreducible $\overline{\Q_\ell}$-hypergeometric sheaf on $\G_m/\overline{\F_p}$. Denote by $G_\infty$ the Zariski closure inside the geometric monodromy group $G_{\geo}$ of the normal subgroup generated by all 
$G_{\geo}$-conjugates of the image $J$ of $I(\infty)$. Then $G_\infty=G_{\geo}$.
\end{lem}
\begin{proof}In this case, representations of the quotient $G_{\geo}/G_\infty$ correspond to lisse sheaves on $\P^1 \setminus \{0\}$ which are tame at $0$, and any such is trivial.
\end{proof}

Here is another companion result to Theorem \ref{generation}.

\begin{thm}\label{inftygeneration}Let $\sH$ be an irreducible $\overline{\Q_\ell}$-hypergeometric sheaf on $\G_m/\overline{\F_p}$ definable on $\G_m/\F_q$ for some finite extension $\F_q/\F_p$, with $p \neq \ell$, and of type $(D,m)$ with $D>m$. Denote by $G_{P(\infty)}$ the Zariski closure inside the geometric monodromy group $G_{\geo}$ of the normal subgroup generated by all 
$G_{\geo}$-conjugates of the image of the wild inertia group $P(\infty)$. Then $G_{\geo}/G_{P(\infty)}$ is a finite cyclic group of order prime to $p$.
\end{thm}
\begin{proof}Let $K:=G_{\geo}/G_{P(\infty)}$. Because $\sH$ is definable on $\G_m/\F_q$, one knows \cite[8.4.2 (4)]{Ka-ESDE} it is pure (of weight $D+m-1$). It is geometrically, and hence arithmetically irreducible; therefore by \cite[1.3.9]{De-Weil II} $G_{\geo}$ is a semisimple group (in the sense that its identity component $G_{\geo}^0$ is semisimple). Therefore the quotient $K$ is semisimple.
Let $V$ be an irreducible representation of $K$. Then $V$ is given by a geometrically irreducible lisse sheaf $\sF$ on  $\G_m/\overline{\F_p}$ which is tame at both $0$ and $\infty$. As a representation of $\pi_1(\G_m/\overline{\F_p})$, it factors through
the quotient $\pi_1(\G_m/\overline{\F_p})^{{\rm tame\ at\ }0,\infty}$, which is the pro-cyclic group $\prod_{\ell \neq p}\Z_\ell(1)$ of 
pro-order prime to $p$. Any irreducible representation of this group is one-dimensional. Therefore $\sF$ is lisse of rank one, and tame at $0$ and $\infty$. Because $K$ is semisimple, it admits a faithful finite dimensional representation, which is necessarity a direct sum of rank one sheaves $\sF$ as above. Therefore $K$ embeds into a finite product of groups $\GL_1(\overline{\Q_\ell})$. Thus $K$ is abelian, and therefore (being semisimple) is finite. But the image of $\pi_1(\G_m/\overline{\F_p})$ in $K$ is Zariski dense (this already being true for its image in $G_{\geo}$). Therefore $K$ is a finite quotient of  $\pi_1(\G_m/\overline{\F_p})^{{\rm tame\ at\ }0,\infty}$, hence is cyclic of order prime to $p$.
\end{proof}

\begin{prop}\label{p-center}
Let $\sH$ be an {\rm (}irreducible{\rm )} hypergeometric sheaf of type $(D,m)$ in characteristic $p$, 
with $D > m$ and with finite geometric monodromy group
$G=G_\geo$. Then the following statements hold for the image $Q$ of $P(\infty)$ in $G$:
\begin{enumerate}[\rm(i)]
\item If $\sH$ is not Kloosterman, i.e. if $m > 0$, then $Q \cap \ZB(G) = 1$.
\item Suppose $\sH$ is Kloosterman and $D > 1$. Then $Q \not\leq \ZB(G)$. If $p \nmid D$, then 
$Q \cap \ZB(G)=1$. If $p|D$ then either $Q \cap \ZB(G)=1$ or $Q \cap \ZB(G) \cong C_p$.
\item If $D>1$, then $1 \neq Q/(Q \cap \ZB(G)) \hookrightarrow G/\ZB(G)$ and $p$ divides $|G/\ZB(G)|$.
\item If $D-m \geq 2$, the determinant of $G$ is a $p'$-group. If moreover $p \nmid D$, then $\ZB(G)$ is a $p'$-group.
\item Suppose $p=2$.
Then the trace of any element $g \in G$ on $\sH$ is $2$-rational {\rm (i.e. lies in a cyclotomic field $\Q(\zeta_N)$ for some odd integer $N$)}; in particular, the $2$-part of $|\ZB(G)|$ is at most $2$.
\end{enumerate}
\end{prop}

\begin{proof}
(a) For (i), note that if $g \in Q \cap \ZB(G)$, then $g$ acts as a scalar on $\sH$ and trivially on the (nonzero) tame part, hence $g = 1$.

Suppose now that $\sH$ is Kloosterman. If $Q \leq \ZB(G)$, then $Q$ acts as scalars on
$\sH$. Hence, the $Q$-module $\sH$ is a direct sum of $D$ copies of a $1$-dimensional module. But 
the wild part, which is $\sH$ in this case, is a direct sum of pairwise non-isomorphic simple $Q$-modules. So
$D=1$.

Next assume that $\sH$ is Kloosterman and $p \nmid D$. Then by \cite[proof of Lemma 1.2]{Ka-RL-T-Co3}, we know that $Q$ can be identified 
with the additive group of the field $\F_{p^a}$, where $\F_{p^a} = \F_p(\zeta_D)$ and $\zeta_D \in \overline{\F_p}^\times$ has order $D$.
Moreover, a generator of the tame quotient acts via conjugation on $Q$ as 
multiplication by $\zeta_D \neq 1$. Hence, if $g \in Q \cap \ZB(G)$, then $g\zeta_D=g$ and $g=0$ in $Q$ viewed as $ \F_p(\zeta_D)$, as stated.

\smallskip
(b) Now we consider the case $\sH$ is Kloosterman and $p|D$.  
Recall \cite[8.6.3]{Ka-ESDE} that the action of $I(\infty)$ on a Kloosterman $\sH$ is uniquely determined by the rank $D$, up to tensoring with a one-dimensional representation and multiplicative translation.

Consider first the case of Kloosterman sheaf $\sH$ of rank $q=p^f$. To analyze the $Q$-action, we may, by  \cite[8.6.3]{Ka-ESDE}, assume that our $\sH$
is $$\sH :=\sK l({\rm all \ nontrivial\ characters\  of \ order\  dividing\ }q+1).$$
The action of $Q$  does not change if we replace this $\sH$ by its (prime to $p$) Kummer pullback 
$$\sF:=[q+1]^\star \sH,$$
which is the local system on $\A^1/\F_{q^2}$ whose trace function at $t \in k$, $k$ a finite extension of $\F_{q^2}$, is
$$\sF: t \mapsto -\sum_{x \in k}\psi_k(x^{q+1}+tx).$$
By a result of Pink \cite[Corollary 20.3]{KT1}, the geometric monodromy group $G_{\geo,\sF}$ of $\sF$ is a finite $p$-group. 
Using \cite[Prop. 1.4.2]{Ka-LGE}, we see that $G_{\geo,\sF}$ is precisely $Q$. This allows us to apply to $Q$ the known results about $G_{\geo,\sF}$, due to Pink and Sawin.
  
By the result \cite[Corollary 20.2]{KT1} of Pink, the image of $Q$ on $\End(\sH)$ is the additive group $\W_q:=\{t \in \F_{q^4} \mid t + t^{q^2}=0\}$. By the irreducibility of the action of $Q$ on $\sH$, this tells us that $Q/\ZB(Q) \cong \W_q$. To compute the order of $\ZB(Q)$, it suffices to compute the order of $Q$. In the first part of the proof of
Sawin's $p$-odd result   \cite[top of page 841]{KT1}, valid in any characteristic, he writes down an explicit description of the action of $Q$ which shows that its order is $pq^2$. Therefore $\ZB(Q) \cong C_p$. Because the $Q$-action is irreducible and faithful,  $\ZB(Q)\cong C_p$ acts by scalars, and faithfully. But any element of $Q$ that acts by a scalar lies in $\ZB(G)$. Thus  $\ZB(Q) \le  \ZB(G)$. Conversely, any element of $Q\cap \ZB(G)$ acts as a scalar, so (by the irrreducibility of the $Q$ action) lies in $\ZB(Q)$.
So in this rank $q$ case, we have $Q\cap \ZB(G) \cong C_p$.

Now we consider the case when our Kloosterman sheaf $\sH$ has rank $dq$ with $d$ prime to $p$. 
As $\ZB(G)$ acts as scalars on $\sH$, 
$Q \cap \ZB(G) = \langle g \rangle$ is cyclic. We also know that $\sH$ is a 
direct sum of $d$  pairwise non-isomorphic simple $Q$-modules, $\Wild_i$ of dimension $q$, $1 \leq i \leq d$.
As $Q$ maps to the image $Q_i$  of $P(\infty)$ on $\Wild_i$, it follows from the preceding rank $q$ result that $g^p$ acts trivially on 
every $\Wild_i$, and hence that $g^p$ acts trivially on $\sH \cong \oplus^d_{i=1} \Wild_i$. By the faithfulness of the action of $Q$ on $\sH$, 
$g^p=1$. Therefore either $Q\cap \ZB(G)$ is trivial, or $Q\cap \ZB(G) \cong C_p$.

\smallskip
(c) For (iii), we note that $\sH$ is not tame at $\infty$, hence $Q \neq 1$. It follows from (i) and (ii) that $Q \not\leq \ZB(G)$, 
and so $1 \neq Q/(Q \cap \ZB(G)) \hookrightarrow G/\ZB(G)$. In particular, $p$ divides $|G/\ZB(G)|$.

\smallskip
(d) Now we establish (iv). By \cite[8.11.6]{Ka-ESDE} $\det(G)$ is equal to the product of the $D$ upstairs characters of $\sH$, 
whence it is a $p'$-group. In particular, for any $g \in G_\geo$, $\det(g)$ is a $p'$-root of unity. If $z \in \ZB(G)$ acts as the scalar 
$\alpha \in \C^\times$, then $\det(z) =\alpha^D$ is a $p'$-root of unity. So if $p \nmid D$, then $\alpha$ is a $p'$-root of unity, and hence 
$z$ has $p'$-order. Thus $\ZB(G)$ is a $p'$-group.

\smallskip
Finally, for (v) we note that any additive character of a finite field of characteristic $2$ takes only integer values $\pm 1$. 
Sp for any finite extension $k/\F_2$, and any multiplicative character $\chi$ of $k^\times$, the Gauss sum $\Gauss(\psi_k,\chi)$ lies
in the field $\Q(\chi)$, which is $\Q(\zeta_N)$ for some odd integer $N$.  View our $\sH$ on $\G_m/\F_q$ for a finite extension $\F_q/\F_2$ such that all the ``upstairs" characters $\chi_i$ and all the ``downstairs" characters $\rho_j$ of $\sH$ are characters of $F_q^\times$, and define
$$\Lambda :=\prod_i \chi_i, \ \ A:=\Lambda((-1)^{D-1})q^{D(D-1)/2}\prod_{i,j}(-\Gauss(\psi_{\F_q},\chi_i/\rho_j)).$$
Notice that $A$ lies in the field $\Q(\zeta_{q-1})$, itself $\Q(\zeta_N)$ for some odd integer $N$.
According to the arithmetic determinant formula \cite[8.12.2]{Ka-ESDE}, we have
$$\det(\sH)\cong \left\{ \begin{array}{ll}\sL_{\Lambda}\otimes A^{\deg/\F_q}, & {\rm\ if\ }D-m \ge 2,\\$$
   \sL_{\psi}\otimes \sL_{\Lambda}\otimes A^{\deg/\F_q}, & {\rm\ if\ }D-m =1.\end{array} \right.$$
The group $G_\geo$ of $\sH$ does not change if we make an extension of the ground field, so we may consider $\sH$ viewed on
$\G_m/\F_{q^D}$. Relative to this ground field, we have
$$\det(\sH)\cong \left\{ \begin{array}{ll}\sL_{\Lambda}\otimes (A^D)^{\deg/\F_{q^D}}, & {\rm\ if\ }D-m \ge 2,\\
   \sL_{\overline{\psi}}\otimes \sL_{\Lambda}\otimes (A^D)^{\deg/\F_{q^D}}, & {\rm\ if\ }D-m =1.\end{array}\right.$$
The key point is that over this ground field, $\sH\otimes A^{-\deg/\F_{q^D}}$ has finite arithmetic determinant,
but all Frobenius traces still lie in $\Q(\zeta_{q-1})$. Because the determinant of $\sH\otimes A^{-\deg/\F_{q^D}}$ is of finite order,
and its $G_{\ari}$ normalizes the irreducible subgroup $G_{\geo}$, $G_{\ari}$ itself is finite. The trace of any element $g \in G$ (indeed of any element $g$ in $G_{\ari}$, being the trace of some Frobenius, is $2$-rational. 
In particular, if a $2$-element
$g \in \ZB(G)$ acts on $\sH$ as scalar $\alpha \in \C^\times$, then $\alpha$ is both a root of unity in some $\Q(\zeta_N)$ with $N$ odd, and 
a $2$-power root of unity, and so $\alpha = \pm 1$.
\end{proof}

\begin{thm}\label{bound1}
In the above situation of \eqref{eq:j1}, suppose that $J$ is a finite group. Let $\F$ be an algebraically closed field 
of characteristic $r=0$ or $r \neq p$, and let
$$\Lambda: J \rightarrow \GL_d(\F)$$
be an $\F J$-representation of dimension $d \ge 1$. If $d <W$, then $\Lambda$ is tame, and the image 
$\Lambda(J)$ is a finite cyclic group of order prime to $p$. If in addition $\Lambda$ is irreducible, then $d=1$.
\end{thm}

\begin{proof}Let
$$\rho_\sH: I(\infty) \rightarrow \GL_n(\overline{\Q_\ell})$$
be the representation of $I(\infty)$ on $\sH$. The highest $\infty$-break of $\rho$ is $1/W$, meaning precisely that the upper numbering subgroup $I(\infty)^{(1/W +)}$ of $I(\infty)$ lies in the kernel of $\rho$. The composite representation
$$\Lambda \circ \rho_\sH: I(\infty) \twoheadrightarrow J \rightarrow  \GL_d(\F)$$
then has $I(\infty)^{(1/W +)}$  in its kernel, and hence has highest slope $\le 1/W$. The Swan conductor of
$\Lambda \circ \rho_\sH$  then satisfies
$$\Swan(\Lambda \circ \rho_\sH) \le {\rm rank }\times {\rm highest\ slope} \le d/W <1,$$
and hence and hence by \cite[1.9]{Ka-GKM} $\Swan(\Lambda \circ \rho_\sH)=0$. Thus $\Lambda \circ \rho_\sH$ is tame, i.e., is a representation of the tame quotient $I(\infty)/P(\infty)$, which is abelian and pro-cyclic, of pro-order prime to $p$. Therefore the image $\Lambda(J)$ is a finite cyclic group of order prime to $p$. If in addition, $\Lambda$ is irreducible, then $d=1$, simply because $I(\infty)/P(\infty)$ is abelian.
\end{proof}

We now give a global version of this result.

\begin{thm}\label{bound1g}Consider a ($\overline{\Q_\ell}$-adic) hypergeometric sheaf, 
$$\sH:=\sH yp_\psi(\chi_1,\ldots ,\chi_D;\rho_1,\ldots , \rho_m)$$ 
of type $(D,m)$ with $D > m \geq 0$, defined over a finite subfield of $\overline{\F_p}$. Suppose that $W :=D-m \geq 2$
and that $\sH$ has finite geometric monodromy group $G_{\geo}$. Suppose further that we are given a finite group $\Gamma$ 
together with a surjective homomorphism
$$\phi: \Gamma \twoheadrightarrow G_{\geo}$$
whose kernel $\Ker(\phi)$ is an abelian group of order prime to $p$.
Let $\F$ be an algebraically closed field 
of characteristic $r=0$ or $r \neq p$, and let
$$\Lambda:\Gamma \rightarrow \GL_d(\F)$$
be an $\F \Gamma$-representation of dimension $d \ge 1$. If $d <W$, then $\Lambda$ is tame, and the image 
$\Lambda(\Gamma)$ is a finite cyclic group of order prime to $p$. If in addition $\Lambda$ is irreducible, then $d=1$.
\end{thm}
\begin{proof}Let us write $\pi_1^{\geo} := \pi_1(\G_m/\overline{\F_p})$, and denote by
$$\rho_\sH: \pi_1^{\geo} \rightarrow \GL_n(\overline{\Q_\ell})$$
the representation which ``is" $\sH$. By definition, $G_{\geo} =\rho_\sH(\pi_1^{\geo})$, and we view $\rho_\sH$ as a
homomorphism
$$\rho_\sH: \pi_1^{\geo} \twoheadrightarrow G_{\geo}.$$
From the short exact sequence
$$1 \rightarrow \Ker(\phi) \rightarrow \Gamma \rightarrow G_{\geo} \rightarrow 1,$$
we see that the obstruction to lifting $\rho_\sH$ to a homomorphism
$$\tilde{\rho}:\pi_1^{\geo} \rightarrow \Gamma$$
lies in the group $H^2(\G_m/\overline{\F_p},\Ker(\phi))=0$, the vanishing because open curves have cohomological dimension $\le 1$, cf. \cite[Cor. 2.7, Exp. IX and Thm. 5.1, Exp. X]{SGA4t3}. Let us choose such a lifting
$$\tilde{\rho}:\pi_1^{\geo} \rightarrow \Gamma.$$
The composite map
$$\pi_1^{\geo} \rightarrow \Gamma \twoheadrightarrow G_{\geo}$$
is tame at $0$, i.e., trivial on $P(0)$, and has highest $\infty$-break $1/W$, i.e., trivial on $I(\infty)^{(1/W +)}$. Because $\Ker(\phi)$ 
has order prime to $p$, the map $\tilde{\rho}$ itself is trivial on the $p$-groups $P(0)$ and $I(\infty)^{(1/W +)}$. Therefore
$\tilde{\rho}$ is tame at $0$ and has  highest $\infty$-break $\le 1/W$. If we now compose $\tilde{\rho}$ with $\Lambda$,
we find that $\Lambda \circ \tilde{\rho}$ is tame at $0$ and has $\Swan_\infty \le d/W$. Hence if $d < W$, then $\Lambda \circ \tilde{\rho}$ has $\Swan_\infty = 0$, hence is tame at both $0$ and $\infty$. But $\pi_1^{\geo,{\rm tame\ at\ }0,\infty}$ is a pro-cyclic group of pro-order prime to $p$. Hence if $d < W$, then the image $\Lambda(\Gamma)$ is a finite cyclic group of order prime to $p$.
In particular, if $\Lambda$ is irreducible and $d < W$, then $d=1$, simply because $\Lambda(\Gamma)$ is abelian.
\end{proof}

\begin{cor}\label{bound1g-bis}In the situation of the theorem above, the group $\Gamma$ has no
faithful  $\F \Gamma$-representation of dimension $d < W$. In particular, $G_{\geo}$ itself has no
faithful  $\F \Gamma$-representation of dimension $d < W$.
\end{cor}
\begin{proof}By Theorem \ref{bound1g}, any such representation $\Lambda$ has image an abelian group. But $\Gamma$ is not abelian,
indeed its quotient $G_{\geo}$ is not abelian, as it has an irreducble $\overline{\Q_\ell}$ representation of dimension $n \ge W \ge 2$, namely the one coming from $\sH$.
\end{proof}

Here is  another application of Theorem \ref{bound1g}.

\begin{thm}\label{bound2}
Let $\sH$ be a hypergeometric sheaf of type $(D,m)$ with $D > m \ge 0$ in characteristic $p$, and let $G$ be the geometric monodromy group
of $\sH$. Suppose that 
\begin{enumerate}[\rm(a)]
\item $G$ is a finite almost quasisimple group: $S \lhd G/\ZB(G) \leq \Aut(S)$ for some finite non-abelian simple group $S$;
\item For some normal subgroup $R$ of $G/\ZB(G)$ containing $S$, $R$ admits either a $d$-dimensional linear representation 
$\Phi:R \to \GL_d(\F)$, or an $e$-dimensional projective representation $\Psi:R \to \PGL_e(\F)$,
over an algebraically closed field $\F$ of characteristic $\neq p$ and nontrivial over $S$.
\end{enumerate}
Then for the dimension $W=D-m$ of the wild part of $I(\infty)$ on $\sH$ we have
$$W \leq d\cdot [G/\ZB(G):R] \leq d\cdot|\Out(S)|,$$
respectively
$$W \leq (e^2-1) \cdot [G/\ZB(G):R] \leq (e^2-1)\cdot|\Out(S)|.$$
\end{thm}

\begin{proof}
In the case $\Psi$ is given, we note that $\Psi$ is faithful. [Indeed, $\Ker(\Psi) \lhd R$ does not contain $S$, and so
intersects $S$ trivially by simplicity of $S$. Because both $S$ and  $\Ker(\Psi)$ are normal in $R$, the commutator $[S,  \Ker(\Psi) ]  \subset S \cap  \Ker(\Psi) =1$. Thus $\Ker(\Psi) \leq \CB_R(S) \le \CB_{\Aut(S)}(S)=1$. Hence
$R$ is embedded in $\PGL(U)$, where $U = \F^d$. Composing this embedding with the faithful action of $\PGL(U)$ on $\End(U)/{\rm scalars}$, we obtain a faithful action of $R$ on a module of dimension $\leq e^2-1$. 
Thus it suffices to prove the bound 
$W \leq d\cdot [G/\ZB(G):R]$ in the case $\Phi: R \to \GL(V)$ is given.

So assume the contrary: $\Phi: R \to \GL(V)$ is faithful with $\dim(V)=d$, but
\begin{equation}\label{bd21}
 W > d \cdot [G/\ZB(G):R].
\end{equation} 
Let $\tilde V$ denote the $\bar{G}$-module $\Ind^{\bar{G}}_R(V)$ for $\bar{G}:=G/\ZB(G)$. 
Note that $\bar{G}$ acts faithfully on $\tilde V$. Indeed, let $K \lhd \bar{G}$ denote the kernel of the action of $\bar{G}$ on $\tilde V$. 
By the construction of $V$ as the induced representation, the $R$-module $\tilde V$ contains $V$ as a submodule. But $S$ acts faithfully on $V$, hence
$S \cap K=1$. As $S \lhd \bar{G}$, it follows that $[S,K]=1$, and so 
$$K \leq \CB_{\bar{G}}(S) \leq \CB_{\Aut(S)}(S)=1.$$  
We also note that
$$\dim(\tilde V) = [\bar{G}:R]\cdot \dim(V) \leq  d \cdot [\bar{G}:R] < W$$
by \eqref{bd21}; in particular, $D \geq W \geq 2$.

Now view $\tilde V$ as a representation of $G$, of dimension $< W$. By Theorem \ref{bound1g}, applied with its $\Gamma$ taken to be $G$, this representation is tame at both $0$ and $\infty$. Thus the image $Q$ in $G$ of  $P(\infty)$ acts trivially on $\tilde V$. But $G/\ZB(G)$ acts faithfully on $\tilde V$. Therefore $Q$ lands in $\ZB(G)$, contradicting Proposition \ref{p-center}(iii). 
\end{proof}

\section{Hypergeometricity results}\label{sec:hypergeom1}
In this section, we consider the question of when a  $\overline{\Q_\ell}$-local system on $\G_m/\overline{\F_p}$, $\ell \neq p$, is given
by a hypergeometric sheaf.
\begin{thm}\label{Brauerp}Let $G$ be a finite group, and
$\phi:\pi_1(\G_m/\overline{\F_p}) \twoheadrightarrow G$
a surjective homomorphism.
Suppose we are given two irreducible representations
$$\Phi_i:G \rightarrow \GL_{D_i}(\overline{\Q_\ell}),\ \ i=1,2.$$
Let $\sH_i$ the the local system on $\G_m/\overline{\F_p}$ which realizes $\Phi_i$; i.e., $\sH_i$ is the local system given by
the composite
$$\pi_1(\G_m/\overline{\F_p}) \xrightarrow{\phi} G  \xrightarrow{\Phi_i} \GL_{D_i}(\overline{\Q_\ell}),\ \ i=1,2,$$
Suppose that there exists an integer $a$ such that for every $g$ in $\phi(P(0)) \cup \phi(P(\infty))$, we have
$$ (**_p)\ \ \ \ \Trace(\Phi_2(g)) =a +\Trace(\Phi_1(g)).$$
Then the following conditions are equivalent.
\begin{itemize}
\item[(i)]$\sH_1$ is hypergeometric, of type $(D_1,m_1)$ with $D_1 > m_1$.
\item[(ii)]$\sH_2$ is hypergeometric, of type $(D_2,m_2)$ with $D_2 > m_2$; moreover $(D_2,m_2) =(D_1+a,m_1+a)$.
\end{itemize}
\end{thm}
\begin{proof}
By symmetry, it suffices to show that (i) implies (ii). Because $D_1 > m_1$, $\sH_1$ is tame at $0$. Apply $(**_p)$ to the image 
$P_G(0)=\phi(P(0))$ of the wild inertia group $P(0)$ at $0$. For any $\gamma \in P_G(0)$, we have
$$\Trace(\Phi_1(\gamma))=D_1,$$
simply because $\sH_1$ is tame at $0$. Therefore we have 
$$\Trace(\Phi_2(\gamma))=D_1+a$$
for every $\gamma \in P_G(0)$. Thus the $P_G(0)$-representation on $\sH_2$ has the same trace at  $D_1+a$ copies of the trivial
representation, and hence $P_G(0)$ acts trivially on $\sH_2$, and $\sH_2$ has rank $D_1+a$. In particular, $\sH_2$ is tame at $0$.

We next consider the action of the image $P_G(\infty)=\phi(P(\infty))$  of the wild inertia group $P(\infty)$ at $\infty$. Because $\sH_1$ is hypergeometric and tame at $0$ (and lisse on $\G_m$), its $P_G(\infty)$-representation has Swan conductor $\Swan_\infty(\sH_1)=1$. [Recall that the Swan conductor of a representation of the inertia group $I(\infty)$ is defined completely in 
terms of its restriction to $P(\infty)$ and of the restriction to $P(\infty)$ of the upper numbering filtration on $I(\infty)$, cf  \cite[1.7]{Ka-GKM}.] From the equality $(**_p)$ applied to elements of $P_G(\infty) < G$, we see that $\sH_2$ as a  $P(\infty)$-representation is isomorphic as a virtual representation to direct sum of $\sH_1$ as a $P(\infty)$-representation and $a$ copies of the trivial representation. As Swan conductors pass to virtual representations, and trivial representations have Swan conductor zero,
it follows that $\Swan_\infty(\sH_2)=1$. By \cite[Theorem 8.5.3]{Ka-ESDE}, it follows that $\sH_2$, being irreducible, tame at $0$ and with $\Swan_\infty(\sH_2)=1$, is hypergeometric, of type $(D_2,m_2)$ with $D_2 > m_2$. We have already seen, from the $P_G(0)$ analysis, that $D_2=D_1 +a$.

We will now show that $m_2 =m_1 +a$. Break the $P_G(\infty)$-representations of $\sH_1$ and $\sH_2$ into tame and totally wild parts, say
$$\sH_1 = \Wild_1 +m_1\triv, \ \  \sH_2 = \Wild_2 + m_2\triv.$$
From the equality of traces on $P_G(\infty)$, we have an equality of virtual representations of $P_G(\infty)$,
$$\Wild_2 +m_2\triv = \Wild_1 +m_1\triv +a\triv,$$
which we rewrite as
$$\Wild_2- \Wild_1= (m_1 +a -m_2)\triv.$$

If, for example, $m_1 +a -m_2 \ge 0$, we get an isomorphism of representations
$$\Wild_2 = \Wild_1 + (m_1 +a -m_2)\triv.$$
But $\Wild_2$ is totally wild, hence it has no trivial components, and hence $m_1 +a -m_2=0$. Similarly, if $m_1 +a -m_2 \le 0$,
then we get an isomorphism of representations
$$\Wild_1 = \Wild_2 + (m_2 -a -m_1)\triv,$$
and again infer that $m_1 +a -m_2=0$.
\end{proof}

Some particularly useful consequences of Theorem \ref{Brauerp} are the following:

\begin{cor}\label{p'-center}Let $G$ be a finite group, and
$\phi:\pi_1(\G_m/\overline{\F_p}) \twoheadrightarrow G$
a surjective homomorphism. Suppose $G = Z \times H$ for a $p'$-subgroup $Z \leq \ZB(G)$ and $H \leq G$.
Denote by $\pi: G \surj H$ the projection and $\iota: H \rightarrow G$ the inclusion. 
Suppose we are given an irreducible representation $\Phi:G \rightarrow \GL_{D}(\overline{\Q_\ell})$, and let $\sH$ and 
$\sH'$ be the local systems on $\G_m/\overline{\F_p}$ given by
$$\pi_1(\G_m/\overline{\F_p}) \xrightarrow{\phi} G  \xrightarrow{\Phi} \GL_{D}(\overline{\Q_\ell}) \text{ and }
   \pi_1(\G_m/\overline{\F_p}) \xrightarrow{\phi} G  \xrightarrow{\pi} H \xrightarrow{\iota} G  \xrightarrow{\Phi} \GL_{D}(\overline{\Q_\ell}),$$
respectively. Then $\sH$ is hypergeometric, of type $(D,m)$ with $D > m$, if and only if 
$\sH'$ is hypergeometric, of type $(D,m)$ with $D > m$.
\end{cor}

\begin{proof}
Note that $(\iota \circ \pi)(h)=h$ for all $h \in H$, and, furthermore, any $p$-element $g \in G$ is contained in $H$ as $p \nmid |Z|$.
Moreover, $(\Phi \circ \iota \circ \pi)(G) = \Phi(H)$ is irreducible since $Z \leq \ZB(G)$. Now for any $p$-element $g \in G$ we have 
$(\Phi \circ \iota \circ \pi)(g) = \Phi(g)$. Hence 
$(**_p)$ holds with $a=0$, and the statement follows from Theorem \ref{Brauerp}.
\end{proof}

\begin{cor}\label{dual}
Let $G$ be a finite group, and
$\phi:\pi_1(\G_m/\overline{\F_p}) \twoheadrightarrow G$
a surjective homomorphism. Suppose we are given an irreducible representation $\Phi:G \rightarrow \GL_{D}(\overline{\Q_\ell})$ and 
a tame representation $\Lambda:G \to \GL_1(\overline{\Q_\ell})$ of odd order such that $\Phi^* \cong \Phi \otimes \Lambda$.
Then there exists a tame representation $\Theta:G \to \GL_1(\overline{\Q_\ell})$ such that $\Phi \otimes \Theta$
is self-dual. Let $\sH$ and $\sH'$ be the local systems on $\G_m/\overline{\F_p}$ given by
$$\pi_1(\G_m/\overline{\F_p}) \xrightarrow{\phi} G  \xrightarrow{\Phi} \GL_{D}(\overline{\Q_\ell}) \text{ and }
   \pi_1(\G_m/\overline{\F_p}) \xrightarrow{\phi} G  \xrightarrow{\Phi\otimes \Theta} \GL_{D}(\overline{\Q_\ell}),$$
respectively. Then $\sH$ is hypergeometric, of type $(D,m)$ with $D > m$, if and only if 
$\sH'$ is hypergeometric, of type $(D,m)$ with $D > m$.
\end{cor}

\begin{proof}
Let $N=2m+1$ denote the order of $\Lambda$, so that $\gcd(N,2p)=1$. Then  
$$(\Phi \otimes \Lambda^{m+1})^* \cong \Phi^* \otimes \Lambda^{-m-1} \cong  \Phi \otimes \Lambda^{-m} \cong 
    \Phi \otimes \Lambda^{m+1},$$
i.e. we can take $\Theta = \Lambda^m$. Now, for any $p$-element $g \in G$, $\Theta(g) = 1$ as $\Theta$ is tame, whence
$(\Phi \otimes \Theta)(g) = \Phi(g)$. Hence 
$(**_p)$ holds with $a=0$, and the statement follows from Theorem \ref{Brauerp}.
\end{proof}

In connection to the last statement, we prove the following useful fact:

\begin{lem}\label{dual-mult}
Let $\sH$ be a hypergeometric sheaf of type $(D,m)$, where $D > m \geq 1$ and $D > 2$, with finite geometric monodromy group
$G = G_\geo$. Let $\Phi:G \to \GL_D(\overline{\Q_\ell})$ denote the corresponding representation, and assume that,
for the image $Q$ of $P(\infty)$ in $G$,
$(\Phi|_Q)^* \cong (\Phi|_Q) \otimes \Lambda$ for some $1$-dimensional $Q$-representation $\Lambda$. 
Then $\Lambda$ is trivial, unless $m=1$, $D$ is a power of $p$, and $Q$ is elementary abelian of order $D$. In all cases, 
the $Q$-representation $\Phi|_Q$ is self-dual.
\end{lem}

\begin{proof}
Let $\varphi$ denote the character of $\Phi$. Write the dimension $w:=D-m$ of $\Wild$ as $tp^n$ with $t$ prime to $p$ and $n \ge 0$. Then one knows \cite[1.14]{Ka-GKM} that $\varphi|_Q = \sum^t_{i=1}\theta_i + m\cdot 1_Q$, for $t$ pairwise distinct nontrivial
irreducible characters $\theta_i$ of $Q$, each of degree $p^n$,  that are permuted transitively by $J$, the image of $I(\infty)$ in $G$. By assumption, there exists $\lambda \in \Irr(Q)$ 
such that $\overline\varphi|_{Q} = \varphi|_{Q}\cdot\lambda$, whence
$$\sum^t_{i=1}\overline\theta_i + m\cdot 1_Q = \overline\varphi|_Q = 
    \varphi|_Q \cdot \lambda=\sum^t_{i=1}\theta_i\cdot\lambda|_Q + m\lambda.$$
Note that all the characters $\overline\theta_i$ are still irreducible and distinct. Hence, if $m \geq 2$, we must have that
$\lambda = 1_Q$, and so $\Phi|_Q$ is self-dual.

Suppose now that $m=1$ but $\lambda \neq 1_Q$. Then there exists some $i$ such that 
$\lambda = \overline\theta_i$. Therefore $\theta_i$ has degree $1$. Therefore $p^n=1$, $t=w$, and every $\theta_j$ is a linear character of order $p$. Therefore $Q$ is elementary abelian, and  $\theta_j(1)=1$ for all $j$. We now have that 
$\theta_i\cdot\overline\sigma = \sigma$ for $\sigma := \varphi|_Q$. Conjugating this equality by elements in $J$ which acts
transitively on $\{\theta_1, \ldots,\theta_t\}$, we see that
\begin{equation}\label{dm11}
  \theta_j\cdot\overline\sigma=\sigma
\end{equation}   
for all $1 \leq j \leq t$. It follows that
$$\sigma\overline\sigma = \bigl(1_Q + \sum^t_{j=1}\theta_j\bigr)\overline\sigma=
    \overline\sigma+\sum^t_{j=1}\theta_j\cdot\overline\sigma = \overline\sigma +t\sigma.$$
Taking complex conjugate and subtracting, we get $(t-1)(\overline\sigma-\sigma)=0$. But note that $t=w=D-m=D-1 > 1$ in this case, so 
$\overline\sigma=\sigma$. Next we show that $\sigma=\varphi|_Q$ is the regular character ${\mathbf {reg}}_Q$ and 
that $D=|Q|$ in this case.
Indeed, consider any $g \in Q$ with $\sigma(g) \neq 0$. Then by \eqref{dm11} for the root of unity 
 $z:= \theta_j(g)$ we have $\sigma(g)=z\overline\sigma(g)$, for all $j$. It follows that $1+tz=z(1+t\bar{z})=z+t$, and so
 $z=1$ as $t > 1$. Thus $\sigma(g) = t+1$ and $g \in \Ker(\Phi)=1$. Thus $\sigma(x)=0$ for all $1 \neq x \in Q$. Now
$$|Q| = |Q| \cdot [\sigma,1_Q]_Q = \sum_{x \in Q}\sigma(x) = \sigma(1) = t+1=D.$$
As the $t+1$ characters $1_Q,\theta_1, \ldots,\theta_t$ are all distinct, we conclude that $\sigma={\mathbf {reg}}_Q$. 
\end{proof}

\section{Almost quasisimple groups containing elements with simple spectra}\label{sec:aqs}
The goal of this section is to describe triples $(G,V,g)$ subject to the following condition:

\begin{tabular}{ll}
$\Cstar$: & \begin{tabular}{l} $G$ is an almost quasisimple finite group, with $S$ the unique non-abelian composition\\ factor,
   $V$ a faithful irreducible $\C G$-module, and $g \in G$ has simple spectrum on $V$. \end{tabular} \end{tabular}

With $G$ as in $\Cstar$, let $E(G)$ denote the {\it layer} of $G$, so that $E(G)$ is quasisimple and $S \cong E(G)/\ZB(E(G))$. On the other hand,
$G/\ZB(G)$ is almost simple: $S \lhd G/\ZB(G) \leq \Aut(S)$. We will frequently identify $G$ with its image in $\GL(V)$.
Let $\dl(S)$ denote the smallest degree of faithful projective irreducible complex 
representations of $S$, and let $\obar(g)$ denote the order of the element $g\ZB(G)$ in $G/\ZB(G)$. 
Adopting the notation of \cite{GMPS}, let $\meo(X)$ denote the largest order of elements in a finite group $X$. 
An element $g \in G \leq \GL(V)$ is called an $\ssp$-element, or an element with simple spectrum, if the multiplicity of any eigenvalue of
$g$ acting on $V$ is $1$.
(Note that in $\Cstar$, we do {\it not} (yet) assume that $V|_{E(G)}$ is irreducible.)

We begin with a useful observation:

\begin{lem}\label{order}
In the situation of $\Cstar$, we have 
$$\dl(S) \leq \dim(V) \leq \obar(g) \leq \meo(G/\ZB(G)) \leq \meo(\Aut(S)).$$
\end{lem}

\begin{proof}
For the first inequality, let $U$ denote an irreducible summand of the $\C E(G)$-module $V$. Since $G$ is almost quasisimple, 
$\ZB(E(G)) \leq \CB_G(S) = \ZB(G)$. As the $G$-module $V$ is faithful and irreducible, it follows that $\ZB(E(G))$ acts faithfully (via scalars)
on $U$, and so $E(G)$ is faithful on $U$. Thus $U$ induces a faithful projective irreducible action of $S$, whence $n:=\dim(V) \geq \dim(U) \geq \dl(S)$.
 
Next, let $\{\eps_1, \ldots ,\eps_n\}$ denote the set of eigenvalues of $g$ acting on $V$, and let $m:=\obar(g)$. Then
$g^m \in \ZB(G)$ acts a scalar $\gamma$ on $V$, hence $\eps_i^m = \gamma$ for all $i$. Since $g$ has simple spectrum on $V$, we conclude
that $n \leq m$, and the statement follows. 
\end{proof}

\subsection*{6A. Non-Lie-type groups}
The goal of this subsection is to address the case where $S=\ABS_n$, the alternating group of degree $n \geq 7$, or one of the $26$ sporadic 
simple groups. (We omit explicit results in the cases $S= \ABS_{5,6}$, since there are too many cases, all of which are tiny and 
can easily be looked up using \cite{GAP}.)
For any partition $\lambda \vdash n$, let $S^\lambda$ denote an irreducible $\C \Sym_n$-module labeled by $\lambda$. In particular,
$S^{(n-1,1)}$ is just the {\it deleted permutation module} of $\Sym_n$.
We will also need to consider the so-called {\it basic spin modules} (acted on faithfully by the double cover
$\hat\ABS_n$), see e.g. \cite[\S2]{KlT}.

The following result extends \cite[Theorem 9.7]{GKT}. (We note that the case $n=6$ of \cite[Theorem 9.7]{GKT} inadvertently omitted a triple
$(G,V,g)$ with $G \cong \ABS_6$, $\dim(V) = 5 = |g|$.)  

\begin{thm}\label{alt}
In the situation of $\Cstar$, assume that $S = \ABS_n$ with $n \geq 8$. Then one of the following statements holds.
\begin{enumerate}[\rm(i)]
\item $E(G) = \ABS_n$ and one of the following holds. \begin{enumerate}[\rm(a)] 
\item $\dim V = n-1$, $V|_{\ABS_n} \cong S^{(n-1,1)}|_{\ABS_n}$, 
and, up to a scalar, $g$ is either an $n$-cycle, or a disjoint product of a $k$-cycle and an $(n-k)$-cycle
for some $1 \leq k \leq n-1$ coprime to $n$.
\item $n=8$, $\dim V = 14$, and, up to a scalar, $g$ is an element of order $15$ in $\ABS_8$. \end{enumerate}
\item $E(G) = \hat{\ABS}_n$ and one of the following holds. \begin{enumerate}[\rm(a)] 
\item $n=8$, $\dim V = 8$, $V|_{E(G)}$ is a basic spin module, and $\obar(g) = 10$, $12$, or $15$.
\item $G/\ZB(G) \cong \ABS_9$, $\dim V = 8$, $V|_{E(G)}$ is a basic spin module, and $\obar(g) = 9$, $10$, $12$, or $15$. 
\item $G/\ZB(G) \cong \Sym_9$, $\dim V = 16$, $V|_{E(G)}$ is the sum of two basic spin modules, and $\obar(g) = 20$. 
\item $G/\ZB(G) \cong \Sym_{10}$, $\dim V = 16$, $V|_{E(G)}$ is a basic spin module, and $\obar(g) = 20$ or $30$. 
\item $G/\ZB(G) \cong \ABS_{11}$, $\dim V = 16$, $V|_{E(G)}$ is a basic spin module, and $\obar(g) = 20$. 
\item $G/\ZB(G) \cong \Sym_{12}$, $\dim V = 32$, $V|_{E(G)}$ is a basic spin module, and $\obar(g) = 60$. 
\end{enumerate}
\end{enumerate}
\end{thm}

\begin{proof}
It is more convenient to work with a modified version $H$ of $G$ which may differ from $G$ only by scalars and whose representation
theory is better understood. If $G/\ZB(G) = S$, we take $H = E(G)$. 

Suppose $G/\ZB(G) \cong \Sym_n$. Then 
there is an element $z \in G$ the conjugation by which induces the same automorphism of $E(G)$ as the one induced by the $2$-cycle 
$(1,2)$. In particular, $z^2$ centralizes $E(G)$ and so $z^2 = \delta \cdot 1_V$ for some $\delta \in \C^\times$. In this case, taking 
$t:=\delta^{-1/2}z$, we have that $t^2 =1_V$ and choose $H:= \langle E(G),t \rangle$. Our construction of $H$ ensures that 
$\ZB(\GL(V))G = \ZB(\GL(V))H$; in particular, $H$ is irreducible on $V$. If furthermore $E(G) = \ABS_n$, then since $|H| = 2|E(G)|$ and 
$H$ induces the full $\Aut(S) \cong \Sym_n$, we have that $H \cong \Sym_n$. Consider the case $E(G) = \hat\ABS_n$. Then 
$\ZB(H) = \ZB(E(G)) < E(G) = [H,H]$ and $H/\ZB(H) \cong G/\CB_G(S) \cong \Aut(S) = \Sym_n$. Thus $H$ is a central extension of 
$\Sym_n$ with kernel $\ZB(H)$ of order $2$ contained in $[H,H]$. By \cite[Corollary (11.20)]{Is}, $H$ is isomorphic to a universal cover of 
$\Sym_n$, namely the one with order $2$ inverse images of transpositions, usually denoted $\hat{\Sym}_n$ \cite[\S1]{KlT}.

From now on, we will replace $G$ by $H$, so that $G \in \{ \ABS_n,\Sym_n\}$ in the case $E(G) = \ABS_n$, and 
$G \in \{\hat\ABS_n,\hat{\Sym}_n\}$ in the case $E(G) = \hat\ABS_n$. We will let $\cyc(g)$ denote the number of disjoint cycles 
of the image of $g$ in $G/\ZB(G) \leq \Sym_n$.

\smallskip
(i) Here we assume that $E(G) \cong \ABS_n \cong S$, in particular, $\ABS_n \lhd G \leq \Sym_n$, 
and proceed by induction on $n \geq 8$. The cases $8 \leq n \leq 14$ can be checked directly using 
\cite{ATLAS} and \cite{GAP}, so we may assume $n \geq 15$. If furthermore $\dim(V) \leq n-1$, then by \cite[Result 1]{Ra} without loss we may
assume that $V = S^{(n-1,1)}|_G$. In this case, if $\cyc(g) \geq 3$, then $\dim V^g \geq 2$, a contradiction. If $\cyc(g) = 2$: $g$ is a product
of disjoint $k$-cycle and $(n-k)$-cycle with $1 \leq k \leq n-1$ but $\gcd(k,n) > 1$, then $\exp(2\pi i/\gcd(k,n))$ is an eigenvalue of $g$ of 
multiplicity $2$, again a contradiction. Thus we arrive at conclusion (i)(a).

We may now assume that $\dim(V) \geq n$. Assume furthermore that $\cyc(g) \leq 3$. Then $|g| \leq n^3/27$, whence 
$\dim(V) \leq n^3/27$ by Lemma \ref{order}. In particular, if $W$ is an irreducible $\C\Sym_n$-module that contains $V$ as a submodule 
upon restriction to $G$, then $\dim W \leq 2n^3/27 < n(n-1)(n-5)/6$. It follows from \cite[Result 3]{Ra} and the assumption
$\dim(V) \geq n$ that $\dim W=\dim V$ and, up to tensoring with the sign representation, $V = S^{(n-2,1^2)}|_G$ or 
$S^{(n-2,2)}|_G$. Direct calculation shows that $\dim V^g \geq 2$ for all $g$ with $\cyc(g) \leq 3$, a contradiction.

Thus we may assume that $s:=\cyc(g) \geq 4$. Let $a_1 \geq a_2 \geq \ldots \geq a_s \geq 1$ denote the length of the disjoint cycles of $g$.  
Then we take $m$ to be $a_s$ if $2 \nmid a_s$, $a_{s-1}$ if $2|a_s$ but $2 \nmid a_{s-1}$, and $a_{s-1}+a_s$ if $2|a_s,a_{s-1}$. Our choice
of $m$ ensures that (a conjugate of) $g$ is contained in $\ABS_{n-m} \times \ABS_m$, with $n-m \geq 8$, and the $\ABS_{n-m}$-component
$h$ of $g$ has disjoint cycles of length $a_1$, $a_2$ (and possibly others) and $\cyc(h) \geq s-2$. Let $U_1 \otimes U_2$ be an irreducible 
summand of the module $V|_{\ABS_{n-m} \times \ABS_m}$ on which $\ABS_{n-m}$ acts nontrivially. Since $\Spec(g,V)$ is simple,
$\Spec(h,U_1)$ is simple. By the induction hypothesis applied to $U_1$, $\cyc(h) \leq 2$, which implies $s=4$, $2|a_3,a_4$, and 
$a_1,a_2$ are coprime. Since $h \in \ABS_{n-m}$, we see that $2 \nmid a_1a_2$. Noting that $a_1+a_3+a_4 \geq 5+2+2=9$, we can now 
put $g$ in $\ABS_{n-a_2} \times \ABS_{a_2}$ and repeat the above argument to get a contradiction, as the $\ABS_{n-a_2}$-component
$h'$ of $g$ now has $\cyc(h') =3$. 

\smallskip
(ii) Now we consider the case $E(G) \cong \hat\ABS_n$, in particular, $\hat\ABS_n \lhd G \leq \hat\Sym_n$.
The cases $8 \leq n \leq 13$ can again be checked directly using \cite{ATLAS} and \cite{GAP} (and they lead to examples 
(i)(a)--(f)), so we may assume $n \geq 14$. Note that 
\begin{equation}\label{dim1}
  \dim(V) = \left\{ \begin{array}{ll} 2^{\lfloor (n-1)/2 \rfloor}, & G = \hat\Sym_n,\\
   2^{\lfloor (n-2)/2 \rfloor}, & G = \hat\ABS_n, \end{array} \right.
\end{equation}
in particular, $\dim(V) \geq 2^{(n-3)/2}$. Now, if $n \geq 40$, then 
$$\dim(V) \geq 2^{(n-3)/2} > e^{1.05314(n \ln n)^{1/2}} > \meo(\Sym_n) \geq \obar(g),$$
(where the second inequality follows from \cite{Mas}), contradicting Lemma \ref{order}. For $20 \leq n \leq 39$, we 
can use the values of $\meo(\Sym_n)$ stored in the sequence A000793 of \cite{Slo} to verify that 
$$\dim(V) \geq 2^{\lfloor (n-2)/2 \rfloor} > \meo(\Sym_n) \geq \obar(g),$$
and again arrive at a contradiction. Using \eqref{dim1} and \cite{GAP}, we can verify that 
$\dim(V) > \meo(G)$ for $17 \leq n \leq 19$. 

Now, the cases $G = \hat\Sym_n$ with $14 \leq n \leq 16$ can be checked using
character tables available in \cite{GAP}. We also have $\dim(V) \geq 128 > 105=\meo(\ABS_{16})$ and 
$\dim(V) \geq 64 > 60=\meo(\ABS_{14})$ when $n=14,16$. It remains to consider the case $G=\hat\ABS_{15}$.
As $\obar(g) \geq \dim(V) \geq 64$, we must have that $\obar(g) = 105$ and that $V$ is a basic spin module of $\hat\ABS_{15}$ of dimension
64 (as non-basic spin modules of $\hat\Sym_{15}$ have dimension $\geq 864$, cf. \cite{GAP}). Without loss, we may assume $|g| = 105$ and 
that $g = g_3g_5g_7$ lies in a central product $\hat\ABS_3 \circ \hat\ABS_5 \circ \hat\ABS_7$, with $g_j \in \hat\ABS_j$ has order $j$ for
$j = 3,5,7$. Note that $g_j$ has $j-1$ distinct eigenvalues on basic spin modules of $\hat\ABS_j$ for $j = 3,5$, all different from $1$. 
Furthermore, the restriction of $V$ to any standard subgroup $\hat\ABS_{n'}$ of $\hat\ABS_n$ involves only basic spin modules of 
$\hat\ABS_{n'}$, see \cite[Lemma 2.4]{KlT}. It follows that $g$ can have at most $2 \times 4 \times 7 = 56 < \dim(V)$ distinct eigenvalues on 
$V$, a contradiction.
\end{proof}

Note that case (i)(b) of Theorem \ref{alt} does give rise to a hypergeometric sheaf in characteristic $2$ with $G_\geo = \ABS_8 \cong \GL_4(2)$,
see \cite[Corollary 8.2]{KT5}. Case (i)(a) is shown to occur in Theorem \ref{alt2}, whereas cases of dimension $16$ or $32$ of Theorem 
\ref{alt}(ii) are ruled out in Lemma \ref{spin32}.

Next we record the following statement, which is useful in studying representations with irrational traces:

\begin{lem}\label{an-root}
Let $\Phi:G \to \GL(V) \cong \GL_{n-1}(\C)$ be a faithful irreducible representation of a finite almost quasisimple group $G$, which 
contains a normal subgroup $S \cong \ABS_n$ with $n \geq 7$. Suppose that 
\begin{enumerate}[\rm(a)]
\item $V|_S \cong S^{(n-1,1)}|_S$, and
\item $\Q(\varphi) \subseteq \K$ for some number field $\K$, if $\varphi$ denotes the character of $\Phi$.
\end{enumerate}
Then $\Q(\varphi) \subseteq \K_0$, the subfield obtained by joining to $\Q$ all roots of unity that belong to $\K$. 
In fact, $\Q(\varphi)$ is some cyclotomic extension $\Q(\zeta_m)$ contained in $\K$, and $\Tr(\Phi(g))$ is an integer multiple of a 
root of unity for any $g \in G$.
\end{lem}

\begin{proof}
(i) By Schur's lemma, $\CB_G(S) = \ZB(G)$ acts in $\Phi$ via scalars,  and so by finiteness $\ZB(G)$ is cyclic of order say $k$.
Then $\varphi(x) \in \Q(\zeta_k) \subseteq \K_0$ for all $x \in \ZB(G)$.
Now if $G$ induces only inner automorphisms of $S$, then $G=\ZB(G)S = \ZB(G) \times S$, and we are done since 
$\varphi(y) \in \Z$ for all $y \in S$; in this case, $\Q(\varphi)=\Q(\zeta_k)$. It is also clear that, for any $g \in G$, 
$\varphi(g) \in \Z\xi$ for some root of unity $\xi$.

\smallskip
(ii) It remains to consider the case $G$ induces some outer automorphisms on $S$. As $n \geq 7$, it follows that $[G:\ZB(G)S] = 2$, 
and we need to look at $\varphi(g)$ for all $g \in G \smallsetminus \ZB(G)S$ with $\varphi(g) \neq 0$. 
Note that we can extend $\Phi|_S$ to 
$\Sym_n$ which without loss we also denote by $\Phi$, and then $\Tr(\Phi(y)) \in \Q$ for all $y \in \Sym_n$. Given 
$g \in G \smallsetminus \ZB(G)S$ with $\varphi(g) \neq 0$, 
we can find $h \in \Sym_n$ that induces the same action on $S$. It follows by Schur's lemma 
that $\Phi(g) = \xi\Phi(h)$ for some $\xi \in \C^\times$. Since both $g$ and $h$ have finite order, $\xi$ is a root of unity. Also we have 
that $\K^\times \ni \varphi(g)=a\xi$ where $a:=\Tr(\Phi(h)) \in \Z$. It follows that $\xi \in \K$, and so $\K_0$ contains $\xi$ and
$\varphi(g)$. We also note that $g^2,h^2 \in \ZB(G)S$, and so 
$\Phi(g^2h^{-2}) = \xi^2 \cdot {\mathrm {Id}}$ belongs to $\Phi(\ZB(G))$, whence $\xi^{2k}=1$. 
Together with (i), we have shown that 
$$\Q(\zeta_k) \subseteq \Q(\varphi) \subseteq \K_0 \cap \Q(\zeta_{2k}).$$
As $[\Q(\zeta_{2k}:\Q(\zeta_k)] \leq 2$, $\Q(\varphi)$ is either $\Q(\zeta_k)$ or $\Q(\zeta_{2k})$. 
\end{proof}

Table 1 summarizes the classification of $\ssp$-elements in the non-generic cases of sporadic groups and $\ABS_7$ and some
small rank Lie-type groups, {\it under the additional condition
that $V|_{E(G)}$ is irreducible}. For each $V$, we list all almost quasisimple groups $G$ with common $E(G)$ that act on $V$, and 
we list the number of isomorphism classes of such representations in a given dimension, for a largest possible $G$ up to scalars 
(if no number is
given, it means the representation is unique up to equivalence in given dimension). For each 
representation, we list the names of conjugacy classes of $\ssp$-elements in a largest possible $G$, as listed in \cite{GAP},
and/or the total number of them. \edit{We also give a reference where a local system realizing the given representation is constructed.
The indicator $^\sharp$ signifies that we have a conjectured local system realizing the given representation, whereas \edit{(-)} means 
that no hypergeometric sheaf with $G$ as monodromy group can exist.}   

\begin{table}
\begin{tabular}{|c|c|c|c|c|c|} \hline
$S$ & $\meo(\Aut(S))$ & $\dl(S)$ & $G$ & $\dim(V)$ & $\ssp$-classes \\ \hline\hline
$\ABS_7$ & $12$ & $4$ & $2\ABS_7$ & $4$ (2 reps) & 9 classes \\ 
 &  & & $\Sym_7$ & $6$ (2 reps)  & $7A$, $6C$, $10A$, $12A$ (4 classes) \\ 
 &  & & $3\ABS_7$ & $6$ (2 reps) & 6 classes \\  
 &  & & $6\ABS_7$ & $6$ (4 reps) & 15 classes \\ \hline
$\mathrm{M}_{11}$ & $11$ & $10$ & $\mathrm{M}_{11}$ & $10$ (3 reps) $^\sharp$& $11AB$ (2 classes) \\ 
 &  & &  & $11$  $^\sharp$& $11AB$ (2 classes) \\ \hline
$\mathrm{M}_{12}$ & $12$ & $10$ & $2\mathrm{M}_{12} \cdot 2$ & $10$ (4 reps) \edit{(-)}& 11 classes \\ 
 &  & & $\mathrm{M}_{12}$ & $11$ (2 reps)  \edit{(-)}& $11AB$ (2 classes) \\ 
 &  & & $2\mathrm{M}_{12} \cdot 2$ & $12$ (2 reps) \edit{(-)}& $24AB$ (2 classes) \\  \hline
$\mathrm{M}_{22}$ & $14$ & $10$ & $2\mathrm{M}_{22} \cdot 2$ & $10$ (4 reps) $^\sharp$& 
10 classes \\ \hline 
$\mathrm{M}_{23}$ & $23$ & $22$ & $\mathrm{M}_{23}$ & $22$ $^\sharp$& $23AB$ (2 classes) \\ \hline 
$\mathrm{M}_{24}$ & $23$ & $23$ & $\mathrm{M}_{24}$ & $23$ $^\sharp$& $23AB$ (2 classes) \\ \hline 
$\mathrm{J}_{2}$ & $24$ & $6$ & $2\mathrm{J}_2$ & $6$ (2 reps) \cite{Ka-RL-2J2} & $17$ classes \\ 
 &  & & $2\mathrm{J}_2 \cdot 2$ & $14$ (2 reps)  $^\sharp$& $28AB$, $24CDEF$ (6 classes) \\ \hline
$\mathrm{J}_3$ & $34$ & $18$ & $3\mathrm{J}_3$ & $18$ (4 reps) & $19AB$, $57ABCD$ (6 classes) \\ \hline
$\mathrm{HS}$ & $30$ & $22$ & $\mathrm{HS} \cdot 2$ & $22$ (2 reps) \edit{(-)}& $30A$  \\ \hline
$\mathrm{McL}$ & $30$ & $22$ & $\mathrm{McL} \cdot 2$ & $22$ (2 reps) $^\sharp$ & $30A$, $22AB$ (3 classes)  \\ \hline
$\mathrm{Ru}$ & $29$ & $28$ & $2\mathrm{Ru}$ & $28$  & $29AB$, $58AB$ (4 classes)  \\ \hline
$\mathrm{Suz}$ & $40$ & $12$ & $6\mathrm{Suz}$ & $12$ (2 reps) \cite{Ka-RL-T-2Co1} & $57$ classes  \\ \hline
$\mathrm{Co}_1$ & $60$ & $24$ & $2\mathrm{Co}_1$ & $24$ \cite{Ka-RL-T-2Co1} & $17$ classes \\ \hline
$\mathrm{Co}_2$ & $30$ & $23$ & $\mathrm{Co}_2$ & $23$ \cite{Ka-RL-T-Co2} & $23AB$, $30AB$ (4 classes) \\ \hline 
$\mathrm{Co}_3$ & $30$ & $23$  & $\mathrm{Co}_3$ & $23$ \cite{Ka-RL-T-Co3} & $23AB$, $30A$ (3 classes) \\ \hline
$\PSL_{3}(4)$ & $21$ & $6$ & $6S \cdot 2_1$ & \edit{$6$} (4 reps)  & many classes \\ 
 &  & & $4_1S \cdot 2_3$ & \edit{$8$} (8 reps)  &  12 classes \\ 
 &  & & $2S \cdot 2_2$ & $10$ (4 reps) $^\sharp$ & $14CDEF$ (4 classes) \\  \hline
$\PSU_4(3)$ & $28$ & $6$ & $6_1S \cdot 2_2$ & \edit{$6$} (4 reps) $^\sharp$ & many classes \\ \hline
$\Sp_{6}(2)$ & $15$ & $7$ & $\Sp_6(2)$ & $7$ & $7A$, $8B$, $9A$, $12C$, $15A$ \\ 
 &  & & $2\Sp_{6}(2)$ & $8$ $^\sharp$  &  8 classes \\ 
 &  & & $\Sp_{6}(2)$ & $15$  \edit{(-)} & $15A$ \\  \hline
$\Omega^+_8(2)$ & $30$ & $8$ & $2\Omega^+_8(2) \cdot 2$ & $8$ $^\sharp$   & 22 classes \\ \hline 
$\tw2 B_2(8)$ & $15$ & $14$ & $\tw2 B_2(8) \cdot 3$ & $14$ (6 reps) $^\sharp$ & $15AB$ (2 classes) \\ \hline
$G_2(3)$ & $18$ & $14$ & $G_2(3)\cdot 2$ & $14$ (2 reps) $^\sharp$ & $14A$, $18ABC$ (4 classes) \\ \hline
$G_2(4)$ & $24$ & $12$ & $2G_2(4)\cdot 2$ & $12$ (2 reps) $^\sharp$  & $20$ classes \\ \hline 
\end{tabular}
\vskip5pt
\caption{Elements with simple spectra in non-generic cases}
\end{table}

\begin{thm}\label{spor}
In the situation of $\Cstar$, assume that $S$ is one of $26$ sporadic simple groups, or $\ABS_7$, and that $V|_{E(G)}$ is irreducible. Then 
$(S,G,V,g)$ are as listed in {\rm Table 1}.
\end{thm}

\begin{proof}
We apply Lemma \ref{order} to $(G,V,g)$ to rule out $12$ sporadic groups, listed in Table 2, because they all satisfy
$\meo(\Aut(S)) < \dl(S)$. For the remaining $15$ cases, we use \cite{GAP} to find possible candidates for $(G,V,g)$ 
(certainly, it suffices to search among representations of dimension at most $\meo(\Aut(S))$).
\end{proof}

\begin{table}[h]
\begin{tabular}{|c|c|c||c|c|c||c|c|c|} \hline
$S$ & $\meo(\Aut(S))$ & $\dl(S)$ & $S$ & $\meo(\Aut(S))$ & $\dl(S)$ & $S$ & $\meo(\Aut(S))$ & $\dl(S)$ \\ \hline\hline
$\mathrm{J}_1$ & $19$ & $56$ & $\mathrm{J}_4$ & $66$ & $1333$ & $\mathrm{He}$ & $42$ & $51$ \\ \hline
$\mathrm{Ly}$ & $62$ & $2480$ & $\mathrm{O'N}$ & $56$ & $342$ & $\mathrm{HN}$ & $60$ & $133$ \\ \hline
$\mathrm{Fi}_{22}$ & $42$ & $78$ & $\mathrm{Fi}_{23}$ & $60$ & $782$ & $\mathrm{Fi}'_{24}$ & $84$ & $783$ \\ \hline
$\mathrm{Th}$ & $39$ & $248$ & $\mathrm{BM}$ & $70$ & $4371$ & $\mathrm{M}$ & $119$ & $196883$ \\ \hline
\end{tabular}
\vskip5pt
\caption{Maximal element order and minimal degree for some sporadic groups}
\end{table}

Furthermore, we list in Table 3 certain hypergeometric sheaves 
$$\sH yp_\psi(\chi_1, \ldots, \chi_D;\rho_1, \ldots,\rho_m)$$ 
in characteristic $p$ that are conjectured to produce $G$ as \edit{geometric} monodromy groups. 
\edit{All of them have been proved in \cite{Ka-RL-T-Spor} to have finite $G_\geo$, and the cases marked with
a reference to \cite{Ka-RL-T-Spor} are proved therein to have the conjectured $G$ as $G_\geo$.}
For any natural number $N$, the notation 
$\Ch_N$ denotes the set of all characters of order dividing $N$, $\Ch^\times_{N}$ denotes the set of all characters of order exactly $N$,
and $\xi_N$ denotes a fixed character of order $N$. The last column indicates the conjectured image of $I(\infty)$.

\begin{table}[h]
\begin{tabular}{|c|c|c|c|c|c|c|} \hline
$S$ & $G$ & $p$ & rank & $\chi_1, \ldots,\chi_D$ & $\rho_1, \ldots,\rho_m$ & Image of $I(\infty)$ \\ \hline\hline
$\mathrm{M}_{11}$ & $S$ & $3$ & 10 \edit{(WS)} & $\Ch^\times_{11}$ & $\Ch_2$ & $3^2:8$ \\ 
  & $S$ & $3$ & 10 [Lem. \ref{m11}] & $\Ch^\times_{11}$ & $\xi_8,\xi_8^3$ & $3^2:8$ \\ 
  & $S$ & $3$ & 11 \edit{(WS)} & $\Ch_{11}$ & $\Ch_4 \smallsetminus \{\triv\}$ & $3^2:8$ \\ \hline
$\mathrm{M}_{22}$ & $2S$ & $2$ & 10  \cite{Ka-RL-T-Spor} & $\Ch^\times_{11}$ & $\xi_7,\xi_7^2,\xi_7^4$ & $2^3:7$ \\ \hline
$\mathrm{M}_{23}$ & $S$ & $2$ & 22& $\Ch^\times_{23}$ & $\Ch_{15} \smallsetminus \Ch^\times_{15}$ & $2^4:15$ \\ \hline
$\mathrm{M}_{24}$ & $S$ & $2$ & 23 \edit{(WS)}& $\Ch_{23}$ & $\Ch^\times_{3}$ & $2^6:21$ \\ \hline 
$\mathrm{McL}$ & $S \cdot 2$ & $3$ & 22 \cite{Ka-RL-T-Spor} & $\Ch_{22}$ & $\Ch^\times_{5}$ & $3^{1+4}:20$ \\ 
  & $S \cdot 2$ & $5$ & 22 \cite{Ka-RL-T-Spor} & $\Ch_{22}$ & $\Ch^\times_{3}$ & $5^{1+2}:24$ \\ \hline
$\mathrm{J}_{2}$ & $2S \cdot 2$ & $5$ & 14 \cite{Ka-RL-T-Spor} 
  & $\Ch_{28} \smallsetminus \Ch_{14}$ & $\xi_8,\xi_8^{-1}$ & $5^2:24$ \\ \hline
$\mathrm{J}_{3}$ & $3S$ & $2$ & 18 \cite{Ka-RL-T-Spor} & $\xi_3\cdot\Ch^\times_{19}$ & $\triv,\xi_5,\xi_5^{-1}$ & $2^4:15$ \\ \hline
$\mathrm{Ru}$ & $2S$ & $5$ & 28 \cite{Ka-RL-T-Spor} & $\Ch^\times_{29}$ & $\xi_{12},\xi_{12}^3,\xi_{12}^5,\xi_{12}^9$ & $5^2:24$ \\ \hline
$\PSU_4(3)$ 
 & $6_1 \cdot S$ & $3$ & 6 \cite{Ka-RL-T-Spor} & $\Ch^\times_{7}$ & $\xi_{2}$ & $3^{4}:10$ \\
\hline
$\Sp_6(2)$ 
& $2S$ & $7$ & 8 & $\Ch_{9} \smallsetminus \{\triv\}$ & $\Ch_2$ & $7:6$ \\ 
\hline
$\Omega^+_8(2)$ & $2S \cdot 2$ & $3$ & 8 \cite{Ka-RL-T-Spor}& $\Ch^\times_{20}$ & $\Ch_2$ & $3^{1+2}:8$ \\ 
& $2S \cdot 2$ & $7$ & 8 \cite{Ka-RL-T-Spor} & $\Ch^\times_{20}$ & $\Ch_2$ & $7:6$ \\ \hline
$\PSL_3(4)$ & $2S \cdot 2_2$ & $3$ & 10 & $\Ch_{14} \smallsetminus \{\triv,\xi_7,\xi_7^2,\xi_7^4\}$ & $\Ch^\times_4$ & $3^2:8$ \\ \hline
$G_2(4)$ & $2 \cdot S$ & $2$ & 12  \cite{Ka-RL-T-Spor} & $\Ch^\times_{13}$ & $\Ch^\times_3$ & $\mbox{2-group}:15$ \\ \hline
$G_2(3)$ & $S \cdot 2$ & $13$ & 14 \cite{Ka-RL-T-Spor} & $\Ch_{18} \smallsetminus \{\triv,\xi_6,\xi_6^2,\xi_6^3\}$ & $\Ch^\times_4$ & $13:12$ \\ 
\hline
$\tw2 B_2(8)$ & $S \cdot 3$ & $13$ & 14 \cite{Ka-RL-T-Spor} & $\Ch_{15} \smallsetminus \{\triv\}$ & $\xi_{12},\xi_{12}^5$ & $13:12$ \\ \hline
\end{tabular}
\vskip5pt
\caption{Hypergeometric sheaves in non-generic cases}
\end{table}


\subsection*{6B. Finite groups of Lie type}
In this subsection, we will deal with almost quasisimple groups $G$, where $S$ is a finite simple group of Lie type. 
We will need the following well-known consequences of the Lang-Steinberg theorem:

\begin{lem}\label{lang}
Let $\sG$ be a connected algebraic group over an algebraically closed field of characteristic $p > 0$ and let 
$\sigma:\sG \to \sG$ be a surjective morphism with finite $\sG^\sigma:= \{ x \in \sG \mid \sigma(x)=x \}$. 

\begin{enumerate}[\rm(i)]
\item Suppose the $\sG$-conjugacy 
class of $g \in \sG$ is $\sigma$-stable. Then some $\sG$-conjugate of $g$ is $\sigma$-fixed,
in particular, $|g| \leq \meo(\sG^\sigma)$.
\item Suppose that $[\sG,\sG]$ is simply connected and $g \in \sG^\sigma$ is semisimple. Then, for any $t \in \sG$ with
$tgt^{-1} \in \sG^\sigma$, $tgt^{-1}$ is $\sG^\sigma$-conjugate to $g$.
\end{enumerate}
\end{lem}

\begin{proof}
(i) By assumption, $\sigma(g) = xgx^{-1}$ for some $x \in \sG$.
Since $\sG$ is connected, the Lang map $y \mapsto y^{-1}\sigma(y)$ is surjective on $\sG$.
Hence $x = y^{-1}\sigma(y)$ for some $y \in \sG$. Thus $\sigma(g) = \sigma(y^{-1})ygy^{-1}\sigma(y)$, whence
$ygy^{-1} \in \sG^\sigma$, and the statement follows.

\smallskip
(ii) By assumption, $t^{-1}\sigma(t) \in \CB_\sG(g)$. Since $\sigma(g)=g$ and $[\sG,\sG]$ is simply connected, by 
\cite[Theorem 3.5.6]{C} $\CB_\sG(g)$ is connected and $\sigma$-stable. By the Lang-Steinberg theorem applied to
$\CB_\sG(g)$, $t^{-1}\sigma(t) = c^{-1}\sigma(c)$ for some $c \in \CB_\sG(g)$. Now $u:=tc^{-1} \in \sG^\sigma$ and 
$tgt^{-1} = tc^{-1}gct^{-1} = ugu^{-1}$ is $\sG^{\sigma}$-conjugate to $g$.
\end{proof}


\begin{thm}\label{simple}
In the situation of $\Cstar$, assume that $S$ is a finite simple group of Lie type. Then one of the following statements holds.
\begin{enumerate}[\rm (i)]
\item $S \cong \PSL_2(q)$ and $\dim(V) \leq \obar(g) \leq q+1$. 
\item $S = \PSL_n(q)$, $n \geq 3$, $E(G)$ is a quotient of $\SL_n(q)$, 
and $V|_{E(G)}$ is one of $q-1$ Weil modules, of dimension $(q^n-1)/(q-1)$ or $(q^n-q)/(q-1)$. Moreover, 
$\dim(V) \leq \obar(g) \leq (q^n-1)/(q-1)$.
\item $S = \PSU_n(q)$, $n \geq 3$, $E(G)$ is a quotient of $\SU_n(q)$, 
and $V|_{E(G)}$ is one of $q+1$ Weil modules, of dimension $(q^n-(-1)^n)/(q+1)$ or $(q^n+q(-1)^n)/(q+1)$.
\item $S = \PSp_{2n}(q)$, $n \geq 2$, $2 \nmid q$, $E(G)$ is a quotient of $\Sp_{2n}(q)$,
every irreducible constituent of $V|_{E(G)}$ is one of four Weil modules, of dimension $d:=(q^n \pm 1)/2$, and 
$\dim(V) = d$ or $2d$.
\item Non-generic cases:
\begin{enumerate}[\rm(a)]
\item $S$ is one of the following groups: $\PSL_3(4)$, 
$\PSU_4(3)$, $\Sp_6(2)$, $\Omega^+_8(2)$, $\tw2 B_2(8)$, $G_2(3)$, $G_2(4)$, 
$V|_{E(G)}$ is simple, and the classification of $\ssp$-elements in $G$ can be read off from 
Table I.
\item $V|_{E(G)}$ is the direct sum of two simple modules of equal dimension, and one of the following possibilities occurs.
\begin{enumerate}
\item[$(\alpha)$] $E(G)=S= \SU_4(2)$, $G/\ZB(G) = \Aut(S)$, 
either $\dim(V) = 8$ and  $\obar(g)=9,10,12$, or $\dim(V) =10$ and $\obar(g) = 10, 12$.
\item[$(\beta)$] $S= \SU_5(2)$, $G/\ZB(G) = \Aut(S)$, $\dim(V) = 22$, and 
$\obar(g)=24$.
\end{enumerate}
\end{enumerate}
\end{enumerate}
\end{thm}

\begin{proof}
By Lemma \ref{order}, 
\begin{equation}\label{dim2}
  \meo(\Aut(S)) \geq \dim(V) \geq \dl(S).
\end{equation}   
We will use the upper bounds on $\meo(\Aut(S))$ available from \cite{KSe} and \cite{GMPS}, on the one hand, and the (precise or lower) bounds on
$\dl(S)$ as recorded in \cite[Table I]{T1}, to show that most of the possibilities for $S$ contradict \eqref{dim2}. We will frequently use the obvious estimate
\begin{equation}\label{dim3}
  \meo(\Aut(S)) \leq \meo(S) \cdot |\Out(S)|.
\end{equation}

\medskip
(A) First we consider exceptional groups of Lie type.

\smallskip
(A1) Assume $S = \tw2 G_2(q)$, with $q = 3^{2a+1} \geq 27$. By \cite[Table A.7]{KSe}, $\meo(S) \leq q+\sqrt{3q}+1$, hence
$\meo(\Aut(S)) \leq (q+\sqrt{3q}+1)(2a+1)$ by \eqref{dim3}. On the other hand, $\dl(S) = q^2-q+1$, contradicting \eqref{dim2}.

Similarly, if $S = \tw2 B_2(q)$ with $q = 2^{2a+1} \geq 128$, then $\meo(S) \leq q+\sqrt{2q}+1$, hence
$\meo(\Aut(S)) \leq (q+\sqrt{2q}+1)(2a+1)$ by \eqref{dim3}. On the other hand, $\dl(S) = (q-1)\sqrt{q/2}$, contradicting \eqref{dim2}.
The cases $S = \tw2 B_2(q)$ with $q=8,32$ can be checked directly using \cite{GAP}.

Let $S = \tw2 F_4(q)$, with $q = 2^{2a+1} \geq 8$. By \cite[Table 5]{GMPS}, 
$$\meo(\Aut(S)) \leq 16(q^2+\sqrt{2q^3}+q+\sqrt{2q}+1)(2a+1).$$
This contradicts \eqref{dim2}, since $\dl(S) = (q^3+1)(q^2-1)\sqrt{q/2}$. If $S = \tw2 F_4(2)'$, then, 
according to \cite{GAP}, $\meo(\Aut(S)) = 20 < 27 = \dl(S)$.

\smallskip
(A2) Assume $S = \tw3 D_4(q)$ with $q=p^f > 2$. We will show that
\begin{equation}\label{3d4}
  \meo(\Aut(S)) < \dl(S) = q(q^4-q^2+1),
\end{equation}  
which contradicts \eqref{dim2}. Indeed, if $p > 2$, then $\meo(S) = (q^3-1)(q+1)$ by \cite[Table A.7]{KSe}.
On the other hand, if $p=2$ then, by Propositions 2.1--2.3 of \cite{DMi}, the order of any element $s \in S$ is 
at most $(q^3-1)(q+1)$ if $s$ is semisimple, and $\max(2(q^3+1),8(q^2+q+1)) \leq (q^3-1)(q+1)$ otherwise, and so
$\meo(S) = (q^3-1)(q+1)$ again. Hence, if $q \neq 3,4,8$, then $\meo(\Aut(S)) \leq 3f(q^3-1)(q+1)$ by 
\eqref{dim3}, and so \eqref{3d4} holds. 

Assume now that $q = 3,4$, or $8$, and view $S = \sG^{\sigma^f\tau}$, where
$\sG = {\mathrm {Spin}}_8(\overline{\F_p})$, $\sigma:\sG\to\sG$ the standard Frobenius morphism induced by 
the map $x \mapsto x^p$ of $\overline{\F_p}$, and $\tau$ a triality automorphism of $\sG$ that commutes with $\sigma$. 
Then the restriction $\alpha :=\sigma|_S$ induces an automorphism of order $3f$ of $S$, and 
$A:=\Aut(S) = S \rtimes \langle \alpha \rangle$, cf. \cite[Theorem 2.5.12]{GLS}. Consider any element $g \in \Aut(S)$. If 
$S\langle g \rangle < A$, then 
$$|g| \leq (3f/2)\meo(S) \leq 3f(q^3-1)(q+1)/2 < q(q^4-q^2+1).$$
In the remaining case, $S \langle g \rangle = A$. Note that $h:= g^{3f} \in S$ is centralized by $g$, and so
$[\CB_A(h):\CB_S(h)]=[A:S]$. Hence $\#(h^A) = \#(h^S)$; in particular,
$$\sigma(h) = \alpha h \alpha^{-1} = tht^{-1}$$
for some $t \in S$. Thus the $\sG$-conjugacy class of $h$ is $\sigma$-stable, and so 
$$|g| \leq 3f\cdot |h| \leq 3f\cdot\meo(\sG^\sigma) = 3f\cdot \meo({\mathrm {Spin}}^+_8(p))$$
by Lemma \ref{lang}(i).  Using \cite{ATLAS} one can check that 
$\meo({\mathrm {Spin}}^+_8(3)) \leq 2\cdot \meo(P\Omega^+_8(3)) = 40$ and $\meo({\mathrm {Spin}}^+_8(2)) = 15$. 
Thus $|g| \leq 120$, respectively $210$, $305$, when
$q=3$, $4$, and $8$, respectively. It follows that $|g| < q(q^4-q^2+1)$, completing the proof of \eqref{3d4}.

If $S = \tw3 D_4(2)$, then $\meo(\Aut(S)) = 28$  and $\dl(S)=26$ according to \cite{GAP}. However, using character tables in
\cite{GAP}, one can check that no $\ssp$-element exists.

\smallskip
(A3) Assume $S = G_2(q)$ with $q=p^f \geq 5$. If $p > 2$, then $\meo(S) = q^2+q+1$ by \cite[Table A.7]{KSe}, and so
$\meo(\Aut(S)) \leq f(q^2+q+1)$ if $p > 3$ and $\meo(\Aut(S)) \leq 2f(q^2+q+1)$ if $p=3$. If $p=2$ and $q \geq 8$, then using \cite{EY} 
one can check that the order of any element $g \in S$ is 
at most $q^2+q+1$ if $g$ is semisimple, and $2(q^2-1)$ otherwise, and so
$\meo(\Aut(S)) < 2f(q^2+q+1)$. 
On the other hand, $\dl(S) \geq q^3-1$ if $p \neq 3$ and $\dl(S)=q^4+q^2+1$ if $p=3$, see \cite[Table I]{T1}, and we arrive at a contradiction
when $q \geq 5$. The cases $q=3,4$ are handled directly using \cite{GAP}.

Let $S = F_4(q)$, with $q = p^f \geq 3$. Arguing as in the proof of \cite[Theorem 1.2]{GMPS}, also using \cite[Table A.7]{KSe}, we get
$\meo(\Aut(S)) \leq 32fq(q^2-1)(q+1)$.
This contradicts \eqref{dim2}, since $\dl(S) \geq q^8-q^4+1$. 
If $S = F_4(2)$, then, according to \cite{GAP}, $\meo(\Aut(S)) = 40 < 52 = \dl(S)$.

Likewise, if $S = \tw2 E_6(q)$ with $q = p^f \geq 3$, then arguing as in the proof of \cite[Theorem 1.2]{GMPS} and using \cite[Table A.7]{KSe}, we get
$\meo(\Aut(S)) \leq 32f(q^3-1)(q^2+1)(q+1)$.
This contradicts \eqref{dim2}, since $\dl(S) = q(q^4+1)(q^6-q^3+1)$. 
If $S = \tw2 E_6(2)$, then, according to \cite{GAP}, $\meo(\Aut(S)) = 105 < 1938 = \dl(S)$.
If $S = E_6(q)$ with $q = p^f \geq 3$, then the same arguments show that
$\meo(\Aut(S)) \leq 32f(q^6-1)/(q-1)$. This contradicts \eqref{dim2}, since $\dl(S) = q(q^4+1)(q^6+q^3+1)$. If $S = E_6(2)$, then 
$\meo(S) = 126$ according to \cite{GAP}, hence $\meo(\Aut(S)) \leq 252 < 2482 = \dl(S)$.

The same arguments apply to the last two exceptional types. If $S = E_7(q)$, then 
$$\meo(\Aut(S)) \leq 32f(q+1)(q^2+1)(q^4+1) < q^{15}(q^2-1) < \dl(S).$$ 
If $S = E_8(q)$, then 
$$\meo(\Aut(S)) \leq 32f(q+1)(q^2+q+1)(q^5-1) < q^{27}(q^2-1) < \dl(S).$$ 

\medskip
(B) Now we analyze the simple classical groups.

\smallskip
(B1) Suppose $S = \Sp_{2n}(q)$ with $n \geq 2$ and $2|q$. Then $\meo(\Aut(S)) \leq q^{n+1}/(q-1)$ by \cite[Theorem 2.16]{GMPS},
whereas $\dl(S) = (q^n-1)(q^n-q)/2(q+1)$ by \cite[Table I]{T1}, and this contradicts \eqref{dim2}, unless $(n,q) = (3,2)$, 
$(2,4)$, $(2,2)$. The remaining exceptions are handled using \cite{GAP}.
Likewise, if $S = \Omega_{2n+1}(q)$ with $n \geq 3$, $2 \nmid q$, and $(n,q) \neq (3,3)$, then $\meo(\Aut(S)) \leq q^{n+1}/(q-1)$ by \cite[Theorem 2.16]{GMPS},
and $\dl(S) \geq (q^n-1)(q^n-q)/(q^2-1)$ by \cite[Table I]{T1}, again contradicting \eqref{dim2}. 
If $S = P\Omega^{\eps}_{2n}(q)$ with $n \geq 4$ and $(n,q,\eps) \neq (4,2,+)$, 
then $\meo(\Aut(S)) \leq q^{n+1}/(q-1)$ by \cite[Theorem 2.16]{GMPS} and $\dl(S) \geq (q^n+1)(q^{n-1}-q)/(q^2-1)$ by \cite[Table I]{T1}, contradicting \eqref{dim2}. The cases $S = \Omega_7(3)$ and $\Omega^+_8(2)$ are handled using \cite{GAP}.

\smallskip
(B2) Assume now that $S = \PSL_n(q)$ with $n \geq 2$, $(n,q) \neq (3,4)$, $(4,3)$, and $q \geq 11$ if $n=2$. Then by \cite[Theorem 2.16]{GMPS} and \eqref{dim2} we have 
$$\dim(V) \leq \obar(g) \leq \meo(\Aut(S)) = (q^n-1)/(q-1).$$
In particular, if $n=2$ then we arrive at conclusion (i).
If $n \geq 3$, then it follows from \cite[Theorem 3.1]{TZ1} that $E(G)$ is a quotient of $\SL_n(q)$ and that 
$V|_{E(G)}$ has an irreducible constituent $U$, which is a Weil module of 
dimension $(q^n-q)/(q-1)$ or $(q^n-1)/(q-1)$. In particular, $\dim(U) > \dim(V)/2$, and so $U=V|_{E(G)}$, and we arrive at conclusion (ii).
The remaining cases are handled using \cite{GAP}.

\smallskip
(B3) Suppose $S = \PSp_{2n}(q)$ with $n \geq 2$, $2 \nmid q$, and $(n,q) \neq (2,3)$. Then by \cite[Theorem 2.16]{GMPS} and \eqref{dim2} we have 
$$\dim(V) \leq \meo(\Aut(S)) \leq q^{n+1}/(q-1).$$
It follows from \cite[Theorem 5.2]{TZ1} that $E(G)$ is a quotient of $\Sp_{2n}(q)$ and that 
$V|_{E(G)}$ has an irreducible constituent $U$, which is a Weil module of 
dimension $d=(q^n \pm 1)/2$. Now, if $q \geq 5$, then $q^{n+1}/(q-1) < 3(q^n-1)/2$, hence $\dim(V) = d$ or $2d$. 
Consider the case $q=3$, for which $q^{n+1}/(q-1)  < 4d$. Here, either $G = \ZB(G)E(G)$, and so $\dim(V) = d$, 
or $[G:\ZB(G)E(G)]=2$, $G$ induces a diagonal automorphism of $E(G)$ and fuses two irreducible Weil modules 
of $E(G)$ of dimension $d$, whence $\dim(V)=2d$. Thus we arrive at conclusion (iv). The remaining case of $S=\PSp_4(3)$ is handled using \cite{GAP}.

\smallskip
(B4) Finally, we consider the case $S = \PSU_n(q)$ with $n \geq 3$. If $n = 3$ and $q \neq 3,5$, then by \cite[Theorem 2.16]{GMPS} 
and \eqref{dim2} we have 
\begin{equation}\label{dim31}
  \dim(V) \leq \meo(\Aut(S)) \leq q(q+1) < (q^2-q+1)(q-1)/\gcd(3,q+1).
\end{equation}
If $n = 4$ and $q \geq 4$ and we have 
\begin{equation}\label{dim32}
  \dim(V) \leq \meo(\Aut(S)) \leq q^3+1 < (q^2-q+1)(q^2+1)/2.
\end{equation}  
If $2|n \geq 6$ and $(n,q) \neq (6,2)$, then we have 
\begin{equation}\label{dim33}
  \dim(V) \leq \meo(\Aut(S)) \leq q^{n-1}+q^2 < (q^n-1)(q^{n-1}-q)/(q+1)(q^2-1).
\end{equation}  
If $2 \nmid n \geq 5$ and $(n,q) \neq (5,2)$ and we have 
\begin{equation}\label{dim34}
  \dim(V) \leq \meo(\Aut(S)) \leq q^{n-1}+q < (q^n+1)(q^{n-1}-q^2)/(q+1)(q^2-1).
\end{equation}  
In all these cases, the upper bound on $\dim(V)$ obtained in \eqref{dim31}--\eqref{dim34} implies by \cite[Theorem 4.1]{TZ1} 
that $E(G)$ is a quotient of $\SU_n(q)$ and that 
$V|_{E(G)}$ has an irreducible constituent $U$, which is a Weil module of 
dimension $(q^n+q(-1)^n)/(q+1)$ or $(q^n-(-1)^n)/(q+1)$. In particular, $\dim(U) > \dim(V)/2$, and so $U=V|_{E(G)}$, and we arrive at conclusion (iii).
The remaining cases $(n,q) = (3,3)$, $(3,5)$, $(4,3)$, $(5,2)$, and $(6,2)$ can be checked directly using \cite{GAP}.
\end{proof}

\section{The characteristic of hypergeometric sheaves}
In this section, we assume $\sH = \sH yp_\psi(\chi_1, \ldots,\chi_D;\rho_1, \ldots,\rho_m)$ is geometrically irreducible (i.e, no $\chi_i$ is any $\rho_j$) $\ell$-adic (Kloosterman or) hypergeometric sheaf of type $(D,m)$, $D > m$, on $\G_m$ over a finite extension of 
$\F_p$ that admits a {\it finite} geometric monodromy group $G_\geo$. In particular, the image of $I(0)$ 
on $\sH$ is a finite cyclic group whose generator 
has $D$ distinct eigenvalues $\zeta^{a_1}, \ldots,\zeta^{a_D}$, where $\zeta \in \overline{\F_p}^\times$ has order $N$, and 
$\chi_i = \chi^{a_i}$ for a fixed multiplicative character $\chi$ of order $N$ and $1 \leq i \leq D$. We will show that, in most cases 
the characteristic $p$ of the sheaf can be read off from the structure of $G_\geo$.

\subsection*{7A. The Lie-type case}
In this subsection, we will assume that $G=G_\geo$ is an {\it almost quasisimple group of Lie type}, that is,
$S \leq G/\ZB(G) \leq \Aut(S)$ for some finite simple group of Lie type, in some characteristic $r$ which may a priori differ from $p$.

The principal result of this section is Theorem \ref{char-sheaf1} stating that in the generic situation we in fact have $r=p$, that is, the characteristic of the sheaf and of the group $S$ are equal.

In view of Theorem \ref{simple}, we will first prove some auxiliary results concerning Weil representations of finite classical groups.

\begin{lem}\label{weil-value}
Let $G$ be a finite classical group and $\varphi$ be a complex irreducible character of $G$, such that at least one of the following
conditions holds:
\begin{enumerate}[\rm(a)]
\item $G = \SL_2(q)$ with $q \geq 7$;
\item $G = \GL_n(q)$ with $n \geq 3$, and $\varphi$ is one of the irreducible Weil characters $\tau^i_{n,q}$, $0 \leq i \leq q-2$; 
\item $G = \GU_n(q)$ with $n \geq 3$ and $(n,q) \neq (3,2)$, 
and $\varphi$ is one of the irreducible Weil characters $\zeta^i_{n,q}$, $0 \leq i \leq q$; or
\item $G= \Sp_{2n}(q)$ with $n \geq 2$, $2 \nmid q$, and $\varphi$ is one of the four irreducible Weil characters $\xi_i,\eta_i$, $i = 1,2$.
\end{enumerate}
Let $g \in G \smallsetminus \ZB(G)$. Then $|\varphi(g)|/\varphi(1) < 2/3$ in the case of {\rm (d)} and $|\varphi(g)|/\varphi(1) \leq 3/5$
in the other cases. Moreover, 
if $G = \SL_2(q)$ with $q \geq 25$ then 
$$|\varphi(g)|/\varphi(1) \leq 1/(\sqrt{q}-1) \leq 1/4.$$
Furthermore, if $G = \Sp_{2n}(q)$ and $g$ is a $p'$-element, then $|\varphi(g)|/\varphi(1) \leq (q^{n-1}+q)/(q^n-1)$.
\end{lem}

\begin{proof}
In the case of (a), one can check using the well-known character tables of $G$, see e.g. \cite[\S38]{Do}, that 
$|\varphi(g)|/\varphi(1) \leq 1/(\sqrt{q}-1) < 3/5$ when $q \geq 8$, and $|\varphi(g)|/\varphi(1) \leq \sqrt{2}/3 < 3/5$ when $q = 7$.
If $q \geq 25$, then $|\varphi(g)|/\varphi(1) \leq 1/(\sqrt{q}-1) \leq 1/4$.

\smallskip
In the remaining cases, we will consider $G$ as a classical group with natural module $V$ and 
let $e(g)$ denote the largest dimension of $g$-eigenspaces on $V \otimes \overline{\F_q}$. As $g \notin \ZB(G)$, we have
$e(g) \leq \dim V-1$.

Consider the case of (b) and view $G = \GL(V)$ with $V = \F_q^n$. If $q=2$, then $\varphi(g)+2=\tau^0_{n,2}(g)+2$ is the 
number of $g$-fixed vectors in $V$, whence $-2 \leq \varphi(g) \leq 2^{n-1}-2$ and so $|\varphi(g)|/\varphi(1) < 1/2$.
Assume $q \geq 3$, and let $\delta \in \overline{\F_q}^\times$ and 
$\td \in \C^\times$ be of order $q-1$. By the character formula \cite[(1.1)]{T2}, 
$$\tau^i_{n,q}(g) = \frac{1}{q-1}\sum^{q-2}_{k=0}\td^{ik}q^{\dim_{\F_q}\Ker(g-\delta^k \cdot 1_V)}-\delta_{i,0}.$$
It is easy to see that $\bigl{|}\sum^{q-2}_{k=0}\td^{ik}q^{\dim_{\F_q}\Ker(g-\delta^k \cdot 1_V)}\bigr{|}$ is at most 
$q^{n-1}+2q-3$ if $e(g)=n-1$, and at most $q^{n-2}(q-1)$ otherwise. It follows that
$$\frac{|\varphi(g)|}{\varphi(1)} \leq \frac{(q^{n-1}+2q-3)/(q-1)+1}{(q^n-q)/(q-1)} < 3/5.$$

\smallskip
In the case of (c), view $G = \GU(V)$ with $V = \F_{q^2}^n$. Let $\xi \in \overline{\F_q}^\times$ and 
$\tx \in \C^\times$ be of order $q+1$. By the character formula \cite[Lemma 4.1]{TZ2}, 
$$\zeta^i_{n,q}(g) = \frac{(-1)^n}{q+1}\sum^{q}_{k=0}\tx^{ik}(-q)^{\dim_{\F_{q^2}}\Ker(g-\xi^k \cdot 1_V)}.$$
Again, it is easy to see that $\bigl{|}\sum^{q}_{k=0}\tx^{ik}q^{\dim_{\F_{q^2}}\Ker(g-\xi^k \cdot 1_V)}\bigr{|}$ is at most 
$q^{n-1}+2q-1$ if $e(g)=n-1$, 
$q^{n-2}+q^2+q-1$ if $e(g)=n-2$, 
and at most $q^{n-3}(q+1)$ otherwise. It follows that
$$\frac{|\varphi(g)|}{\varphi(1)} \leq \frac{(q^{n-1}+2q-1)/(q+1)}{(q^n-q)/(q+1)} \leq 3/5$$
unless $(n,q)=(4,2)$, $(5,2)$. In the cases $(n,q)=(4,2)$, $(5,2)$, the desired bound can be checked directly using \cite{ATLAS}.
If $(n,q) = (3,16)$, then $|\varphi(g)|/\varphi(1) < 1/14$.

\smallskip
Finally, we consider the case of (d), where $G = \Sp(V)$ and $V = \F_q^{2n}$. By the character formula for the Weil characters,
see e.g. Theorem 2.1 and Lemma 3.1 of \cite{GMT},
$$|\varphi(g)| \leq (q^{\dim_{\F_q}\Ker(g-1_V)}+q^{\dim_{\F_q}\Ker(g+1_V)})/2 \leq (q^{n-1/2}+q^{1/2})/2$$
as $e(g) \leq 2n-1$. It follows that 
$$\frac{|\varphi(g)|}{\varphi(1)} \leq \frac{q^{n-1/2}+q^{1/2}}{q^n-1} \leq 2/3$$
unless $(n,q)=(2,3)$. The remaining case $(n,q)=(2,3)$ can be checked directly using \cite{ATLAS}. If in addition
$g$ is a $p'$-element, then $\dim_{\F_q}\Ker(g \pm 1_V) \leq 2n-2$, and so $|\varphi(g)| \leq (q^{n-1}+q)/2$.
\end{proof}

\begin{prop}\label{bound-Q}
Let $G$ be a finite almost quasisimple group: $S \lhd G/\ZB(G) \leq \Aut(S)$ for some simple non-abelian group $S$.
Suppose that at least one of the following conditions holds for $(L,\varphi)$, where $L:=G^{(\infty)}$ and $\varphi \in \Irr(G)$ is any
faithful irreducible character:
\begin{enumerate}[\rm(a)]
\item $L$ is a quotient of $\SL_2(q)$ for some prime power $q \geq 7$ and $\varphi(1) \leq q+1$;
\item $L$ is a quotient of $\SL_n(q)$ for some prime power $q$ and some $n \geq 3$, and $\varphi$ viewed as 
a character of $\SL_n(q)$ is one of the irreducible Weil characters $\tau^i_{n,q}$, $0 \leq i \leq q-2$;
\item $L$ is a quotient of $\SU_n(q)$ for some prime power $q$ and some $n \geq 3$, and $\varphi$ viewed as 
a character of $\SU_n(q)$ is one of the irreducible Weil characters $\zeta^i_{n,q}$, $0 \leq i \leq q$; or
\item $L$ is a quotient of $\Sp_{2n}(q)$ for some odd prime power $q$ and some $n \geq 2$, and every irreducible constituent of
$\varphi|_L$ viewed as 
a character of $\Sp_{2n}(q)$ is one of the four irreducible Weil characters $\xi_i,\eta_i$, $i =1,2$.
\end{enumerate}
Let $1 \neq Q \leq G$ be any subgroup and let $w(Q) := \varphi(1)-[\varphi|_Q,1_Q]_Q$ be the codimension of the 
fixed point subspace of $Q$ in a $\C G$-representation $\Phi$ affording the character $\varphi$. Then 
$$\frac{w(Q)}{\varphi(1)} \geq \left\{ 
     \begin{array}{ll}(1/3)\cdot \bigl( 1-1/|Q| \bigr) \geq 1/6, & \mbox{ in the case of {\rm (d)}},\\ 
     (1/10)\cdot \bigl( 1-1/|Q| \bigr) \geq 1/20, & \mbox{ in the cases of {\rm (a)--(c)},}\\
     (3/16)\cdot \bigl( 1-1/|Q| \bigr), &  \mbox{ in the case of {\rm (a)}, with }q \geq 25,\\
     1/4-2/(5|Q|), & \mbox{ in the cases of {\rm (b), (c)}, with }q\mbox{ prime},\\
     0.377-0.345/|Q|, & \mbox{ in the cases of {\rm (c)}, with }(n,q)=(6,3).
     \end{array}\right.$$
\end{prop}

\begin{proof}
(i) The faithfulness of $\varphi$ implies that any non-identity central element $z \in \ZB(G)$ acts without nonzero fixed points in 
$\Phi$. In particular, $w(Q)=\varphi(1)$ if $Q \cap \ZB(G) \neq 1$. So in what follows we may assume that 
$Q \cap \ZB(G)=1$. Suppose we can find an explicit constant $0 < \alpha < 1$ such that 
$$|\varphi(g)|/\varphi(1) \leq \alpha$$ 
for all $g \in Q \smallsetminus \ZB(G)$. Then 
\begin{equation}\label{bd10}
  \frac{[\varphi|_Q,1_Q]_Q}{\varphi(1)} = \biggl{|}\frac{1}{|Q| \cdot \varphi(1)}\sum_{g \in Q}\varphi(g) \biggr{|} \leq 
  \frac{1+\alpha(|Q|-1)}{|Q|} = \alpha + \frac{1-\alpha}{|Q|},
\end{equation} 
and so 
\begin{equation}\label{bd11}
  \frac{w(Q)}{\varphi(1)} \geq (1-\alpha)\bigl(1-\frac{1}{|Q|} \bigr).
\end{equation}  

\smallskip
(ii) Assume we are in the case of (d). First we consider the case where all irreducible constituents of $\varphi|_L$ are equal to 
a single irreducible Weil character, say $\theta$ (when considered as a character of $\Sp_{2n}(q)$). 
It is well known that, each such $\theta$ is stable under field automorphisms
of $\Sp_{2n}(q)$ -- in fact, it extend to a certain extension $\Sp_{2n}(q) \rtimes C_f \leq \Sp_{2nf}(p)$ 
that induces the full subgroup of outer 
field automorphisms of $\Sp_{2n}(q)$, where $q=p^f$ and $p$ is prime -- but $\theta$ is not stable under outer diagonal automorphisms. 
See e.g. \cite[\S6]{KT6}. As $\ZB(G)$ acts via scalars in $\Phi$, we can extend $\theta$ to a character of $\ZB(G)L$, which is 
still $G$-invariant. But $G/\ZB(G)L$ embeds in the subgroup $C_f$ of field automorphisms of $L$ and so it is cyclic. Hence, by 
\cite[(6.17), (11.22)]{Is}, $\theta$ extends to $G$ and in fact $\varphi|_L = \theta$.  
Thus we may assume that $\Phi$ extends to $\Phi:\Sp_{2nf}(p) \to \GL(V)$ and that
$$\Phi(G) \leq \NB_{\GL(V)}(\Phi(\Sp_{2n}(q))) \leq \Phi(\Sp_{2nf}(p))\ZB(\GL(V)).$$
It follows that $\Phi(g)$ is a scalar multiple of $\Phi(h)$ for some non-central element $h \in \Sp_{2nf}(p)$. Applying 
case (d) of Lemma \ref{weil-value} to $\varphi(h)$, we obtain $|\varphi(g)|=|\varphi(h)| \leq (2/3)\varphi(1)$. Thus we can take 
$\alpha=2/3$ in this case.

Assume now that the set of irreducible constituents of $\varphi|L$ is $\{\xi_1,\xi_2\}$ or $\{\eta_1,\eta_2\}$. By Clifford's theorem
$G$ permutes these two constituents transtively; let $H$ denote the stabilizer of one of them, say $\theta_1$. Then $|G/H| = 2$ and
$H$ fixes both $\theta_1$ and the other constituent $\theta_2$. Moreover, $\Phi|_H = \Phi_1 \oplus \Phi_2$, where 
all irreducible constituents of the character $\varphi_i$ on restriction to $L$ are equal to $\theta_i$ for $i=1,2$, and 
$\ZB(G)$ acts the same in $\Phi_1$ and $\Phi_2$. The preceding analysis applied to $\varphi_i$ shows that 
$|\varphi_i(g)|/\varphi_i(1) \leq 2/3$ and so $|\varphi(g)|/\varphi(1) \leq 2/3$ for all $g \in (Q \cap H) \smallsetminus \ZB(G)$. On the 
other hand, if $g \in Q \smallsetminus H$, then $g$ interchanges $\Phi_1$ and $\Phi_2$ and so $\varphi(g)=0$. Thus 
we have $|\varphi(g)| \leq (2/3)\varphi(1)$ for all $g \in Q \smallsetminus \ZB(G)$, and can take $\alpha = 2/3$ as above.

\smallskip
(iii) In the remaining cases of (a)--(c), note that for any $g \in G \smallsetminus \ZB(G)$, we can find $h \in L$ such that 
$[g,h] \leq L \smallsetminus \ZB(G)$. [Indeed, suppose $[g,x] \in \ZB(G)$ for all $x \in L$. Then for all $y \in L$ we have
$[[x,y],g] = \bigl([y,g],x][[g,x],y]\bigr)^{-1} =1$, and so $g$ centralizes $[L,L]=L$. But this implies $g \in \CB_G(L) = \ZB(G)$.] 
By Lemma \ref{weil-value} (applied to each irreducible constituent of $\varphi|_L$), 
$|\varphi(h)| \leq (3/5)\varphi(1)$. Hence, by \cite[Corollary 2.14]{GT3} we have
$$|\varphi(g)| \leq (3/4)\varphi(1)+(1/4)|\varphi(h)| \leq (9/10)\varphi(1).$$
Thus we can take $\alpha = 9/10$ in these remaining cases.  If $L$ is a quotient of $\SL_2(q)$ with $q \geq 25$ in (a), then 
$|\varphi(h)| \leq \varphi(1)/4$ by Lemma \ref{weil-value}, and so we 
can take $\alpha = 3/4+1/16 = 13/16$. 

\smallskip
(iv) Finally, assume we are in the case of (b) or (c), and $q$ is prime. Then by \cite[Theorem 2.5.12]{GLS}, 
$\Aut(S) = \PGL^\eps_n(q) \rtimes \langle \tau \rangle$ if $S = \PSL^\eps_n(q)$, where $\eps = +$ in the $\GL$-case and 
$\eps=-$ in the $\GU$-case. In particular, at least half of the elements $g \in Q$ must induce only inner-diagonal automorphisms of
$S$. Also, irreducible Weil representations of $L=\SL^\eps_n(q)$ extend to $\Phi:\GL^\eps_n(q) \to \GL(V)$.  Hence, for any such element 
$g \in Q$, we may assume that $\Phi(g) = \gamma \Phi(g')$ for some $g' \in \GL^\eps_n(q)$ and a root of unity $\gamma \in \C^\times$, 
whence $|\varphi(g)| = |\varphi(g')| \leq (3/5)\varphi(1)$ by Lemma \ref{weil-value}. It follows that
$$\frac{[\varphi|_Q,1_Q]_Q}{\varphi(1)} = \bigl{|}\frac{1}{|Q| \cdot \varphi(1)}\sum_{g \in Q}\varphi(g) \bigr{|} \leq 
  \frac{1+(9/10)(|Q|/2)+(3/5)(|Q|/2-1)}{|Q|} = \frac{3}{4} + \frac{2/5}{|Q|},$$
whence  $w(Q)/\varphi(1) \geq 1/4-2/(5|Q|)$. 

Assume furthermore that we are in (c) and $L$ is a quotient of $\SU_6(3)$. The proof of Lemma \ref{weil-value} for 
$\SU_6(3)$ and the character table of $\SU_5(3)$ in \cite{GAP} show that 
$|\varphi(g')|/\varphi(1) < 0.345$ for the aforementioned element $g'$. 
Hence, replacing $3/5$ by $0.345$ in the above estimate, we obtain 
$w(Q)/\varphi(1) \geq 0.377-0.345/|Q|$. 
\end{proof}

\begin{prop}\label{Pinftyimage}Let $\sH$ be an irreducible hypergeometric sheaf on $\G_m/\overline{\F_p}$ of type $(D,m)$
with $W:=D-m > 0$ the dimension of the wild part $\Wild$ of the $I(\infty)$-representation. If $p \nmid W$, then we have the following results.
\begin{itemize}
 \item[(i)]$\Wild$ is the Kummer direct image $[W]_\star(\sL)$ of some linear character $\sL$ of Swan conductor $1$.
 \item[(ii)]$\Wild$ as a $P(\infty)$ representation is the direct sum of the $W$ multiplicative translates of $\sL|_{P(\infty)}$ by $\mu_W$.
 \item[(iii)] Any element of $I(\infty)$ of pro-order prime to $p$ which maps onto a generator of $I(\infty)/P(\infty)$ acts on the set of the $W$ irreducible consituents of $\Wild|_{P(\infty)}$ through the quotient $\mu_W$ of $I(\infty)$, cyclically permuting these multiplicative translates of $\sL|_{P(\infty)}$. 
\item[(iv)]The image of $P(\infty)$ is isomorphic to the additive group of the finite field $\F_p(\mu_W)$.
 \end{itemize}
\end{prop}
\begin{proof}Statement (i) is proven in  \cite[1.14 (2)]{Ka-GKM}. Statements (ii) and (iii) result from (i), cf. \cite[proof of 3.1]{Ka-RL-T-2Co1}.
For (iv), there is nothing to prove if $W=1$. If $W \ge 2$, by \cite[8.6.3]{Ka-ESDE}, the $I(\infty)$-isomorphism class of $\Wild$
up to multiplicative translation depends only on $\det(\Wild)$, a tame character which we can change as we like by tensoring $\Wild$ with a tame character. Such tensoring with a tame character does not alter the action of $P(\infty)$, and allows us to reduce to the case where $\det(\Wild)=\chi_2^{W-1}$, and then apply \cite[3.1]{Ka-RL-T-2Co1}.
\end{proof}

\begin{thm}\label{char-sheaf1}
Let $\sH$ be an irreducible hypergeometric sheaf in characteristic $p$ of rank $D$ with finite geometric monodromy group $G=G_\geo$. Suppose that $G$ is an almost quasisimple group of Lie type: 
$$S \lhd G/\ZB(G) \leq \Aut(S)$$ 
for some finite simple group $S$ of Lie type in 
characteristic $r$. 
Then at least one of the following statements holds.
\begin{enumerate}[\rm(i)] 
\item $p=r$, i.e. $\sH$ and $S$ have the same characteristic.
\item $D \leq \edit{22}$ and $S$ is one of the following simple groups: $\PSL_2(5,7,8,9,11,25)$, 
$\SL_{3,4}(2)$, $\PSL_3(3,4)$, $\PSU_{4,5,6}(2)$, $\PSU_{3,4}(3)$, $\PSU_3(4,5)$, $\Sp_6(2)$, 
$\PSp_{4,6}(3)$, $\PSp_{4}(5)$, $\Omega^+_8(2)$, $\tw2 B_2(8)$, $G_2(3,4)$. 
\end{enumerate}
\end{thm}

\begin{proof}
(a) As explained above, a generator $g$ of the image $I(0)$ on $\sH$ has simple spectrum on $\sH$. Hence we can apply Theorem
\ref{simple} to the faithful irreducible representation $\Phi:G \to \GL(\C^D)$ induced by the action of $G$ on $\sH$.
Since the non-generic cases of Theorem \ref{simple}(v) are already included in (ii), 
we may assume that the character $\varphi$ of $\Phi$ and 
the subgroup $L=G^{(\infty)}$ fulfills the assumptions of Proposition \ref{bound-Q}. We can therefore apply Proposition 
\ref{bound-Q} to the subgroup $Q=\OB_p(J)$, the image of $P(\infty)$ on $\sH$, where $J=QC$ is the image of $I(\infty)$ on
$\sH$, with $C$ the cyclic tame quotient, and $W=w(Q)$ is the dimension of the wild part for $I(\infty)$ on $\sH$.     

First we note that if $|Q| = 2$, then $Q$ has a unique nontrivial irreducible character (of degree $1$), and so $W\leq 1$ and 
$D \leq 20$ by Proposition \ref{bound-Q}. Thus, by assuming $D \geq \edit{21}$, we may assume that $|Q| \geq 3$. In the rest of 
the proof we will assume that $|Q| \geq 3$ and that $r \neq p$, and work to bound $D = \rank(\sH)$.

\medskip
(b) Here we consider the symplectic case: $S = \PSp_{2n}(q)$, where $n \geq 2$, $q=r^f$, and $r \neq 2$. By Proposition
\ref{bound-Q}, $W \geq (2/9)D$, and $D \geq (q^n-1)/2$. Since any irreducible Weil character is invariant under field
automorphisms of $\Sp_{2n}(q)$, we have by Gallagher's theorem \cite[(6.17)]{Is} that $D$ or $D/2$ is
$(q^n \pm 1)/2$.
 
We can view $\Sp_{2n}(q)$ as $\Sp(V)$ with $V = \F_q^{2n}$, 
where $V$ is endowed with a non-degenerate symplectic form $(\cdot,\cdot)$. We also consider the {\it conformal 
symplectic group}
$$\CSp(V) = \left\{ X \in \GL(V) \mid \exists \kappa(X) \in \F_q^\times,~(Xu,Xv) = \kappa(X)(u,v),~ \forall u,v \in V\right\}$$
which contains $\Sp(V)$ as a normal subgroup with cyclic quotient of order $q-1$. Next we consider the
representation $\Lambda:\GL(V) \to \GL(\wedge^2(V))$, and its twisted restriction
$$\Lambda': \CSp(V) \to \GL(\wedge^2(V)),~\Lambda'(X) = \kappa(X)^{-1}\Lambda(X)$$
to $\CSp(V)$. It is straightforward to check that $\Ker(\Lambda') = \ZB(\CSp(V))$, and $\Lambda'$ induces 
a faithful action of $S$ on the quotient of $\wedge^2(V)$ by the trivial submodule that extends to 
$\PCSp(V) := \CSp(V)/\ZB(\CSp(V))$, which is the group of inner-diagonal automorphisms of $S$, cf. \cite[Theorem 2.5.12]{GLS}.

Now we can embed $G/\ZB(G)$ in $\Aut(S) \cong \PCSp(V) \rtimes C_f$ and set $R := G/\ZB(G) \cap \PCSp(V)$.
Applying Theorem \ref{bound2} with $d:= (2n+1)(n-1)$, we get 
\begin{equation}\label{cs11}
  (q^n-1)/9  \leq 2D/9 \leq W \leq d \cdot [G/\ZB(G):R] \leq (2n+1)(n-1)f.
\end{equation}  
Now, if $n \geq 6$, then $q^{n-1} \geq 3^{n-1} > 6n^2$, whence
$$q^n > 6n^2q \geq 18n^2f,$$
contradicting \eqref{cs11}. If $2 \leq n \leq 5$, then \eqref{cs11} leads us to one of the following possibilities.

\smallskip
(b1) $q=3$, $2 \leq n \leq 5$, $D$ or $D/2$ is $(q^n \pm 1)/2$. Assume $(n,q) = (5,3)$ or $(4,3)$. 
In this case, as $p \neq r=3$, any $g \in Q \smallsetminus \ZB(G)$ is 
a $3'$-element, and so $|\varphi(g)|/\varphi(1) \leq 3/8$ by Lemma \ref{weil-value} if $g \in \ZB(G)L$. 
If $g \in Q \smallsetminus \ZB(G)L$, then $g$ fuses the two irreducible Weil characters of any given degree 
$(q^n \pm 1)/2$, whence we must have that $D =q^n \pm 1$ and $\varphi(g)=0$. Thus we have 
$|\varphi(g)|/\varphi(1) \leq 3/8$ for all $g \in Q \smallsetminus \ZB(G)$. Also, since $|Q| > 2$ and $Q$ is not a $3$-group,
we have that $|Q| \geq 4$. 
The proof of Proposition \ref{bound-Q} now implies that 
$$W \geq D\bigl(1-\frac{3}{8}\bigr)\bigl(1-\frac{1}{|Q|}) \geq 15D/32,$$
whence \eqref{cs11} yields that $n=4$, $D \in \{40,41\}$, $G = \ZB(G)L$, and 
\begin{equation}\label{cs11a}
  19 \leq W \leq 27,~Q \cap \ZB(G)=1,
\end{equation}  
where the second conclusion follows from $W < D$ and Proposition \ref{p-center}(i). 
Assume in addition that $D=40$.
Then $L = \Sp_8(3)$, and it is easy to see that $G/\OB_3(\ZB(G))$ admits a faithful $8$-dimensional representation $\Lambda$ over 
$\overline{\F_3}$. Applying Theorem \ref{bound1} to $\Lambda|_J$ and using \eqref{cs11a}, we obtain that $Q \leq \ZB(G)$, 
a contradiction. Now we consider the other possibility $D=41$, for which $G=\ZB(G) \times S$. By \eqref{cs11a}, $p \in \{2,5,7,13,41\}$ and
$Q$ embeds in a Sylow $p$-subgroup of $S = \PSp_8(3)$. Also, recall that the $Q$-module $\Wild$ is a sum of $a$ 
pairwise non-isomorphic simple $Q$-modules of dimension $p^c$, where $W=ap^c$ and $p \nmid a$. If $p=7$ or $13$, then 
$Q \cap C_p$, whence $a \leq p-1 \leq 12$ and $p^c = 1$, yielding the contradiction 
$W \leq 12$. If $p=5$ or $13$, then using the character table of $S$ given in \cite{GAP} one check that 
$|\varphi(g)| \leq 4$ for all $g \in Q \smallsetminus \ZB(G)$, whence $W \geq (37/41)(1-1/|Q|)D > 29$, contradicting \eqref{cs11a}.
Thus we must have $p=2$. Again using the character table of $S$ we get 
$|\varphi(g)| \leq 15$ for all $g \in Q \smallsetminus \ZB(G)$. Also, if $|Q| \leq 16$, then $Q$ has at most 
$16$ irreducible characters of degree $1$, $3$ of degree $2$, and none of degree $> 2$, whence $W \leq 15$, contradicting 
\eqref{cs11a}. Thus $|Q| \geq 32$, and so $W \geq (26/41)(1-1/|Q|)D > 25$, i.e. $W=27$ or $26$. In the former case, we know
that $Q$ can be identified with the additive group of $\F_2(\tilde\zeta_{27}) \cong \F_{2^{18}}$, which is impossible since $Q$ embeds in
$\PSp_8(3)$. Consider the latter case $W=26$, in which a generator $g$ of the tame quotient of $I(\infty)$ permutes cyclically 
$a=13$ simple $Q$-modules in $\Wild$, and has a simple spectrum on the tame part of dimension $15$, since $I(\infty)$ has 
finite image $J$. Thus the $2'$-element $g$ has order divisible by $13$ but larger than $15$. Hence, we can write 
$g=zh$, where $z \in \ZB(G)$ acts as a scalar on $\sH$, and $h \in S$ is an element of order $39$. 
Using the character table of $S$ and letting $\zeta=\zeta_{39} \in \C^\times$, we have that the spectrum of $h$ on $\sH$ is the 
disjoint union $X \sqcup Y \sqcup Y \sqcup \{1,\zeta^{26}\}$ (with counting multiplicities), 
where $X = \mu_{13} = \langle \zeta^3 \rangle$,  and $Y= \zeta^2\mu_{13}$. 
On the other hand, because of the cyclic action of $g$ (and $h$) on the $13$ $Q$-summands in 
$\Wild$, the spectrum of $g$ (and $h$) on $\Wild$ must be the union of some $\mu_{13}$-cosets, whereas the spectrum on
the tame part is simple. Now, if the spectrum $A$ of $h$ on $\Wild$ is $Y \sqcup Y$, then the spectrum $B$ of $h$ on the tame part 
contains $1$ twice. If $A$ is $X \sqcup Y$, then $B$ contains $\zeta^{26}$ twice, again a contradiction.

Next, we consider the case $n=3$ and $D = q^n \pm 1$. Using the character table of $\Sp_6(3) \cdot 2$ \cite{GAP}, one 
can check that $|\varphi(g)| \leq D-16$ for all $g \in Q \smallsetminus \ZB(G)$. As $|Q| \geq 4$, it follows
that $W \geq 16(1-1/4) = 12$. Since $Q$ acts on the wild part of $\sH$ with pairwise non-isomorphic simple 
summands, it follows that $|Q| \geq 9$, forcing $W \geq 16(8/9) > 14$, contradicting \eqref{cs11}.
We have shown that $n \leq 3$ and $D \leq \edit{14}$.

\smallskip
(b2) $(n,q) = (3,5)$, $(2,9)$, $(2,5)$. 
In the case $(n,q)=(3,5)$, the same arguments as in (b1) also apply, yielding $|\varphi(g)|/\varphi(1) \leq 15/62$, whence
$W \geq (47/62)(1-1/3)D > 14$, contradicting \eqref{cs11}. In the remaining cases \eqref{cs12} forces $D = (q^n \pm 1)/2$, and 
so $G$ can only induce inner and field automorphisms of $S$. In the case $(n,q) = (2,9)$, this implies that, modulo scalars,
$\Phi(Q)$ is contained in the image of $\Sp_8(3)$ in a Weil representation, and so the arguments in (b1) again apply, 
yielding $W \geq 15D/32 > 18$, contradicting \eqref{cs11}. 
Thus $(n,q) = (2,5)$, and $D \leq \edit{13}$. 

\medskip
(c) Next we consider the linear case: $S = \PSL_{n}(q)$, where $n \geq 3$ and $q=r^f$. By Proposition
\ref{bound-Q}, $W \geq D/15$, and $D \geq (q^n-q)/(q-1)$. Recall \cite[Theorem 2.5.12]{GLS} that 
$$\Aut(S) \cong \PGL_n(q) \rtimes C,$$
where $C$ is an abelian group of order $2f$. Embedding $\bar{G}:=G/\ZB(G)$ in $\Aut(S)$, we let 
$$R_1 := \bar{G} \cap \PGL_n(q) \lhd \bar{G},$$ 
so that $[\bar{G}:R_1]$ divides $|C|=2f$. As noted in the proof of Theorem \ref{bound2}, $\PGL_n(q) \leq \PGL(V)$  acts
faithfully and  irreducibly on a subquotient $A(V)$ of $V \otimes V^*$ (of dimension $n^2-1$ if $r\nmid n$ and $n^2-2$ if 
$r|n$), where $V = \overline{\F_q}^n$, and moreover this action is extendible to $\PGL(V) \rtimes \langle \tau \rangle$, where 
$\tau$ is the transpose-inverse automorphism of $\PGL(V)$. Viewing $R_1$ inside $\PGL(V)$, 
we get a faithful irreducible action of $R_1$ which also extends to $R_1 \rtimes \langle \tau \rangle$. 

If $[\bar{G}:R_1] \leq f$, we set $R:=R_1$. If $\bar{G}:R_1] > f$, then $\bar{G}/R_1 \cong C = C_f \rtimes \langle \tau \rangle$. 
In this latter case, there is some element $\bar{x} \in \bar{G} \smallsetminus R_1$ such that $\bar{x}^2 \in R_1$ and
$\bar{x}$ induces the automorphism $\tau$ on $R_1$. Then we set $R:=\langle R_1,\bar{x} \rangle$ and obtain a faithful (at least 
on $R_1$) irreducible action on $A(V)$. Now we can apply Theorem \ref{bound2} with $d:= \dim(A(V)) \leq n^2-1$ to get 
\begin{equation}\label{cs12}
  \frac{q^n-q}{15(q-1)}  \leq D/15 \leq W \leq (n^2-1) \cdot [\bar{G}:R] \leq (n^2-1)f.
\end{equation}  
Note that if $\gcd(n,q-1)=1$, then $\PGL_n(q) \cong \SL_n(q)$ (but the action of $S$ on $V$ does {\it not} extend to 
$S \rtimes \langle \tau \rangle$), and so we can apply Theorem \ref{bound2} with $d:=n$ to get 
\begin{equation}\label{cs13}
  \frac{q^n-q}{15(q-1)}  \leq D/15 \leq W \leq d \cdot [\bar{G}:R_1] \leq 2nf.
\end{equation}  
Furthermore, if $q=r$ is prime, then by Proposition \ref{bound-Q} we have $W > D/8.6$, hence the constants $15$ in \eqref{cs12} and 
\eqref{cs13} can be replaced throughout by $8.6$. Another observation is that, when $r=2$ {\it and} $f$ is a $2$-power, 
since $p \neq r=2$, $Q\ZB(G)/\ZB(G) \leq R_1$ for the $p$-group $Q$. Hence 
$|\varphi(g)|/\varphi(1) \leq 0.6$ for all $g \in Q \smallsetminus \ZB(G)$ by Lemma \ref{weil-value}. Now 
the proof of Proposition \ref{bound-Q} show that $W/D > 0.4(1-1/|Q|) \geq 4/15$, hence the constant $15$ 
in \eqref{cs12} and \eqref{cs13} can be replaced by $15/4$.

Now, if $n \geq 11$, or if $n \geq 7$ and $q \geq 3$, then $(q^{n-1}-1)/(q-1) > 7.5n^2$, whence
$$q^n-q > 7.5n^2q(q-1) \geq 15n^2(q-1)f,$$
contradicting \eqref{cs12}. If $2 \leq n \leq 10$, then \eqref{cs12} and \eqref{cs13} imply that one of the following holds.

\smallskip
(c1) $q=2$, $n \leq 5$, and $D=2^n-2$. If $n=5$, then since $Q$ is not a $2$-group, the character table of $\SL_5(2)$ \cite{GAP} shows
that $|\varphi(g)|/\varphi(1) \leq 1/3$, whence $W \geq (2/3)(1-1/|Q|)D > 13$, contradicting \eqref{cs13}.
Thus $n \leq 4$ and $D=2^n-2 \in \{6,\edit{14}\}$.

\smallskip
(c2) $q=3$, $n = 4$, 
and $D = (3^n-3)/2, (3^n-1)/2$. Using the character tables of $L \cdot 2_1$, $L \cdot 2_2$, and $L \cdot 2_3$
given in \cite{ATLAS}, one can check that $|\varphi(g)|/\varphi(1) \leq 1/3$ for all $g \in Q \smallsetminus \ZB(G)$, whence
$W \geq (2/3)(1-1/|Q|)D > 18$, contradicting \eqref{cs12}. 

\smallskip
(c3) $n=3$, $3 \leq q \leq 9$, $D = q(q+1)$ or $q^2+q+1$. Suppose $q=9$, whence $S = \SL_3(9)$ and 
$\Aut(S) = S \rtimes C_2^2$, or $q=8$, whence $S = \SL_3(8)$ and $\Aut(S)=S \rtimes C_3$.
The character tables of all groups between $S$ and $\Aut(S)$ are known in \cite{GAP}, from which 
we can check that $|\varphi(g)|/\varphi(1) \leq 1/7$ for all $g \in Q \smallsetminus \ZB(Q)$. Hence the proof of Proposition 
\ref{bound-Q} implies that $W/D \geq 6/7(1-1/|Q|) > 1/2$. Thus we can use $2$ instead of the constant
$15$ in \eqref{cs13}, which now leads to a contradiction. 

Next suppose that $q=7$, whence $D = 56$ or $57$. Note that when $D=57$, $L = \SL_3(7)$ has its center inverted by 
the transpose-inverse automorphism of $S$, hence, up to scalars, $\Phi(G)$ is contained in the image of 
$\GL_3(7)$ in a Weil representation of degree $57$. If $D = 56$, then either $p = 2$ and, up to scalars, $\Phi(Q)$ is contained
the image of $\PSL_3(7) \rtimes C_2$, or $p \neq 2$ and, up to scalars, $\Phi(Q)$ is contained
the image of $\PGL_3(7)$, in a representation of degree $56$. Hence, using the proof of Lemma \ref{weil-value} for $\GL_3(7)$
and the character table of $\PSL_3(7) \cdot C_2$ in \cite{GAP}, we can check that $|\varphi(g)|/\varphi(1) \leq 11/56$ for 
all $g \in Q \smallsetminus \ZB(G)$. It follows that $W \geq (45/56)(1-1/3)D \geq 30 > 8$, contradicting \eqref{cs12}.

If $q=5$, then the character table of $\SL_3(5) \cdot 2$ \cite{GAP} shows
that $|\varphi(g)|/\varphi(1) \leq 7/31$, whence $W \geq (24/31)(1-1/|Q|)D > 15$, contradicting \eqref{cs13}.
Hence $q \leq 4$ and $D \leq \edit{21}$.

\medskip
(d) Now we handle the unitary case: $S = \PSU_{n}(q)$, where $n \geq 3$ and $q=r^f$. By Proposition
\ref{bound-Q}, $W \geq D/15$, and $D \geq (q^n-q)/(q+1)$. Here we have 
$$\Aut(S) \cong \PGU_n(q) \rtimes C,$$
where $C \cong C_{2f}$, by \cite[Theorem 2.5.12]{GLS}. Embedding $\bar{G}:=G/\ZB(G)$ in $\Aut(S)$, we let 
$$R_1 := \bar{G} \cap \PGU_n(q) \lhd \bar{G},$$ 
so that $[\bar{G}:R_1]$ divides $2f$. As in (c), $\PGU_n(q) \leq \PGL(V)$  acts
faithfully and  irreducibly on a subquotient $A(V)$ of $V \otimes V^*$ (of dimension $n^2-1$ if $r\nmid n$ and $n^2-2$ if 
$r|n$), where $V = \overline{\F_q}^n$, and this action is extendible to $\PGL(V) \rtimes \langle \tau \rangle$, 
where $\tau$ is the transpose-inverse automorphism of $\PGL(V)$. Viewing $R_1$ inside $\PGL(V)$, we get a faithful irreducible action of $R_1$ which also extends to $R_1 \rtimes \langle \tau \rangle$. Note that $\tau$ can be identified with an involution in the 
subgroup $C_{2f}$ of $\Aut(S)$.

If $[\bar{G}:R_1] \leq f$, we set $R:=R_1$. If $\bar{G}:R_1] > f$, then $\bar{G}/R_1 \cong C$. 
In this latter case, there is some element $\bar{x} \in \bar{G} \smallsetminus R_1$ such that $\bar{x}^2 \in R_1$ and
$\bar{x}$ induces the automorphism $\tau$ on $R_1$. Then we set $R:=\langle R_1,\bar{x} \rangle$ and obtain a faithful (at least 
on $R_1$) irreducible action on $A(V)$. Now we can apply Theorem \ref{bound2} with $d:= \dim(A(V)) \leq n^2-1$ to get 
\begin{equation}\label{cs14}
  \frac{q^n-q}{15(q+1)}  \leq D/15 \leq W \leq (n^2-1) \cdot [\bar{G}:R] \leq (n^2-1)f.
\end{equation}  
Note that if $\gcd(n,q+1)=1$, then $\PGU_n(q) \cong \SU_n(q)$ (but the action of $S$ on $V$ does {\it not} extend to 
$S \rtimes \langle \tau \rangle$), and so we can apply Theorem \ref{bound2} with $d:=n$ to get 
\begin{equation}\label{cs15}
  \frac{q^n-q}{15(q+1)}  \leq D/15 \leq W \leq d \cdot [\bar{G}:R_1] \leq 2nf.
\end{equation}  
Furthermore, if $q=r$ is prime, then by Proposition \ref{bound-Q} we have $W > D/8.6$, hence the constants $15$ in \eqref{cs14} and 
\eqref{cs15} can be replaced throughout by $8.6$. If, on the other hand, $r=2$ and $f$ is a $2$-power, then as in (c) the constant $15$ 
in \eqref{cs12} and \eqref{cs13} can be replaced by $15/4$.

Now, if $n \geq 13$, or if $n \geq 8$ and $q \geq 3$, then $(q^{n-1}-1)/(q+1) > 7.5n^2$, whence
$$q^n-q > 7.5n^2q(q+1) \geq 15n^2(q+1)f,$$
contradicting \eqref{cs14}. If $2 \leq n \leq 12$, then \eqref{cs14} and \eqref{cs15} imply that one of the following holds.

\smallskip
(d1) $q=2$ and $n \leq 9$ but $n \neq 8$.  Assume we are in the case
$(n,q) = (9,2)$, so that $D = 170$ or $171$. As mentioned above, since $r=2$ and $f=1$, we have the bound 
$W/D \geq (4/15)D$  and so $W \geq 46$. Now if the $p$-abelian group $Q$ is non-abelian, then $|Q| \geq p^3 \geq 27$.
If $Q$ is abelian, then since the wild part of $\sH$ is a sum on non-isomorphic irreducible $Q$-modules, we get
$|Q| \geq w+1 \geq 47$. Thus in either case $|Q| \geq 27$. Since $Q\ZB(G)/\ZB(G) \leq R_1$, the proof of Lemma 
\ref{weil-value} shows that $|\varphi(g)|/\varphi(1) < 0.508$ for all $g \in Q \smallsetminus \ZB(G)$. Arguing as in 
the proof of Proposition \ref{bound-Q}, we obtain 
$$W/D > (1-0.508)(1-1/27) > 1/2.12.$$ 
Thus we can use $2.12$ instead of the constant
$15$ in \eqref{cs14}, which now leads to a contradiction. 

Next suppose that $n=7$. Since $Q$ is not a $2$-group and $|\Out(S)|=2$, $Q\ZB(G) \leq \ZB(G)L$. Using the character table
of $\SU_7(2)$ \cite{GAP} one can check that $|\varphi(g)|/\varphi(1) < 0.51$ for all $g \in Q \smallsetminus \ZB(G)$, whence
$W > 0.49(1-1/3)D > 13$. This in turn implies that $|Q| \geq 5$, and so 
$W > 0.49(1-1/5)D > 16$, contradicting \eqref{cs15}.
Hence, in fact we have $n \leq 6$, and so
$D=(2^n+2(-1)^n)/3, (2^n-(-1)^n)/3 \leq \edit{22}$. 

\smallskip
(d2) $q=3$ and $n \leq 6$. As $|Q| > 2$ and $Q$ is not a $3$-group, we have $|Q| \geq 4$. If moreover
$n = 5$ or $6$, then by Proposition \ref{bound-Q}, $W > D/4$, hence we can use \eqref{cs15}, respectively \eqref{cs14}, 
with $8.6$ replaced by $4$, yielding a contradiction ruling out this case. Hence $3 \leq n \leq 4$, and 
$D = (3^n+3(-1)^n)/4, (3^n-(-1)^n)/4 \in \{ 6, 7, 20, \edit{21}\}$.


\smallskip
(d3) $n=4$ and $q=4,5$. Suppose $q=5$. Then the proof of Lemma \ref{weil-value} shows that $|\varphi(g')|/\varphi(1) < 1/4$
for all $g' \in \GU_4(5) \smallsetminus \ZB(\GU_4(5))$. Arguing as in part (iii) of the proof of Proposition \ref{bound-Q} we 
get $|\varphi(g)|/\varphi(1) \leq 3/4+1/16=13/16$ for all $g \in G \smallsetminus \ZB(G)$. Now, arguing as in part (iv) of 
the proof of Proposition \ref{bound-Q} we obtain $W/D \geq 15/32-3/4|Q| > 1/5$. Thus we can use $5$ instead of the constant
$15$ in \eqref{cs14}, which now leads to a contradiction. Next suppose that $q=4$. As $\Out(S) = C_4$ and $Q$ is not a $2$-group,
$Q/\ZB(G) \leq S$. Hence, using the character table of $\SU_4(4)$ \cite{GAP}, one can check that 
$|\varphi(g)|/\varphi(1) < 1/2$ for all $g \in Q \smallsetminus \ZB(G)$, and so $W/D > (1/2)(1-1/|Q|) > 1/3$. 
Thus we can use $3$ instead of the constant
$15$ in \eqref{cs15}, which again leads to a contradiction.

\smallskip
(d4) $n=3$, $4 \leq q \leq 9$, and $D = q(q-1)$ or $q^2-q+1$. Suppose $q=9$, whence $S = \SU_3(9)$ and 
$\Aut(S) = S \rtimes C_4$.
The character tables of all groups between $S$ and $\Aut(S)$ are known in \cite{GAP}, from which 
we can check that $|\varphi(g)|/\varphi(1) < 1/7$ for all $g \in Q \smallsetminus \ZB(Q)$. Hence the proof of Proposition 
\ref{bound-Q} implies that $W/D \geq 6/7(1-1/|Q|) > 1/2$. Thus we can use $2$ instead of the constant
$15$ in \eqref{cs15}, which now leads to a contradiction. 

Next suppose that $q=8$, whence $\Out(S) = C_3 \times {\mathsf{S}}_3$. As $Q$ is not a $2$-group, for any 
$g \in Q \smallsetminus \ZB(G)$ the coset $g\ZB(G)$ belongs to one of the three almost simple groups $S \cdot 3_1$, $S \cdot 3_2$, and 
$S \cdot 3_3$ listed in \cite{ATLAS}. Using the character tables of covers of these groups given in \cite{ATLAS}, we can 
check that $|\varphi(g)|/\varphi(1) \leq 1/7$, whence $W \geq (6/7)(1-1/|Q|)D > 8f$, contradicting \eqref{cs14}.

If $q=7$, then using \cite{ATLAS} one can check that $|\varphi(g)|/\varphi(1) \leq 7/43$, whence $W \geq (36/43)(1-1/|Q|)D > 8f$, 
again contradicting \eqref{cs14}.
Hence $q \leq 5$ and $D \leq \edit{21}$.

\medskip
(v) Finally, we consider the case $S = \PSL_2(q)$ with $q=r^f$, whence $D \leq q+1$.
First we analyze the cases with $D \geq \edit{25}$; in particular, $q \geq 25$. 
By Proposition \ref{bound-Q}, $W \geq D/8$, and $D \geq (q-1)/\gcd(2,q-1)$. 
We also note in this case that $|Q| \geq 5$ (because if $|Q| \leq 4$, then $Q$ is abelian and has at most $3$ nontrivial irreducible 
characters, all of degree $1$, when $W \leq 3$ and so $D \leq 24$). Lemma \ref{weil-value} and Proposition \ref{bound-Q} now imply that
$W/D \geq 3/20$ (with equality possibly only when $q=25$).  Arguing as in (c)
using $R = \bar{G} \cap \PGL_2(q)$, instead of \eqref{cs12} and \eqref{cs13} we now have 
\begin{equation}\label{cs16}
  q-1 \leq \left\{ \begin{array}{ll} 2D \leq 40f, & r > 2,\\D \leq 40f/3, & r=2. \end{array} \right.
\end{equation}  
This can happen only when $q \leq 3^4$. 

\smallskip
We will now analyze the remaining cases $q \leq 3^4$ further, following the proof of Proposition \ref{bound-Q}
(and using \eqref{cs16} only when $D \geq 25$). 
If $q= 2^6$, then the character table of $\SL_2(64)$ \cite{GAP} shows that  
$|\varphi(g')|/\varphi(1) \leq 2/63$ for all $g' \in L \smallsetminus \ZB(G)$, whence 
$|\varphi(g)|/\varphi(1) \leq 3/4 + (1/4)(2/63)=191/252$. Thus $W \geq (61/252)(1-1/5)D > 12 = 2f$, a contradiction. If $q= 2^5$, then the character tables of $\SL_2(32)$ and $\SL_3(32) \cdot C_5$ in \cite{GAP} shows that  
$|\varphi(g)|/\varphi(1) \leq 2/31$ for all $g \in G \smallsetminus \ZB(G)$. Hence $W \geq (29/31)(1-1/5)D > 23 > 2f$, again a contradiction.

If $q=3^4$, then since $D < 80$ by \eqref{cs16} (note that the equality in \eqref{cs16} when $r>2$ can occur only when $q=25$),
we must have $D = (q \pm 1)/2$. In particular, $G$ can only induce inner and fields automorphisms of $S$. Thus, modulo scalars,
$\Phi(G)$ is contained in the image of $\Sp_8(3)$ in a Weil representation of degree $D$. Arguing as in (b1), we obtain
$W > 15D/32 > 18 > 3f$, a contradiction.

If $q=7^2$, then since $D \leq 40$, we must have $D = (q \pm 1)/2$, whence again $G$ can only induce inner and fields automorphisms of $S$. Thus, modulo scalars, $\Phi(G)$ is contained in the image of $\Sp_4(9)$ in a Weil representation of degree $D$. Arguing as in (b1), we obtain $W > 15D/32 > 11 > 3f$, again a contradiction.

If $q=3^3$ or $q=r \geq 13$, then since $Q$ is not an $r$-group, we must have that $Q\ZB(G)/\ZB(G) \leq \PGL_2(q)$. 
The character tables of $\SL_2(q)$ and $\GL_2(q)$ \cite[Ch. 15]{DM} show that $|\varphi(g)|/\varphi(1) \leq 2/(q-1) \leq 1/6$ 
for all $g \in Q \smallsetminus \ZB(G)$. Hence, $W \geq (5/6)(1-1/|Q|)D \geq 6(5/6)(2/3) > 3$, contradicting \eqref{cs12} when $q=r$.
When $q=3^3$, the same bound but using $D \geq 13$ yields $W > 7$, and so $|Q| \geq 8$. Using the same bound again,
we get $W \geq (5/6)(7/8)13 > 9=3f$, again contradicting \eqref{cs12}.

If $q=5^2$, then using the character tables in \cite{ATLAS} we can check that $|\varphi(g)|/\varphi(1) \leq 5/13$ for 
all $g \in Q \smallsetminus \ZB(G)$. By \eqref{cs12} we now have $6 \geq W > (8/13)(2/3)D$, whence $D = (q \pm 1)/2\leq \edit{13}$.
Thus we have shown that $D \leq \edit{13}$, and either $q = \edit{25}$ or $q \leq \edit{11}$.
\end{proof}

\subsection*{7B. The extraspecial case}
Next we determine the characteristic of hypergeometric sheaves $\sH$ whose geometric monodromy groups
are in the extraspecial case (iii) of \cite[Proposition 2.8]{G-T}.

\begin{thm}\label{char-sheaf2}
Let $\sH$ be a hypergeometric sheaf in characteristic $p$, of type $(D,m)$ with $D > m$. Suppose that 
$D=r^n > 1$ for some prime $r$ and that the geometric monodromy group $G=G_\geo$ of $\sH$ contains a normal $r$-subgroup $R$, such that $R=\ZB(R)E$ for an extraspecial $r$-group $E$ of order $r^{1+2n}$ that acts irreducibly on $\sH$, and either $R=E$ or $\ZB(R) \cong C_4$.
Then either $p=r$, or $D \in \{2,3,4,5,8,9\}$.
\end{thm}

\begin{proof}
Assume $p \neq r$, and let $J = QC$ denote the image of $I(\infty)$ on $\sH$, with $Q = \OB_p(J)$ being the image of 
$P(\infty)$ and $C$ the image of the tame quotient. Also let $\Phi$ denote the representation of $G$ on $\sH$, and $\varphi$ 
denote the character of $\Phi$. As in the proof of Theorem \ref{char-sheaf1}, first we show that the dimension 
$W = \varphi(1)-[\varphi|_Q,1_Q]_Q$ of the wild part of $\sH$ satisfies
\begin{equation}\label{cs21}
 W \geq D/3.
\end{equation}
Note that $D > m$ implies that $Q \neq 1$. Moreover, $W=D$ if $Q \cap \ZB(G) \neq 1$. Hence we may assume that $Q \cap \ZB(G)=1$,
and consider any $g \in Q \smallsetminus \ZB(G)$. We will use the well-known fact (see e.g. \cite[Theorem 1]{Wi}) 
that the group $\Out_1(R)$ of all outer automorphisms of $R$ 
that act trivially on $\ZB(R)$ is contained in $\Sp_{2n}(r)$, and so (identifying the groups in consideration with their images on $\sH$)
$$G \leq \NB_{\GL(\sH)}(R) \leq \ZB(\GL(\sH))R \cdot \Sp_{2n}(r).$$ 
Since $p \neq r$ and $g \notin \ZB(G)$, $g$ projects onto a nontrivial {\it semisimple} element $\bar{g}$ of $\Sp_{2n}(r)$. In 
particular, if we view $\Sp_{2n}(r)$ as $\Sp(U)$ with $U := \F_r^{2n}$, then $\dim\Ker_{\F_r}(\bar{g}-1_U) \leq 2n-2$.
Applying \cite[Lemma 2.4]{GT1}, we obtain 
\begin{equation}\label{cs22}
  |\varphi(g)| \leq r^{n-1} = \varphi(1)/r.
\end{equation}  
Now, using \eqref{bd11}, we see that $W/D \geq (1-1/3)(1-1/2) = 1/3$ if $r \geq 3$. If $r=2$, then 
$|Q| \geq 3$ (as $Q \neq 1$ is a $p$-group), and so $W/D \geq (1/2)(1-1/3) = 1/3$ again.

\smallskip
(ii) Recall that the conjugation action of $G$ on $R$ induces a homomorphism $\Psi: G \to \Aut(R)$, with 
$\Ker(\Psi) = \CB_G(R) = \ZB(G)$. Composing with the projection $\Aut(R) \surj \Out(R)$ (with kernel
$R/\ZB(R)$), we obtain a homomorphism $\Lambda: G \to \Out_1(R) \leq \Sp(U)$, with 
$\Ker(\Lambda) = \ZB(G)R$.  Suppose now that $2n < W$. Then, by Theorem \ref{bound1}, $\Lambda|_J$ is tame, i.e. 
$Q \leq \ZB(G)R$. As $Q$ is a $p$-group and $p \neq r$, it follows that $Q \leq \ZB(G)$. But this is impossible when 
$D > 1$ by Proposition \ref{p-center}.

Together with \eqref{cs21}, we have shown that
\begin{equation}\label{cs23}
  2n \geq W \geq D/3 = r^n/3.
\end{equation}  
This is possible only when $D \in \{2,3,4,5,8,9,16\}$. Assume now that $D=16$, so that $r=2$ and $p>2$. Then $W \in \{6,7,8\}$ by \eqref{cs23}. 
We now show that $W=8$. First, as $W \geq 6$, we must have by Proposition \ref{Pinftyimage} 
that $|Q| \geq 7$, whence \eqref{cs22} and \eqref{bd11} imply $W \geq 7$. This in turn implies by Proposition \ref{Pinftyimage}
that $Q| \geq 9$, whence \eqref{cs22} and \eqref{bd11} ensure that $W \geq 8$, i.e. $W=8$ by \eqref{cs23}. Suppose that 
$\dim\Ker_{\F_r}(\bar{h}-1_U) \leq 2n-3$ for all $1 \neq h \in Q$. Then instead of \eqref{cs22} we now have 
$|\varphi(h)| \leq \varphi(1)/4$, and so \eqref{bd11} implies $W > 10$, a contradiction. Thus $Q$ contains an element 
$g \neq 1$ with $|\varphi(g)|=\varphi(1)/2$, hence necessarily $\dim\Ker_{\F_r}(\bar{g}-1_U) = 2n-2$. 
As $\gbar$ is a $2'$-element in $\Sp(U) = \Sp_8(2)$, by \cite[Lemma 2.4]{GT1} this can happen only when
$\gbar \in \GO^-_2(2) < \Sp_8(2)$ is an element of order $3$, whence $p=3$, and $g$ has eigenvalues $\lambda\zeta_3$ and 
$\lambda\zeta_3^2$, both with multiplicity $8$, for some root of unity $\lambda \in \C^\times$. On the other hand, $g$ acts trivially on the tame
part of dimension $D-W=8$, so we may assume $\lambda=\zeta_3^2$, whence $\varphi(g)+\varphi(g^{-1}) = 8$. Let $x \geq 1$ be 
the number of pairs $(h,h^{-1})$ of elements $h \in Q$ with $|\varphi(h)| = \varphi(1)/2$, and let $y \geq 0$ by the number of 
remaining pairs of nontrivial elements in $Q$ (for which we have $|\varphi(h)| \leq \varphi(1)/4$). Then 
$$8(1+2x+2y) =W \cdot |Q|=[\varphi|_Q,1_Q]_Q \cdot |Q| \leq 16+8x + 8y,$$
whence $(x,y)=(1,0)$. Thus $|Q|=3$, and this contradicts Proposition \ref{Pinftyimage} since $W=8$. 
\end{proof}

\section{Elements with simple spectra in finite groups of Lie type}\label{sec:simple-Lie}
In this section, we continue the classification of triples $(G,V,g)$ satisfying the condition $\Cstar$ introduced 
at the beginning of \S\ref{sec:aqs} in the {\it generic} situation,
that is, when $\dim V \geq 23$ and $S={\mathrm {soc}}(G/\ZB(G))$ is a simple group of Lie type in characteristic $p$. 
The non-generic cases,
that is where either $\dim V \leq 22$ or $S$ is an alternating group, have already been dealt with in \S\ref{sec:aqs}.
Furthermore, because of the main application to hypergeometric sheaves, by Theorem \ref{char-sheaf1} and using the assumption
$\dim V \geq 23$, we will assume in some, explicitly described, cases that $g$ is a semisimple element. The 
$\ssp$-elements $g$ will be classified modulo scalars, that is, inside $G/\ZB(G) \leq \Aut(S)$, and we let $\gbar$ denote the coset
$g\ZB(G)$ as an element of $G/\ZB(G)$.

First we start with the linear case:

\begin{thm}\label{ss-sl}
In the situation of $\Cstar$, suppose that $S = \PSL_n(q)$ with $n \geq 3$ and $(n,q) \neq (3,2)$, $(3,3)$, $(3,4)$, 
so that case {\rm (ii)} of Theorem \ref{simple} holds. Then $\obar(g) = (q^n-1)/(q-1)$, $\gbar \in \PGL_n(q)$, and $\gbar$ 
generates the unique, up to $\PGL_n(q)$-conjugacy, maximal torus of order $(q^n-1)/(q-1)$ of $\PGL_n(q)$. 
\end{thm}

\begin{proof}
(i) The cases $(n,q) = (4,2)$, $(4,4)$, $(5,2)$, $(6,2)$, $(7,2)$ can be checked directly using \cite{GAP}, so we will assume 
none of these cases occurs. By Theorem \ref{simple}(ii), 
\begin{equation}\label{ss11}
  (q^n-1)/(q-1) \geq \obar(g) \geq \dim(V) \geq (q^n-q)/(q-1). 
\end{equation}  
Recall \cite[Theorem 2.5.12]{GLS}
that $\Aut(S) = Y \rtimes A$, where $Y := \PGL_n(q)$ and $A = \langle \phi,\tau \rangle \cong C_f \times C_2$, with 
$\phi$ the field automorphism induced by the Frobenius map $x \mapsto x^p$, $q=p^f$, and $\tau$ the transpose-inverse automorphism.
We will follow the analysis in the proof of \cite[Theorem 2.16]{GMPS} to show that $\gbar \in Y$. 
 
First suppose that $\gbar \in \langle Y,\phi \rangle \smallsetminus Y$. Then, as shown on \cite[p. 7679]{GMPS}, there is a divisor
$e>1$ of $f$ such that 
$$\obar(g) \leq e\cdot\meo\bigl(\PGL_n(q^{1/e})\bigr) \leq e\frac{r^n-1}{r-1},$$
where $r := q^{1/e} \geq 2$. Note that, since $n \geq 3$ we have 
\begin{equation}\label{ss12}
  r^{(n-1)e-n} \geq \left\{ \begin{array}{ll}e, & \mbox{ if }e \geq 2,\mbox{ with equality only when }(n,r,e) =(3,2,2),\\ 2e, & \mbox{ if }e \geq 3, \end{array} \right.
\end{equation}   
in particular, $q^{n-1} \geq er^n$. Hence 
$$\frac{q^n-1}{q-1}-e\frac{r^n-1}{r-1} > q^{n-1}+q^{n-2}-er^n \geq 2,$$
and so $\obar(g) < (q^n-1)/(q-1)-2$, violating \eqref{ss11}.

Next suppose that $\gbar = y\psi\tau$, where $y \in Y$ and $\psi \in \langle \phi \rangle$ has order $e|f$. If $2|e$, then, as shown 
on \cite[p. 7680]{GMPS}, $\obar(g) \leq e\cdot\meo\bigl(\PGU_n(q^{1/e}) \bigr)$. By \cite[Lemma 2.15]{GMPS}, 
$\meo\bigl(\PGU_n(q^{1/e}) \bigr) < \meo\bigl(\PGL_n(q^{1/e}) \bigr)$, so as above we have $\obar(g) < (q^n-q)/(q-1)$, contradicting 
\eqref{ss11}. On the other hand, if $2 \nmid e \geq 3$, then
$$\obar(g) \leq 2e\frac{r^n-1}{r-1} < \frac{q^n-1}{q-1}-2,$$
where we use \eqref{ss12} for $r=q^{1/e}$, and this contradicts \eqref{ss11}.

It remains to consider the case $e=\ord(\psi) = 1$. As shown on \cite[p. 7680]{GMPS}, we have one of the following cases:

$\bullet$ $2|n$ and $\obar(g) \leq 2q^{n/2+1}/(q-1) < (q^n-q)/(q-1)$, since $(n,q) \neq (4,2)$;

$\bullet$ $n=3$ and $\obar(g) \leq \max(8,2q+2) < (q^3-q)/(q-1)$, since $(n,q) \neq (3,2)$;

$\bullet$ $n \geq 4$ and $\obar(g) \leq 2p^{\lceil \log_p (2k+1) \rceil}q^{(n-2k+1)/2}$ for some $1 \leq k \leq (n-1)/2$. Since 
$(n,q) \neq (4,2)$, $(5,2)$, we again have $\obar(g) < (q^n-q)/(q-1)$.

\smallskip
(ii) We have shown that $\gbar \in \PGL_n(q)$. The cases $(n,q) = (3,3)$ or $(3,7)$ can be checked directly using \cite{GAP}, so assume
we are not in these cases. Consider an inverse image $h$ of $\gbar$ in $\GL_n(q) = \GL(V)$ with
$V = \F_q^n$, and suppose first that $h$ is not semisimple. Then $p$ divides $\ord(h)$ and $\obar(g)$, and so $\obar(g) = (q^n-q)/(q-1)$
by \eqref{ss11}.
Note that $h$ centralizes its unipotent part $u \neq 1$. If $u$ is regular unipotent, then $\ord(h)$ divides 
$|\CB_{\GL_n(q)}(u)|=q^{n-1}(q-1)$, a contradiction. If, on the opposite, $u$ is a transvection, then  
$\ord(h)$ divides $|\CB_{\GL_n(q)}(u)|=q^{2n-3}(q-1)^2\cdot|\GL_{n-2}(q)|$, a contradiction when $n = 3$ since $(n,q) \neq (3,3)$, $(3,7)$.
In particular, we may assume now that $n \geq 4$. Our assumptions on
$(n,q)$ imply that there exists a primitive prime divisor $\ell=\ell(p,(n-1)f)$ of $p^{(n-1)f}-1=q^{n-1}-1$ by \cite{Zs}. Since 
$\ell$ divides $\obar(g)$ but not $|\GL_{n-2}(q)|$, this rules out the case $u$ is a transvection. Thus $u$ is neither regular nor a 
transvection, whence the $u$-fixed point subspace $U$ on $V$ has dimension $2 \leq m \leq n-2$. Now $h$ fixes $U$, so it belongs to
$\Stab_{\GL(V)}(U)$, which is a $p$-group extended by $\GL_m(q) \times \GL_{n-m}(q)$, and so has order coprime to $\ell$, again a 
contradiction.

We have shown that $h$ is semisimple, and so $\obar(g) = (q^n-1)/(q-1)$ by \eqref{ss11}. Our assumptions on
$(n,q)$ imply that there exists a primitive prime divisor $\ell_1=\ell(p,nf)$ of $p^{nf}-1=q^n-1$ by \cite{Zs}. Let $h_1$ denote the 
$\ell_1$-part of $h$. The structure of centralizers of semisimple elements in $\GL_n(q)$ is well known, in particular, 
the choice of $\ell_1$ implies that $\CB_{\GL_n(q)}(h_1) \cong \GL_1(q^n)$, and this maximal torus is unique in 
$\GL_n(q)$ up to conjugacy. It is now clear that $h \in \GL_1(q^n)$, and, since $\obar(g) = (q^n-1)/(q-1)$, $\gbar$ generates 
$\GL_1(q^n)$ modulo $\ZB(\GL_n(q))$.
\end{proof}

Next we consider the symplectic case:

\begin{thm}\label{ss-sp}
In the situation of $\Cstar$, suppose that $S = \PSp_{2n}(q)$ with $n \geq 2$, $2 \nmid q=p^f$, and $(n,q) \neq (2,3)$, 
so that case {\rm (iv)} of Theorem \ref{simple} holds. Then $\gbar \in \PCSp_{2n}(q)$. Assume furthermore that $\gbar$ is a 
$p'$-element. Then one of the following cases occurs.
\begin{enumerate}[\rm(i)]
\item $V|_{E(G)}$ is an irreducible Weil module, $\gbar \in \PSp_{2n}(q)$, and one of the following statements hold.
\begin{enumerate}
\item[{\rm $(\alpha)$}] $\obar(g) = (q^n \pm 1)/2$, and $\gbar$ generates a unique, up to $\PSp_{2n}(q)$-conjugacy, cyclic maximal torus 
$T_{\pm} < \PSL_2(q^n)$ of order $(q^n \pm 1)/2$ in $\PSp_{2n}(q)$.
\item[{\rm $(\beta)$}] $n=a+b$ with $a,b \in \Z_{\geq 1}$, $2e|a$ for $e:=\gcd(a,b)$, $\obar(g) = (q^a+1)(q^b+1)/2$, and $\gbar$ generates a unique, up to $\PSp_{2n}(q)$-conjugacy, cyclic maximal torus 
$T_{a,b} < (\Sp_{2a}(q) \times \Sp_{2b}(q))/\ZB(\Sp_{2n}(q))$ of order $(q^a+1)(q^b+1)/2$ in $\PSp_{2n}(q)$.
\end{enumerate}
\item $V|_{E(G)}$ is reducible, $\dim(V) = q^n \pm 1$, $\gbar \notin \PSp_{2n}(q)$, and its square $\gbar^2$ fulfills the conclusions of {\rm (i)}.
\end{enumerate}
\end{thm}

\begin{proof}
(A) By Theorem \ref{simple}(iv) and \cite[Theorem 2.16]{GMPS}, 
\begin{equation}\label{ss21}
  q^{n+1}/(q-1) \geq \obar(g) \geq \dim(V) \geq (q^n-1)/2. 
\end{equation}  
Recall \cite[Theorem 2.5.12]{GLS}
that $\Aut(S) = Y \rtimes \langle \phi \rangle$, where $Y := \PCSp_{2n}(q)$ and 
$\phi$ is the field automorphism induced by the Frobenius map $x \mapsto x^p$. Now suppose that 
$\gbar \notin Y$. Then, as shown on \cite[p. 7679]{GMPS}, there is a divisor
$e>1$ of $f$ such that 
$$\obar(g) \leq e\cdot\meo\bigl(\PCSp_{2n}(q^{1/e})\bigr) \leq er^{n+1}/(r-1),$$
where $r := q^{1/e} \geq 3$. By \eqref{ss12} applied to $(n+1,r,e)$, we have that
$$q^n=r^{ne} \geq (e+1)r^{n+1} > er^{n+1}+1,$$ 
and so $\obar(g) \leq er^{n+1}/(r-1) \leq er^{n+1}/2 < (q^n-1)/2$,
violating \eqref{ss21}. Thus we have shown that $\gbar \in \PCSp_{2n}(q)$. 

\smallskip
(B) From now on we will assume that $\gbar$ is a $p'$-element. First 
we consider the case $V|_{E(G)}$ is irreducible, and so it is a Weil module of dimension 
$d=(q^n \pm 1)/2$. Since the outer diagonal automorphism of $E(G)$ fuses the two irreducible Weil modules of
dimension $d$ but $\phi$ stabilizes each of them, $\gbar \in Y \cap \langle S,\phi \rangle =S = \PSp_{2n}(q)$.
View $S = \PSp(W)$ with $W = \F_q^{2n}$, and let $h \in \Sp(W)$ be a (semisimple) inverse image of $\gbar$.

\smallskip
(B1) Here we consider the case where the $\langle h \rangle$-module $W$ cannot be decomposed as an 
orthogonal sum of $h$-invariant nonzero non-degenerate subspaces, and, for further use, we also allow $n=1$ and $(n,q) = (2,3)$ here. 
In this case, by \cite[Satz 2]{Hu}, either 
$\ord(h)|(q^n+1)$ and $W$ is an irreducible $\F_q\langle h \rangle$-module, or 
$\ord(h)|(q^n-1)$ and $W = W_1 \oplus W_2$ with $W_i$ an irreducible $\F_q\langle h \rangle$-module, also being 
a totally isotropic subspace of $W$.  Set $\eps = -$, respectively $\eps = +$, in these two cases. Then, up to $\Sp(W)$-conjugacy,
there is a unique cyclic maximal torus $\hat{T}_\eps = \langle h_\eps \rangle \cong C_{q^n-\eps}$, which can be chosen to be inside a standard subgroup $\SL_2(q^n)$ of $\Sp(W)$ and to contain $h$. Note that $\obar(h_\eps) = (q^n-\eps)/2$; on the other hand,
$\obar(g) \geq (q^n-1)/2$ by \eqref{ss21}. Hence, if $\obar(g) > (q^n-1)/2$, we must have that $\eps = -$, $\obar(g) = (q^n+1)/2$,
and $\langle \gbar \rangle = \hat{T}_-/\ZB(\Sp(W)) =: T_-$. Otherwise we have $\obar(g) = (q^n-1)/2$. If moreover
$(n,q) \neq (1,3)$, then $(q^n+1)/4 < (q^n-1)/2$, and so $\eps = +$, $\obar(g) = (q^n-1)/2$,
and $\langle \gbar \rangle = \hat{T}_+/\ZB(\Sp(W)) =: T_+$. In the remaining case, we have $(n,q)=(1,3)$ and $g \in \ZB(G)$. In 
particular, we have arrived at conclusion ($\alpha$).

\smallskip
(B2) Now we may assume that $W = \oplus^k_{i=1}W_i$ is an orthogonal sum of minimal $h$-invariant nonzero non-degenerate subspaces
$W_i$ for some $k \geq 2$.
Correspondingly, we can write 
$$h = \diag(h_1, h_2, \ldots,h_k) \in H:= \Sp(W_1) \times \Sp(W_2) \times \ldots \times \Sp(W_k),$$
with $h_i \in \Sp(W_i)$, $\dim W_i = 2n_i$ and $\sum^k_{i=1}n_i=n$.

By the analysis in (B1), $\obar(h_i) \leq (q^{n_i}+1)/2$ for all $i$.
Suppose for instance that $\obar(h_1) \leq (q^{n_1}-1)/2$. Recall that, a total Weil module  of $\Sp_{2n}(q)$ has 
character $\omega_n = \xi_n+\eta_n$, with $\xi_n(1) = (q^n+1)/2$ and $\eta_n(1) = (q^n-1)/2$; in particular, the character of
$V$ considered as $\Sp(W)$-module is either $\xi_n$ or $\eta_n$. Using the branching rule \cite[Proposition 2.2(iv)]{TZ2}, we see
that at least one irreducible constituent of $V|_{\Sp(W_1)}$ affords the character $\xi_{n_1}$ of degree $(q^{n_1}+1)/2$. It follows that 
at least one irreducible constituent $V_0$ of $V|_H$ affords the character 
$$\xi_{n_1} \boxtimes \alpha_2 \boxtimes \ldots \boxtimes \alpha_k$$
for some irreducible (Weil) characters $\alpha_i$ of $\Sp(W_i)$, $2 \leq i \leq k$. Since $\xi_{n_1}(1) > \obar(h_1)$, 
$\Spec(h_1,\xi_{n_1})$ is {\it not} simple, whence the same holds for $\Spec(h,V_0)$ and $\Spec(g,V)$, a contradiction. 
 
We have shown that $\obar(h_i) =  (q^{n_i}+1)/2$ for all $i$. The analysis in (B1) shows that
$\ord(h_i^{(q^{n_i}+1)/2}) \leq 2$. In particular, $h^{2M}=1$ and $\obar(g) \leq 2M$, where 
\begin{equation}\label{ss22}
  M := \lcm\bigl( (q^{n_1}+1)/2, (q^{n_2}+1)/2, \ldots,(q^{n_k}+1)/2 \bigr).
\end{equation}
Suppose that $k \geq 3$. If $q \geq 5$, or if $q=3$ but $k \geq 4$, then
$$M \leq \frac{q^n}{2^k} \cdot \prod^k_{i=1}\bigl( 1+\frac{1}{q^{n_i}}\bigr) \leq  \frac{q^n}{2^k} \cdot \bigl( 1 +\frac{1}{q} \bigr)^k  
    < \frac{q^n}{4} \cdot \bigl(1 - \frac{1}{q^n} \bigr) = \frac{q^n-1}{4},$$
and so $\obar(g) \leq 2M < (q^n-1)/2$, contradicting \eqref{ss21}.  If $q=k=3$ but $n \geq 5$, then   
$$M \leq \frac{q^n}{2^3} \cdot \prod^3_{i=1}\bigl( 1+\frac{1}{q^{n_i}}\bigr) \leq  \frac{q^n}{8} \cdot \bigl( 1 +\frac{1}{3} \bigr)^2 \cdot 
   \bigr(1 + \frac{1}{9} \bigr)   
    < \frac{q^n}{4} \cdot \bigl(1 - \frac{1}{q^n} \bigr) = \frac{q^n-1}{4},$$ 
again leading to the same contradiction. If $q=k=n=3$, then $2M = 4 < (q^3-1)/2$ by \eqref{ss22}, and if 
$q=k=3$ but $n=4$, then $2M = 20 < (q^4-1)/2$, again contradicting \eqref{ss21}.

We have shown that $k=2$. Now we have $n=a+b$ with $a:=n_1$ and $b:=n_2$. Let $e:=\gcd(a,b)$, and consider
the case both $a/e$ and $b/e$ are odd. Then by \eqref{ss22} we have 
$$\obar(g) \leq 2M \leq (q^a+1)(q^b+1)/(q^e+1) \leq (q^a+1)(q^b+1)/4 < (q^n-1)/2,$$
unless $(n,q) = (2,3)$ which is ruled out by assumption. Thus, renaming $a$ and $b$ if necessary, we have 
that $2e|a$ (and so $2 \nmid (b/e)$ necessarily). In this case, $(q^a-1)/2$ is divisible by $(q^{2e}-1)/2$, and one can check
that $\gcd((q^a+1)/2,(q^b+1)/2))=1$. Now, $\obar(h_1)=(q^a+1)/2$ and 
$\obar(h_2) = (q^b+1)/2$, hence $\obar(g)=\obar(h)$ is divisible by $M = (q^a+1)(q^b+1)/4 > (q^n+1)/4$. Together with 
\eqref{ss21} and $h^{2M}=1$, this implies that $\obar(g) = 2M$. Also, we have shown in (B1) that $h_i$ generates a cyclic maximal torus 
(of order $(q^{n_i}+1)/2$) in $\PSp(W_i)$ for $i = 1,2$, and this torus is unique in $\PSp(W_i)$ up to conjugacy. Hence, 
$\gbar$ generates a cyclic maximal torus $T_{a,b}$ (of order $2M=(q^a+1)(q^b+1)/2$) in $\PSp(W)$. Note that such a torus $T$ is unique in $\PSp(W)$ up to conjugacy. [Indeed, applying the above analysis to an inverse image $h' \in \Sp(W)$ of a  generator
of $T$ we see that case (B1) does not occur for $h'$, since $\obar(h') =|T| > (q^n+1)/2$. Next, the analysis in (B2) using \eqref{ss22}
shows that the $\langle h' \rangle$-module $W$ decomposes as the orthogonal sum $W'_1 \oplus W'_2$ of two minimal $h'$-invariant non-degenerate 
subspaces of dimension $2c$ and $2d$, with $1 \leq c \leq d$ and $c+d=n$. Now, using 
$(q^c-1)(q^d-1) < (q^c-1)(q^d+1) < q^n$ and $(q^c+1)(q^d-1) \equiv -1 \not\equiv (q^a+1)(q^b+1) (\bmod\ p)$ but 
$\obar(h') = \obar(h)=(q^a+1)(q^b+1)/2$, we must have that $\obar(h') = (q^c+1)(q^d+1)/2$ and $\{c,d\} = \{a,b\}$, and thus 
$T$ is conjugate to $T_{a,b}$.] We have arrived at conclusion ($\beta$).

\smallskip
(C) Now we consider the case $V|_{E(G)}$ is reducible, whence $\dim(V) = q^n \pm 1$ and 
$V|_{E(G)} = A \oplus B$ is the sum of two irreducible Weil modules $A, B$ of dimension
$d=(q^n \pm 1)/2$ by Theorem \ref{simple}(iv). If moreover $A \cong B$, then $G$ cannot induce an outer diagonal 
automorphism of $E(G)$, and so $G/\ZB(G) \leq \langle S,\phi\rangle$. In particular, $G/\ZB(G)E(G)$ is cyclic. On the other hand,
the $E(G)$-module $A$ extends to a simple $\ZB(G)E(G)$-module $\tilde A$ which is $G$-stable. It follows from Gallagher's theorem
\cite[(6.17)]{Is} that $\dim(V)=\dim(A)$, a contradiction. Thus $A \not\cong B$, but $A$ and $B$ are fused by any element 
$t \in Y \smallsetminus S$: $B \cong A^t$. 

Suppose $\gbar \in S$. Then, up to scalar, $g$ acts on $V$ as some $p'$-element $h \in \Sp(V)$ and stabilizes each of $A$ and $B$. 
As $g$ has simple spectrum on $V$, the same holds for the actions of $h$ on $A$ and on $B$. Next, viewing 
$\Sp_{2n}(q) = \sG^\sigma$ for a Frobenius endomorphism $\sigma:\sG \to \sG$ of the simply connected algebraic group
$\sG = \Sp_{2n}(\overline{\F_q})$, we have that $Y = (\sG/\ZB(\sG))^\sigma$. In particular, we can take $t \in \sG$ and 
also have that $tht^{-1} \in \sG^{\sigma}$ (since $S \lhd Y$). By Lemma \ref{lang}(ii), $tht^{-1} = uhu^{-1}$ for some 
$u \in \sG^\sigma$. It follows that
$$\Spec(h,B) = \Spec(h,A^t) = \Spec(tht^{-1},A) = \Spec(uhu^{-1},A) = \Spec(h,A),$$ 
and this contradicts the simple spectrum of $g$ on $V = A \oplus B$.

We have shown that $\gbar \in Y \smallsetminus S$, whence $g$ interchanges $A$ and $B$ and $g^2$ stabilizes each of $A$ and 
$B$. In this case, $\Spec(g,V) = \{ \pm \sqrt{\alpha} \mid \alpha \in \Spec(g^2,A)\}$ is simple if and only if $\Spec(g^2,A)$ is simple.
It follows that $\gbar^2$ fulfills the conclusions of (i). 
\end{proof}

Next, we treat the unitary case:

\begin{thm}\label{ss-su}
In the situation of $\Cstar$, suppose that $S = \PSU_{n}(q)$ with $n \geq 3$, $q=p^f$, and $(n,q) \neq (3,2)$, $(3,3)$, $(3,4)$,
$(4,2)$, $(4,3)$, $(5,2)$, $(6,2)$, 
so that case {\rm (iii)} of Theorem \ref{simple} holds. Then $\gbar \in \PGU_{n}(q)$. Assume furthermore that $\gbar$ is a 
$p'$-element. Then $G/\ZB(G) \rhd \PGU_n(q)$ and one of the following cases occurs.
\begin{enumerate}[\rm(i)]
\item $\obar(g) = (q^n-(-1)^n)/(q+1)$, and $\gbar$ generates a unique, up to $\PGU_{n}(q)$-conjugacy, cyclic maximal torus 
of order $(q^n-(-1)^n)/(q+1)$ in $\PGU_{n}(q)$. Moreover, if $2|n$ then $\dim(V) = (q^n-1)/(q+1)$.
\item $2\nmid n$, $\obar(g) = q^{n-1}-1$, and $\gbar$ generates a unique, up to $\PGU_{n}(q)$-conjugacy, cyclic maximal torus 
$T_{n-1,1}$ of order $q^{n-1}-1$ in $\PGU_{n}(q)$. Moreover, $\dim(V) = (q^n-q)/(q+1)$.
\item $2|n=a+b$ with $2 \nmid a,b \in \Z_{\geq 1}$, $\gcd(a,b) =1$, $\obar(g) = (q^a+1)(q^b+1)/(q+1)$, and $\gbar$ generates a unique, up to 
$\PGU_{n}(q)$-conjugacy, cyclic maximal torus 
$T_{a,b} < (\GU_{a}(q) \times \GU_{b}(q))/\ZB(\GU_{n}(q))$ of order $(q^a+1)(q^b+1)/(q+1)$ in $\PGU_{n}(q)$.
\end{enumerate}
\end{thm}

\begin{proof}
(A) The cases $(n,q) = (4,4)$, $(4,5)$ can be checked directly using \cite{GAP}, so we will assume $(n,q) \neq (4,4)$, $(4,5)$.
By Theorem \ref{simple}(iii) and \cite[Theorem 2.16]{GMPS}, 
\begin{equation}\label{ss31}
  q^{n-1}+q^{\min(2,n-2)} \geq \obar(g) \geq \dim(V) \geq (q^n-q)/(q+1). 
\end{equation}  
Recall \cite[Theorem 2.5.12]{GLS}
that $\Aut(S) = Y \rtimes \langle \phi \rangle$, where $Y := \PGU_{n}(q)$ and 
$\phi$ is an outer automorphism of order $2f$. Now suppose that 
$\gbar \notin Y$, and write $\gbar = x\psi$ with $x \in Y$ and $\psi \in \langle \phi \rangle$ of order $1 < e|2f$.
Suppose first that $2\nmid e$. Then, as shown on \cite[p. 7679]{GMPS},
$$\obar(g) \leq e\cdot\meo\bigl(\PGU_{n}(q^{1/e})\bigr) \leq e\bigl(r^{n-1}+r^{\min(2,n-2)}\bigr) < (8/9)(q^{n-1}-1) \leq (q^n-q)/(q+1),$$
where $r := q^{1/e} \geq 2$, provided $(n,r) \neq (5,2)$; and this contradicts \eqref{ss31}. We also achieve a contradiction
in the case $(n,r) = (5,2)$ using $\meo(\PGU_5(2)) = 24$.
Next we consider the case $2|e \geq 4$. Then
$$\obar(g) \leq e\cdot\meo\bigl(\PGL_{n}(q^{2/e})\bigr) \leq e(r^n-1)/(r-1) <  (q^n-q)/(q+1),$$
with $r:=q^{2/e}$, $(n,r) \neq (3,2)$, and $(n,q) \neq (4,4)$, again contradicting \eqref{ss31}. 
If $(n,r) = (3,2)$, then $e \geq 6$ as $(n,q) \neq (3,4)$,
and we also achieve a contradiction using $\meo(\PGL_3(2)) = 8$. 

It remains to consider the case $e=\ord(\psi) = 2$. As shown on \cite[p. 7680]{GMPS}, we have one of the following cases:

$\bullet$ $2|n$ and $\obar(g) \leq 2q^{n/2+1}/(q-1) < (q^n-q)/(q+1)$, since $(n,q) \neq (4,2)$, $(4,3)$, and $(6,2)$;

$\bullet$ $n=3$ and $\obar(g) \leq \max(8,2q+2) < (q^3-q)/(q+1)$, since $(n,q) \neq (3,2)$, $(3,3)$;

$\bullet$ $n \geq 4$ and $\obar(g) \leq 2p^{\lceil \log_p (2k+1) \rceil}q^{(n-2k+1)/2}$ for some $1 \leq k \leq (n-1)/2$. Since 
$(n,q) \neq (4,2)$, $(4,3)$, $(4,4)$, $(4,5)$, and $(5,2)$, we again have $\obar(g) < (q^n-q)/(q+1)$.

\noindent
Thus we have shown that $\gbar \in \PGU_{n}(q)$. Note that the same conclusion holds in the cases $(n,q) = 
(3,4)$, $(4,2)$, $(4,4)$, and $(6,2)$, since in these cases $p=2$ and $f$ is a $2$-power.

\smallskip
(B) From now on we will assume that $\gbar$ is a $p'$-element.
View $S = \PSU(W)$ with $W = \F_{q^2}^n$, and let $h \in \GU(W)$ be a (semisimple) inverse image of $\gbar$.

\smallskip
(B1) Here we consider the case where the $\langle h \rangle$-module $W$ cannot be decomposed as an 
orthogonal sum of $h$-invariant nonzero non-degenerate subspaces, and, for further use, we assume only that $n \geq 2$.
In this case, by \cite[Satz 2]{Hu}, either 
$2 \nmid n$, $\ord(h)|(q^n+1)$ and $W$ is an irreducible $\F_{q^{2}}\langle h \rangle$-module, or 
$2|n$, $\ord(h)|(q^n-1)$ and $W = W_1 \oplus W_2$ with $W_i$ an irreducible $\F_{q^{2}}\langle h \rangle$-module, also being 
a totally isotropic subspace of $W$.  Furthermore, up to $\GU(W)$-conjugacy,
there is a unique cyclic maximal torus $\hat{T} = \langle t \rangle \cong C_{q^n-(-1)^n}$, which can be chosen to contain $h$. Note that 
$\obar(t) = (q^n-(-1)^n)/(q+1)$; on the other hand,
$\obar(g) > (q^n-(-1)^n)/2(q+1)$ by \eqref{ss31}. Hence, we must have that $\obar(g) = (q^n-(-1)^n)/(q+1)$,
and $\langle \gbar \rangle = \hat{T}/\ZB(\GU(W))$. 
In particular, we have arrived at conclusion (i) (with the value of $\dim V$ following from \eqref{ss31} when $2|n$).

\smallskip
(B2) Now we may assume that $W = \oplus^k_{i=1}W_i$ is an orthogonal sum of minimal $h$-invariant nonzero non-degenerate subspaces
$W_i$ for some $k \geq 2$.
Correspondingly, we can write 
$$h = \diag(h_1, h_2, \ldots,h_k) \in H:= \GU(W_1) \times \GU(W_2) \times \ldots \times \GU(W_k),$$
with $h_i \in \GU(W_i)$, $\dim W_i = n_i$ and $\sum^k_{i=1}n_i=n$.

By the analysis in (B1), $\ord(h_i^{(q^{n_i}-(-1)^{n_i})/(q+1)})|(q+1)$. In particular, $h^{(q+1)M}=1$ and $\obar(g)|(q+1)M$, where 
\begin{equation}\label{ss32}
  M := \lcm\bigl( (q^{n_1}-(-1)^{n_1})/(q+1), (q^{n_2}-(-1)^{n_2})/(q+1), \ldots,(q^{n_k}-(-1)^{n_k})/(q+1) \bigr).
\end{equation}
Also note that for any $c,d \in \Z_{\geq 1}$,
\begin{equation}\label{ss33}
  \frac{q^a-(-1)^a}{q+1} \cdot \frac{q^b-(-1)^b}{q+1} \leq \frac{q^{a+b-1}-(-1)^{a+b-1}}{q+1}.
\end{equation}  
Suppose that $n \geq 4$ and $k \geq 3$. Applying \eqref{ss33} repeatedly, we obtain
$$(q+1)M \leq (q+1)\prod^{k}_{i=1}\frac{q^{n_i}-(-1)^{n_i}}{q+1} \leq q^{n-k+1}-(-1)^{n-k+1} \leq q^{n-2}-(-1)^{n-2} < \frac{q^n-q}{q+1},$$
and so $\obar(g) < (q^n-q)/(q+1)$, contradicting \eqref{ss31}. If $n=k=3$, then $q > 2$ and $n_i=1$ for all $i$, and so 
$(q+1)M = q+1 < (q^3-q)/(q+1)$ by \eqref{ss32}, again contradicting \eqref{ss21}.

We have shown that $k=2$. Now we have $n=a+b$ with $a:=n_1$ and $b:=n_2$. If $a,b \geq 2$, then we note
that $(q^a-(-1)^a)(q^b-(-1)^b) < 2(q^{a+b}-q)$. Hence, if $\gcd(q^a-(-1)^a,q^b-(-1)^b) \geq 2(q+1)$, then 
$$\obar(g) \leq (q+1)M \leq (q+1) \frac{(q^a-(-1)^a)(q^b-(-1)^b)}{2(q+1)^2}  < \frac{q^n-q}{q+1},$$
contradicting \eqref{ss31}. Since $\gcd(q^a-(-1)^a,q^b-(-1)^b) = q^e-(-1)^e$ for $e:=\gcd(a,b)$, we must therefore have 
that $e=1$, or $e=q=2$. In the latter case we also have 
$$\obar(g) \leq (q+1)M = (q+1) \frac{(q^a-(-1)^a)(q^b-(-1)^b)}{(q+1)^2}  =  \frac{(q^a-1)(q^b-1)}{q+1}< \frac{q^n-q}{q+1},$$
again a contradiction. Thus $a$ and $b$ are coprime. 

Consider the case $a,b \geq 2$, but $2|b$. As $n \geq 3$, all irreducible Weil characters of $\SU_n(q)$ extend to 
$\GU_n(q)$, see e.g. \cite[\S4]{TZ2}, so we may extend $V$ to $\GU(W)$. 
Using the branching rule \cite[(2.0.3)]{KT4}, we see that at least one irreducible constituent of 
$V|_{\SU(W_2)}$ affords the Weil character $\zeta^0_{b,q}$ of degree $(q^b+q)/(q+1)$, and the same holds for 
the restriction to $\GU(W_2)$. Thus at least one irreducible constituent $V_0$ of $V$ to $H=\GU(W_1) \times \GU(W_2)$ 
affords the character $\beta_1 \boxtimes \beta_2$
for some irreducible (Weil) characters $\beta_i$ of $\GU(W_i)$, $i= 1,2$, and 
$\beta_2(1)=(q^b+q)/(q+1) > \obar(h_2)$. It follows that 
$\Spec(h_2,\beta_{2})$ is {\it not} simple, whence the same holds for $\Spec(h,V_0)$ and $\Spec(g,V)$, a contradiction. 

Thus either $2 \nmid n$ and $(a,b) = (n-1,1)$, or $2|n$ and $\gcd(a,b) = 1$. In the former case, $h$ is contained 
in a maximal torus $C_{q^{n-1}-1} \times C_{q+1} < \GU_{n-1}(q) \times \GU_1(q)$ which projects onto a 
cyclic maximal torus $T_{n-1,1} \cong C_{q^{n-1}-1}$ of $\PGU_n(q)$. Since 
$$\obar(h)=\obar(g) \geq (q^n-q)/(q+1) > (q^{n-1}-1)/2,$$ 
we must have that $\langle \gbar \rangle = T_{n-1,1}$. In fact, multiplying $h$ by a suitable central element of 
$\GU(W)$, we may assume that $h_2=1_{W_2}$. Now, if $\dim V = (q^n+1)/(q+1)$, then again using the branching rule, 
we see that at least one irreducible constituent $V_1$ of $V|_{\GU(W_1)}$ has degree $(q^{n-1}+q)/(q+1)$. On the other hand,
$h_1$ has order $(q^{n-1}-1)/(q+1)$ {\it modulo $\ZB(\GU(W_1))$}. It follows that $\Spec(h_1,V_1)$ is not simple, whence so is
$\Spec(g,V)$ by the above argument. Hence $\dim(V) = (q^n-q)/(q+1)$, and we arrive at conclusion (ii).

In the latter case, $h$ is contained 
in a maximal torus $C_{q^a+1} \times C_{q^b+1} < \GU_a(q) \times \GU_b(q)$ which again projects onto a 
maximal torus $T_{a,b}$ of $\PGU_n(q)$; moreover,
$T_{a,b}$ is cyclic of order $(q^a+1)(q^b+1)/(q+1)$, since $\gcd(q^a+1,q^b+1)=q+1$.
Since 
$$\obar(h)=\obar(g) \geq (q^n-q)/(q+1) > (q^a+1)(q^b+1)/2(q+1),$$ 
we must have that $\langle \gbar \rangle = T_{a,b}$, and so we arrive at conclusion (iii).

In both cases of (ii) and (iii), the uniqueness of cyclic maximal tori $T_{a,b}$ of order $(q^a-(-1)^a)/(q^b-(-1)^b)/(q+1)$
follows from the well-known order formula and classification of maximal tori in $\GU_n(q)$ (or from repeating the analysis 
in (B1) and (B2) for an inverse image $h' \in \GU(W)$ of a generator of such a torus). Finally, since $\gbar$ generates a maximal torus
of $\PGU_n(q)$, we have $G/\ZB(G) \rhd \PGU_n(q)$.
\end{proof}

\begin{cor}\label{simple2}
In the situation of $\Cstar$, assume we are in one of the cases considered in Theorem \ref{ss-sl}, respectively Theorem \ref{ss-sp}(i), 
Theorem \ref{ss-su}. Suppose $\Cstar$ gives rise to a hypergeometric sheaf $\sH$ of type $(D,m)$ with $D-m \geq 2$, with
$G=G_\geo$, $g$ a generator of the image of $I(0)$ in $G$, and $V$ realizes the action of $G$ on $\sH$. Then 
$G/\ZB(G) \cong \PGL_n(q)$, respectively $\PSp_{2n}(q)$, $\PGU_n(q)$.
\end{cor}

\begin{proof}
Since $G$ is almost quasisimple, $G^{(\infty)}$ is a quasisimple cover of $S=\PSL_n(q)$, respectively $\PSp_{2n}(q)$, $\PSU_n(q)$, 
and $S \lhd G/\ZB(G) \leq \Aut(S)$. By Theorem \ref{ss-sl}, respectively Theorem \ref{ss-sp}(i), 
Theorem \ref{ss-su}, $H/\ZB(G)$, with $H:=\langle G^{(\infty)},\ZB(G),g \rangle$, is the normal subgroup $\PGL_n(q)$, respectively
$\PSp_{2n}(q)$, $\PGU_n(q)$, of $\Aut(S)$; in particular, $H \lhd G$. As $H$ contains the normal closure of the image $\langle g \rangle$
of $I(0)$ in $G$, it follows by Theorem \ref{generation} that $G=H$, whence the statement follows. 
\end{proof}

Finally, we treat the extraspecial normalizers:

\begin{thm}\label{ss-extr}
Let $p$ be a prime. Let $G$ be a finite irreducible subgroup of $\GL(V) \cong \GL_{p^n}(\C)$ that satisfies $\CSP$ and is an 
extraspecial normalizer, so that $G \rhd R = \ZB(R)E$ for some 
some extraspecial $p$-group $E$ of order $p^{1+2n}$ that acts irreducibly on $V$, and furthermore either $R=E$ or $\ZB(R) \cong C_4$, 
as in \cite[Proposition 2.8(iii)]{G-T}. Suppose that
a $p'$-element $g \in G$ has simple spectrum on $V$ and that $p^n \geq 11$. Then the following statements hold.
\begin{enumerate}[\rm(i)]
\item Suppose $p > 2$. Then $\exp(R)=p$, $\obar(g) = p^n+1$, and the coset $g\ZB(G)R$ as an element of $G/\ZB(G)R \hookrightarrow \Sp_{2n}(p)$ generates a cyclic maximal torus $C_{p^n+1}$ of $\Sp_{2n}(p)$. 
\item Suppose $p=2$. Then one can find integers $a_1 > a_2 > \ldots > a_t \geq 1$ such that $n=\sum^t_{i=1}a_i$, 
$\gcd(2^{a_i}+1,2^{a_j}+1)=1$ if $i \neq j$, $\obar(g) = \prod^t_{i=1}(2^{a_i}+1)$, and the coset $g\ZB(G)R$ as an element of 
$G/\ZB(G)R \hookrightarrow \Sp_{2n}(2)$ generates a cyclic maximal torus 
$C_{2^{a_1}+1} \times  \ldots \times C_{2^{a_t}+1}$
of $\Sp_{2n}(2)$. 
\end{enumerate}
\end{thm}

\begin{proof}
By Schur's lemma, the irreducibility of $R$ on $V$ implies that $\CB_G(R) = \ZB(G) < \ZB(\GL(V))$, and so $G/\ZB(G)$ embeds in 
the group $\Aut_1(R)$ of all automorphisms of $R$ that act trivially on $\ZB(R)$, and 
$G/\ZB(G)R \hookrightarrow \Out_1(R)=\Aut_1(R)/(R/\ZB(R)) \hookrightarrow \Sp_{2n}(p)$,  see e.g. \cite[Theorem 1]{Wi}. As 
$p \nmid \ord(g)$, $\obar(g)$ is equal to the order of the coset $g\ZB(G)R$ in $\Out_1(R) \leq \Sp_{2n}(p)$. On the other hand,
\begin{equation}\label{extr10}
  \obar(g) > p^n
\end{equation}
as $g$ is an $\ssp$-element on $V$.

\smallskip
(i) First we consider the case $p>2$.  Suppose that $\exp(R) > p$. Then $\Out_1(R)$ is isomorphic to a semidirect product of a $p$-group of order $p^{2n-1}$ by $\Sp_{2n-2}(p)$ by \cite[Theorem 1]{Wi}. As $g$ is a $p'$-element, 
$\obar(g)$ is at most the maximum order of elements in $\Sp_{2n-2}(p)$, which is at most twice of $\meo(\PSp_{2n-2}(p)) < p^n/(p-1)$ by
\cite[Lemma 2.10]{GMPS} (where the strict inequality holds because $\meo(\cdot)$ is an integer). It follows that 
$\obar(g) <2p^n/(p-1)\leq p^n = \dim(V)$, contradicting \eqref{extr10}.

We have shown that $\exp(R)=p$, i.e. $R \cong p^{1+2n}_{+}$. In this case, it is known that $\Aut_1(R)$ is a split extension of 
${\mathrm {Inn}}(R) \cong R/\ZB(R)$ by $\Sp_{2n}(p)$. Now, $\Sp_{2n}(p)$ as a subgroup of $\Aut_1(R)$ preserves the equivalence class
of the representation of $R$ on $V$, hence it admits a projective representation on $V$, which must be linearized since 
$\Sp_{2n}(p)$ has trivial Schur multiplier when $p^n > 9$, and by a faithful representation because $\Sp_{2n}(p)$ acts faithfully on $R$. 
Thus we have shown that $$G \leq \NB_{\GL(V)}(R) = \ZB(\GL(V))R \rtimes \Sp_{2n}(p);$$
in particular, by conjugating the $p'$-element $g$ (applying the Schur-Zassenhaus theorem to $\ZB(G)R\langle g\rangle$), we can 
write $g=zh$ for some $z \in \ZB(G)$ and some $p'$-element $h \in \Sp_{2n}(p)$ with $\obar(h)=\obar(g) \geq p^n$. If $n=1$, it follows
that $\obar(h)=\ord(h)=p+1$, and we are done in this case.

Assume now that $n \geq 2$, and apply Theorem \ref{ss-sp}(i) to $h \in \Sp_{2n}(p)$ acting on $V$. In case $(\alpha)$,
we have that $h$ generates a cyclic maximal torus $C_{p^n-\eps}$ of $\Sp_{2n}(p)$ for some $\eps = \pm$. As $\obar(h)\geq p^n$, we 
must have that $\eps=-$ and $\ord(h)=p^n+1=\obar(g)$, as stated. In case $(\beta)$, $h$ belongs to a maximal torus
$C_{p^a+1} \times C_{p^b+1} < \Sp_{2a}(p) \times \Sp_{2b}(p)$ with $n=a+b$ and $a,b \in \Z_{\geq 1}$. In this case, 
$h^{(p^a+1)(p^b+1)/4} \in \ZB(\Sp_{2a}(p) \times \Sp_{2b}(p))$ and so $h^{(p^a+1)(p^b+1)/2}=1$. Thus 
$\obar(h) \leq \ord(h) \leq (p^a+1)(p^b+1)/2 < p^{a+b}=p^n$, a contradiction.

\smallskip
(ii) Let $\gbar \in \Sp_{2n}(2)$ denote the image of $g$ in $G/\ZB(G)R$. By \cite[Lemma 5.8]{GT1}, there is some $\eps=\pm$ such
that $\gbar$ preserves a quadratic form of type $\eps \in \{+,-\}$ on the natural module $\F_2^{2n}$ for $\Sp_{2n}(2)$: 
$\gbar \in \GO^\eps_{2n}(2)$. This implies that we may take $E$ to be of type $\eps$ and $g$-invariant. Let $\Sp^+$ denote 
$\GO$ and let $\Sp^-$ denote $\Sp$. As shown in \cite[Theorem 4.2]{KT8}, the action of $E$ on $V$ then preserves a 
non-degenerate bilinear form of type $\eps$, and 
$$\NB_{\Sp^\eps(V)}(E) = E \cdot \GO^\eps_{2n}(2),~~\NB_{\GL(V)} = \ZB(\GL(V))\NB_{\Sp^\eps(V)}(E).$$
In particular, we can write $g=zh$ with $z \in \ZB(\GL(V))$ and $h \in \NB_{\Sp^\eps(V)}(E)$ having odd order
(as $2 \nmid \obar(g)$). In turn, we can embed the image of $h$ in $\GO^\eps_{2n}(2)$ in a maximal torus 
$$T = C_{2^{a_1}-\eps_1} \times C_{2^{a_2}-\eps_2} \times \ldots \times C_{2^{a_t}-\eps_t} < 
    \GO^{\eps_1}_{2a_1}(2) \times \GO^{\eps_2}_{2a_2}(2) \times \ldots \times \GO^{\eps_t}_{2a_t}(2) \leq \GO^\eps_{2n}(2),$$
where $a_i \in \Z_{\geq 1}$, $n = \sum^t_{i=1}a_i$, $\eps_i = \pm$, and $\eps = \prod^t_{i=1}\eps_i$. Correspondingly, we can decompose
$$V = V_1 \otimes V_2 \otimes \ldots \otimes V_t,~~E= E_1 \circ E_2 \circ \ldots \circ E_t,$$
where $V_i = \C^{2^{a_i}}$, $E_i \cong 2^{1+2a_i}_{\eps_i}$, and $\NB_{\Sp^{\eps_i}(V_i)}(E_i) \cong E_i \cdot \GO^{\eps_i}_{2a_i}(2)$, and 
then put $h$ in 
$$\NB_{\Sp^{\eps_1}(V_1)}(E_1) \otimes  \NB_{\Sp^{\eps_2}(V_2)}(E_2) \otimes \ldots \otimes \NB_{\Sp^{\eps_t}(V_t)}(E_t).$$  
By \cite[Lemma 4.3]{KT8}, a generator of the maximal torus $C_{2^{a_i}-\eps_i}$ of 
$\NB_{\Sp^{\eps_i}(V_i)}(E_i)/E_i \cong \GO^{\eps_i}_{2a_i}(2)$ can be lifted to
an element $s_i$ of $\Sp^{\eps_i}(V_i)$ that has order $2^{a_i}-\eps_i$ and spectrum $\mu_{2^{a_i}+1} \smallsetminus \{1\}$ when $\eps_i=-$, and $\mu_{2^{a_i}-1}$ with $1$ occurring twice when $\eps_i=+$. Thus the $2'$-element $h$ is contained in 
$E \cdot T= E \langle s_1, s_2, \ldots ,s_t \rangle$, with $s_1, \ldots,s_t$ centralizing each other.  Applying the Schur-Zassenhaus theorem to 
$E \cdot T$ with normal Hall subgroup $E$, we may assume that $h \in \langle s_1, s_2, \ldots,s_t \rangle$.

Recall that $g$ is an $\ssp$-element on $V$, whence so is $h$. Now, if $\eps_i=+$ for some $i$, then $s_i$ has eigenvalue $1$ with multiplicity
$2$ on $V_i$, precluding $h$ from being an $\ssp$-element. Thus $\eps_i=-$ for all $i$. Now, the order of $h$ is at most
$L:=\lcm(2^{a_1}+1,2^{a_2}+1, \ldots ,2^{a_t}+1)$. 
Denoting by $b_1 > b_2 > \ldots > b_{t'}$ the 
distinct values among $a_1, a_2, \ldots,a_t$ we have 
\begin{equation}\label{extr11}
 L = \lcm(2^{b_1}+1, 2^{b_2}+1, \ldots, 2^{b_{t'}}+1) < 2^{b_1+b_2+ \ldots + b_{t'}+1}
\end{equation}  
by \cite[Lemma 2.9]{GMPS} (with strict inequality because $2 \nmid L$). Now, if some $a_j$ is repeated, 
then $\sum^{t'}_{i=1}b_i  \leq n-a_j \leq n-1$, and so \eqref{extr11} implies
$\ord(h) \leq L < 2^n$,  whence $\obar(g) < 2^n$, contradicting \eqref{extr10}. Thus $a_1 > a_2 > \ldots >a_t \geq 1$.  
If $\gcd(2^{a_i}+1,2^{a_j}+1) > 1$ for some $i \neq j$, then by \cite[Lemma 4.1(iii)]{LMT} we have 
$\ord(h) \leq L \leq \prod^t_{i=1}(2^{a_i}+1)/3 < (2.4)2^n/3 < 2^n$,
again a contradiction. We also achieve the same contradiction, if $h$ does not generate $\langle s_1, s_2, \ldots,s_t \rangle$, which is 
now a cyclic group of order $\prod^t_{i=1}(2^{a_i}+1)$. Hence $\obar(g) = \prod^t_{i=1}(2^{a_i}+1)$, as stated.
\end{proof}

As one can see, the results in this section leave out the case (i) of Theorem \ref{simple}, where the almost quasisimple group
$G$ has $S = \PSL_2(q)$ as its unique non-abelian composition factor. In this case, many complex representations of $G$,
particularly the ones irreducible and nontrivial on $L = G^{(\infty)}$, have dimension $\leq q+1$ always admit $\ssp$-elements.
On the other hand, if $q$ is not small, say $q \geq 27$, then any hypergeometric sheaf $\sH$ admitting $G$ as its (finite) geometric 
monodromy group, must be in characteristic $p$ dividing $q$. As a direct application of Theorem \ref{Brauerp} and results of
\cite{KT5}, we show that all nontrivial irreducible  representations of $\GL_2(q)$ do lead to hypergeometric sheaves.

\begin{thm}\label{gl2}
Let $q=p^f \geq 4$ be a power of a prime $p$. Then the following statements hold.

\begin{enumerate}[\rm(i)]
\item Let $\Phi$ be any irreducible $\overline{\Q_\ell}$-representation of 
$G =\GL_2(q)$ of degree $>1$. Then, there exists a hypergeometric 
sheaf $\sH$ over $\overline{\F_p}$ that has $G/\Ker(\Phi)$ as its geometric monodromy group.

\item Let $\Theta$ be any irreducible $\overline{\Q_\ell}$-representation of 
$H =\GU_2(q)$ of degree $>1$ that is trivial at $\OB_{2'}(\ZB(H))$. Then, there exists a hypergeometric 
sheaf $\sH$ over $\overline{\F_p}$ that has $H/\Ker(\Theta)$ as its geometric monodromy group.
\end{enumerate}
\end{thm}

\begin{proof}
(i) We use the character table of $G$ as given in \cite[Table 1, p. 155]{DM}. In particular, if $T \cong \mu_{q-1} \times \mu_{q-1}$ denotes
a diagonal maximal torus of $G$, then the irreducible representations of $G$ of degree $q+1$ are $R^G_T(\alpha,\beta)$ which
are Harish-Chandra induced from $\alpha \boxtimes\beta: T \to \overline{\Q_\ell}^\times$, 
where $\alpha,\beta$ are distinct characters of $\mu_{q-1}$. The nontrivial irreducible components of the total Weil representation of $G$ considered in \cite{KT5} are precisely $R^G_T(\alpha,\triv)-\delta_{\alpha,\triv}1_G$ with $\alpha \in \Irr(\mu_{q-1})$.
Now we pick $\alpha \in \Irr(\mu_{q-1})$ to be faithful. By \cite[Corollary 8.2]{KT5}, there exists a hypergeometric 
sheaf $\sH_\alpha$ over $\overline{\F_p}$ that has $G$ as its geometric monodromy group, acting on $\sH_{\alpha}$ via a 
representation $\Psi_\alpha$ with character $R^G_T(\alpha,\triv)$. Inspecting the character table of $G$, we see that
$$\Trace(\Phi(g))-\Trace(\Psi_\alpha(g)) = \Trace(\Phi(1))-\Trace(\Psi_\alpha(1))$$ 
for all $p$-elements $g \in G$. Hence the statement follows from Theorem \ref{Brauerp}.

\smallskip
(ii) As in (i), we appeal to \cite[Corollary 8.2]{KT5} to get a hypergeometric sheaf $\sH$ of rank $q$ with geometric monodromy group
$\PGL_2(q) \cong \PGU_2(q)$, which utilizes a surjection $\phi: \pi_1(\G_m/\overline{\F_p}) \surj \PGU_2(q)$ together with
a representation $\Phi:\PGU_2(q) \to \GL(\sH)$. 
View $\PGU_2(q)$ as $H/Z$, where $Z := \ZB(H) \cong C_{q+1}$, and decompose $Z = Z_1 \times Z_2$, where
$Z_1 := \OB_{2'}(Z)$ and $Z_2 := \OB_2(Z)$.  Then observe that $H = Z_1H^\circ$ and $H^\circ \cap Z = Z_2$, where
$$H^\circ := \left\{ X \in \GU_2(q) \mid \det(X) \in \OB_2(\mu_{q+1}) \right\}.$$
It follows that $\PGU_2(q) \cong ZH^\circ/Z \cong H^\circ/Z_2$. Now, 
the obstruction to lifting $\phi$ to a homomorphism
$\varphi: \pi_1(\G_m/\overline{\F_p})\to H^\circ$
lies in the group $H^2(\G_m/\overline{\F_p},Z_2)=0$, the vanishing because open curves have cohomological dimension $\le 1$, cf. \cite[Cor. 2.7, Exp. IX and Thm. 5.1, Exp. X]{SGA4t3}. We claim that $\varphi$ is surjective. [Indeed, for the image $J$ of $\varphi$ 
we have $Z_2J = H^\circ$, and so 
$$J \geq [J,J] = [Z_2J,Z_2J] = [H^\circ,H^\circ] = \SU_2(q).$$ 
Hence, if $\det$ maps $J$ onto
$\OB_2(\mu_{q+1}) =: C_{2^a}$, then $J = H^\circ$ as claimed. 
Otherwise we have $\det(J) \cong C_{2^b}$ with $0 \leq b \leq a-1$; in particular, $q$ is odd.
In this case, one can check that 
$$J \cap Z_2 = \left\{ \diag(x,x) \mid x \in \mu_{q+1},~x^{2^{b+1}} = 1 \right\} \cong C_{2^{b+1}},$$
and so 
$$|\PGU_2(q)| = |H^\circ/Z_2| = |JZ_2/Z_2| =  |J|/|J \cap Z_2|  = 2^b|\SU_2(q)|/2^{b+1} = |\PGU_2(q)|/2,$$
a contradiction.]

Now, consider any irreducible representation $\Theta$ of $H$ that is trivial on $Z_1$. Then we can view $\Theta$ as a representation
of $H/Z_1 = Z_1H^\circ/Z_1 \cong H^\circ$, and inflate $\Phi$ to a representation of $H^\circ$. 
Checking the well-known character table of $\GU_2(q)$, we see that 
$$\Trace(\Phi(g))-\Trace(\Theta(g))=\Trace(\Phi(1))-\Trace(\Theta(1))$$
for all $p$-elements $g \in \GU_2(q)$. Hence the statement
follows from Theorem \ref{Brauerp}.  
\end{proof}

\section{(Non-)existence theorems}\label{sec:hypergeom2}
In this section we will prove various theorems that rule out the existence of (irreducible) hypergeometric sheaves of type 
$(D,m)$ with $D > m$ and  certain kind of finite monodromy groups $G=G_\geo$. 
For a hypergeometric sheaf $\sH$ in question, we will denote by $Q$ the image
of $P(\infty)$ on $\sH$; note that $Q \neq 1$ as $\sH$ is  not tame at $\infty$. We also use the fact that, since $\sH$ is tame at $0$, $P(0)$ acts trivially on $\sH$ and a generator $g_0$ of the image of the $p'$-group $I(0)/P(0)$ has simple spectrum
on $\sH$, if $p = {\rm {char}}(\sH)$. Furthermore, if $D > 1$, then $p$ divides $|G/\ZB(G)|$ by Proposition \ref{p-center}(iii). 

\subsection*{9A. Alternating groups}
First, we rule out the cases (c)--(f) of Theorem \ref{alt}(ii) for hypergeometric sheaves.

\begin{lem}\label{spin32}
There does not exist any hypergeometric sheaf $\sH$ of type $(D,m)$ with $D>m$ and $D=16$ or $32$ that has finite geometric monodromy group $G$ such that $G/\ZB(G) \cong \Sym_{9}$, $\Sym_{10}$, $\ABS_{11}$, or $\Sym_{12}$ as listed in Theorem \ref{alt}(ii).
\end{lem}

\begin{proof}
Assume the contrary, and let $p$ denote the characteristic of such a sheaf $\sH$, and $\varphi$ denote the character of $G$ acting on 
$\sH$. As mentioned above, a generator $g_0$ of the image of $I(0)/P(0)$ has simple spectrum
on $\sH$.

\smallskip
(i) Consider the case of Theorem \ref{alt}(ii)(f), i.e. $D=32$ and $\obar(g_0) = 60$. As $g_0$ is a $p'$-element, $p > 5$. Now, by Proposition 
\ref{p-center}, $Q \cap \ZB(G)=1$, whence $Q$ embeds in $G/\ZB(G) \cong \Sym_{12}$.
It follows that $p = 7$ or $11$, and $Q \cong C_p$. By checking the character table of $2\Sym_{12}$ as given in \cite{GAP}, we
see that the spectrum of a generator $g$ of $Q$ on $\sH$ consists of all $p^{\mathrm {th}}$ roots of unity, each with multiplicity at least 
$4$ if $p=7$ and at least $2$ if $p=11$. On the other hand, the action of $g$ on the wild part $\Wild$ yields a (nontrivial) eigenvalue of
$g$ with multiplicity $1$, a contradiction.

\smallskip
(ii) Now we consider the cases (c)--(e) of Theorem \ref{alt}(ii), i.e. $D=16$ and $\obar(g_0) = 20$ or $30$, whence $p \neq 2,5$. Again by Proposition \ref{p-center}, $Q \cap \ZB(G)=1$, whence $Q$ embeds in $G/\ZB(G) \cong \Sym_{9}$, $\Sym_{10}$, or $\ABS_{11}$.
In all cases, $G/\ZB(G)$ embeds in $\GL_{10}(2)$, hence 
\begin{equation}\label{spin11}
  W=\dim\Wild \leq 10
\end{equation}
by Theorem \ref{bound2}. Inspecting the character table of $2\Sym_{9}$, $2\Sym_{10}$, and $2\ABS_{11}$ as given in \cite{GAP}, we see that
$|\varphi(g)| \leq 8$, whence 
\begin{equation}\label{spin12}
  W \geq 6 
\end{equation}
by \eqref{bd11}. It follows from Proposition \ref{p-center}(iv), $p \nmid |\ZB(G)|$. In turn,
this implies that $Q \leq G^{(\infty)} = 2\ABS_9$, $2\ABS_{10}$, or $2\ABS_{11}$. Now, if $Q$ contains an element $g$ of order $3$ that 
projects onto a $3$-cycle, then $\varphi(g)=-8$ and $g$ has no eigenvalue $1$ on $\sH$, whence $W = 16$, contradicting \eqref{spin11}. 
In all other cases, we have $|\varphi(x)| \leq 4$ for $1 \neq x \in Q$. If moreover $|Q| \geq 7$, then using \eqref{bd11} we obtain
$W \geq 16\cdot(3/4)\cdot(6/7) > 10$, again contradicting \eqref{spin11}. As $p \neq 2,5$ and $Q \neq 1$, we conclude that $p=3$ and 
$|Q|=3$. But then $Q$ has at most $2$ nontrivial irreducible characters, all of degree $1$, and this contradicts \eqref{spin12}.
\end{proof}

We now give a result due to Sawin.
\begin{lem}\label{Sawin}{\rm (Sawin)} Given positive integers $A,B$ with $\gcd(A,B)=1$, and $C:=A+B$ consider the polynomial 
$$f(x):=x^A(1-x)^B,$$
viewed as a map from $\P^1 \setminus \{0,1,\infty\}$ to $\G_m$.
Then we have the following results.
\begin{itemize}
\item[(i)] Let $p$ be a prime with $p|C$. Write $C=C_0p^e$ with $C_0$ prime to $p$. Then in characteristic $p$, we have, for any $\ell \neq p$, and any nontrivial additive character $\psi$ of $\F_p$, the sheaf $$f_\star \overline{\Q_\ell}/ \overline{\Q_\ell} $$
is geometrically isomorphic to a multiplicative translate of the hypergeometric sheaf
$$ \sH yp_\psi(\Ch(A)\sqcup \Ch(B) \setminus \{\triv\};\Ch(C_0) \setminus \{\triv\}).$$
\item[(ii)]Let $p$ be a prime with $p|A$. Write $A=A_0p^e$ with $A_0$ prime to $p$. Then in characteristic $p$, we have, for any $\ell \neq p$, and any nontrivial additive character $\psi$ of $\F_p$, the sheaf $$(1/f)_\star \overline{\Q_\ell}/ \overline{\Q_\ell} $$
is geometrically isomorphic to a multiplicative translate of the hypergeometric sheaf
$$\sH yp_\psi(\Ch(C) \setminus \{\triv\};\Ch(A_0)\sqcup \Ch(B) \setminus \{\triv\}.$$
\end{itemize}
\end{lem}
\begin{proof}In either of the situations (i) or (ii), we work in the specified characteristic $p$. Both $f$ and $1/f$ are finite etale maps from $\P^1 \setminus \{0,1,\infty\}$ to $\G_m$, cf. \cite[proof of 1.2]{Ka-RL-T-2Co1}. The constant sheaf $\overline{\Q_\ell}$ has Euler characteristic $-1$ on $\P^1 \setminus \{0,1,\infty\}$, hence $f_\star \overline{\Q_\ell}$ and $(1/f)_\star \overline{\Q_\ell}$ are lisse sheaves on $\G_m$ with Euler characteristic $-1$. Each is pure of weight zero, so is geometrically semisimple. By 
\cite[8.5.2 and 8.5.3]{Ka-ESDE}, each of these direct images is the direct sum of a single irreducible hypergeometric sheaf $\sH$ with
some Kummer sheaves $\sL_\chi$. We detect the $\sH$ by listing the characters which occur in  $f_\star \overline{\Q_\ell}$ and $(1/f)_\star \overline{\Q_\ell}$ respectively at $0$ and at $\infty$, and cancelling those which appear at both $0$ and $\infty$, cf. 
\cite[9.3.1]{Ka-ESDE}. Because $\gcd(A,B)=1$, the only character to cancel is $\triv$, hence the assertion that the $\sH$, namely
$f_\star \overline{\Q_\ell}/ \overline{\Q_\ell}$ or  $(1/f)_\star \overline{\Q_\ell}/ \overline{\Q_\ell}$, has the asserted ``upstairs" and
``downstairs" characters. By \cite[8.5.5]{Ka-ESDE}, this local monodromy data at $0$ and $\infty$ determines $\sH$ up to multiplicative translation.
\end{proof}

\begin{thm}\label{alt2}
Let $n \geq 5$. Then the cases listed in Theorem \ref{alt}(i)(a) give rise to hypergeometric sheaves. More precisely,
\begin{enumerate}[\rm(i)]
\item For any prime $p \leq n-3$ with $p \nmid n$, there exists a hypergeometric sheaf $\sH$ over $\G_m/\overline{\F_p}$, with 
$G_\geo=\ABS_n$ if $2 \nmid n$ and $G_\geo = \Sym_n$ if $2|n$, and with the image of $I(0)$ generated by an $n$-cycle.
\item Suppose that $1 \leq k \leq n/2$ is coprime to $n$, and $p$ is any prime dividing $n$. If $k=1$, suppose in addition that
$n$ is not a $p$-power, and that $p=3$ if $n=6$ or $24$, and $p=2$ if $n=12$. Then  
there exists a hypergeometric sheaf $\sH$ over $\G_m/\overline{\F_p}$, with 
$G_\geo=\ABS_n$ if $2|n$ and $G_\geo = \Sym_n$ if $2 \nmid n$, and with the image of $I(0)$ generated by the disjoint product of an 
$(n-k)$-cycle and a $k$-cycle.
\end{enumerate}
\end{thm}

\begin{proof}
(i) By Sawin's Lemma \ref{Sawin}(ii),  by considering  
$f_{\star}\overline{\Q_{\ell}}/\overline{\Q_{\ell}}$ with $f(x)=x^{-p}(x-1)^{p-n}$ in characteristic $p$, we get 
$$\sH=\sH yp(\Ch_n \smallsetminus \{\triv\};\Ch_{n-p}),$$ 
with $G_\geo \leq \Sym_n$ acting (irreducibly) via the restriction of the deleted natural permutation module of $\Sym_n$. This irreducibility
implies that $G_\geo$ is a doubly transitive subgroup of $\Sym_n$, in particular a primitive subgroup. The wild part $\Wild$ has 
dimension $p-1$, whence by Proposition \ref{Pinftyimage} the image $Q$ of $P(\infty)$ is of order $p$. Now a generator $g \in \Sym_n$ of $Q$ has order $p$,  
and it acts trivially on the tame part of dimension $n-p$, and thus
it is a $p$-cycle. By Jordan's theorem \cite{J}, $G_\geo = \ABS_n$ or $\Sym_n$. Since a generator of 
the image of $I(0)$ has its spectrum on $\sH$ consisting of all nontrivial $n^{\mathrm {th}}$ roots of unity, it must act as an $n$-cycle. Applying 
Theorem \ref{generation}, we conclude that $G_\geo=\ABS_n$ if $2\nmid n$ and $G_\geo=\Sym_n$ if $2|n$.

\smallskip
(ii) Now we choose any prime $p|n$, and again follow Lemma \ref{Sawin}(i) to consider $f_{\star}\overline{\Q_{\ell}}/\overline{\Q_{\ell}}$ with $f(x)=x^k(x-1)^{n-k}$ in characteristic $p$,
to get 
\begin{equation}\label{eq:alt21}
  \sH:= \sH yp(\Ch_k \smallsetminus \{\triv\} \sqcup \Ch_{n-k}; \Ch_{n_0} \smallsetminus \{\triv\}),
\end{equation}  
where $n_0$ is the $p'$-part of $n$ (also see \cite[Proposition 1.2(ii)]{Ka-RL-T-2Co1}). As in (i), 
$G_\geo \leq \Sym_n$ is a doubly transitive subgroup. But now a generator $g_0$ of the image of $I(0)$ has its (simple) spectrum on 
$\sH$ consisting of all $(n-k)^{\mathrm {th}}$ and all $k^{\mathrm {th}}$ roots of unity, hence it must act as a product of 
an $(n-k)$-cycle and a $k$-cycle. 
We note that $g_0 \in \ABS_n$ if and only if $2|n$. Hence, using Theorem \ref{generation} and assuming $G_\geo \geq \ABS_n$, 
we can say that $G_\geo=\Sym_n$ is $2 \nmid n$ and $G_\geo=\ABS_n$ if $2|n$.

Since $\gcd(k,n-k)=1$, $g_0^{n-k}$ is a $k$-cycle.
Suppose in addition that $2 \leq k \leq n/8$. As $g_0^{n-k}$ fixes $n-k$ points, we have that 
$G_\geo \geq \ABS_n$ by Bochert's theorem \cite{Bo}; in fact, the same is true
by Manning's theorem \cite{Man} if $k \leq n/3-2\sqrt{n/3}$. The same is true by Jordan's theorem if $k$ is 
a prime. Note that, up until this point of this proof, we have not used the Classification of 
Finite Simple Groups.

Suppose now that $n/8 < k \leq n/2$ and $k$ is not a prime. As the element $g_0^{n-k}$ of order $k$ fixes $n-k \geq n/2$ points, we can 
quote either \cite[Theorem 1]{GM} or \cite[Theorem 1.2]{Jo}, which both use the Classification, to conclude that $G_\geo \geq \ABS_n$.

Finally, assume that $k=1$, in which case $g_0$ is an $(n-1)$-cycle. If $n$ is not a prime power (equivalently,
$n$ is not a $p$-power since $p|n$) and $n-1$ is not a prime, then 
$G_\geo \geq \ABS_n$ by \cite[Theorem 1.2]{Jo}. 

We note that when $n=p^a$, the Kloosterman sheaf $\sH$ is Kummer induced.
Consider the case $n=r+1$ for a prime $r \geq 5$ and assume that $G_\geo \not\geq \ABS_n$. 
Suppose $(n,p)=(6,3)$. Then $\dim\Wild = 4$ and so $|Q| = 3^2$ by Proposition \ref{Pinftyimage}, but $\rank(\sH)=5$ divides $|G_\geo|$, 
and this forces $G_\geo \leq \Sym_6$ to contain $\ABS_6$ by \cite{ATLAS}. 
Suppose $(n,p)=(12,2)$. By \cite[Theorem 1.2]{Jo}, $G_\geo \in \{M_{11},M_{12},\PSL_2(11), \PGL_2(11)\}$. As $\dim\Wild = 9$, 
$G_\geo$ must contain an element of order divisible by $9$ (namely a generator for the tame quotient $I(\infty)/P(\infty)$) by
Proposition \ref{Pinftyimage}(ii), which is impossible in all the four listed groups. Next suppose that $(n,p)=(24,3)$. By \cite[Theorem 1.2]{Jo}, $G_\geo \in \{M_{24},\PSL_2(23), \PGL_2(23)\}$. 
Now $\dim\Wild = 16$, whence $G_\geo$ must contain 
an element of order divisible by $16$, which is again impossible in all the three listed groups. 
Assume now that $r \neq 5,11,23$ and $n \neq p^a$. Then by \cite[Theorem 1.2]{Jo} we have 
$\PSL_2(r) \leq G_\geo \leq \PGL_2(r)$. This last possibility is ruled out by Theorem \ref{char-sheaf1},
which implies that $p=\Char(\sH)$ must have been equal to $r$ and so coprime to $n$.
\end{proof}

\begin{rmk}
Let us comment on $G_\geo$ of Sawin's sheaf $\sH$ of \eqref{eq:alt21} in the exceptional cases $(n,p) = (6,2)$, $(12,3)$, and 
$(24,2)$ of Theorem \ref{alt2}(ii) when $k=1$.
If $(n,p)=(12,3)$, then $\sH$ is recorded in Table 3 and it is shown in Lemma \ref{m11} that $G_\geo = M_{11}$. 
\edit{If $(n,p)=(24,2)$, then $\sH$ is recorded in Table 3 and we show that $G_\geo = M_{24}$.} Finally, let $(n,p) = (6,2)$. Then 
\cite[Corollary 8.2]{KT8} yields a hypergeometric sheaf $\sH_1$ of type $(4,1)$ in characteristic $2$ 
with $G_\geo = \PGL_2(4) \cong \ABS_5$, and with 
$C_5$ as the image of $I(0)$. Applying Theorem \ref{Brauerp} to the two irreducible representations of degree $4$ and $5$ of 
$\ABS_5$, we get a hypergeometric sheaf $\sH_2$ of type $(5,2)$ in characteristic $2$ 
with $G_\geo = \ABS_5$, and with $C_5$ as the image of $I(0)$; in particular, the set of ``upstairs'' characters of $\sH_2$ is $\Ch_5$. 
A $2'$-generator $g \in \ABS_5$ of $I(\infty)/P(\infty)$ has order divisible by $3=\dim\Wild$, hence $\ord(g)=3$, and the set of ``downstairs''
characters of $\sH_2$ is $\Ch_3^\times$. Thus $\sH=\sH_2$ has $G_\geo = \ABS_5$.
\end{rmk}

\subsection*{9B. Sporadic groups}
\begin{lem}\label{m11}
The first three lines of Table 3 give hypergeometric sheaves over $\G_m/\overline{\F_3}$, each with
finite geometric monodromy group 
$G_\geo = {\rm M}_{11}$.
\end{lem}

\begin{proof}
We start with the rank 11 sheaf $\sH_1$. This can be obtained as a Sawin's sheaf with $(n,k)=(12,11)$; in fact, as \edit{Sawin \cite{Sa} kindly explained to us}, it follows from previous results of \edit{Adler} and \edit{Abhyankar} that this sheaf has 
$G_\geo = G \cong {\rm M}_{11}$, with the image of $I(0)$ in $G$ being $\langle g_0 \rangle \cong C_{11}$. 
As $p=3$ and $\dim \Wild = 8$, we see that the image $Q$ of $P(\infty)$ in 
$G$ is $Q \cong C_3^2$. Now, the image $J$ of $I(\infty)$ in $G$ permutes cyclically the $8$ linear characters of $Q$ on 
$\Wild$, and checking the character table of $G$ \cite{GAP}, we see that $J \cong C_3^2 \rtimes C_8$ (as listed in Table 3).

Let $\Phi_1:G \to \GL(\sH_1)$ denote the representation of $G$ on $\sH_1$. Let $\Phi_2:G \to \GL_{10}(\overline{\Q_\ell})$ and 
$\Phi_3:G \to \GL_{10}(\overline{\Q_\ell})$ denote irreducible representations of $G$ that afford a rational, respectively
non-real, character of degree $10$. Using \cite{GAP} we can check that 
$\Trace(\Phi_i(g))-\Trace(\Phi_1(g)) = -1$ for all $3$-elements $g \in G$ and $i = 2,3$. It follows from Theorem \ref{Brauerp}
that $\Phi_i$, $i = 2,3$, gives rise to a hypergeometric sheaf $\sH_i$ over $\G_m/\overline{\F_3}$ with $G$ as its geometric monodromy group. The ``upstairs'' and ``downstairs'' characters of $\sH_i$ can be seen by inspecting the 
spectra of $g_0$ and an element of order $8$ in $J$ in $\Phi_i$, which are precisely those listed in Table 3. We also note
that $\sH_2$ is Sawin-like, with $(n,k)=(11,9)$, see Lemma \ref{Sawin}(ii).
\end{proof}

\begin{lem}\label{m12}
There does not exist any hypergeometric sheaf $\sH$ of type $(D,m)$ with $12 \geq D >m$, that has finite geometric monodromy group $G$ such that $S  \lhd G/\ZB(G) \leq \Aut(S)$ for $S \cong {\mathrm {M}}_{12}$.
\end{lem}

\begin{proof}
Assume the contrary, and let $p$ and $\varphi$ denote the characteristic of such a sheaf $\sH$ and 
the character of $G$ acting on $\sH$. Then a generator $g_0$ of the image of $I(0)/P(0)$ has simple spectrum
on $\sH$, and we can apply Theorem \ref{spor} to arrive at one of the following cases.

\smallskip
{\it Case 1}: $D=12$, $\obar(g_0) = 24$, and $G/\ZB(G) = S \cdot 2$. 

As $g_0$ is a $p'$-element, $p \geq 5$,
whence $p=5$ or $11$, and moreover  $Q/(Q \cap \ZB(G))$ embeds in a Sylow $p$-subgroup which is cyclic of order $p$. Thus $Q$ is abelian and $Q/(Q \cap \ZB(G)) \cong C_p$.  Next, observe that $L:=G^{(\infty)}$ is a quasisimple cover of $S$ acting on 
$\sH$ of rank $12$, whence $L = 2S$ and $x \in \ZB(G)L$ for any $x \in Q \smallsetminus \ZB(G)$. 
Checking the character table of $L$ as given in \cite{GAP}, we see that $|\varphi(x)|/\varphi(1) \leq 1/6$, and 
so $W = \dim \Wild \geq 8$ by \eqref{bd11}. As $Q$ is abelian, we see that $Q$ admits 
at least $8$ distinct linear characters on $W$. This is impossible when $p=5$, since, with the action of $Q \cap \ZB(G)$ fixed on 
$\sH$, $Q$ can have at most $|Q/(Q \cap \ZB(G))| = p$ linear characters lying above it. Thus $p=11$, whence 
$|\varphi(x)|/\varphi(1) \leq 1/12$ by \cite{GAP} and $W \geq 10$ by \eqref{bd11}. In particular, 
$p \nmid |\ZB(G)|$ by Proposition \ref{p-center}(iv), and so $Q \cong C_{11}$ and 
$W=10$. Since $\obar(g_0)=24$ and $G/\ZB(G) = S \cdot 2$, by inspecting the spectrum of such an element on $\sH$, we see
that the ``upstairs'' characters of $\sH$ must be $\Ch_{12}\chi$ for a fixed $\chi$. Next, a $p'$-element $g$ in the image of $I(\infty)$
permutes cyclically the $10$ characters of $Q$ on $\Wild$, hence $g$ is a scalar multiple of an element of class $10B$ or 
$10C$ of $2S \cdot 2$ in \cite{GAP}. Checking the spectrum of $g$, we see that the ``downstairs'' characters of $\sH$ must be 
$\Ch_{12}\rho$ for a fixed $\rho$. Thus $\sH$ is stable under multiplication by $\chi_2$, and so it is induced from a sheaf of rank
$6$. But this is also impossible, since this does not hold for the $12$-dimensional representations of $2S.2$.

\smallskip
{\it Case 2}: $D=10$. 

Checking  the character table of quasisimple covers of $S$, we see that $L:=G^{(\infty)} = 2S$. 
First we consider the case $p \neq 2$. Then $|\varphi(x)|/\varphi(1) \leq 1/5$ for all $x \in Q \smallsetminus \ZB(G)$ 
and $|Q| \geq 3$, whence $W \geq 6$ by \eqref{bd11}. If $p=5$, then $Q/(Q \cap \ZB(G)) \cong C_5$, and so $Q$ is abelian
and cannot have $6$ distinct linear characters of $\Wild$, a contradiction. Hence $p = 3$ or $11$, and $p \nmid D$. This in turn 
implies by Proposition \ref{p-center}(iv) that $p \nmid |\ZB(G)|$. As $p > 2$, we have $Q \leq \ZB(G)L$, and so in fact $Q \leq L$.
Now observe that $\varphi(y) - \psi(y) = -2$ for all $p$-elements $y \in L$, if $\psi \in \Irr(L)$ has degree $12$. If $G/\ZB(G) = S$,
then $G = \ZB(G)L$ and $\psi$ extends to $G$. In the remaining case we have $G = \langle \ZB(G)L,y \rangle$ where 
$y^2 \in \ZB(G)L$. As $y$ centralizes $\ZB(G)$, and $\psi$ extends to $L \cdot 2$, we have that $y$ fixes an extension of 
$\psi$ to $\ZB(G)L$, and so this extension extends to $G$. Thus in all cases $\psi$ extends to a character $\tilde\psi$ of $G$ of
degree $12$, and $\varphi(y) - \tilde\psi(y) = -2$ for all $p$-elements $y \in G$. This implies by Theorem \ref{Brauerp} that
there exists a rank 12 hypergeometric sheaf realizing $G$ in a representation with character $\tilde\psi$, contrary to the result of 
Case 1.

We have shown that $p=2$. Since we still have  $|\varphi(x)|/\varphi(1) \leq 1/5$ for all $x \in Q \smallsetminus \ZB(G)$, 
$W \geq 4$ by \eqref{bd11} using $|Q| \geq 2$. This in turn implies that $|Q| \geq 8$, and so in fact $W \geq 7$.  If $2 \nmid W$,
then a $p'$-element in the image of $I(\infty)$ permutes $W$ linear characters of $Q$ on $\Wild$ and so $\obar(g)$ is divisible
by $W$. But this is impossible, since $L.2$ does not possess any element of such order modulo $\ZB(G)$. Suppose $W=8$, whence
$Q$ acts irreducibly on $\Wild$. Now if $1 \neq z \in \ZB(Q)$, then $z$ acts as $1$ on $\Tame$ and as a scalar on 
$\Wild$, whence $|\varphi(z)| \geq 8-2 = 6$. Checking the character table of $L.2$, we see that $|\varphi(z)| = 10$, and 
so $z$ acts trivially on $\Wild$ and on $\sH$, contrary to $z \neq 1$. Thus $W = 10$. In this case, $\sH$ is Kloosterman,
and the $2'$-element $g_0$ has order $\geq 10$ modulo $\ZB(G)$. It follows that $\obar(g_0)=11$ and the ``upstairs'' characters of
$\sH$ should be $\Ch_{11}^\times\chi$ for a fixed character $\chi$. \edit{Presumably this case should however lead to 
$\SU_5(2)$ by [KT11].}

\smallskip
{\it Case 3}: $D=11$ and $\obar(g_0)=11$.

In this case we have $|\varphi(x)|/\varphi(1) \leq 3/11$ for all $x \in Q \smallsetminus \ZB(G)$, and so $W \geq 4$ by \eqref{bd11}. 
Also, checking the representations of quasisimple covers of $S$, we see that $G^{(\infty)} = S$, and moreover $G = \ZB(G) \times S$,
as the two $11$-dimensional irreducible representations of $S$ are fused by outer automorphisms of $S$. First we consider
the case $p=11$. Then $Q/(Q \cap \ZB(G)) \cong C_p$ and so $Q$ is abelian. This in turn implies that $W \neq 11$ (as otherwise
$Q$ is irreducible on $\Wild$ of dimension $p$), whence $Q \cap \ZB(G) = 1$ and $Q \cong C_p$ by Proposition \ref{p-center}(i).
As $\varphi(x) = 0$ for all $x \in Q \smallsetminus \ZB(G)$, we now have $W \geq 10$ by \eqref{bd11}, and so in fact $W=10$. 
Thus some $p'$-element $g$ of $\ZB(G) \times S$ permutes cyclically the $10$ distinct characters of $Q \cong C_{11}$ on $\Wild$. We will write a generator of $Q$ as $zh$, with $z \in \ZB(G)$ and $h \in S$ of order $11$. As $g$ normalizes $Q$, we have 
$z^ih^i = g(zh)g^{-1} = z(ghg^{-1})$, implying $z^i=z$ and $ghg^{-1} = h^i$. Checking the latter relation in $S$, we see that 
$g$ can permute cyclically only $5$ eigenspaces for $h$, and so for $zh$ as well, a contradiction.

We have shown that $p \neq 11$. Then $p \nmid |\ZB(G)|$ by Proposition \ref{p-center}(iv), and so we may assume $G = S$ by 
Corollary \ref{p'-center}. In the cases $p = 3,5$, we can further lift the surjection $\pi_1(\G_m/\overline{\F_p}) \surj S$ to 
a surjection $\pi_1(\G_m/\overline{\F_p}) \surj 2S$ and then consider an irreducible character $\psi$ of $2S$ of degree $12$. 
After inflating $\varphi$ to a character of $2S$, we get that $\varphi(y)-\psi(y) = -1$ for all $p$-elements $g \in 2S$. But this 
leads by Theorem \ref{Brauerp} to a hypergeometric sheaf of rank $12$ realizing $2S$, contradicting the result of Case 1.
Thus $p=2$. As $W \geq 4$, we must have $|Q| \geq 8$. Another application of \ref{bd11} now shows that $W \geq 7$. As in 
Case 2, we can rule out $W = 7,9$ as $G$ has no elements of order $7$ and $9$. Likewise, the case $W=8$ would lead 
to an element $1 \neq z \in \ZB(Q)$ acting as a scalar on $\Wild$ and $1$ on $\Tame$, whence $|\varphi(z)| \geq 8-3 = 5$,
which is impossible by \cite{GAP}. If $W=11$, then, as $\F_2(\zeta_{11}) = \F_{2^10}$, we must have 
$|Q|=2^{10}$, too big for a subgroup of $S$. Thus $W=10$, and the ``upstairs'' characters of the sheaf $\sH$ is now
$\Ch_{11}$. Now a $2'$-element $g$ in the image of $I(\infty)$ cyclically permutes the $5$ summands of $P(\infty)$ acting on $\Wild$.
Since $g \in S$, we see that $g$ has order $5$. Checking the spectrum of $g$, we get that the ``downstairs'' character of $\sH$ is
$\triv$, which also occurs upstairs, violating the irreducibility of $\sH$.
\end{proof}

\begin{lem}\label{hs}
There does not exist any hypergeometric sheaf $\sH$ of type $(D,m)$ with $D=22>m$, that has finite geometric monodromy group $G$ such that $S  \lhd G/\ZB(G) \leq \Aut(S)$ for $S \cong {\mathrm {HS}}$.
\end{lem}

\begin{proof}
Assume the contrary, and let $p$ and $\varphi$ denote the characteristic of such a sheaf $\sH$ and 
the character of $G$ acting on $\sH$. Then a generator $g_0$ of the image of $I(0)/P(0)$ has simple spectrum
on $\sH$, and so by Theorem \ref{spor}, $\obar(g_0) = 30$. As $g_0$ is a $p'$-element, $p \geq 7$.
On the other hand, $p$ divides $|\Aut(S)| = 2|S|$,
whence $p=7$ or $11$, and moreover  $Q/(Q \cap \ZB(G))$ embeds in a Sylow $p$-subgroup which is cyclic of order $p$. Thus $Q$ is abelian and $Q/(Q \cap \ZB(G)) \cong C_p$. 
Next, observe that $G^{(\infty)}$ is a quasisimple cover of $S$. Since the Schur multiplier of $S$ is $C_2$ and $2S$ cannot act faithfully on any space of dimension $< 56$, see \cite{GAP}, $G^{(\infty)} \cong S$, and we will identify it with $S$. 
Moreover, $[G:\ZB(G)S] \leq 2$. Now, for any $1 \neq x \in Q$, $x$ belongs to $\ZB(G)S$. Checking the character table of $S$ as given in \cite{GAP}, we see that $|\varphi(x)|/\varphi(1) \leq 1/22$ if $x \in Q \smallsetminus \ZB(G)$. An application of \eqref{bd11} then gives $W =\dim\Wild \geq 21(1-1/|Q|) \geq 21(1-1/7) =18$. As $Q$ is abelian, we see that $Q$ admits 
at least $18$ distinct linear characters on $W$. But this is impossible, since, with the action of $Q \cap \ZB(G)$ fixed on 
$\sH$, $Q$ can have at most $|Q/(Q \cap \ZB(G))| = p$ linear characters lying above it.
\end{proof}

\begin{lem}\label{sp62}
There does not exist any hypergeometric sheaf $\sH$ of type $(D,m)$ with $D=15>m$, that has finite geometric monodromy group $G$ such that $G/\ZB(G) \cong \Sp_6(2)$.
\end{lem}

\begin{proof}
Assume the contrary, and let $p$ and $\varphi$ denote the characteristic of such a sheaf $\sH$ and 
the character of $G$ acting on $\sH$. Then a generator $g_0$ of the image of $I(0)/P(0)$ has simple spectrum
on $\sH$, and so by Theorem \ref{simple}, $\obar(g_0) = 15$. As $g_0$ is a $p'$-element, $p \nmid D=15$. Now, as in the proof of Corollary
\ref{spin32}, $Q \cap \ZB(G)=1$, whence $Q$ embeds in $S:=G/\ZB(G) \cong \Sp_6(2)$ and $p = 2$ or $7$. 
As $S$ is simple, $G^{(\infty)}\ZB(G) = G$, and so $G^{(\infty)}$ is a quasisimple cover of $S$ which acts irreducibly on
$\sH$ of rank $15$. It follows that $G^{(\infty)}$ is isomorphic to $S$ and so we can identity it with $S$. Now
$S \cap \ZB(G) = 1$, so $G = \ZB(G) \times S$. Checking the character table of $S$ as given in \cite{GAP}, we see that 
$|\varphi(x)| \leq 7$ for any $1 \neq x \in Q$, whence 
\begin{equation}\label{sp621}
  m \leq 7 + 8/|Q| \leq 11 
\end{equation} 
by \eqref{bd10}.
It follows from Proposition \ref{p-center}(iv) that $p \nmid |\ZB(G)|$. In turn, the latter and Corollary \ref{p'-center} allow 
us to assume that $G = S$ and so $\sH$ is self-dual. Now, $G$ has a faithful irreducible $\C$-representation of degree $7$ \cite{GAP}, 
so by Theorem \ref{bound2} and \eqref{sp621} we now have
\begin{equation}\label{sp622}
  4 \leq W = \dim\Wild = D-m \leq 7
\end{equation}

Assume $p=7$. Then $\varphi(x)=1$ for all $1 \neq x \in Q$ and $Q \cong C_7$. It follows that $m = 3$ and $W=\dim\Wild = 12$, a contradiction.
Thus $p=2$. As $W \geq 4$, $Q$ cannot be 
(abelian) of order $\leq 4$, whence $|Q| \geq 8$, $m \leq 8$ by \eqref{sp621}, and so \eqref{sp622} implies that $W=7$ and $|Q|=8$. 
Now a generator of the tame quotient $I(\infty)/P(\infty)$ maps onto an element $h \in S$ which permutes cyclically the seven characters of
$Q$ on $\Wild$. It follows that $\ord(h) = 7$, and so $h$ cannot have $8$ distinct eigenvalues on $\Tame$, again a contradiction.
\end{proof}

\subsection*{9C. Symplectic groups}
The next result is well known; we recall a proof for the reader's convenience:

\begin{lem}\label{sp-mod1}
Let $q$ be an odd prime power, $n \in \Z_{\geq 1}$, and let $\om_n=\xi_n+\eta_n$ denote the character of a total Weil module
$M$ of $L:=\Sp_{2n}(q)$, so that $\xi_n,\eta_n$ are irreducible Weil characters of degree $(q^n+1)/2$ and $(q^n-1)/2$, respectively. Then
for any $2'$-element $g \in L$, $\xi_n(g)=\eta_n(g)+1$. 
\end{lem}

\begin{proof}
Let $\bj$ denote the central involution of $G$, and let $M_2$ denote a reduction modulo $2$ of the complex module $M$. As shown in
\cite[\S5]{GMST}, $M_2$ has a composition series 
$$0 < (\bj-1_{M_2})M_2 < \CB_{M_2}(t) < M_2,$$ 
with three successive simple 
quotients $X$, $Y$, and $X$, where $Y$ is trivial and $\dim(X) = (q^n-1)/2 = \eta_n(1)$. It follows that the restrictions to $2'$-elements
of $\eta_n$ and $\xi_n$ must equal to $\varphi$ and $\varphi+1_L$, where $\varphi$ is the Brauer character of $Y$. Hence
$\xi_n(g)=\eta_n(g)+1$ for all $2'$-elements $g \in L$. 
\end{proof}

\begin{prop}\label{sp-mod2}
Let $q=p^f$ be a power of  an odd prime $p$, $n \in \Z_{\geq 1}$, $(n,q) \neq (1,3)$, 
and let $\Phi:G \to \GL(V) \cong \GL_{(q^n-1)/2}(\C)$ be a faithful irreducible representation of a finite almost quasisimple group $G$. 
Suppose that $\det(\Phi(G)) \cong \mu_N$ is a $p'$-group, $E(G)$ is a quotient of $L:=\Sp_{2n}(q)$ by a central subgroup, and that 
$\Phi|_{E(G)}$ inflated to $L$, is an irreducible Weil representation of $L$. Then there exists a finite almost quasisimple group 
$\hat G$, a surjection $\pi: \hat G \surj G$ with kernel a central subgroup, of order $1$ if $2|D$ and $2$ if $2 \nmid D$, 
and an irreducible representation $\Psi:\hat G \to \GL_{(q^n+1)/2}(\C)$ such that 
$$\Trace(\Psi(g)) = \Trace(\Phi(\pi(g)))+1$$ 
for all $p$-elements $g \in \hat G$. 
\end{prop}

\begin{proof}
Write $\ZB(G) = Z_1 \times Z_2$, where $2 \nmid |Z_1|$ and $Z_2$ is a $2$-group. 
Note that if $2|D:=(q^n-1)/2$, then $E(G)=L = \Sp_{2n}(q)$ and $Z \cap L = Z_2 \cap L = \ZB(L)$, whereas 
if $2 \nmid D$, then $E(G)=L/\ZB(L) = \PSp_{2n}(q)$ and $Z \cap E(G) = 1$.  
By \cite[Lemma 4.3]{KT2}, we can embed $L$ in $\tilde{L} = \Sp_{2nf}(p)$ and 
extend $\Phi|_{E(G)}$ to an irreducible Weil representation $\tilde{\Phi}: \tilde{L} \to \GL(V)$. As $E(G) \lhd G$ and
no outer automorphism of $\tilde{L}$ fixes the equivalence class of $\tilde{\Phi}$, we have
$$\Phi(G) \leq \NB_{\GL(V)}(\Phi(E(G)) \leq \ZB(\GL(V))\tilde\Phi(\tilde{L}).$$
In fact, 
\begin{equation}\label{spm11}
  \Phi(G) \leq \tilde{Z}\tilde\Phi(\tilde{L}), 
\end{equation}
where $\tilde{Z} \cong \mu_{ND}$ is a cyclic subgroup of order $ND$ of $\ZB(\GL(V))$. 
[Indeed, for any $x \in G$, we can write $\Phi(x) = \alpha \tilde\Phi(y)$ for some $\alpha \in \C^\times$ and $y \in \tilde{L}$. By assumption,
$$1= \det(\Phi(x))^N = \alpha^{ND}\det(\tilde\Phi(y))^N = \alpha^{ND}$$ 
as $\tilde{L}$ is perfect, whence $\alpha^{ND}=1$.] 

Write $\tilde{Z} = T_1 \times T_2$, with $T_1 = \langle t_1 \rangle \geq Z_1$ cyclic of odd order, and $T_2 = \langle t_2 \rangle \geq Z_2$ 
a cyclic $2$-group. We also write $t_i = \gamma_i \cdot 1_V$ with $\gamma_i \in \C^\times$.  
Let $\tilde\Psi:\tilde{L} \to \GL_{D+1}(\C)$ be the other constituent of the total Weil representation of $\tilde{L}$ having $\tilde\Phi$ as one constituent. Now, if $2|D$, we extend $\tilde\Psi$ to $\tilde{Z}$ by letting $t_1$ act as scalar $\gamma_1$ and $t_2$ act as
scalar $\gamma_2^2$. We also take $\hat{G} = G$, $\tilde\pi:=1_{\tilde Z\tilde L}$, $\pi:=1_G$, and 
choose $\Psi$ to be the restriction of $\tilde\Psi$ to $G$. On the other hand, 
if $2 \nmid D$, then note that $\tilde\Phi$ is trivial and $\tilde\Psi$ is faithful at $\ZB(\tilde L)$. In this case, we consider 
$\hat{\tilde Z} = \langle t_1,\hat{t}_2 \rangle$, and extend $\tilde\Psi$ to $\hat{\tilde Z}$ by letting $t_1$ act as scalar $\gamma_1$ and 
$\hat{t}_2$ act as scalar $\sqrt{\gamma_2}$. We can also define a surjection $\tilde\pi: \hat{\tilde Z}*L \to \tilde{Z} \times L/\ZB(L)$ by sending 
$t_1$ to $t_1$, $\hat{t}_2$ to $t_2$, and $y \in \tilde L$ to $y\ZB(L)$, 
and note that $\Ker(\tilde \pi) = \ZB(\tilde L) = \ZB(L)$. 
Finally, we take $\hat{G} = \tilde \pi^{-1}(G)$, and choose $\Psi$ and $\pi$ to be the restrictions of $\tilde\Psi$ and $\tilde\pi$ to $\hat{G}$. 

Now consider any $p$-element $g \in \hat{G}$, of order say $p^a$, and write $\Phi(\pi(g)) = \beta \tilde\Phi(\tilde\pi(h))$ for some 
$\beta \in \tilde{Z}$ and $h \in \tilde{L}$, using \eqref{spm11}. Then $\Phi(\pi(g))^{p^a}=1_V \in \tilde\Phi(\tilde{L})$ and 
$$\Phi(\pi(g))^{ND} = \beta^{ND}\tilde\Phi(\tilde\pi(h)^{ND}) = \tilde\Phi(\tilde\pi(h)^{ND}) \in \tilde\Phi(\tilde{L}).$$ 
As $p \nmid ND$, it follows that $\Phi(\pi(g)) \in \tilde\Phi(\tilde{L})$, and so we may assume that $\beta=1$, $h$ is a $p$-element, and 
$\tilde\Phi(\pi(g)) = \tilde\Phi(\tilde\pi(h))$. Recall that $\tilde\Phi$ is faithful, so $\tilde\pi(g)=\pi(g)=\tilde\pi(h)$. But $\Ker(\tilde\pi) \leq C_2$, so we now have that $g=h$. Thus
$$\Trace(\Psi(g))-\Trace(\Phi(\pi(g)))=\Trace(\tilde\Psi(h))-\Trace(\tilde\Phi(\tilde\pi(h))) =1$$
by Lemma \ref{sp-mod1} applied to $\tilde L$.
\end{proof}

\begin{thm}\label{sp-mod-main}
Let $q$ be an odd prime power, $n \in \Z_{\geq 1}$, $(n,q) \neq (2,3)$, $(3,3)$, and $q \neq 5,7,9,11,25$ if 
$n=1$. Then case {\rm (i)($\alpha$)} of Theorem \ref{ss-sp} with 
$\obar(g) = (q^n-1)/2$ does not lead to hypergeometric sheaves in dimension $(q^n \pm 1)/2$. More precisely, there is no hypergeometric sheaf $\sH$ of type $(D,m)$ with $m < D=(q^n \pm 1)/2$, with finite geometric monodromy group $G=G_\geo$ such that  
$G$ is almost quasisimple with $S = \PSp_{2n}(q)$ as a non-abelian composition factor, and with the image of $I(0)$ being a cyclic 
group $\langle g \rangle$ where $\obar(g) = (q^n-1)/2$.
\end{thm}

\begin{proof}
The finiteness of $G$ implies that $\Phi(g)$ has simple spectrum on $\sH$, where $\Phi$ denotes the representation of $G$ on $\sH$.
In particular, $\obar(g) \geq D$, and so the case $D = (q^n+1)/2$ is impossible. 

Consider the case $D=(q^n-1)/2$ and assume the contrary that $\sH$ exists. The assumptions on
$(n,q)$ imply by Theorem \ref{char-sheaf1} that the characteristic of $\sH$ is the prime $p$ dividing $q$. Now we consider the 
surjection $\phi:\pi_1(\G_m/\overline{\F_p})\surj G$ underlying $\sH$ and apply Theorem \ref{sp-mod2} to get the surjection
$\pi: \hat{G} \surj G$ with $\Ker(\pi) \leq \ZB(\hat{G})$ of order $1$ or $2$ and the representation $\Psi:\hat{G} \to \GL_{D+1}(\C)$.
If $2|D$, then $\Ker(\pi) = 1$, then we may trivially lift $\phi$ to a surjection $\varphi:\pi_1(\G_m/\overline{\F_p})\surj \hat{G}$
such that $\pi \circ \varphi=\phi$. Assume $2 \nmid D$, so that $\Ker(\pi) \cong C_2$. The obstruction to lifting $\phi$ to a homomorphism
$\varphi: \pi_1(\G_m/\overline{\F_p})\to \hat{G}$
lies in the group $H^2(\G_m/\overline{\F_p},\Ker(\pi))=0$, the vanishing because open curves have cohomological dimension $\le 1$, cf. \cite[Cor. 2.7, Exp. IX and Thm. 5.1, Exp. X]{SGA4t3}. We claim that $\varphi$ is surjective. Indeed, we have $H \leq \hat{G}$ for
$H := {\mathrm {Im}}(\varphi)$ and $\pi(H) = \phi\bigl(\pi_1(\G_m/\overline{\F_p})\bigr) = G$. Now, if $H \geq \Ker(\pi)$, then 
$|H| \geq 2|G| = |\hat{G}|$ and so $H = \hat{G}$. Otherwise $H \cap \Ker(\pi) = 1$, $H \cong G$ and so 
$H^{(\infty)} \cong G^{(\infty)} = E(G) \cong \PSp_{2n}(q)$. Also, $|H|=|G|=|\hat{G}|/2$, so $H \lhd \hat{G}$. Thus 
$\hat{G}^{(\infty)} = H^{(\infty)} \cong \PSp_{2n}(q)$. On the other hand, the construction of $\hat{G}$ in Theorem \ref{sp-mod2} ensures that 
$\hat{G}^{(\infty)} \cong \Sp_{2n}(q)$, a contradiction. 

Now we can apply Theorem \ref{sp-mod2} and Theorem \ref{Brauerp} to the surjection $\varphi:\pi_1(\G_m/\overline{\F_p})\surj \hat{G}$ together with $\Phi \circ \pi$ and $\Psi$ and obtain another hypergeometric sheaf $\sH'$ of type $(D+1,m+1)$, also tame at $0$. So the image $I$ of $I(0)$ in $\hat{G}$ is a cyclic group $\langle \hat{g} \rangle$ which projects onto $\langle g \rangle$ via $\pi$ (seen 
by the action on $\sH$). In the case $\Ker(\pi)=1$, $\obar(\hat{g}) = \obar(g) = (q^n-1)/2 = D < \rank(\sH')$, and so $\sH'$ cannot have
finite monodromy. In the other case, $2 \nmid D$, recall by Theorem \ref{ss-sp} that $\gbar \in \PSp_{2n}(q) \lhd G$ and so 
$\ord(g) = D$. As $\Ker(\pi) \leq \ZB(\hat{G})$ and $\pi(\hat{g}) = g$, we have $\hat{g}^D \in \ZB(\hat{G})$, and so
again $\obar(\hat{g}) \leq D < \rank(\sH')$, a contradiction.
\end{proof}

\subsection*{9D. Unitary groups}

The unitary analogue of Lemma \ref{sp-mod1} was proved in \cite[Theorem 7.2]{DT}; we will give a slight extension of it:

\begin{lem}\label{su-mod}
Let $q$ be any prime power, $n \in \Z_{\geq 2}$, and let $\zeta_n=\sum^q_{i=0}\zeta^i_n$ denote the character of a total Weil representation of $G:=\GU_{n}(q)$, as described in \cite[\S4]{TZ2}. Then for any $g \in G$ and any $0 \leq i,j \leq q$ we have
$$\zeta^i_n(g)-\zeta^j_n(g)=\zeta^i_n(1)-\zeta^j_n(1),$$ 
if at least one of the following two conditions holds.

\begin{enumerate}[\rm(a)]
\item $\ell$ is any prime divisor of $q+1$, $\ell \nmid \ord(g)$, and $i-j$ is divisible by the $\ell'$-part $s$ of $q+1$.
\item $\ord(g)$ is coprime to $q+1$.
\end{enumerate} 
\end{lem}

\begin{proof}
Let $\xi \in \overline{\F_q}^\times$ and $\tx \in \C^\times$ be primitive $(q+1)^{\mathrm {th}}$ roots of unity.
We write $q+1 = \ell^{c}s$ with $c \in \Z_{\geq 1}$ and $\ell \nmid s$ in the case of (a), and choose 
$(\ell,c,s) = (q+1,1,1)$ in the case of (b). Also let $V = \F_{q^2}^n$ denote the natural module of $G$. Then,  
by \cite[Lemma 4.1]{TZ2}, $\zeta^{i}_{n}(g) - \zeta^{j}_{n}(g)$ is equal to
$$\frac{(-1)^{n}}{q+1}\sum_{0 \leq k \leq q,~\ell^{c} \nmid k}
  (\tx^{ik} - \tx^{jk})(-q)^{\dim\,\Ker(g-\xi^{k}\cdot 1_V)} +
  \frac{(-1)^{n}}{q+1}\sum_{0 \leq k \leq q,~\ell^{c}|k}
  (\tx^{ik} - \tx^{jk})(-q)^{\dim\,\Ker(g-\xi^{k} \cdot 1_V)}.$$
Since $i \equiv j (\bmod\ s)$, in the second summation we have
$\tx^{ik} = \tx^{jk}$. In the case of (a), the condition
$\dim\, \Ker(g-\xi^{k} \cdot 1_V) \neq 0$ implies $\ell^{c}|k$, hence in the first summation
we have $\dim\,\Ker(g-\xi^{k} \cdot 1_V) = 0$. The latter equality also holds for $1 \leq k \leq q$ in the case of (b).
Thus $\zeta^{i}_{n}(g) - \zeta^{j}_{n}(g)$ is equal to
$$\frac{(-1)^{n}}{q+1}\sum_{0 \leq k \leq q}
  (\tx^{ik} - \tx^{jk}) = (-1)^{n} (\delta_{i,0} - \delta_{j,0})=\zeta^i_n(1)-\zeta^j_n(1),$$
as stated.
\end{proof}

\begin{thm}\label{su1a-mod-main}
Let $q$ be an odd prime power, $2 \nmid n \in \Z_{\geq 3}$, and $(n,q) \neq (3,3)$, $(3,5)$. 
Then case {\rm (ii)} of Theorem \ref{ss-su} does not lead to hypergeometric sheaves in dimension $(q^n-q)/(q+1)$. More precisely, there is no hypergeometric sheaf $\sH$ of type 
$(D,m)$ with $m < D=(q^n-q)/(q+1)$, with finite geometric monodromy group $G=G_\geo$ such that  
$G$ is almost quasisimple with $S = \PSU_n(q)$ as a non-abelian composition factor, and with the image of $I(0)$ being a cyclic 
group $\langle g_0 \rangle$ where $\obar(g_0) = q^{n-1}-1$.
\end{thm}

\begin{proof}
(i) Assume the contrary that such a hypergeometric sheaf $\sH$ exists, and let $\Phi:G \to \GL(V)$ denote the corresponding 
representation, with $V = \overline{\Q_\ell}^D$, and with character $\varphi$. The assumptions on $(n,q)$ imply by Theorem
\ref{char-sheaf1} that the characteristic $p$ of $\sH$ divides $q$, that is, $q=p^f$ for some $f \in \Z_{\geq 1}$.
Recall that $V$ is irreducible over $G^{(\infty)}$, which in 
turn is a quotient of $\SU_n(q)$. As $D=(q^n-q)/(q+1)$, we see that in fact $G^{(\infty)} = S$. 
Also recall that $\varphi|_S$ extends to the Weil character $\zeta^0_n$ of $\PGU_n(q)$, which is fixed by, and hence, extends to 
$A := \PGU_n(q) \rtimes C_{2f} = \Aut(S)$. Thus we can extend $\Phi$ to an $A$-representation on $V$ which we also 
denote by $\Phi:A \to \GL(V)$. Certainly, $\Phi(A)$ and $\Phi(G)$ both normalize $\Phi(S)$. Using the finiteness of $G$, we can 
find a finite cyclic subgroup $\mu_N$ of order $N$ of $\ZB(\GL(V))$ such that
$$\Phi(G) \leq \mu_N \times \Phi(A) < \ZB(\GL(V))\Phi(A) = \NB_{\GL(V)}(\Phi(S)).$$
Here, $\mu_N \cap \Phi(A) = 1$ since $\CB_A(S) = 1$ and ${\rm {soc}}(A)=S$. We now define $\Gamma := C_N \times A$
and extend $\Phi$ to $C_N = \ZB(\Gamma)$ via scalar action, so that $\Phi(\Gamma) = \mu_N \times \Phi(A)$.

We also note that $m \geq 1$. Indeed, if $m=0$, then $\sH$ is Kloosterman, and the $D$ ``upstairs'' characters of 
$\sH$ can be read off from the spectrum of the image $\langle g_0\rangle$ of $I(0)$, and seen to be
$\bigl(\Ch_{E(q+1)} \smallsetminus \Ch_E\bigr)\chi$, where $E:=(q^{n-1}-1)/(q+1)=D/q$ and $\chi$ is some character. This set is 
stable under the multiplication by $\xi_E$, and so $\sH$ is induced from a rank $q$ sheaf. But this is impossible, since 
$S$ has no proper subgroups of index $\leq E$, cf. \cite[Table 5.2.A]{KlL}.

\smallskip
(ii) Next we consider the Weil character $\zeta^{(q+1)/2}_n$ of $\GU_n(q)$, which restricts irreducibly to $\SU_n(q)$ and 
in fact factors through $S=\PSU_n(q)$, since its kernel is the subgroup of order $(q+1)/2$ of $\ZB(\GU_n(q))$ and so contains
$\ZB(\SU_n(q))$ since $2 \nmid n$. Thus we obtain a self-dual representation $\Psi:S \to \GL(W)$, with 
$W = \overline{\Q_\ell}^{D+1}$, whose character is invariant under $A$, since $\zeta^{(q+1)/2}_n$ is fixed by the subgroup
$C_{2f}$ of field automorphisms. As $\deg(\Psi) = D+1$ is odd and $S$ is perfect, $\Psi$ extends to a self-dual representation
$\Psi:A \to \GL(W)$ by \cite[Theorem 2.3]{NT}, which we inflate to a self-dual representation $\Psi:\Gamma \to \GL(W)$ by letting $C_N = \ZB(\Gamma)$ act trivially.

\smallskip
(iii) Recall that $\sH$ gives rise to a surjection $\phi:\pi_1(\G_m/\overline{\F_p}) \surj G$. We will now compose $\phi$ with
$$\Phi: G \hookrightarrow \Gamma \to \GL(V) \mbox{ and }\Psi: G \hookrightarrow \Gamma \to \GL(W)$$ 
and compare their traces at elements in $\phi(P(0)) \cup \phi(P(\infty))$. 
First, if $y \in \phi(P(0))$, then, since $\sH$ is tame at $0$, $y$ acts trivially in
$\Phi$. But the latter is faithful, so $y = 1$, i.e. $\phi(P(0))=1$ and we trivially have 
\begin{equation}\label{su11}
  \Trace(\Phi(y))-\Trace(\Psi(y)) = -1
\end{equation}
for all $y \in \phi(P(0))$. 

Now, let $\varphi$ and $\psi$ denote the character of the $\Gamma$-representations $\Phi$ and $\Psi$, 
and let $\varphi^\circ$ and $\psi^\circ$ denote their restrictions to $2'$-elements. By \cite[Theorem 7.2(ii)]{DT}, 
$\theta:=\varphi^\circ|_S$ is an irreducible $2$-Brauer character of $S$.
Next, by Lemma \ref{su-mod}(a), 
$\psi^\circ|_S=\theta+1_S$. As $S \lhd \Gamma$ and $D > 1$, it follows from Clifford's theorem that 
$\psi^\circ = \alpha+\beta$ is the sum of two irreducible Brauer characters, with $\alpha$ lying above $\theta$ and 
$\beta$ lying above $1_S$. Furthermore, as $\Psi$ is self-dual, we have that $\alpha$ and $\beta$ are both real-valued.
But $\beta(1)=1$, so in fact $\beta=1_\Gamma$.  We have shown that $\psi^\circ-1_\Gamma=\alpha$ and 
$\varphi^\circ$ are two extensions to $\Gamma$ of $\theta$. By \cite[Cor. (8.20)]{N}, there exists a linear character 
$\lambda$ of $\Gamma/S$ such that 
\begin{equation}\label{su12}
  \varphi^\circ = (\psi^\circ-1_\Gamma)\lambda.
\end{equation}  
Taking the complex conjugate and using $\psi=\overline\psi$, we obtain
\begin{equation}\label{su13}
  \overline{\varphi^\circ} = \overline{(\psi^\circ-1_\Gamma)\lambda} =  (\psi^\circ-1_\Gamma)\overline\lambda = 
   \varphi^\circ\lambda^2.
\end{equation}
In particular, we have that $\overline\varphi|_Q = \varphi|_Q \cdot \lambda^2|_Q$. Note that $D = (q^n-q)/(q+1)$ is not 
a $p$-power and so $|Q| \neq D$. Hence $\lambda^2|_Q=1_Q$ by Lemma \ref{dual-mult}.
But $Q$ has odd order, so in fact $\lambda|_Q = 1_Q$. The relation \eqref{su12} applied to 
$y \in Q$ now implies that \eqref{su11} holds for all $y \in Q = \phi(P(\infty))$.
 
Now we can apply Theorem \ref{Brauerp} to $\Phi$ and $\Psi$ to conclude that $\Psi$ leads to a hypergeometric sheaf $\sH'$
of rank $D+1$ with geometric monodromy group $\Psi(G)$. In particular, the image $\langle g_0 \rangle$ of $I(0)$ in 
$\Psi(G)$ has simple spectrum on $\sH'$. But this contradicts Theorem \ref{ss-su}, since $\obar(g_0) = q^{n-1}-1$.
\end{proof}

Next we prove the $q$-even analogue of Theorem \ref{su1a-mod-main}.

\begin{thm}\label{su1b-mod-main}
Let $q=2^f$, $2 \nmid n \in \Z_{\geq 3}$, and $(n,q) \neq (3,2)$, $(3,4)$, $(5,2)$.
Then case {\rm (ii)} of Theorem \ref{ss-su} does not lead to hypergeometric sheaves in dimension $(q^n-q)/(q+1)$. More precisely, there is no hypergeometric sheaf $\sH$ of type 
$(D,m)$ with $m < D=(q^n-q)/(q+1)$, with finite geometric monodromy group $G=G_\geo$ such that  
$G$ is almost quasisimple with $S = \PSU_n(q)$ as a non-abelian composition factor, and with the image of $I(0)$ being a cyclic 
group $\langle g_0 \rangle$ where $\obar(g_0) = q^{n-1}-1$.
\end{thm}

\begin{proof}
(i) Assume the contrary that such a hypergeometric sheaf $\sH$ exists, and let $\Phi:G \to \GL(V)$ denote the corresponding 
representation, with $V = \overline{\Q_\ell}^D$, and with character $\varphi$. The assumptions on $(n,q)$ imply by Theorem
\ref{char-sheaf1} that the characteristic $p$ of $\sH$ divides $q$, that is, $p=2$. 
Recall that $V$ is irreducible over $G^{(\infty)}$, which in 
turn is a quotient of $\SU_n(q)$. As $D=(q^n-q)/(q+1)$, we see that in fact $G^{(\infty)} = S$ and that $\Phi|_S$ affords the
Weil character $\zeta^0_n$ of $\SU_n(q)$, which is real-valued. 

First suppose that $\dim\Wild=D-m=1$. 
By Proposition 2.22 and Lemma 2.19 of \cite{GT3}, 
$|\varphi(g)|/\varphi(1) \leq (3.95)/4$ for all $1 \neq g \in Q$. It now follows from \eqref{bd11} that $D \leq 160$, which is ruled out 
by our assumptions on $(n,q)$, unless possibly $(n,q) = (3,8)$ or $(7,2)$. When $(n,q)=(3,8)$, using the character table
of $\Aut(\PSU_3(8))$ given in \cite{GAP} we can check that $|\varphi(g)|/\varphi(1) \leq 1/7$ for all $1 \neq g \in Q$. 
When $(n,q)=(7,2)$, using the character table
of $X:=\SU_7(2)$ given in \cite{GAP} we can check that $|\varphi(x)|/\varphi(1) \leq 1/2$ for all $1 \neq x \in X$, whence 
by \cite[Lemma 2.19]{GT3} we have 
$|\varphi(g)|/\varphi(1) \leq (3.5)/4$ for all $1 \neq g \in Q$. Hence, in these two cases we have $D \leq 16$ by \eqref{bd11}, which 
is impossible.

Hence we may assume that $D-m \geq 2$, and therefore 
\begin{equation}\label{su10b}
  G=\OB^2(G) 
\end{equation}  
by Theorem \ref{generation}. Now both $\Phi$ and its dual $\Phi^*$ are extensions to 
$G$ of $\Phi|_S$, hence $\Phi^* \cong \Phi \otimes \Lambda$ for some $1$-dimensional representation $\Lambda$ by 
Gallagher's theorem \cite[Cor. (6.17)]{Is}. By \eqref{su10b}, $\Lambda$ has odd order. Applying Corollary \ref{dual}, we can find 
a power $\Theta$ of $\Lambda$ so that $\Phi \otimes \Theta$ is self-dual and giving rise to a hypergeometric sheaf, and 
$(\Phi \otimes \Theta)(g_0)$ has the same central order $q^{n-1}-1$.
Replacing $(G,\Phi)$ by $(G/\Ker(\Phi \otimes \Theta),\Phi \otimes \Theta)$, we may assume that $\Phi$ is self-dual, and 
$\Phi(g_0)$ has central order $q^{n-1}-1$.
This in turn implies that  $|\ZB(G)| \leq 2$. Note furthermore that $\ZB(G) \cap S = 1$ and $G/\ZB(G) \cong \PGU_n(q)$ by Corollary \ref{simple2}. Hence, $G/S \cong \ZB(G) \cdot C_d$, where $d:=\gcd(n,q+1)$ is odd and $C_d\cong \PGU_n(q)/S$. 
The oddness of $d$ allows us to write $G/S \cong \ZB(G) \times C_d$. Applying \eqref{su10b} again, we get that $\ZB(G) = 1$, and 
thus 
\begin{equation}\label{su10b2}
  G \cong \PGU_n(q). 
\end{equation}

\smallskip
(ii) Let $r_1, \ldots,r_m$ be all the distinct primes divisors of $d = \gcd(n,q+1)$ (with $m=0$ if $d=1$). Then we can write 
$$n=n_0r_1^{a_1} \ldots r_m^{a_m} = n_0n',~q+1=q_0r_1^{b_1} \ldots r_m^{b_m}=q_0q'$$
for some integers $n_0,q_0,a_i,b_i \geq 1$, such that $\gcd(n_0,r_1 \ldots r_m)=\gcd(q_0,r_1 \ldots r_m)=1$. This implies 
\begin{equation}\label{su10b3}
  \gcd(q_0,q')=\gcd(q_0,n)=1.
\end{equation}
Here we prove that if $H \leq \Gamma:=\GU_n(q)$ is a subgroup that contains the central subgroup
$Z_0 \cong C_{q_0}$ of $Z:=\ZB(\GU_n(q))$ and maps onto $\PGU_n(q)$ under the surjection $\Gamma \surj \Gamma/Z$, then $H=\Gamma$.  
Indeed, let $\gamma \in \F_{q^2}^\times$ be of order $q+1$, so that the determinantal map 
$\det$ maps $\Gamma$ onto $\mu_{q+1}=\langle \gamma \rangle$.
Any element of $Z_0$ is a scalar matrix $x=\gamma^{q'i} \cdot I_n$ in $\Gamma$, with $i \in \Z/q_0\Z$ and with $\det(x) = (\gamma^{q'})^{ni}$. As $n$ is 
coprime to $q_0$ by \eqref{su10b3}, we see that $\det(x)$ runs over the subgroup $\mu_{q_0}$ of $\mu_{q+1}$, i.e. 
$\det(Z_0) = \mu_{q_0}$. Next, the condition that $H$ maps onto $\PGU_n(q)$ implies that $HZ = \Gamma$. In particular,
\begin{equation}\label{su10b4}
  H \geq [H,H]=[HZ,HZ] = [\Gamma,\Gamma] = \SU_n(q),
\end{equation}  
and there are some $h \in H$ and $j \in \Z/(q+1)\Z$ such that $\gamma = \det(h(\gamma^j \cdot I_n))$, i.e. $\det(h) = \gamma^{1-jn}$. As 
$n=n_0n'$, it follows that the order of $\gamma^{1-jn}$ in $\mu_{q+1}$ is divisible by $q'$. Thus $\det(H)$ has order divisible by 
both $q_0$ and $q'$, and so by \eqref{su10b3} we have $\det(H) = \mu_{q+1}=\det(\Gamma)$. But $H \geq \SU_n(q)=\Ker(\det)$ by \eqref{su10b4}, 
hence $H = \Gamma$, as stated.

\smallskip
(iii) Now we consider the surjection $\phi:\pi_1(\G_m/\overline{\F_p})\surj G$ underlying $\sH$ and 
recall that $G \cong \PGU_n(q) = \Gamma/Z$ by \eqref{su10b2}. Also, consider the surjection $\pi:\hat{G}=\Gamma/Z_0 \surj G$
with kernel $\Ker(\pi) \cong C_{q'}$. The obstruction to lifting $\phi$ to a homomorphism
$\varpi: \pi_1(\G_m/\overline{\F_p})\to \hat{G}$
lies in the group $H^2(\G_m/\overline{\F_p},\Ker(\pi))=0$, the vanishing because open curves have cohomological dimension $\le 1$, cf. \cite[Cor. 2.7, Exp. IX and Thm. 5.1, Exp. X]{SGA4t3}. Write $\varpi(\hat{G})=H/Z_0$
for some subgroup $H \leq \Gamma$ containing $Z_0$. Then 
$$\pi(H/Z_0) = (\pi \circ \varpi)\bigl( \pi_1(\G_m/\overline{\F_p}) \bigr) = \phi\bigl( \pi_1(\G_m/\overline{\F_p}) \bigr) = \Gamma/Z,$$
i.e. $H$ maps onto $\Gamma/Z$. By (ii), $H = \Gamma$, that is, $\varpi$ is surjective. 

\smallskip
(iv) Next, if $q_0=q+1$, equivalently $\gcd(n,q+1)=1$, then $\Gamma = S \times Z_0$, $\Gamma/Z_0 \cong G \cong \SU_n(q)$;
set $l:=1$ in this case and consider the Weil character $\zeta^l_n$ of $\hat{G}=\Gamma/Z_0$.  If $q_0 < q+1$, then we set $l:=q_0$ and 
consider the Weil character $\zeta^l_n$ of $\GU_n(q)$, which restricts irreducibly to $\SU_n(q)$ and 
factors through $\hat{G}=\Gamma/Z_0$, since its kernel is the subgroup $Z_0$ of $Z=\ZB(\Gamma)$. 
Thus in both cases we obtain an irreducible representation $\Psi:\hat{G} \to \GL(W)$, with 
$W = \overline{\Q_\ell}^{D+1}$, whose character is $\zeta^l_n$. We can also inflate $\Phi$ to a $\hat{G}$-representation $\hat{\Phi}$ with 
kernel $Z_0$. We will now compose $\varpi$ with
$$\hat\Phi: \hat{G} \to \GL(V) \mbox{ and }\Psi: \hat{G} \to \GL(W)$$ 
and compare their traces at any element $y \in \varpi(P(0)) \cup \varpi(P(\infty))$. Then $y$ is a $2$-element, and so by 
Lemma \ref{su-mod}(ii) we have 
$$\Trace(\hat\Phi(y))-\Trace(\Psi(y)) = -1.$$
Now we can apply Theorem \ref{Brauerp} to $\hat\Phi$ and $\Psi$ to conclude that $\Psi$ leads to a hypergeometric sheaf $\sH'$
of rank $D+1$ with geometric monodromy group $\Psi(\hat{G})$. In particular, the image $\langle g_0 \rangle$ of $I(0)$ in 
$\Psi(\hat{G})$ has simple spectrum on $\sH'$. But this contradicts Theorem \ref{ss-su}, since $\obar(g_0) = q^{n-1}-1$ but
$\rank(\sH') = (q^n+1)/(q+1)$.
\end{proof}

To handle other unitary cases, we will need some auxiliary statements.

\begin{lem}\label{su2-mod1}
Let $2|n \geq 4$, $q = p^f$ a power of a prime $p>2$, and let $H$ be a finite group with $p \nmid |\ZB(H)|$ and 
$H/\ZB(H) \cong \PGU_n(q)$. Let $P < H$ be the full inverse image in $H$ of a Siegel parabolic subgroup $\bar P$ {\rm [i.e. with Levi subgroup 
$\GL_{n/2}(q^2)/\ZB(\GU_n(q))$]} of $\PGU_n(q)$. Let $Q := \OB_p(P)$ and let $J = Q \rtimes C < P$ such that $Z:=\ZB(H) \leq C$ and 
$C/Z$ projects onto a maximal torus $C_{(q^n-1)/(q+1)}$ of $\PGU_n(q)$. Then there exists a linear character $\theta \in \Irr(Q)$ 
such that the following statements hold.

\begin{enumerate}[\rm(i)]
\item If $\xi \in \Irr(H)$ is any irreducible character of degree $D:=(q^n-1)/(q+1)$, then $\xi|_J$ is irreducible and 
there exists a linear character $\xi^* \in \Irr(Z)$ such that 
$$\xi|_Z = \xi(1)\cdot\xi^*,~~\xi|_J = \Ind^J_{QZ}(\theta \boxtimes \xi^*).$$
\item Moreover, if $\sigma$ is an automorphism of $H$ of $p$-power order that fixes $\xi \in \Irr(H)$ of degree $D$, then there exists a $\sigma$-invariant 
linear character $\tilde\xi \in \Irr(C)$ such that 
$$\tilde\xi|_Z = \xi^*,~~\xi|_J = \tilde\xi \cdot \Ind^J_{QZ}(\theta \boxtimes 1_Z).$$
if we inflate $\tilde\xi$ to a linear character of $J$.
\end{enumerate}
\end{lem}

\begin{proof}
(i) Let $\hat{H} := \GU_n(q) = \GU(V)$ for a Hermitian space 
$V = \F_{q^2}^n$, $\hat{Z} := \ZB(\hat{H})$, so that $H/Z \cong \hat{H}/\hat{Z}$. Now we can write 
$P/Z = \bar P = \hat{P}/\hat{Z}$, where $\hat{P} = \Stab_{\hat{H}}(U)$ for a totally singular subspace $\langle e_1, \ldots ,e_{n/2} \rangle$
of $V$. Then $\hat{Q} := \OB_p(\hat{P})$ is elementary abelian of order $q^{n(n+2)/8}$. Moreover, as shown in the proof of 
\cite[Lemma 12.5]{GMST}, $\hat{P}$ acts on $\Irr(\hat{Q})$ with exactly one orbit of length $1$, namely 
$\{1_{\hat{Q}}\}$, one orbit $\sO_1$ of length $D$, and all other orbits have length larger than $D$. Certainly, this action factors 
through $\hat{Z}$.

Let $\hat{C} \cong C_{q^n-1}$ be a maximal torus of $\hat{H}$ contained in $\hat{P}$. We may assume that $\hat{C} = \langle \hat{g} \rangle$, where $\hat{g}$ has simple spectrum 
$$\{\eps,\eps^{-q}, \ldots,\eps^{(-q)^{n-1}}\}$$
on $V \otimes \overline{\F_q}$ for a generator $\eps$ of $\mu_{q^n-1}=\F_{q^n}^\times$. Note that 
$\langle \hat{g}^D \rangle = \ZB(\hat{H})$ fixes every $\lambda \in \sO_1$. Furthermore, as shown in the proof of 
\cite[Lemma 12.5]{GMST}, if some power $\hat{g}^m$ fixes some character $\lambda \in \sO_1$, then the $p'$-element
$\hat{g}^m$ belongs to a subgroup $\GU_1(q) \times \GL_{n/2-1}(q^2)$ of a Levi subgroup $\GL_{n/2}(q^2)$ of $\hat{P}$. This implies
that $\hat{g}^m$ has an eigenvalue belonging to $\mu_{q+1} \subseteq \F_{q^2}^\times$, and the latter is possible only when 
$D|m$. We have shown that $\hat{C}$ acts transitively on $\sO_1$, with any point stabilizer equal to $\ZB(\hat{H})$.

Since $p \nmid |Z|$, the full inverse image of $\hat{Q}\hat{Z}/\hat{Z} \cong \hat{Q}$ in $H$ is precisely $Z \times Q$. Hence, without loss
of generality, we may identify $\hat{Q}$ with $Q$, and conclude that $P$ acts on $\Irr(Q)$ with exactly one orbit of length $1$, namely 
$\{1_Q\}$, one orbit $\sO_1$ of length $D$, and all other orbits have length larger than $D$; furthermore, $C$ acts transitively on $\sO_1$, with any point stabilizer equal to $Z$. Now, as $\xi \in \Irr(H)$ has degree $D$, it follows that 
$\xi|_Q = \sum_{\lambda \in \sO_1}\lambda$. 
Fixing $\theta \in \sO_1$, we then have by Clifford's theorem that $\xi|_J = \Ind^J_{QZ}(\theta \boxtimes \xi^*)$, since
the inertia group of $\theta$ in $J$ is precisely $Q \times Z$ and $\xi|_Z = \xi(1) \cdot \xi^*$ by Schur's lemma. In particular,
$\xi|_J$ is irreducible.

\smallskip
(ii) First note that $C$ is abelian, since $C/Z \cong C_{(q^n-1)/(q+1)}$ is cyclic. As $\sigma$ fixes $\xi$, it also fixes the central character 
$\xi^*$. Hence $\sigma$ acts on the set of $D=|C/Z|$ irreducible constituents of the character $\Ind^C_Z(\xi^*)$, as $C$ is abelian.
But $\ord(\sigma)$ is a $p$-power and $p \nmid D$. Therefore, $\sigma$ fixes some irreducible constituents $\tilde\xi$ of $\Ind^C_Z(\xi^*)$,
and we have that $\tilde\xi|_Z = \xi^*$. By Frobenius' reciprocity, we also have that 
$$\tilde\xi \cdot \Ind^J_{QZ}(\theta \boxtimes 1_Z) = \Ind^J_{QZ}\bigl((\tilde\xi)|_{QZ} \cdot (\theta \boxtimes 1_Z)\bigr) = 
    \Ind^J_{QZ}(\theta \boxtimes \xi^*) = \xi|_J.$$
%
\end{proof}

\begin{prop}\label{su2-mod2}
Let $q = p^f$ be a power of a prime $p$, $2|n \geq 4$, $(n,q) \neq (4,2)$, $(4,3)$, $(6,2)$, and let $G$ be a finite almost quasisimple 
group with a normal subgroup $H \geq \ZB(G)$ such that 
$H/\ZB(G) \cong \PGU_n(q)$ and $p \nmid |\ZB(G)|$. Then 
$L:=G^{(\infty)}$ is a quotient of $\SU_n(q)$ by a central subgroup. Suppose that $G$ has a faithful irreducible 
complex representation $\Phi:G \to \GL(V)$ of degree 
$$D := (q^n-1)/(q+1)$$ 
such that $\Phi|_L$ induces 
a Weil representation of $\SU_n(q)$ with character $\zeta^i_n$ for some $1 \leq i \leq q$ and $p \nmid |\det(\Phi(G))|$.
Then $G$ admits an irreducible complex representation $\Psi:G \to \GL(W)$ of degree $D+1$ such that
$$\Trace(\Psi(y))-\Trace(\Phi(y)) = 1$$
for all $p$-elements $y \in G$.
\end{prop}

\begin{proof}
(i) Note that $G$ has a unique non-abelian composition factor $S \cong \PSU_n(q)$, and 
\begin{equation}\label{su20}
  \PGU_n(q) \cong H/\ZB(G) \lhd G/\ZB(G) \leq \Aut(S) = \PGU_n(q) \rtimes C_{2f}
\end{equation}  
by hypothesis. In particular,  $L$ is a quasisimple cover of $S$, so $L$ is a quotient of $\SU_n(q)$ by the assumptions on 
$(n,q)$. Assume in addition that $p=2$. If $2|[G:H]$, then $G$ induces an involutive field automorphism, 
namely the transpose-inverse
automorphism, on $S$, which sends the character $\zeta^i_n$ of $\SU_n(q)$ to $\zeta^{q+1-i}_n \neq \zeta^i_n$, and this contradicts
the existence of $\Phi$. Hence $2 \nmid [G:H]$. Next, both $|\ZB(G)|$ and $|\PGU_n(q)/S| = \gcd(n,q+1)$ are odd, so $2\nmid [G:L]$ and 
$y \in L$ for all $2$-elements $y \in G$. Now recall that the Weil character $\zeta^0_n$ of $\SU_n(q)$ factors through $S$ and so 
can be inflated to a real-valued, $\Aut(S)$-invariant, irreducible character with trivial determinant (as $L$ is perfect) of $L$. By
\cite[Lemma 2.1]{NT}, the latter character extends to the character of some representation $\Psi:G \to \GL(\C^{D+1})$.  Now the 
relation $\Trace(\Psi(y))-\Trace(\Phi(y)) = 1$ follows from Lemma \ref{su-mod}(b).

\smallskip
(ii) From now on we will assume $p > 2$. 
Also let $\varphi$ denote the character of $\Phi$, and $\varphi^\circ$ denote its restriction to $2'$-elements of $G$ (and similarly 
for any character of $G$). Now, the Weil character 
$\zeta^0_n$ of $\SU_n(q)$, of {\it odd} degree $(q^n+q)/(q+1)=D+1$, factors through $S$ and yields  
a real-valued, $\Aut(S)$-invariant, irreducible character with trivial determinant (as $S$ is simple) of $S$. By
\cite[Theorem 2.3]{NT}, the latter character extends uniquely to a real character $\psi$ with trivial determinant, 
of some representation $\Psi:G \to \GL(\C^{D+1})$ that is trivial at $\ZB(G)$.  

Here we consider the case $i=(q+1)/2$.
By \cite[Theorem 7.2(ii)]{DT}, 
$\theta:=\varphi^\circ|_L$ is an irreducible $2$-Brauer character of $L$.
Next, by Lemma \ref{su-mod}(a), 
$\psi^\circ|_L=\theta+1_L$. As $L \lhd G$ and $D > 1$, it follows from Clifford's theorem that 
$\psi^\circ = \alpha+\beta$ is the sum of two irreducible Brauer characters, with $\alpha$ lying above $\theta$ and 
$\beta$ lying above $1_L$. Furthermore, as $\vartheta$ is real, we have that $\alpha$ and $\beta$ are both real-valued.
But $\beta(1)=1$, so in fact $\beta=1_G$.  We have shown that $\psi^\circ-1_G=\alpha$ and 
$\varphi^\circ$ are two extensions to $L$ of $\theta$. By \cite[Cor. (8.20)]{N}, there exists a linear character 
$\lambda$ of $G/L$ such that 
\begin{equation}\label{su21}
  \varphi^\circ = (\psi^\circ-1_G)\lambda.
\end{equation}  
Taking the complex conjugate and using $\psi=\overline\psi$, we obtain
\begin{equation}\label{su22}
\overline{\varphi^\circ} = \overline{(\psi^\circ-1_G)\lambda} =  (\psi^\circ-1_G)\overline\lambda = 
   \varphi^\circ\lambda^2.
\end{equation}
Now we consider any $p$-element $y \in G$ and let $Y:=\langle y \rangle$. Restricting \eqref{su22} to $Y$, we see
that $\Phi^*|_Y \cong \Phi|_Y \otimes \Lambda$ for some representation $\Lambda: Y \to \GL_1(\C)$ with character
$\lambda|_Y$. In particular, 
$$\det(\Phi(y))^{-1} = \det(\Phi^*(y)) = \det(\Phi(y))\lambda(y)^D.$$
As $2 < p \nmid |\det(\Phi(G))|$ and $p \nmid D$, it follows that $\lambda(y) \in \C^\times$ is a $p'$-root of unity. 
On the other hand, $\ord(\lambda(y))$ is a $p$-power, since $y$ is a $p$-element. Hence $\lambda(y)=1$,   
and so $\Trace(\Psi(y))-\Trace(\Phi(y)) = 1$ by \eqref{su21}.

\smallskip
(iii) In the rest of the proof, we consider the case $i \neq (q+1)/2$. Recalling \eqref{su20}, we can find a $p$-element 
$\bar\sigma \in G/\ZB(G)$, which is induced by a field automorphism of $\GU_n(q)$, such that $p$ does not divide the index of 
$\langle H/\ZB(G),\bar\sigma \rangle = H \rtimes \langle \bar\sigma \rangle$ in $G/\ZB(G)$. Using $p \nmid |\ZB(G)|$, we can 
find a lift $\sigma$ of $p$-power order of $\bar\sigma$ in $G$, and note that 
$G_1:= \langle H,\sigma \rangle = H \rtimes \langle \sigma \rangle$
satisfies $p \nmid [G:G_1]$; in particular, $G_1$ contains all $p$-elements of $G$. The field automorphism action of 
$\bar\sigma$ induces an action of $\sigma$ on $H_2:=\GU_n(q)/C_{(q+1)/2}$ (where $C_{(q+1)/2}$ has index $2$ in 
$\ZB(\GU_n(q))$ and leads to $G_2:=H_2 \rtimes \langle \sigma \rangle$. Also recall that the real-valued Weil character $\zeta^{(q+1)/2}_n$
of $\GU_n(q)$ factors through $H_2$ and is invariant under all field automorphisms of $\GU_n(q)$, in particular under $\sigma$.
By \cite[Lemma 2.1]{NT}, $\zeta^{(q+1)/2}_n$ has a unique real-valued extension $\xi$ to $G_2$, afforded by a representation 
$\Xi:G_2 \to \GL(\C^D)$. The analysis in (ii) applied to $(G_2,\Xi)$ then shows that
\begin{equation}\label{su23}
  \Trace(\Psi(y))-\Trace(\Xi(y)) = 1
\end{equation}
for all $p$-elements $y \in G_2$. (Note that $\Psi$ is trivial at $\ZB(G)$ and so can be viewed as defined on $G_2/\ZB(G_2)$
and then inflated to $G_2$.)

Next we can find a $\bar\sigma$-stable Siegel parabolic subgroup $\bar P$ of $H/\ZB(G)$ with unipotent radical $\bar Q$ and 
a $\bar\sigma$-stable maximal torus $\bar C \cong C_{(q^n-1)/(q+1)}$ in $\bar P$ (using the field automorphism action of $\bar\sigma$),
and embed $\langle \bar Q,\bar\sigma \rangle$ in a Sylow $p$-subgroup $\bar R \rtimes \langle \bar\sigma \rangle$ of $G/\ZB(G)$,
where $\bar R \in {\mathrm {Syl}}_p(H/\ZB(G))$. Let $J_1=Q_1 \rtimes C_1$ and 
$J_2=Q_2 \rtimes C_2$ be the full inverse images of $\bar Q \rtimes \bar C$ in $G_1$ and $G_2$, respectively, with 
$Q_i:= \OB_p(J_i) \cong \bar Q$ as $p$ is coprime to $|\ZB(G)|$ and $|\ZB(G_2)|$. Similarly, there exist a unique Sylow $p$-subgroup
$R_1 > Q_1$ of $H$ and a unique Sylow $p$-subgroup $R_2 > Q_2$ of $H_2$ that project isomorphically onto $\bar R$, and we 
may identify $R_2$ with $R_1$ and $Q_2$ with $Q_1$ (via some fixed isomorphism). Note that 
$R_1 < L=G^{(\infty)}$ and $R_2 < L_2:=G_2^{(\infty)}$. By hypothesis and by the construction of $(G_2,\Xi)$, $\Phi(L)$ and 
$\Xi(L_2)$ afford the Weil characters $\zeta^i_n$ and $\zeta^{(q+1)/2}_n$, both of degree $D$. By Lemma \ref{su-mod}(b), 
$\Trace(\Phi(y))= \Trace(\Xi(y))$ for all $y \in R_1$. Conjugating $\Xi$ suitably, we achieve that
\begin{equation}\label{su24}
  \Phi(y)=\Xi(y)
\end{equation}
for all $y \in R_1$.

\smallskip
(iv) Denote $Z_1:=\ZB(G) < C_1$ and $Z_2:=\ZB(G_2) < C_2$. Certainly, $\varphi|_H$ and $\xi|_{H_2}$ are both $\sigma$-invariant of
degree $D$. By Lemma \ref{su2-mod1}(ii), there exist $\theta \in \Irr(Q_1)$ and $\sigma$-invariant linear characters 
$\lambda_1 \in \Irr(J_1/Q_1)$ and $\lambda_2 \in \Irr(J_2/Q_2)$ such that
$$\varphi|_{J_1} = \lambda_1 \cdot \Ind^{J_1}_{Q_1Z_1}(\theta \boxtimes 1_{Z_1}),~~\xi|_{J_2} = 
     \lambda_2 \cdot \Ind^{J_2}_{Q_2Z_2}(\theta \boxtimes 1_{Z_2}).$$
Note that  $\Ind^{J_1}_{Q_1Z_1}(\theta \boxtimes 1_{Z_1})$ is trivial at $Z_1$ and $\sigma$-invariant, and similarly 
$\Ind^{J_2}_{Q_1Z_2}(\theta \boxtimes 1_{Z_2})$ is trivial at $Z_2$. So both of them can be viewed as the 
character of the same representation $\Theta$ of $J_1/Z_1 \cong J_2/Z_2 \cong \bar Q \rtimes \bar C$.
Let $\Lambda_i$ denote the one-dimensional representation of $J_i$ with character $\lambda_i$. Then 
$\Phi|_{J_1 \rtimes \langle \sigma \rangle}$ is an extension of $\Phi|_{J_1}=\Lambda_1 \otimes \Theta$ to 
$J_1 \rtimes \langle \sigma \rangle$, and 
$\Xi|_{J_2 \rtimes \langle \sigma \rangle}$ is an extension of $\Xi|_{J_2}=\Lambda_2 \otimes \Theta$ to 
$J_2 \rtimes \langle \sigma \rangle$ (with $\Theta$ being viewed as 
representations of $J_1$ and $J_2$, respectively). Thus, for all $x \in J_1$ we have, using $\sigma$-invariance of $\lambda_1$,
$$\lambda_1(x)\Theta(x^\sigma) = \Lambda_1(x^\sigma) \otimes \Theta(x^\sigma) = \Phi(x^\sigma) = \Phi(\sigma)\Phi(x)\Phi(\sigma)^{-1} = 
   \lambda_1(x)\Phi(\sigma)\Theta(x)\Phi(\sigma)^{-1},$$
and so $\Theta(x^\sigma) =  \Phi(\sigma)\Theta(x)\Phi(\sigma)^{-1}$ for all $x \in J_1/Z_1$. Similarly, 
$\Theta(x^\sigma) =  \Xi(\sigma)\Theta(x)\Xi(\sigma)^{-1}$ for all $x \in J_2/Z_2$. Since $\Phi|_{J_1}$ is irreducible by 
Lemma \ref{su2-mod1}(i), $\Theta$ is irreducible. Hence, the equality 
$\Phi(\sigma)\Theta(x)\Phi(\sigma)^{-1}= \Xi(\sigma)\Theta(x)\Xi(\sigma)^{-1}$ for all $x \in J_1/Z_1$ implies by Schur's lemma that 
\begin{equation}\label{su25}
  \Xi(\sigma) = \alpha\Phi(\sigma)
\end{equation}  
for some $\alpha \in \C^\times$. As $\sigma$ is a $p$-element, we see that $\ord(\alpha)$ is a $p$-power. On the other hand,
$\xi = \bar\xi$, so $\Xi$ is self-dual, whence $\det(\Xi(\sigma)) = \pm 1$. As $p \nmid |\det(\Phi(G))|$, we also have 
that $\ord(\Phi(\sigma))$ is coprime to $p$. Taking the determinant of \eqref{su25}, we now see that $\alpha^D$ has $p'$-order,
whence so does $\alpha$, since $p \nmid D$. Consequently, $\alpha = 1$. 

Now, \eqref{su24} and \eqref{su25} show that $\Phi(y) = \Xi(y)$ for all $y \in R_2 \cup \{\sigma\}$. It follows that
$\Phi(y)=\Xi(y)$ for all $y \in R_1 \rtimes \langle \sigma \rangle$, a Sylow $p$-subgroup of $G_2$. Hence,
$\Trace(\Phi(y)) = \Trace(\Xi(y))$ for all $p$-elements $y \in G_2$, and, together with \eqref{su23}, this implies that
$$\Trace(\Psi(y)) - \Trace(\Phi(y))=1$$ 
for all $p$-elements $y \in G$.
\end{proof}

\begin{thm}\label{su2-mod-main}
Let $q$ be a prime power, $2|n \geq 4$, and $(n,q) \neq (4,2)$, $(4,3)$, $(6,2)$. 
Then case {\rm (i)} of Theorem \ref{ss-su} does not lead to hypergeometric sheaves in dimension $(q^n-1)/(q+1)$. More precisely, there is no hypergeometric sheaf $\sH$ of type 
$(D,m)$ with $m < D=(q^n-1)/(q+1)$, with finite geometric monodromy group $G=G_\geo$ such that  
$G$ is almost quasisimple with $S = \PSU_n(q)$ as a non-abelian composition factor, and with the image of $I(0)$ being a cyclic 
group $\langle g_0 \rangle$ where $\obar(g_0) = D$.
\end{thm}

\begin{proof}
Assume the contrary that such a hypergeometric sheaf $\sH$ exists, and let $\Phi:G \to \GL(V)$ denote the corresponding 
representation, with $V = \overline{\Q_\ell}^D$, and with character $\varphi$. The assumptions on $(n,q)$ imply by Theorem
\ref{char-sheaf1} that the characteristic $p$ of $\sH$ divides $q$, that is, $q=p^f$ for some $f \in \Z_{\geq 1}$. The existence of 
$g_0$ implies by Theorem \ref{ss-su}(i) that $G$ contains a normal subgroup $H > \ZB(G)$ such that $H/\ZB(G) \cong \PGU_n(q)$.

\smallskip
Assume in addition that $D-m \geq 2$. Then both $\det(\Phi(G))$ and $\ZB(G)$ are $p'$-groups by Proposition \ref{p-center}(iv).
Now, Proposition \ref{su2-mod2} implies that $G$ admits an irreducible representation $\Psi:G \to \GL(\overline{\Q_\ell}^{D+1})$
with $\Trace(\Psi(y))-\Trace(\Phi(y))=1$ for all $p$-elements $y \in G$. Applying Theorem \ref{Brauerp} to $\Phi$ and $\Psi$, we conclude that $\Psi$ leads to a hypergeometric sheaf $\sH'$ of rank $D+1$ with geometric monodromy group $\Psi(G)$. In particular, the image 
$\langle g_0 \rangle$ of $I(0)$ in $\Psi(G)$ has simple spectrum on $\sH'$. But this is impossible, since 
$\obar(g_0) = D < \rank(\sH')$.

\smallskip
It remains to consider the case $\dim\Wild=D-m=1$. 
By Proposition 2.22 and Lemma 2.19 of \cite{GT3}, 
$|\varphi(g)|/\varphi(1) \leq (3.95)/4$ for all $1 \neq g \in Q$. It now follows from \eqref{bd11} that $D \leq 160$, whence
$(n,q) =  (4,4)$, $(4,5)$, $(8,2)$. However, if $2|q$ then p. (i) of the proof of Proposition \ref{su2-mod2} shows that 
$Q \leq \ZB(G)G^{(\infty)}$, and so we have $|\varphi(g)|/\varphi(1) \leq 0.95$ for all $1 \neq g \in Q$, whence $D \leq 40$, ruling out
two of these possible exceptions. The same arguments apply to the remaining exception $(n,q) = (4,5)$.
\end{proof}

\subsection*{9E. Extraspecial normalizers}
We will now show that case (i) of Theorem \ref{ss-extr} cannot lead to hypergeometric sheaves of rank $>9$. First we need the following 
technical result:

\begin{lem}\label{ext-eig}
Let $p > 2$ be a prime, and let $n \in \Z_{\geq 1}$ with $p^n \geq 11$. Let $E < \GL(V)$ be an irreducible extraspecial $p$-subgroup
of order $p^{1+2n}$ for $V:=\C^{p^n}$, and let $Q \leq \NB_{\GL(V)}(E)$ be a nontrivial $p$-subgroup. Then the following statements
hold for $W:=p^n-\dim \CB_V(Q)$.
\begin{enumerate}[\rm(i)]
\item $W \geq 7$.
\item If $|Q| \geq 9$, then $W > p^n/2$.
\end{enumerate}
\end{lem}

\begin{proof}
(a) It is well known, see e.g. \cite[Theorem 1]{Wi}, that 
\begin{equation}\label{eq:ext20}
  \NB_{\GL(V)}(E)/ZE \hookrightarrow \Sp(E/\ZB(E)) \cong \Sp_{2n}(p),~\CB_{\NB_{\GL(V)}(E)}(E/\ZB(E))=ZE
\end{equation}  
for $Z := \ZB(\GL(V))$. 
The statements are obvious in the case $Q \cap Z \neq 1$, so we may assume $Q \cap Z = 1$. Let 
$\varphi$ denote the character of $Q$ acting on $V$. Now, if $\varphi(x)=0$ for some $1 \neq x \in Q$ of order $p$,  then $\varphi|_X$ contains $1_X$ with multiplicity $p^{n-1}$ for $X:=\langle x\rangle$,
and so $W \geq p^{n-1}(p-1)$, yielding both statements. We also note that $\varphi(y)=0$ for all $y \in ZE \smallsetminus Z$ of order
$p$. Hence, arguing by contradiction, we may assume that 
\begin{equation}\label{eq:ext21}
  Q \cap ZE = 1,~W < p^n/2,~\varphi(x) \neq 0 \mbox{ for all }x \in Q \mbox{ of order }p.
\end{equation}

Consider any element $1 \neq g \in Q$ and let $\gbar$ denote its image in $\Sp(E/\ZB(E))$. We write 
$|\CB_{E/\ZB(E)}(g)|=p^{e(g)}$ for $e(g) \in \Z$. As $g \notin ZE$ by \eqref{eq:ext21}, the second part of \eqref{eq:ext20} implies
that $0 \leq e(g) \leq 2n-1$. Hence 
\begin{equation}\label{eq:ext22}
  |\varphi(g)| \leq p^{e(g)/2} \leq p^{n-1/2} 
\end{equation}
by \cite[Lemma 2.4]{GT1}.  
Applying \eqref{bd11}, \eqref{eq:ext22}, and using $p^n \geq 11$ but assuming
$p > 3$, we again obtain both (i) and (ii). We also obtain (i) when $p=3$, since $p^n \geq 27$ in this case. 

\smallskip
(b) It remains to prove (ii) for $p=3$, in which case we have $p^n \geq 27$ and may assume $|Q|=9$. 
First suppose that $Q$ contains some $g$ with $e(g) \leq 2n-3$. If $\ord(g)=3$, then as $|\varphi(g)| \leq 3^{n-3/2}$ by 
\eqref{eq:ext22}, we again have $W>p^n/2$ from \eqref{bd11}. If $\ord(g)=9$, then $Q$ contains $6$ elements 
$x$ with $e(x) \leq 2n-3$ and two more elements $y$ with $e(y) \leq 2n-1$. Using the bound \eqref{eq:ext22} for $x$ and 
$y$, we can see that $[\varphi|_Q,1_Q]/3^n < 0.37$, contradicting \eqref{eq:ext21}. Hence $e(g) \geq 2n-2$ for all $1 \neq g \in Q$.
Since $\gbar \in \Sp_{2n}(p)$, it follows that $g$ can have either one Jordan block of size $2$ (and so $\gbar$ is a transvection),
or two Jordan blocks of size $2$, and all other blocks of size $1$ while acting on $E/\ZB(E)$. In particular, $g^3$ centralizes
$E/\ZB(E)$, whence $g^3 \in Q \cap GE=1$ by \eqref{eq:ext20} and \eqref{eq:ext21}, and so $\exp(Q)=3$.

We have shown that $Q \cong C_3 \times C_3$ and $Q$ consists of $2a$ elements $x$ with $\bar{x}$ being transvections and 
$8-2a$ elements $y$ with $e(y)=2n-2$, where $0 \leq a \leq 4$. If $a \leq 1$, then \eqref{bd11} implies $W > 3^n/2$, contradicting \eqref{eq:ext21}.
Hence $Q$ contains at least $4$ elements $x$ with $\bar{x}$ being transvections. As $Q$ embeds in $\Sp_{2n}(3)$ by 
\eqref{eq:ext21} and $|Q|=9$, we have that 
$Q = \langle  g,h \rangle$ with $\gbar$ and $\bar{h}$ two distinct, commuting transvections in $\Sp_{2n}(3)$.
By \cite[Lemma 4.5]{GMST}, this pair $(\gbar,\bar{h})$ is unique in $\Sp_{2n}(3)$, up to conjugacy. Now we can readily check that
$a=2$.

Recall by \cite[Lemma 2.4]{GT1} and \eqref{eq:ext21} that $|\varphi(g)| = 3^{n-1/2}$; moreover, $g$ centralizes an extraspecial subgroup
$E_1$ of order $3^{2n-1}$ of $E$. The $E_1$-module $V$ is the sum of $3$ copies of a simple module of dimension $3^{n-1}$. On 
the other hand, $E_1$ preserves each of $g$-eigenspaces on $V$, and moreover the $1$-eigenspace has dimension $> 3^n/2$ 
by \eqref{eq:ext21}. Replacing $g$ by $g^{-1}$ if necessary, it follows that $g$ has eigenvalues $1$ with multiplicity $2 \cdot 3^{n-1}$ and 
$\zeta:=\zeta_3$ with multiplicity $3^{n-1}$ on $V$, whence 
\begin{equation}\label{eq:ext24}
  \varphi(g)+\varphi(g^{-1}) =  3^n,
\end{equation}
and the same holds for $h$.   

Next we look at $u:=gh$, for which we have $e(u)=2$ and $\ord(u)=3$. As $\varphi(u) \neq 0$ by \eqref{eq:ext21}, \cite[Lemma 2.4]{GT1}
implies that $u$ acts trivially on the inverse image of order $3^{2n-1}$ of $\CB_{E/\ZB(E)}(u)$ in $E$, and this contains an extraspecial subgroup
$E_2$ of order $3^{2n-3}$ of $E$. The $E_2$-module $V$ is the sum of $9$ copies of a simple module of dimension $3^{n-2}$. On 
the other hand, $E_2$ preserves each of $u$-eigenspaces on $V$. Hence, we may denote by 
$3^{n-2}b$, $3^{n-2}c$, and $3^{n-2}d$ the dimensions of the $u$-eigenspaces for eigenvalues $1$, $\zeta$, and $\bar\zeta$, respectively,
with $b,c,d \in \Z_{\geq 0}$ and $b+c+d=9$. We also have by \cite[Lemma 2.4]{GT1} that
$$9 = 3^{e(u)}/3^{2n-4} = |\varphi(u)|^2/3^{2n-4} = |b+c\zeta+d\bar\zeta|^2 = \bigl((b-c)^2+(c-d)^2+(d-b)^2\bigr)/2.$$
Note that $18=16+1+1=9+9+0$ are the only two ways to write $18$ as the sum of three squares. One can now readily check
that $\{b,c,d\} = \{5,2,2\}$ or $\{4,4,1\}$. But $b > 4$ by \eqref{eq:ext21}, so $(b,c,d) = (5,2,2)$, and 
\begin{equation}\label{eq:ext25}
  \varphi(gh)+\varphi((gh)^{-1}) =  2 \cdot 3^{n-1},
\end{equation}
and the same holds for $gh^{-1}$. Now using \eqref{eq:ext24} and \eqref{eq:ext25}, we can compute
$[\varphi_Q,1_Q]$ to be  $3^n \cdot(13/27)$, i.e. $W/3^n=14/27$, contradicting \eqref{eq:ext21}.
\end{proof}

\begin{thm}\label{ext-main}
Let $\sH$ be an irreducible hypergeometric sheaf of type $(D,m)$ in characteristic $p$ with
$D > m$, $D \geq 11$, such that its geometric monodromy group $G=G_\geo$ is a finite extraspecial normalizer
in some characteristic $r$. Then $p=r$, $D=p^n$ for some $n \in \Z_{\geq 1}$, and the following statements hold.
\begin{enumerate}[\rm(i)]
\item Suppose $p > 2$. Then $\sH$ is Kloosterman, in fact the sheaf
$\sK l(\Ch_{p^n+1} \smallsetminus \{\triv\})$ {\rm (}studied by Pink \cite{Pink} and Sawin \cite[p. 841]{KT1}{\rm )}.
\item Suppose $p=2$. 
Then $\ZB(G) \cong C_2$, and so in Lemma \ref{basic}(i)(c) we have that $R=E$ is a normal extraspecial $2$-group $2^{1+2n}_{\eps}$
of $G$ for some $\eps=\pm$.
\end{enumerate}
\end{thm}

\begin{proof}
Since $G$ is a finite extraspecial normalizer, $E \lhd G < \NB_{\GL_{r^n}(\C)}(E)$ for an irreducible extraspecial $r$-group 
$E < \GL_{r^n}(\C)$ of order $r^{1+2n}$, and $D=r^n$. By Theorem \ref{char-sheaf2}, $p=r$. If furthermore $p=2$, then, since
$\ZB(E) \leq \ZB(R) \leq \ZB(G)$ in Lemma \ref{basic}(c),  (ii) follows from Proposition \ref{p-center}(v).

From now on, assume $p >2$, and 
let $Q \neq 1$ be the image in $G$ of $P(\infty)$. By Lemma \ref{ext-eig}(i), $W := \dim\Wild \geq 7$. 
Now, if $|Q| < 9$, then the $p$-group $Q$ has order $p \leq 7$, whence $Q$ affords at most $p-1$ distinct, nontrivial, irreducible characters on 
$\Wild$, which are all linear, and so $W \leq p-1 \leq 6$ by Proposition \ref{Pinftyimage}, a contradiction. 
Hence $|Q| \geq 9$, and so $W > D/2$ by Lemma \ref{ext-eig}(ii).   

Now, \eqref{eq:ext20} implies that $G/(G \cap ZE)$ embeds in $\Sp_{2n}(p)$, and so admits a complex representation $\Lambda$ of degree 
$(p^n-1)/2 < W$, with kernel $K$ of of order at most $2$. Applying Theorem \ref{bound1g} to $\Gamma:=G$ and $\Lambda$, we conclude
that $\Lambda(G/(G \cap ZE))$ is a finite cyclic $p'$-group. As $|K| \leq 2$ and $p > 2$, it follows that $G/(G \cap ZE)$ is an abelian 
$p'$-group.
Also, note that $G \cap ZE = \ZB(G)E$ since $G \geq E$.
Now, applying Theorem \ref{ss-extr} to a generator 
$g_0$ of the image of $I(0)$ in $G$, we see that the coset $g_0\ZB(G)E$ generates a cyclic, self-centralizing, maximal torus $C_{p^n+1}$ of 
$\Sp_{2n}(p)$. It follows that $G = \ZB(G)E\langle g_0\rangle$. In fact, since $G$ normalizes $E\langle g_0\rangle$ and
$E\langle g_0 \rangle$ contains $g_0$, 
by Theorem \ref{generation} we have $G = E\langle g_0 \rangle \cong E \rtimes C_{p^n+1}$. 

As $g_0$ is a generator of $I(0)$, the ``upstairs'' characters of $\sH$ are determined by the spectrum of $g_0$ on $\sH$,
which consists of all nontrivial $(p^n+1)^{\mathrm {th}}$ roots of unity, hence they are just $\Ch_{p^n+1} \smallsetminus \{\triv\}$.
Suppose $\sH$ is not Kloosterman, and we look at the image $Q \rtimes \langle g_\infty \rangle$ of $I(\infty)$ in $G$
for some $p'$-element $g_{\infty}$. 
By Hall's theorem applied to the solvable group $G$, $\langle g_{\infty} \rangle$ is contained in a conjugate of the Hall subgroup
$\langle g_0\rangle$. In particular, the spectrum of $g_\infty$ on $\sH$ is the spectrum of some power $g_0^i$ on $\sH$. But $\sH$ is 
irreducible, so the ``downstairs'' characters of $\sH$, which are determined by the spectrum of $g_{\infty}$ on the tame part $\Tame$, must be
disjoint from the ``upstairs'', whence $m=\dim\Tame=1$ and the single ``downstairs'' character is $\triv$. Thus $\sH$ is the sheaf 
$\sH_1$ studied in \cite[Corollary 8.2]{KT5}, which, however, was shown therein to have $G_\geo\cong \PGL_2(p^n)$, a contradiction.
\end{proof}

Note that the case $(p,\ZB(G))=(2,C_2)$ in Theorem \ref{ext-main}(ii) can lead to non-Kloosterman sheaves 
-- indeed in \cite{KT8} \edit{we constructed 
hypergeometric sheaves with $G_\geo = 2^{1+2n}_- \cdot \Omega^-_{2n}(2)$ for any $n \geq 4$}.
 
\section{Converse theorems} 
Let us recall that, in Theorems \ref{alt}, \ref{spor}, and \ref{simple} we have classified all pairs $(G,V)$, where $G$ is a finite almost quasisimple group and $V$ a faithful irreducible $\C G$-representation of $G$ in which  
some element $g$ has simple spectrum. Next, in Theorem \ref{char-sheaf1} we show that if such a group $G$ occurs as $G_{\geo}$
for a hypergeometric sheaf in characteristic $p$ and in addition $G$ is a finite group of Lie type in characteristic $r$, then $p=r$ unless 
$\dim(V) \leq 22$. Theorem \ref{char-sheaf2} gives an analogous result in the case $G$ is an extraspecial normalizer. The classification of triples $(G,V,g)$ that satisfy the Abhyankar condition at $p$, with $G$ being almost quasisimple or an extraspecial normalizer, is completed in  
Theorems \ref{ss-sl}, \ref{ss-sp}, \ref{ss-su}, and \ref{ss-extr}. 

However, as shown in \S\ref{sec:hypergeom2}, not all of these almost quasisimple triples 
$(G,V,g)$ can give rise to hypergeometric sheaves. The cases that can possibly occur hypergeometrically are the following:

\begin{enumerate}[\rm(a)]
\item $G/\ZB(G) \cong \ABS_n$ or $\mathsf{S}_n$ with $n \geq 5$, and $g$ is as described in Theorem \ref{alt}(i);

\item $G$ is a quotient of $\GL_n(q)$ in a Weil representation $V$ of degree $(q^n-q)/(q-1)$ or $(q^n-q)/(q-1)$, $q=p^f$, and $\obar(g)=(q^n-1)/(q-1)$
with $n \geq 3$ (cf. Theorem \ref{ss-sl}), or $G$ is a quotient of $\GL_2(q)$ or $\GU_2(q)$;

\item $G$ is a quotient of $\Sp_{2n}(q)$ in a Weil representation $V$ of degree $(q^n \pm 1)/2$ with $n \geq 2$, $2 \nmid q=p^f$, and 
$g$ is of type $(\alpha)$ or type $(\beta)$ as described in Theorem \ref{ss-sp}(i);

\item $G$ is a quotient of $\GU_{n}(q)$ in a Weil representation $V$ of degree $(q^n-q)/(q+1)$ or $(q^n+1)/(q+1)$ with $2 \nmid n \geq 3$, $q=p^f$, and $\obar(g)=(q^n+1)/(q+1)$, cf. Theorem \ref{ss-su}(i); 

\item $G$ is a quotient of $\GU_{n}(q)$ in a Weil representation $V$ of degree $(q^n+q)/(q+1)$ or $(q^n-1)/(q+1)$ with $2|n \geq 4$, $q=p^f$, and $g$ is is as described in Theorem \ref{ss-su}(iii);

\item A finite and explicit list of ``non-generic'' cases, including sporadic groups, as listed in Table 1.
\end{enumerate}
The cases (a)--(e) are indeed shown to occur. 
Namely, the respective hypergeometric sheaves $\sH$ (in characteristic $p$ in (b)--(e)) 
are explicitly constructed in Theorem \ref{alt2} for case (a), in \cite{KT5} and Theorem \ref{gl2} for case (b), 
in \cite{KT6} for type $(\alpha)$ in (c) and for (d)
with $2 \nmid q$, in \cite{KT7} for type $(\beta)$ in (c) and for (e), and in 
\cite{KT8} for case (d) with $2|q$. The extraspecial normalizers,
and the sporadic and non-generic cases in (f), are handled in \cite{KT8} and  
\cite{Ka-RL-2J2}, \cite{Ka-RL-T-Co3}--\cite{Ka-RL-T-Spor}.


\begin{thebibliography}{ABCD}
\bibitem[Abh]{Ab}
Abhyankar, S., Coverings of algebraic curves, {\it Amer. J. Math.} {\bf 79} (1957), 825--856.
\medskip
\bibitem[Asch]{Asch} Aschbacher, M., Maximal subgroups of classical groups, 
  On the maximal subgroups of the finite classical groups,
{\it Invent. Math.} {\bf 76} (1984),  469--514. 
\medskip
\bibitem[BNRT]{BNRT} Bannai E., Navarro, G., Rizo, N., and Tiep, P. H., Unitary $t$-groups, {\it J. Math. Soc. Japan} (to appear).
\medskip
\bibitem[Atlas]{ATLAS} Conway, J. H.,  Curtis, R. T.,  Norton, S. P.,  Parker, R. A. and 
Wilson, R. A.,  Atlas of finite groups. Maximal subgroups and ordinary characters for simple groups. 
Oxford University Press, Eynsham, 1985.
\medskip
\bibitem[Bo]{Bo}
  Bochert, A., \"Uber die Classe der transitiven Substitutionengruppen, {\it Math. Ann.} {\bf 40} (1892), 192--199.
\medskip
\bibitem[BHR]{BHR}
 Bray, J. N., Holt, D. F., and Roney-Dougal, C. M., The maximal subgroups of the low-dimensional finite classical groups. 
London Mathematical Society Lecture Note Series, {\bf 407}, 2013.
\medskip
\bibitem[C]{C}
Carter, R. W., Finite groups of Lie type: conjugacy classes
and complex characters, Wiley, Chichester, 1985.
\medskip
\bibitem[De-Weil II]{De-Weil II}Deligne, P., La conjecture de Weil II.
Publ. Math. IHES 52 (1981), 313-428.
\medskip
\bibitem[DM]{DM}
  Digne, F. and Michel, J.,  Representations of Finite Groups of
Lie Type, London Mathematical Society Student Texts 21, 
Cambridge University Press, 1991.
\medskip
\bibitem[DMi]{DMi}
  Deriziotis, D. I., and Michler, G. O., Character table and blocks of finite simple triality groups $\tw3 D_4(q)$, {\it Trans. Amer. Math. Soc.} 
{\bf 303} (1987), 39--70.
\medskip
\bibitem[Do]{Do}
  Dornhoff, L., Group representation theory, Dekker, New York, 1971.
\medskip
\bibitem[DT]{DT}
  Dummigan, N., and Tiep, P. H., Lower bounds for the minima of certain symplectic and unitary group lattices, {\it Amer. J. Math.} 
{\bf 121} (1999), 889--918.
\medskip
\bibitem[EY]{EY} 
  Enomoto, H., and Yamada, H., The characters of
$G_{2}(2^{n})$, {\it Japan. J. Math.} {\bf 12} (1986), 325--377.
\medskip
\bibitem[GAP]{GAP}
The GAP group, {\sf GAP} - groups, algorithms, and
programming, Version 4.4, 
2004, \url{http://www.gap-system.org}.
\medskip
\bibitem[GLS]{GLS}
Gorenstein, D., Lyons, R., and Solomon, R.M.,
The Classification of the Finite Simple Groups, Number 3.
Part {I}. {C}hapter {A}, {\bf 40}, Mathematical Surveys and
Monographs, Amer. Math. Soc., Providence, RI, 1998.
\medskip
\bibitem[Gri]{Gri}
  Griess, R.L., Automorphisms of extra special groups and nonvanishing degree $2$ cohomology, {\it Pacif. J. Math.} {\bf 48} (1973), 
403--422.  
\medskip
\bibitem[Gr]{Gross}Gross, B. H., Rigid local systems on $\G_m$ with finite monodromy, {\it  Adv. Math.} {\bf 224} (2010), 2531--2543.
\medskip
\bibitem[GMPS]{GMPS}
  Guest, S., Morris J., Praeger, C. E., and Spiga, P., On the maximum orders of elements of finite almost simple groups and primitive permutation groups,
{\it Trans. Amer. Math. Soc.} {\bf 367} (2015), 7665--7694.
\medskip
\bibitem[GKT]{GKT}
  Guralnick, R. M., Katz, N., and Tiep, P. H., Rigid local systems and alternating groups, {\it Tunisian J. Math.} {\bf 1} (2019), 295--320.
\medskip
\bibitem[GM]{GM}
  Guralnick, R. M., and Magaard, K., On the minimal degree of a primitive permutation group, {\it J. Algebra} {\bf 207} (1998), 127--145.
\medskip
\bibitem[GMST]{GMST}
Guralnick, R. M., Magaard, K., Saxl, J., and Tiep, P. H., Cross characteristic
representations of symplectic groups and unitary groups, {\it J. Algebra} {\bf 257} (2002), 
291--347.
\medskip
\bibitem[GMT]{GMT}
Guralnick, R. M., Magaard, K., and Tiep, P. H., Symmetric and alternating powers of Weil representations of finite symplectic groups, 
{\it Bull. Inst. Math. Acad. Sinica} {\bf 13} (2018), 443--461.
\medskip
\bibitem[GPPS]{GPPS}
  Guralnick, R. M., Penttila, T., Praeger, C., Saxl, J., Linear groups with orders
having certain large prime divisors, {\it Proc. London Math. Soc.} {\bf 78} (1999), 167--214.
\medskip
\bibitem[GT1]{GT1} Guralnick, R. M. and Tiep, P. H., Cross characteristic representations of even characteristic symplectic groups, 
{\it Trans. Amer. Math. Soc.} {\bf 356} (2004), 4969--5023.
\medskip
\bibitem[GT2]{G-T} Guralnick, R. M. and Tiep, P. H., Symmetric powers and a conjecture of Koll\'ar and Larsen, {\it Invent. Math.} {\bf 174} (2008), 505--554.
\medskip
\bibitem[GT3]{GT3} Guralnick, R. M. and Tiep, P. H., A problem of Koll\'ar and Larsen on finite linear groups and crepant resolutions, 
{\it J. Europ. Math. Soc.} {\bf 14} (2012), 605--657.
\medskip
\bibitem[Har]{Har}
  Harbater, D., Abhyankar's conjecture on Galois groups over curves, {\it Invent. Math.} {\bf 117} (1994), 1--25.
\medskip
\bibitem[Hum]{Hum}Humphreys, J., Linear Algebraic Groups, Graduate Texts in Mathematics, {\bf 21}, Springer-Verlag, New York-Heidelberg, 1975, xiv+247 pp.
\medskip
\bibitem[Hup]{Hu} Huppert, B., Singer-Zyklen in klassischen Gruppen, {\it Math. Z.} {\bf 117} (1970), 141--150.
\medskip
\bibitem[Is]{Is}
 Isaacs, I. M., Character Theory of Finite Groups, AMS-Chelsea, Providence, 2006.
\medskip
\bibitem[Jo]{Jo}
  Jones, G., Primitive permutation groups containing a cycle, {\it Bull. Aust. Math. Soc.} {\bf 89} (2014), 159--165.
\medskip
\bibitem[J]{J}
  Jordan, C.,  Th\'eor\`emes sur les groupes primitifs, {\it J. Math. Pures Appl.} {\bf 16} (1871), 383--408.
\medskip
\bibitem[KSe]{KSe}
  Kantor, W. M., and Seress, A.,  Large element orders and the characteristic of Lie-type
simple groups, {\it J. Algebra} {\bf 322} (2009), 802--832.
\medskip  
\bibitem[Ka-CC]{Ka-CC}Katz, N.. From Clausen to Carlitz: low-dimensional spin groups and identities among character sums, 
{\it Mosc. Math. J.} {\bf 9} (2009), 57--89.
\medskip
\bibitem[Ka-ESDE]{Ka-ESDE}Katz, N., Exponential sums and 
differential equations. Annals of Mathematics Studies, 124. Princeton Univ. Press, Princeton, NJ, 1990. xii+430 pp.
\medskip
\bibitem[Ka-GKM]{Ka-GKM}Katz, N., Gauss sums, Kloosterman sums, and monodromy groups, Annals of Mathematics Studies, {\bf 116}. 
Princeton Univ. Press, Princeton, NJ, 1988. ix+246 pp.
\medskip
\bibitem[Ka-LGE]{Ka-LGE}Katz, N., Local-to-global extensions of representations of fundamental groups, Annales de L'Institut Fourier, tome 36, $n^o$ 4(1986), 59-106.
\medskip
\bibitem[Ka-MG]{Ka-MG}Katz, N., On the monodromy groups attached to certain families of exponential sums,  Duke Math. J. 54 (1987), no. 1, 41--56.
\medskip
\bibitem[KRL]{Ka-RL-2J2}Katz, N., and Rojas-Le\'{o}n, A., A rigid local system with monodromy group $2.J_2$, {\it  Finite Fields Appl.}
{\bf 57} (2019), 276--286.
\medskip
\bibitem[KRLT1]{Ka-RL-T-Co3}Katz, N., Rojas-Le\'{o}n, A., and Tiep, P.H., Rigid local systems with monodromy group the Conway group ${\mathsf {Co}}_3$, {\it J. Number Theory} {\bf 206} (2020), 1--23.
\medskip
\bibitem[KRLT2]{Ka-RL-T-Co2}Katz, N., Rojas-Le\'{o}n, A., and Tiep, P.H., Rigid local systems with monodromy group the Conway group ${\mathsf {Co}}_2$,  {\it Int. J. Number Theory} (to appear), arXiv:1811.05712.
\medskip
\bibitem[KRLT3]{Ka-RL-T-2Co1}Katz, N., Rojas-Le\'{o}n, A., and Tiep, P.H., A rigid local system with monodromy group the big Conway group 
$2.{\mathsf {Co}}_1$ and two others  with monodromy group the Suzuki group $6.{\sf{Suz}}$, {\it Trans. Amer. Math. Soc.} {\bf 373} (2020), 
2007--2044.
\medskip
\bibitem[KRLT4]{Ka-RL-T-Spor}Katz, N., Rojas-Le\'{o}n, A., and Tiep, P.H., Rigid local systems for some sporadic simple groups,
(in preparation).
\medskip
\bibitem[KT1]{KT1}
Katz, N., with an Appendix by Tiep, P.H., Rigid local systems on $\A^1$ with finite monodromy, {\it Mathematika} {\bf 64} (2018), 785--846.
\medskip
\bibitem[KT2]{KT2}
Katz, N., and Tiep, P.H., Rigid local systems and finite symplectic groups, {\it Finite Fields Appl.} {\bf 59} (2019), 134--174.
\medskip
\bibitem[KT3]{KT3}
Katz, N., and Tiep, P.H., Local systems and finite unitary and symplectic groups, {\it Advances in Math.} {\bf 358} (2019), 106859, 37 pp.
\medskip
\bibitem[KT4]{KT4}
Katz, N., and Tiep, P.H., Moments of Weil representations of finite special unitary groups, {\it J. Algebra} (to appear).
\medskip
\bibitem[KT5]{KT5}
Katz, N., and Tiep, P.H., Rigid local systems and finite general linear groups, (submitted).
\medskip
\bibitem[KT6]{KT6}
Katz, N., and Tiep, P.H., Exponential sums and total Weil representations of finite symplectic and unitary groups, (preprint).
\medskip
\bibitem[KT7]{KT7}
Katz, N., and Tiep, P.H., Hypergeometric sheaves and finite symplectic and unitary groups, (in preparation).
\medskip
\bibitem[KT8]{KT8}
Katz, N., and Tiep, P.H., Local systems, extraspecial groups, and finite unitary groups in characteristic $2$, (in preparation).
\medskip
\bibitem[KlL]{KlL}
  Kleidman, P. B., and Liebeck, M. W., {The subgroup structure of the
finite classical groups}, London Math. Soc. Lecture Note Ser. no.
{\bf 129}, Cambridge University Press, 1990.
\medskip
\bibitem[KlT]{KlT}
  Kleshchev, A. S., and Tiep, P. H., On restrictions of modular spin representations of symmetric and alternating groups, {\it Trans. Amer. Math. Soc.} 
{\bf 356} (2004), 1971--1999.
\medskip
\bibitem[LMT]{LMT}
  Larsen, M., Malle, G., and Tiep, P. H., The largest irreducible representations of simple groups, {\it Proc. Lond. Math. Soc.} {\bf 106} (2013), 
65--96.
\medskip
\bibitem[LBST]{LBST}
  Liebeck, M. W., O'Brien, E., Shalev, A., and Tiep, P. H.,  The Ore conjecture, {\it J. Eur. Math. Soc.} {\bf 12} (2010), 939--1008.
\medskip
\bibitem[LPS]{LPS} Liebeck, M. W., Praeger, C. E., and Saxl, J., On the O'Nan-Scott reduction theorem for finite primitive permutation groups, 
{\it J. Austral. Math. Soc.} {\bf 44} (1988), 389--396.  
\medskip
\bibitem[Man]{Man}
  Manning, W. A., The degree and class of multiply transitive groups, {\it Trans. Amer. Math. Soc.} {\bf 35} (1933), 585--599.
\medskip
\bibitem[Mas]{Mas}
  Massias, J.-P., Majoration explicite de l'ordre maximum d'un \'el\'ement du groupe sym\'etrique, 
{\it Ann. Fac. Sci. Toulouse Math.} {\bf 6} (1984), 269--281. 
\medskip
\bibitem[N]{N}
  Navarro, G., Characters and blocks of finite groups,
Cambridge University Press, 1998.
\medskip
\bibitem[NT]{NT}
  Navarro, G., and Tiep, P. H., Rational irreducible characters and rational conjugacy classes in finite groups, 
{\it Trans. Amer. Math. Soc.} {\bf 360} (2008), 2443--2465.
\medskip
\bibitem[Pink]{Pink}Pink, R., Lectures at Princeton University, May 1986.
\medskip
\bibitem[Pop]{Pop}
Pop, F., \'Etale Galois covers of affine smooth curves. The geometric case of a conjecture of Shafarevich. On Abhyankar's conjecture,
{\it Invent. Math.} {\bf 120} (1995), 555--578. 
\medskip
\bibitem[Ra]{Ra} Rasala, R., On the minimal degrees of characters of $S_n$, {\it J. Algebra} {\bf 45} (1977), 132--181.
\medskip
\bibitem[Ray]{Ray}Raynaud, M. Rev\^{e}tements de la droite affine en caract\'{e}ristique $p>0$  et conjecture d'Abhyankar, 
{\it Invent. Math.} {\bf 116} (1994), 425--462. 
\medskip
\bibitem[R-L]{R-L}Rojas-Le\'{o}n, A., Finite monodromy of some families of exponential sums, {\it J. Number. Theory} {\bf 197} (2019), 37--48.
\medskip
\bibitem[Sa]{Sa}
  Sawin, W., \edit{Private communication}.
\medskip
\bibitem[SGA4t3]{SGA4t3}A. Grothendieck et al, S\'{e}minaire de G\'{e}ometrie Alg\'{e}brique du Bois-Marie, SGA 4, Part III,
Springer Lecture Notes in Math. 305, Springer, 1973.
\medskip
\bibitem[Slo]{Slo}
  Sloane, N. J. A., The on-line encyclopedia of integer sequences, {\tt https://oeis.org/}.   
\medskip
\bibitem[Such]{Such}\v{S}uch, O., Monodromy of Airy and Kloosterman sheaves, {\it Duke Math. J.} {\bf 103} (2000), 397--444. 
\medskip
\bibitem[T1]{T1}
Tiep, P. H., Low dimensional representations of finite quasisimple groups, in: ``Groups, 
Geometries, and Combinatorics'', Durham, July 2001, A. A. Ivanov et al eds., World Scientific, 2003, N. J. 277--294.  
\medskip
\bibitem[T2]{T2}
  Tiep, P. H., Weil representations of finite general linear groups and finite special linear groups, {\it Pacific J. Math.} {\bf 279} (2015), 
481--498.
\medskip
\bibitem[TZ1]{TZ1}
Tiep, P. H. and Zalesskii, A. E., Minimal characters of the finite classical
groups, {\it Comm. Algebra} {\bf 24} (1996), 2093--2167.
\medskip
\bibitem[TZ2]{TZ2}
Tiep, P. H. and Zalesskii, A. E., Some characterizations of the Weil
representations of the symplectic and unitary groups, {\it J. Algebra} {\bf 192}
(1997), 130--165.
\medskip
\bibitem[vdG-vdV]{vdG-vdV}van der Geer, G., van der Vlugt, M., Reed-Muller codes and supersingular curves. I, {\it  Compos. Math.} 
{\bf 84} (1992), 333--367.
\medskip
\bibitem[Wi]{Wi}
  Winter, D., The automorphism group of an extraspecial $p$-group, {\it Rocky Mount. J. Math.} {\bf 2} (1972), 159--168.
\medskip
\bibitem[Zs]{Zs}
Zsigmondy, K., Zur Theorie der Potenzreste, {\it Monatsh. Math. Phys.} {\bf 3} 
(1892), 265--284.
\end{thebibliography}
\end{document}